\documentclass[11pt]{article}
\usepackage[kerning, tracking, spacing]{microtype}
\microtypecontext{spacing=nonfrench}
\usepackage[margin=1in]{geometry}
\usepackage{amsmath, amsthm}
\usepackage{amssymb}
\usepackage{mathrsfs}
\usepackage{enumitem}
\usepackage{graphicx}
\usepackage{mathdots}
\usepackage{color}
\usepackage{hyperref,url}
\usepackage{tikz, tikz-cd}
\usetikzlibrary{matrix}
\usetikzlibrary{shapes}
\usetikzlibrary{arrows,decorations.markings, tikzmark}
\usepackage{mathtools}
\usepackage{ stmaryrd }
\usepackage[normalem]{ulem}
\usepackage[utf8]{inputenc}
\usepackage{xcolor}
\hypersetup{
	colorlinks,
	linkcolor={red!50!black},
	citecolor={blue!50!black},
	urlcolor={blue!80!black}
}

\pgfarrowsdeclare{bad to}{bad to}
{
	\pgfarrowsleftextend{-2\pgflinewidth}
	\pgfarrowsrightextend{\pgflinewidth}
}
{
	\pgfsetlinewidth{0.8\pgflinewidth}
	\pgfsetdash{}{0pt}
	\pgfsetroundcap
	\pgfsetroundjoin
	\pgfpathmoveto{\pgfpoint{-3\pgflinewidth}{4\pgflinewidth}}
	\pgfpathcurveto
	{\pgfpoint{-2.75\pgflinewidth}{2.5\pgflinewidth}}
	{\pgfpoint{0pt}{0.25\pgflinewidth}}
	{\pgfpoint{0.75\pgflinewidth}{0pt}}
	\pgfpathcurveto
	{\pgfpoint{0pt}{-0.25\pgflinewidth}}
	{\pgfpoint{-2.75\pgflinewidth}{-2.5\pgflinewidth}}
	{\pgfpoint{-3\pgflinewidth}{-4\pgflinewidth}}
	\pgfusepathqstroke
}

\DeclareMathAlphabet{\mymathbb}{U}{BOONDOX-ds}{m}{n}

\def\Mbar{\overline{\mathcal{M}}}
\def\M{\mathfrak{M}}
\def\Mct{\mathcal{M}^{\,\mathrm{ct}}}

\def\Msf{\mathsf{M}}
\def\C{\mathfrak{C}}
\def\CH{\mathsf{CH}}
\def\CHop{\mathsf{CH}_{\mathsf{op}}}
\def\logCH{\mathsf{logCH}}
\def\logR{\mathsf{logR}}

\def\Ms{\mathsf{M}}
\def\R{\mathsf{R}}
\def\st{\mathrm{st}}
\def\aut{\mathrm{Aut}}

\def\CO{\mathcal{O}}
\def\A{\mathcal{A}}

\def\F{\mathcal{F}}

\def\L{\mathcal{L}}
\def\P{\mathcal{P}}
\def\X{\mathcal{X}}
\def\id{\mathrm{id}}
\def\AA{\mathbb{A}}

\def\CC{\mathbb{C}}
\def\DD{\mathbb{D}}
\def\EE{\mathbb{E}}
\def\FF{\mathbb{F}}

\def\NN{\mathbb{N}}
\def\PP{\mathbb{P}}
\def\QQ{\mathbb{Q}}

\def\ZZ{\mathbb{Z}}

\def\lpic{\mathrm{LogPic}}
\def\pbar{\overline{\mathcal{P}}}
\def\pic{\mathfrak{Pic}}

\def\pth{\mathcal{P}^\theta}

\def\logP{\mathsf{LogPic}}

\def\Rels{\mathcal{R}}
\def\aj{\mathsf{aj}}

\def\hom{\mathcal{H}om}
\def\Hom{\mathrm{Hom}}

\def\Gm{\mathbb{G}_m}
\def\Glog{\mathbb{G}_{\log}}
\def\sR{\mathsf{R}}

\def\Tq{\mathsf{P}}

\def\im{\mathrm{im}~}
\def\codim{\mathrm{codim}}
\def\Hom{\mathrm{Hom}}
\def\vir{\mathrm{vir}}

\def\sm{\mathrm{sm}}

\def\aut{\mathrm{Aut}}
\def\ch{\mathrm{ch}}

\def\tC{\widetilde{C}}
\def\dx{\frac{d}{dx}}

\def\st{\textup{st}}

\def\bp{\textbf{p}}
\def\bq{\textbf{q}}

\def\Ssig{\mathscr{S}(\sigma)}

\def\im{\mathrm{Im}\,}
\def\kd{\mathrm{kd}}
\def\sfa{\mathsf{v}}
\def\sfa{\mathsf{a}}

\def\Jbar{\overline{J}}
\def\Jzero{J^{\underline{0}}}

\def\Td{\mathrm{td}}
\def\fm{\mathfrak{F}}
\def\dbcoh{\mathrm{D}^b_{\mathrm{coh}}}
\def\Th{\Theta}
\def\Pbar{\overline{\mathcal{P}}}
\def\Mint{\overline{\mathcal{M}}^{\, \mathrm{tl}}}
\def\DR{\mathsf{DR}}
\def\DRP{\mathsf{DRP}}
\def\DRD{\mathsf{DRD}}
\def\uniDR{\mathsf{uniDR}}
\def\uniDRP{\mathsf{uniDRP}}
\def\uniDRD{\mathsf{uniDRD}}

\def\cR{\mathcal{R}}
\def\pure{\mathrm{pure}}

\def\qs{\mathrm{qs}}
\def\ajp{\mathsf{AJ}}
\def\m{\mathfrak{m}}
\def\p{\mathfrak{p}}

\def\dual{\omega^\bullet}
\def\Hom{\mathcal{H}om}
\def\gp{\mathrm{gp}}
\def\T{\mathbb{T}}
\def\sfm{\mathsf{m}}
\def\tl{\,\mathrm{tl}}
\def\Abar{\overline{\mathcal{A}}}
\def\Xbar{\overline{\mathcal{X}}}
\def\Abarone{\mathcal{A}'}
\def\Xbarone{\mathcal{X}'}
\def\Atrop{\mathcal{A}^{\text{trop}}}
\def\Xtrop{\mathcal{X}^{\text{trop}}}
\usepackage{amsthm}
\theoremstyle{definition}
\newtheorem{definition}{Definition}[section]
\newtheorem{theorem}[definition]{Theorem}
\newtheorem{example}[definition]{Example}
\newtheorem{proposition}[definition]{Proposition}
\newtheorem{corollary}[definition]{Corollary}
\newtheorem{lemma}[definition]{Lemma}

\newtheorem{question}[definition]{Question}
\newtheorem{remark}[definition]{Remark}

\newtheorem*{conjecture*}{Conjecture}
\newtheorem*{question*}{Question}
\newtheorem*{theorem*}{Theorem}

\title{Fourier transforms and Abel-Jacobi theory}
\author{Younghan Bae, Sam Molcho and Aaron Pixton}
\date{}
\newcommand{\Ycomment}[1]{{\color{blue} Y: #1}}

\newcommand{\floor}[1]{\left \lfloor{#1}\right \rfloor}

\begin{document}
	\maketitle
	\begin{abstract}
		We relate Fourier transforms between compactified Jacobians over the moduli space of stable curves to logarithmic Abel-Jacobi theory. As an application, we compute the pushforward of divisor monomials on compactified Jacobians in terms of the twisted double ramification cycle formula. %Our proof relies on a construction of extended Poincar\'e line bundle for Fourier transform, logarithmic structure of compactified Jacobians, the logarithmic Abel-Jacobi theory, and combinatorial identities of the (universal) double ramification formula.
	\end{abstract}
	%\tableofcontents
	\section{Introduction}
	\subsection{Overview}
	For an $L^2$-function $f:\T^g:=\mathbb{R}^g/\mathbb{Z}^g \to \mathbb{C}$, the integral of $f$ can be evaluated as the $0$-th Fourier coefficient 
	\[ \int_{\T^g} f(x)dx= \widehat{f}(0)\,\]
	of the Fourier transform $\widehat{f}(\sfm) := \int_{\T^g}f(x) e^{-2\pi i\sfm\cdot x}dx$, $\sfm \in \mathbb{Z}^g$. 
	
	An analogous statement holds for a family of abelian varieties. Let $\pi: A \to B$ be a family of principally polarized abelian varieties with unit section $e$. Let $\fm: \CH^*(A,\QQ) \to \CH^*(A,\QQ)$ be the Fourier-Mukai transform, given by the Poincar\'e line bundle. For a class $\alpha\in \CH^*(A,\QQ)$, we have
	\begin{equation}\label{eq:intro1}
		\pi_*(\alpha) = e^*\fm(\alpha)\,,
	\end{equation}
	replacing integration with the pushforward $\pi_*$ and replacing evaluation at $0$ with the pullback $e^*$. Therefore, \eqref{eq:intro1} provides a useful tool for studying the pushforward, if one knows enough about the Fourier transform  -- for example, its compatibility with the weight decomposition of $\CH^*(A,\QQ)$.
	%It turns out that this equation is quite helpful in computing the pushforward $\pi_*$, as the Fourier transform is computable due to its compatibility with the structures induced by the group structure of $A$ on $\CH^*(A,\QQ)$ (for example, its compatibility with the weight decomposition).
	
	We study here the intersection theory of certain families of degenerating abelian schemes -- the fine compactified Jacobians $\pi:\Jbar_{g,n}^{\,\epsilon}\to\Mbar_{g,n}$ over the moduli space of stable curves -- using the Fourier transform. Two challenges arise: (i) the Poincaré line bundle does not extend to a line bundle on the product of the compactifications, and (ii) the relative group structure is lost. To address these issues, firstly, we construct a `minimal' logarithmic modification of the product of two compactified Jacobians where the Poincar\'e line bundle admits a unique line bundle extension. This construction yields a logarithmic desingularization of Arinkin’s kernel \cite{arinkin1, arinkin2} and provides the Fourier transform with a recursive structure.  Secondly, we link the Fourier transform with logarithmic Abel-Jacobi theory. The resolved Abel-Jacobi section of Holmes-Molcho-Pandharipande-Pixton-Schmitt \cite{HMPPS} is related to the universal double ramification (uniDR) formula of Bae-Holmes-Pandharipande-Schmitt-Schwarz \cite{BHPSS}. We study the Fourier transform of the class of the resolved Abel-Jacobi section and show that the $\uniDR$ formula and Fourier transform share the same genus recursive structure, enabling us to compute the Fourier transform effectively.
	
	We establish formula \eqref{eq:intro1} for compactified Jacobians under an appropriate normalization of Fourier kernel. By combining this formula with the calculation of the Fourier transform, we show that the pushforward of certain monomials of divisors can be expressed in terms of the twisted double ramification (DR) cycle formula $\DR^c_{g}(b;\sfa)$ introduced by Janda-Pixton-Pandharipande-Zvonkine \cite{JPPZ1}.
	%For monomials of arbitrary divisor classes, including weight $2$ divisors, we also obtain a closed formula in terms of $\DR^c_{g}(b;\sfa)$ for $c<g$, by showing that  the leading part of the pushforward of the twisted $\DR$ formula along the forgetful morphism $p:\Mbar_{g,n}\to\Mbar_{g,n-1}$ satisfies a genus recursion.
	Our formula involves all codimensions $c\leq g$\footnote{If the codimension $c$ is greater than $g$, then $\DR^c_g(b;\sfa)$ vanishes by \cite{claderjanda}.} and considers individual coefficients of the twisted $\DR$ formula as a polynomial in $b\in\ZZ$ and $\sfa\in\ZZ^n$. Thus,  we offer a geometric interpretation of $\DR^c_g(b;\sfa)$ in both contexts in terms of intersection theory on compactified Jacobians.
	
	\subsection{Compactified Jacobians}\label{sec:intro1.2}
	Let $p:\overline{\mathcal{C}}_{g,n}\to \Mbar_{g,n}$ be the universal curve with $n\geq 1$. A {\em stability condition} $\epsilon$ of degree $d$ for $p$ assigns a rational number to every irreducible component of every stable curve $(C,x_1,\ldots, x_n)$ of genus $g$ with $n$ marked points satisfying the following two conditions:
	\begin{enumerate}
		\item[(a)] the sum of the values of $\epsilon$ over the irreducible components of $C$ equals $d$, and
		\item[(b)] $\epsilon$ is additive under all partial smoothings of the curve $C$.
	\end{enumerate}
	
	A prestable curve $C'$ is {\em quasi-stable} if the relative dualizing
	sheaf $\omega_{C'}$ is nef and chains of unstable components have length at most $1$. A line bundle $L$ on $C'$ is {\em admissible} if $L$ has degree $1$ on each unstable component of $C'$. 
	
	With respect to a stability condition $\epsilon$, there exists a stability inequality defining when an admissible line bundles on a quasi-stable curve is $\epsilon$-stable (resp. $\epsilon$-semistable). A stability condition $\epsilon$ is {\em non-degenerate} if there are no strictly $\epsilon$-semistable admissible line bundles. A stability condition $\epsilon_0$ of degree $0$ is {\em small} if line bundles of multidegree $0$ on underlying curves are $\epsilon_0$ stable. 
	
	For a non-degenerate stability condition $\epsilon$ of degree $d$ for $p$, there exists a moduli space  $\Jbar_{g,n}^\epsilon$ of $\epsilon$-stable admissible line bundles on quasi-stable curves over $\Mbar_{g,n}$ (\cite{KP19}). It is a proper, smooth Deligne-Mumford stack of dimension $4g-3+n$.  In particular, for small stability condition $\epsilon_0$, $\Jbar^{\epsilon_0}_{g,n}$ contains the Jacobian of multidegree zero line bundles $\Jzero_{g,n}$ as an open substack.
	
	The compactified Jacobian $\Jbar_{g,n}^\epsilon$ carries a universal family which can be used to define tautological classes on $\Jbar_{g,n}^\epsilon$. There exists a universal quasi-stable curve with $n$ sections,
	\[p:\overline{\mathcal{C}}_{g,n}^{\text{qs}}\to \Jbar_{g,n}^\epsilon\,,\hspace{2mm} x_i:\Jbar_{g,n}^\epsilon\to \overline{\mathcal{C}}_{g,n}^{\text{qs}}, \, i=1,\ldots, n\,,\]
	with a universal admissible line bundle $\L$ on $\overline{\mathcal{C}}_{g,n}^{\text{qs}}$. We choose the universal admissible line bundle to be trivialized along the first marking. 
	%
	%We introduce tautological divisors on $\Jbar_{g,n}$. 
	For $1\leq i \leq n$, the $\xi$-class at the $i$th marked point is defined by $\xi_i := x_i^*c_1(\L)$ and the $\kappa_{0,1}$ class is defined by $\kappa_{0,1}=p_*(c_1(\omega_{p,\log})c_1(\L))$.
	The  $\Theta$-divisor is defined by $\Theta:=-\frac{1}{2}p_*(c_1(\L)^2)$.
	The divisor classes $\xi_i,\Theta$ and $\kappa_{0,1}$ span the rational Picard group $\mathrm{Pic} (\Jbar^\epsilon_{g,n})\otimes_\ZZ\QQ$ modulo  divisor classes pulled back from $\Mbar_{g,n}$ (\cite{FV22}).
	\subsection{Twisted double ramification cycle formula}
	%We recall the twisted double ramification cycle formula  \cite{JPPZ1}.
	Let $b\in\ZZ$ be an integer and let $\sfa=(a_1,\ldots, a_n)\in\ZZ^n$ be a vector of integers satisfying 
	\begin{equation}\label{eq:intro2}
		a_1+\cdots + a_n=(2g-2+n)b\,.
	\end{equation}
	Let $r$ be a positive integer. We denote by $\DR^{c,r}_g(b;\sfa)$\footnote{In the literature, this formula is denoted by $\mathsf{P}^{c,r,b}_g(\sfa)$, while $\DR_g^b(\sfa)$ refers to the twisted double ramification cycle, which corresponds to the codimension $g$ part of the double ramification cycle formula.} the codimension $c$ component of the tautological class
	\begin{equation*}
		\exp\left(-\frac{b^2}{2}\kappa_1+\sum_{i=1}^n\frac{a_i^2}{2}\psi_i\right)\cdot\sum_{
			\substack{\Gamma\in \mathsf{G}_{g,n} \\
				w\in \mathsf{W}_{\Gamma,r,b}}
		}
		\frac{r^{-h^1(\Gamma)}}{|\aut(\Gamma)| }
		\;
		(j_{\Gamma})_*\Bigg[
		\prod_{e=(h,h')\in E(\Gamma)}
		\frac{1-\exp\left(-\frac{w(h)w(h')}2(\psi_h+\psi_{h'})\right)}{\psi_h + \psi_{h'}} \Bigg]    
	\end{equation*}
	in $\CH^c(\Mbar_{g,n},\QQ)$. For a detailed explanation of the notation, we refer to Section \ref{subsec:DRintro}. 
	For sufficiently large $r$, $\DR^{c,r}_g(b;\sfa)$ becomes polynomial in $r$ \cite{JPPZ1}. We denote by $\DR^c_g(b;\sfa)$ the value at $r = 0$ of the polynomial associated with $\DR^{c,r}_g(b;\sfa)$.
	
	The formula $\DR^c_g(b; \sfa)$ is a polynomial of degree $2c$ in the variables $b,a_1,\ldots, a_n$ (see \cite{Spelier,Pixton_Poly}), well-defined only modulo the relation \eqref{eq:intro2}. If we rewrite $a_1$ in terms of the other variables using \eqref{eq:intro2}, then we can treat it as a polynomial in the other variables alone, with no relations. 
	It then makes sense to take a specific coefficient of the polynomial, and we write
	\[[\DR^c_g(b;\sfa)]_{b^m a_2^{k_2}\cdots a_n^{k_n}}\]
	for the coefficient of the monomial $b^m a_2^{k_2} \cdots a_n^{k_n}$ in $\DR^c_g(b;\sfa)$ in $\CH^c(\Mbar_{g,n},\QQ)$.
	\subsection{Intersection theory on compactified Jacobians}
	For a nondegenerate stability condition $\epsilon$ for the universal curve $\overline{\mathcal{C}}_{g,n}\to\Mbar_{g,n}$ of degree $d$, let $\pi:\Jbar_{g,n}^{\epsilon}\to\Mbar_{g,n}$ be the compactified Jacobian. For the definition of the tautological classes of $\Jbar_{g,n}^\epsilon$, we choose the universal line bundle which is trivialized along the first marking -- so in particular, $\xi_1 = 0$. 
	\begin{theorem}\label{thm:pushforward}
		Let $\pi: \Jbar_{g,n}^\epsilon\to\Mbar_{g,n}$ be the compactified Jacobian for a  non-degenerate stability condition $\epsilon$. Let $\ell,m,k_2,\ldots, k_n$ be non-negative integers. 
		\begin{enumerate}[label=(\alph*)]
			\item If $2\ell+m+\sum_i k_i<2g$, then $\pi_*(\Theta^\ell\kappa_{0,1}^m\xi_2^{k_2}\cdots\xi_n^{k_n})=0$ in $\CH^*(\Mbar_{g,n},\QQ)$.
			\item If $2\ell+m+\sum k_i=2g$, then
			\[\pi_*\Big(\frac{\Theta^\ell}{\ell!}\cdot\frac{(-\kappa_{0,1})^m}{m!}\cdot\prod_{i=2}^n\frac{\xi_i^{k_i}}{k_i!}\Big) = (-1)^{g-\ell} [\DR_g^{g-\ell}(b;\sfa)]_{b^m a_2^{k_2}\cdots a_n^{k_n}}\,.\]
			In particular, the pushforward is independent of the choice of stability condition.
			\item For any $\ell,m, k_2,\ldots,k_n$, the pushforward $\pi_*(\Theta^\ell\kappa_{0,1}^m\xi_2^{k_2}\cdots\xi_n^{k_n})$ lies in the tautological ring $\R^*(\Mbar_{g,n})$.
		\end{enumerate}
	\end{theorem}
	When $\epsilon$ is small, the number $2\ell+m+\sum_i k_i$ is the weight of the class $\Theta^\ell\kappa_{0,1}^m\xi_2^{k_2}\cdots\xi_n^{k_n}$ over the locus $\Mct_{g,n}\subset\Mbar_{g,n}$ where $\pi$ is an abelian scheme. Passing to rational cohomology, Theorem \ref{thm:pushforward} (a) can be obtained by the perverse filtration (Section \ref{sec:vanish}). The closed formula and the independence of stability condition in Theorem \ref{thm:pushforward} (b) is not at all clear from previous results.
	
	Theorem \ref{thm:pushforward} (b) consists of the leading terms of the $\DR$ cycle formula, viewed as a polynomial in $b,\sfa$. By Theorem \ref{thm:unidr1} below, the full $\DR$ cycle formula can be recovered from these leading terms. Therefore Theorem \ref{thm:pushforward} gives refined geometric meaning of the $\DR$ cycle formula.

    When $2\ell+m+\sum_i k_i$ is greater than $2g$, Theorem \ref{thm:pushforward} (c) says that classes in the subring of $\CH^*(\Jbar_{g,n}^\epsilon,\QQ)$ generated by divisors push forward to tautological classes. In this case, the pushforward depends on the choice of stability condition $\epsilon$, so explicit formulas will necessarily be more complicated. Based on this result, the forthcoming paper \cite{BM} will show that arbitrary tautological classes on the logarithmic Picard group push forward to tautological classes.
    
	%Although our methods do not give a simple formula for the pushforward when $2\ell+m+\sum_i k_i$ is greater than $2g$, we can at least say that the pushforward is tautological. In this case, the pushforward depends on the choice of stability condition $\epsilon$, so explicit formulas will necessarily be more complicated.
	%
	%Theorem \ref{thm:pushforward} (c) only implies that classes in the subring of $\CH^*(\Jbar_{g,n}^\epsilon,\QQ)$ generated by divisors push forward to tautological classes. The forthcoming paper \cite{BM} will show that arbitrary tautological classes on the logarithmic Picard group push forward to tautological classes.
	\subsection{Logarithmic Poincar\'e line bundle and Arinkin's kernel}\label{sec:log_dual}
	Our construction of the extended Poincaré line bundle is motivated by the duality of the logarithmic Picard group. For smooth curves and their relative Jacobians,  the Poincar\'e line bundle can be described as the Deligne pairing of two universal line bundles. This construction extends to logarithmically smooth curves. For a log smooth curve $C\to B$ and the logarithmic Picard group $\lpic_C\to B$, introduced by Molcho-Wise \cite{MW22}, the logarithmic Poincar\'e line bundle\footnote{The construction of $\P^{\log}$ will appear in the forthcoming work \cite{MWPoincare} and it will not be used in this paper.}
	\begin{equation*}
		\P^{\log}\to \lpic_C\times_B\lpic_C
	\end{equation*}
	is defined via the Deligne pairing of the two universal logarithmic line bundles, which induces a ``geometric" duality isomorphism $\lpic^0 (\lpic^0_C)\cong\lpic^0_C$. Any log line bundle on a log scheme (or algebraic log space) $X$ is representable by a line bundle on a logarithmic modification of $X$, and thus we can represent $\P^{\textup{log}}$ by a line bundle on some logarithmic modification of $\lpic_C^0 \times_B \lpic_C^0$. 
	%
	%Logarithmic line bundles on a log scheme $X$ can be (non-uniquely) represented by line bundles on logarithmic modifications of $X$. Thus, the logarithmic autoduality of $\lpic_C$ can be exhibited through some line bundle on some logarithmic modification of $\lpic_C \times_B \lpic_C$. As logarithmic modifications of $\lpic_C$ are precisely the toroidal compactifications of $J_C$, the Poincar\`e log line bundle can be represented by some line bundle on some logarithmic modification of the product $\Jbar_C \times_B \Jbar_C$ of two compactified Jacobians. 
	
	To connect the geometric duality obtained from the logarithmic Poincar\'e line bundle to derived equivalences among compactified Jacobians is however a more delicate issue: the logarithmic modification and line bundle representing the log Poincar\'e must be chosen carefully. For nondegenerate stability conditions $\epsilon_1$ and $\epsilon_2$, $\Jbar_C^{\epsilon_1}\times_B\Jbar^{\epsilon_2}_C$ is a natural logarithmic modification of $\lpic^0_C\times_B\lpic^0_C$, but the classical Poincar\'e line bundle on $J_C \times_B J_C$ does not extend to a line bundle on it; it extends to a line bundle $\P$ only up to the open locus $J_C^{\epsilon_1}\times_B\Jbar_C^{\epsilon_2}\cup \Jbar_C^{\epsilon_1}\times_BJ_C^{\epsilon_2}$, where $J^{\epsilon}_C\subset\Jbar^{\epsilon}_C$ denotes the locus of line bundles. 
	
	\if0
	Suppose now $C = \overline{\mathcal{C}}_{g,n} \to B=\Mbar_{g,n}$. For nondegenerate stability conditions $\epsilon_1$ and $\epsilon_2$, $\Jbar_C^{\epsilon_1}\times_B\Jbar^{\epsilon_2}_C$ is a logarithmic modification of $\lpic_C\times_B\lpic_C$. The representablilty of $\P^{\log}$  poses two questions: 
	\begin{enumerate}
		\item Can we describe a log modification and representing line bundle explicitly? 
		\item Can we choose a log modification and representing line bundle in a way in which we can extract useful algebrogeometric consequences?
	\end{enumerate}
	In Section \ref{sec:3}, we address both questions: first, we construct a modular logarithmic modification $\Jbar_C ^{(2)} \to \Jbar_C^{\epsilon_1}\times_B\Jbar^{\epsilon_2}_C$ and an explicit line bundle $\widetilde{\P}$ on $\Jbar_C^{(2)}$ that is defined via the modular description of $\Jbar_C^{(2)}$. The classical Poincar\'e line bundle on $J_C \times J_C$ does not extend to a line bundle on $\Jbar_C^{\epsilon_1}\times_B\Jbar^{\epsilon_2}_C$, but does extend to a line bundle $\P$ on the open locus of $J_C^{\epsilon_1}\times_B\Jbar_C^{\epsilon_2}\cup \Jbar_C^{\epsilon_1}\times_BJ_C^{\epsilon_2}$, where $J^{\epsilon}_C\subset\Jbar^{\epsilon}_C$ denotes the locus of line bundles. The modification $\Jbar_C^{(2)} \to \Jbar_C^{\epsilon_1}\times_B\Jbar^{\epsilon_2}_C$ is an isomorphism over this locus, and our line bundle $\widetilde{\P}$ extends $\P$. We then prove: 
	\fi
	\begin{theorem}\label{thm:Plog=arinkin}
		For $B=\Mbar_{g,n}$, there exists a  logarithmic modification $f:\Jbar^{(2)}_C\to\Jbar_C^{\epsilon_1}\times_B\Jbar_C^{\epsilon_2}$, where 
		\begin{enumerate}[label=(\alph*)]
			\item $\Jbar^{(2)}_C$ is smooth and log smooth. The map $f$ is an isomorphism over $J_C^{\epsilon_1}\times_B\Jbar_C^{\epsilon_2}\cup \Jbar_C^{\epsilon_1}\times_BJ_C^{\epsilon_2}$ and its complement has codimension $2$;
			\item if $\widetilde{\P}$ denotes the unique line bundle on $\Jbar_C^{(2)}$ extending $\P$, then $\widetilde{\P}$ represents $\P^{\log}$;
			\item the pushforward $\pbar:=Rf_*\widetilde{\P}$ is a maximal Cohen-Macaulay sheaf extending $\P$ which is flat relative to each projection $\pi_i :\Jbar_C\times_B\Jbar_C\to \Jbar_C$, $i=1,2$;
			\item $\fm:= -\otimes \Pbar:\dbcoh(\Jbar^{\epsilon_1}_C)\to\dbcoh(\Jbar^{\epsilon_2}_C)$ is an equivalence between derived categories.
		\end{enumerate}
	\end{theorem}
	This provides an independent proof that the line bundle $\P$ admits a $\pi_i$-flat maximal Cohen-Macaulay extension, a result originally proven by Arinkin \cite{arinkin1, arinkin2} (later generalized in \cite{MRV1, MRV2}). Following Arinkin's argument, the kernel $\Pbar$ induces a derived equivalence between $\Jbar^{\epsilon_1}_C$ and $\Jbar^{\epsilon_2}_C$. Consequently, we can use the explicit line bundle $\widetilde{\P}$ as the kernel of our derived equivalence.
	
	\subsection{Structure of the universal double ramification cycle formula}
	The universal double ramification formula is a natural lift of the twisted $\DR$ formula. Let $\pic_{g,n,0}$ denote the universal Picard stack, which parametrizes prestable curves of genus $g$ with $n$ markings and total degree $0$ line bundles. For $b\in\ZZ$ and $\sfa \in\ZZ^n$ satisfying \eqref{eq:intro2}, the $\uniDR$ formula is a tautological class $\uniDR^c_{g}(b;\sfa)$ in $\CH^c(\pic_{g,n,0},\QQ)$ that generalizes the $\DR$ formula \cite{BHPSS}. The $\uniDR$ formula is also polynomial in $b,a_i$. For a detailed explanation, we refer to Section \ref{subsec:topdeg_uniDR}.
	
	The top degree part of the $\uniDR$ formula is defined by the relative group structure of $\pic_{g,n,0}$ over the moduli stack of prestable curves $\mathfrak{M}_{g,n}$. Under the pullback along the ``multiplication by $N$" map $[N]:\pic_{g,n,0}\to\pic_{g,n,0}$, the $\uniDR$ formula becomes polynomial in $N$. We define the {\it top degree part} $\widetilde{\uniDR}^c_g(b;\sfa)$ as the part such that the sum of the weight (with respect to $[N]$) and the polynomial degree (with respect to $b,a_i$) is exactly $2c$. 
	
	We prove a correspondence between $\uniDR$ and its top degree part $\widetilde{\uniDR}$. The correspondence is most naturally stated using the negative zeta value regularization convention:
	\begin{equation}\label{eq:intro4}
		\sum_{k=1}^{\infty}k^{d+1} := \zeta(-d-1) = -\frac{B_{d+2}}{d+2}\quad\text{ for $d\ge 0$},
	\end{equation}
	where $B_n$ are the Bernoulli numbers.
	\begin{theorem}[Theorem~\ref{thm:top}]\label{thm:unidr1}
		Let $g,c,n\geq 0$. For each $0\leq m\leq g$, let $j_m:\pic_{\Gamma_m}\to\pic_{g,n,0}$ be the stratum parameterizing curves with degenerations forced by gluing the last $m$ pairs of markings and let $q_m:\pic_{\Gamma_m}\to\pic_{g-m,n+2m,0}$ be the $(\Gm)^m$-torsor associated to partial normalization. Then
		\[
		\uniDR_g^c(b;\sfa) = \sum_{m=0}^{\min(g,c)}\frac{1}{2^mm!} (j_m)_*(q_m)^*\Bigg[\sum_{k_1,\ldots,k_m>0}\Big(\prod_{i=1}^m k_i\Big) \widetilde{\uniDR}^{c-m}_{g-m}(b;\sfa,k_1,-k_1,\ldots,k_m,-k_m)\Bigg]\,,
		\]
		where the infinite sum over $k_i$ of a polynomial in $k_i$ is evaluated via \eqref{eq:intro4}.
	\end{theorem}
	By pulling back Theorem \ref{thm:unidr1} along the unit $e:\Mbar_{g,n}\to\pic_{g,n}$, we obtain genus recursion for the $\DR$ cycle formula (Theorem \ref{thm:topdeg_correspondence}).
	
	We also prove an identity involving the pushforward of the regular (twisted) $\DR$ cycle formula under the map $p:\Mbar_{g,n+1}\to\Mbar_{g,n}$ forgetting the last marking. 
	\begin{theorem}[Theorem \ref{thm:DR_push}]\label{thm:unidr2}
		Let $g,c,n\ge 0$. Let $F = p_*\DR_g^c(b;a_1,\ldots,a_{n+1})\in \CH^{c-1}(\Mbar_{g,n})$ be viewed as a polynomial in $b, a_1,\ldots, a_{n+1}$ modulo the relation \eqref{eq:intro1}.
		Then $F$ is a multiple of $(a_{n+1}-b)^2$. Moreover, we have the identity
		\[
		\left[\frac{F}{(a_{n+1}-b)^2}\right]_{a_{n+1}:=b} = (g+1-c)\DR_g^{c-1}(b;a_1,\ldots,a_n)\,.
		\]
	\end{theorem}
	A generalization of this theorem to the $\uniDR$ formula is given in Section~\ref{subsec:pushforward_uniDR} (Theorem \ref{thm:pushforward_uniDR}).
	\subsection{Fourier transform and the resolved Abel-Jacobi sections}
	We sketch our proof of Theorem~\ref{thm:pushforward} here; full details are found in Section~\ref{sec:7}. Our key approach to computing the Fourier transform is by linking it to logarithmic Abel-Jacobi theory.
	
	Let $\epsilon_0, \epsilon$ be non-degenerate stability conditions. Following the appropriate normalization as in Maulik-Shen-Yin \cite{MSY23}, we consider the Chow-theoretic realization of the Fourier transform of Theorem \ref{thm:Plog=arinkin}(d) given by the following relative algebraic correspondence:
	\begin{equation}\label{eq:intro3}
		\fm := f_*\big(\Td(T_{\Jbar^{(2)}_C}-f^*T_{\Jbar^{\epsilon_0}_{g,n}\times_{\Mbar_{g,n}}\Jbar^{\epsilon}_{g,n}})\cdot\ch(\widetilde{\P})\big): \CH^*(\Jbar_{g,n}^{\,\epsilon_0},\QQ)\xrightarrow{\cong} \CH^*(\Jbar_{g,n}^{\,\epsilon},\QQ)\,.
	\end{equation}
	We denote its inverse by $\fm^{-1}$. Choose $\epsilon_0$ to be small, so that the unit section $e$ exists in $\Jbar_{g,n}^{\epsilon_0}$. 
	
	We start from the Fourier image of the resolved Abel-Jacobi section. By \cite{HMPPS}, the rational Abel-Jacobi section admits a logarithmic resolution $\aj_{b;\sfa} :\Mbar_{g,b;\sfa}^{\,\epsilon_0}\to\Jbar_{g,n}^{\, \epsilon_0}$ . Using Theorem \ref{thm:Plog=arinkin}, the image of $\aj_{b;\sfa}$ has the form:
	\begin{equation}\label{eq:intro_fm_aj}
		\fm\big([\aj_{b;\sfa}]\big) = \exp\left(-b\kappa_{0,1}+\sum_{i=1}^n a_i\xi_i\right)\cdot (1+\gamma_{b;\sfa})\,,
	\end{equation}
	where $\gamma_{b;\sfa}$ is a class associated with an explicit piecewise polynomial (in the sense of \cite[Section 5.5]{MPS}) on $\Jbar_{g,n}^{\epsilon}$ which is supported away from the locus of integral curves (Proposition \ref{pro:FM_aj}).

	Logarithmic Abel-Jacobi theory computes the locus of the resolved Abel-Jacobi section. Using \cite{BHPSS} (Proposition \ref{pro:aj_P}), the class of the image of $\aj_{b;\sfa}$ can be written as the codimension $g$ part of the $\uniDR$ formula:
	\begin{equation}\label{eq:aj_uniDR}
		[\aj_{b;\sfa}] = \uniDR^g_g(b;\sfa)\Big|_{\Jbar^{\epsilon_0}_{g,n}}\in\CH^g(\Jbar^{\epsilon_0}_{g,n},\QQ)\,.
	\end{equation}
	Consider the inverse Fourier transform from $\Jbar_{g,n}^{\epsilon}$ to $\Jbar_{g,n}^{\epsilon_0}$.  Combining \eqref{eq:intro_fm_aj} and \eqref{eq:aj_uniDR}, we obtain
	\begin{equation}\label{eq:fm_aj}
		\uniDR^g_g(b;\sfa)\Big|_{\Jbar^{\epsilon_0}_{g,n}} = [\aj_{b;\sfa}] = \fm^{-1}\left(\exp\left(-b\kappa_{0,1}+\sum_{i=1}^n a_i\xi_i\right)\cdot (1+\gamma_{b;\sfa})\right)\,.
	\end{equation}

	We are then able to match the leading term of the inverse Fourier transform $\fm^{-1}$ and the top degree term of the $\uniDR$ formula by induction on genus (Theorem \ref{thm:leading}). After restricting $\fm^{-1}$ to $\Jbar^{\epsilon}_{g,n}\times_{\Mbar_{g,n}}\Jzero_{g,n}$, the inverse transform $\fm^{-1}$ reduces to
	\begin{equation*}\label{eq:Finv_F}
		\fm^{-1}= (-1)^g\cdot\ch(\widetilde{\P}^\vee)\cdot\pi_1^*\Td^\vee(\cR_\pi)^{-1}:\CH^*(\Jbar_{g,n}^{\epsilon},\QQ)\to\CH^*(\Jzero_{g,n},\QQ)\,,
	\end{equation*}
	where $\cR_\pi$ is the sheaf of residues of logarithmic differentials on $\Jbar^{\epsilon}_{g,n}$. Let $\fm^\circ:=\ch(\widetilde{\P}^\vee)$ be, up to sign, the leading term of $\fm^{-1}$.
	We show that the image of $\fm^\circ$ of classes supported on boundary strata of $\Jbar_{g,n}^{\,\epsilon}$ can be computed from $\fm^\circ$ for lower genera (Theorem \ref{thm:inductive}). By Theorem \ref{thm:unidr1}, the $\uniDR$ formula satisfies a genus recursion with respect to the top degree part. Combining the two recursive structures via \eqref{eq:fm_aj}, we show that the contribution of the inverse Fourier transform of the class $\gamma_{b;\sfa}$ does not contribute in codimension $g$ and the contribution from $\Td^\vee(\cR_\pi)^{-1}$ exactly matches with the recursion formula in Theorem \ref{thm:unidr1}. Therefore, the codimension $c<g$ part of the image of monomials of weight $1$ divisors vanishes and we get
	\[\left[\fm^\circ\left(\exp\left(-b\kappa_{0,1}+\sum_{i=1}^n a_i\xi_i\right)\right)\right]_{\codim =g} = (-1)^g\cdot\widetilde{\uniDR}_g^g(b;\sfa)\Big|_{\Jzero_{g,n}}\,.\]
	Pulling back both sides along the zero section $e$ and applying \eqref{eq:intro1} yields Theorem~\ref{thm:pushforward} when $\ell = 0$.
	
	For monomials with a positive exponent on $\Theta$, we use Theorem \ref{thm:unidr2} to reduce to $\ell=0$.

	\subsection{Other toroidal abelian fibrations}
	Let $\A_g$ be the moduli space of principally polarized abelian varieties of dimension $g$ and let $\A_g\subset\A_g'$ be the canonical partial compactification of rank $1$ degenerations. Let $\pi:\X'_g\to\A'_g$ be the universal family and let $\X^\circ_g\to\A'_g$ be the universal semi-abelian scheme. Let $\mu:\X^\circ_g\times_{\A'_g}\X'_g\to\X'_g$ be the multiplication map. By the work of Arinkin-Fedorov \cite{AF16}, the Poincar\'e line bundle extends to a line bundle $\P$ on  $\X^\circ_g\times_{\A'_g}\X'_g\cup\X'_g\times_{\A'_g}\X^\circ_g$. 
	\begin{theorem}\label{thm:Xdual}
		There exists a unique maximal Cohen-Macaulay sheaf $\overline{\P}$ on $\X'_g\times_{\A'_g}\X'_g$ extending $\P$. Moreover, $\fm:=-\otimes\Pbar :\dbcoh(\X'_g)\to\dbcoh(\X'_g)$ induces a derived equivalence.
	\end{theorem}
	We prove Theorem \ref{thm:Xdual} by resolving the indeterminacy of the multiplication map $\mu$, following the approach of Theorem \ref{thm:Plog=arinkin}. It gives a new derived equivalence  which does not follow from \cite{arinkin2}.
	%For degenerate abelian fibrations that do not originate from compactified Jacobians, Arinkin's construction cannot be directly applied. 
	
	We connect the partial Fourier transformation and the weight decomposition for semi-abelian group schemes and compute the class of the unit section $e$ inside $\X^\circ_g$. 
	\begin{theorem}\label{thm:unit_Ag}
		Let $e:\Abarone_g\to \X^\circ_g$ be the unit section. Then we have
		\[
		[e] = \Big[\exp(\Theta)\cdot\Big(1+
		\frac{1}{2}i_*q^*\sum_{a=1}^g\frac{- B_{2a}}{a!}\cdot\frac{(\Theta_1)^{a-1}}{2}\Big)\Big]_{\codim=g} \in \CH^g(\X^\circ_g)\,.
		\]
	\end{theorem}
	For a detailed explanation of notations, we refer to Section \ref{sec:trk1}.
	Along the way, we provide a new proof of the $\uniDR$ formula on the relative Jacobian over the treelike locus $\Mint_{g,n}\subset\Mbar_{g,n}$.

    \subsection{Relation to other work and further discussion}

    The results and the methods developed in this paper provide a new framework for studying the intersection theory of compactified Jacobians and other partial toroidal compactifications of abelian schemes. 
    
    In \cite{BMSY2}, Theorem \ref{thm:pushforward} is used to prove that the graded ring structure on $H^*(\Jbar_{g,1}^\epsilon,\QQ)$ depends on the stability condition $\epsilon$, even though the underlying vector space is independent of the stability. In loc. cit. the {\em intrinsic cohomology ring} is defined using the perverse filtration. Theorem \ref{thm:pushforward} completely determines the intersections divisors inside the intrinsic cohomology ring:
    \begin{equation}\label{eq:intrinsic}
            \pi_*\Big(\exp \big(-\Theta -b \kappa_{0,1}+\sum_{i=2}^na_i \xi_i\big)\Big) = (-1)^g \cdot \widetilde{\DR}_g(b;\sfa)\,,
    \end{equation}
    where $\widetilde{\DR}_g(b;\sfa)$ denotes the top-degree part of the $\DR$ formula. We refer to Section \ref{sec:vanish} for a possible connection with integrable hierarchies.
    
    In the forthcoming paper \cite{BM}, these methods will be used to show that the pushforward of piecewise polynomial classes on compactified Jacobians lies in the ring generated by piecewise polynomials, a property that appears to be exceptional among toroidal fibrations, and which we expect to extend to broader families of toroidal abelian fibrations.
    
    Moreover, in the forthcoming work \cite{BFLM}, our techniques will be applied to study the toroidal abelian fibration over the partial compactification of the moduli space of principally polarized abelian varieties $\A_g$ along the torus rank $1$ locus, and extend Theorem \ref{thm:unit_Ag} to $\X_g'$.
    
	\subsection*{Conventions}
	We work over a field $k$ of characteristic zero. All Chow groups are taken with $\QQ$-coefficients. All monoids and log structures will always be fine and saturated (f.s.) in the sense of \cite{Kato,ogus}. For a scheme $X$, let $\dbcoh(X)$ denote the bounded derived category of coherent sheaves on $X$. Let  $K_0(X)$ (resp. $K^0(X)$) denote the Grothendieck group of coherent sheaves (resp. locally free sheaves). For a morphism $X\to Y$ and $Y'\to Y$, we denote $X|_{Y'} :=X\times_Y Y'$.
	\subsection*{Acknowledgements}
	We would like to thank Carel Faber, David Holmes, Seungsu Lee, Aitor Iribar L\'opez, Davesh Maulik, Miguel Moreira, Rahul Pandharipande, Alex Perry, Junliang Shen, and Qizheng Yin for discussions on  weak abelian fibrations, Fourier transforms and related topics. Parts of this work were completed during visits to the University of Michigan, La Sapienza, and Tor Vergata, and we are grateful to these institutions for their hospitality. Y.~B. was supported by the SNSF Postdoc.Mobility fellowship and by a June E Huh Visiting Fellowship supported by the June E Huh Center for Mathematical Challenges. He thanks Utrecht University and MIT for its hospitality. A.~P. was supported by the NSF grant DMS-2301506.
	
	\section{Compactified Jacobians}
	\subsection{Compactified Jacobian and logarithmic Picard group}\label{sec:jbar}
	Relative Jacobians of a family of prestable curves (\cite[Definition 0E6T]{stacks-project}) can be compactified in two distinct ways: through the compactified Jacobian, and via logarithmic Picard groups. %While our intersection-theoretic computations are carried out on compactified Jacobians, the framework of logarithmic Picard groups offers key insights for our geometric constructions.
	
	%We begin by discussing compactified Jacobians. 
	Let $p: C\to B$ be a prestable curve of genus $g$ equipped with a section. For  $b\in B$, let $\Gamma_b$ denote the dual graph of the prestable curve $C_b$. A {\em stability condition} $\epsilon$ of degree $d$ for $p$ is a function
	\[\epsilon: V(\Gamma_b)\to\QQ\]
	for each geometric point $b\in B$ which has total degree $d$ and is compatible with degenerations (Section \ref{sec:intro1.2}).  A prestable curve $C$ is called quasi-stable if the dualizing sheaf $\omega_C$ is nef and chains of unstable components have length at most $1$. A line bundle $L$ on $C$ is {\em admissible} if $L$ has degree $1$ on each unstable component of $C$. An admissible line bundle $L$ on a quasi-stable curve $C$ is $\epsilon$-stable (resp. $\epsilon$-semistable) if for every proper subcurve $\emptyset\subsetneq \underline{C}\subsetneq C$ with neither $\underline{C}$ nor its complement consisting entirely of unstable components, we have
	\begin{equation*}
		\epsilon(\underline{C}) -\frac{E(\underline{C},\underline{C}^c)}{2}<(\leq) \deg(L|_{\underline{C}})<(\leq)  \,\epsilon(\underline{C})+\frac{E(\underline{C},\underline{C}^c)}{2}\,,
	\end{equation*}
	where  $E(\underline{C},\underline{C}^c)$ is the number of intersection points of $\underline{C}$ with the complement $\underline{C}^c$ and $\epsilon(\underline{C})$ is the sum of the values of $\epsilon$ over all irreducible components of $\underline{C}$. A stability condition is called {\em nondegenerate} if $\epsilon$-semi-stability equals $\epsilon$-stability.  %A stability condition is called {\em small} if the trivial line bundle $\CO_{C_b}$ is always $\epsilon$-semistable.
	
	For a stability condition $\epsilon$, a compactified Jacobian $\Jbar_C^\epsilon\to B$ of $p$ is a good moduli space parametrizing $\epsilon$-semistable rank $1$ torsion free sheaves (\cite{Caporaso94,KP19}). When the stability $\epsilon$ is clear from context, or the precise choice of $\epsilon$ is irrelevant to the argument, we simply denote $\Jbar_C:=\Jbar_C^\epsilon$. Let $\Jzero_C\to B$ be a relative Jacobian parametrizing multidegree $0$ line bundles. Tensor product induces a natural action
	\begin{equation}\label{eq:action}
		\mu: \Jzero_C\times_B\Jbar^\epsilon_C\to\Jbar_C^\epsilon\,. 
	\end{equation}
	
	The compactified Jacobian has several desirable properties  when the base $B$ is $\Mbar_{g,n}$. For nondegenerate $\epsilon$, $\Jbar^\epsilon_{C}$ is a smooth Deligne-Mumford stack, the projection map $\pi:\Jbar_C^\epsilon\to B$ is representable, proper and flat (\cite[Corollary 4.4]{KP19}). When the family $p$ has a section, rigidifying along that section produces a universal sheaf on the universal curve $C\to\Jbar_C^\epsilon$.
	
	To describe the second compactification of relative Jacobians, we begin by recalling definitions from logarithmic geometry. A {\em log scheme} $X=(X,\Msf_X)$ consists of a scheme $X$ together with a sheaf of commutative monoids $\Msf_X$ on the \'etale site of $X$ together with a morphism $\alpha: \Msf_X\to \CO_X$ with $\alpha^{-1}\CO_X^*\cong\CO_X^*$. For a sheaf of monoids $\Msf_X$, let $\Msf_X^\gp$ be the sheaf of abelian groups associated with the Grothendieck group of $\Msf_X$. For a log scheme $B$, we write $\textbf{LogSch}/B$ the category of f.s. log schemes over $B$. We refer to \cite{Kato,ogus} for the foundational concepts of logarithmic geometry.
	
	The logarithmic Picard group, introduced by Molcho and Wise \cite{MW22}, offers a canonical compactification of the relative Jacobian within the framework of algebraic logarithmic stacks. Let $B$ be a logarithmic scheme which is log smooth, and let $p:C\to B$ be a proper vertical log smooth curve. 
	\begin{theorem}[\cite{MW22}]\label{thm:MW}
		Let $\textbf{LogPic}_C: (\textbf{LogSch}/B)^\text{op} \to \text{(Grp)}$ be the functor defined by
		\[(B'\to B)\mapsto \{\text{$\Msf^\gp_{C|_{B'}}$-torsor of bounded monodromy}\}\,.\]
		Then $\textbf{LogPic}_C$ is an algebraic logarithmic stack. The relative rigidification $\lpic_C$ is log smooth with proper components, and forms a commutative group object over $B$.
	\end{theorem}
	For details on the notion of bounded monodromy, we refer to \cite[Section 3.5]{MW22}. In brief, it amounts to an infinitesimal smoothability condition.

	Compactified Jacobians serve as a ``scheme-theoretic" birational model of the logarithmic Picard group. In \cite[Proposition 4.4.8]{MW22} a natural morphism
	\begin{equation}\label{eq:Jbar_lpic}
		\Jbar_C\to\lpic_C\,
	\end{equation}
	is constructed. While the logarithmic Picard stack is not representable by an algebraic stack, the morphism \eqref{eq:Jbar_lpic} is a log modification by \cite{AP20}, which implies that the compactified Jacobian $\Jbar_C$ can be interpreted as a representable, birational model of $\lpic_C$.\footnote{A morphism $f:X\to Y$ between algebraic logarithmic stacks is called a {\em log modification} if for any morphism $T\to Y$ from a log scheme, $f_T:X_T\to Y_T$ is a log modification (in particular, proper, representable and birational). } 
	\subsection{Universal Picard stack and tautological classes}\label{sec:taut_class}
	Tautological classes on the relative compactified Jacobian are derived from tautological classes pulled back from the universal Picard stack. Let $\M_{g,n}$  denote the moduli stack of prestable (not necessarily stable) curves of genus $g$ with $n$ marked points, and let $p:\C_{g,n}\to\M_{g,n}$ be the universal curve together with sections $x_i:\M_{g,n}\to \C_{g,n}$ for $1\leq i\leq n$. The {\em universal Picard stack}
	\[\pic_{g,n}\to\M_{g,n}\]
	is the relative Picard stack over $\M_{g,n}$, as described in \cite[Section 2]{BHPSS}. It is a smooth algebraic stack of locally of finite type over $\M_{g,n}$ and decomposes into a disjoint union of connected components $\pic_{g,n,d}$, indexed by $d$, the total degree of a line bundle. Under the natural boundary stratification, $\pic_{g,n,d}$ is an algebraic stack with a log structure. 
	
	The compactified Jacobian admits a morphism to the universal Picard stack. For $n\geq 1$, let $B=\Mbar_{g,n}$ and let $\Jbar^\epsilon_C\to B$ denote the compactified Jacobian associated with a non-degenerate stability condition $\epsilon$. We choose a universal rank $1$ torsion-free sheaf $F\to C$. According to \cite{EP16}, there exists a quasi-stable model
	\[\nu:C^\mathrm{qs}\to C\]
	along with an admissible line bundle
	\[L\to C^\mathrm{qs}\]
	such that $R\nu_*L\cong F$. The pair $(C^\mathrm{qs},L)$  is referred to as the {\em quasi-stable model}. This quasi-stable model, defined over $\Jbar_C$ induces a morphism 
	\begin{equation}\label{eq:varphi}
		\varphi : \Jbar_C^\epsilon\to\pic_{g,n}\,,
	\end{equation}
	which is dependent on the chosen universal sheaf. For $n\geq 1$, let $\pic^{\text{rel}}_{g,n}$ be the relative rigidification of $\pic_{g,n}\to\M_{g,n}$. Then $\varphi$ factors through $\overline{\varphi}: \Jbar_{C}^\epsilon\to\pic_{g,n}^\text{rel}$.
	\begin{lemma}\label{lem:smoothness}
		For a nondegenerate stability condition $\epsilon$, the morphism  \eqref{eq:varphi} is smooth. Moreover $\overline{\varphi}$ is an open immersion.
	\end{lemma}
	\begin{proof}
		Clearly $\varphi$ is locally of finite presentation. Fibers of $\varphi$ have constant dimension. Since $\Jbar_C^\epsilon$ and $\pic_{g,n}$ are both smooth, we conclude that $\varphi$ is flat by the miracle flatness lemma. Moreover, fibers of $\varphi$ are smooth, hence it is smooth. Since $\overline{\varphi}$ is a momonorphism, it is an open immersion.
	\end{proof}
	
	We define a natural notion of tautological ring for a compactified Jacobian by pulling back the tautological ring of $\pic_{g,n}$ via the morphism $\varphi$ defined in \eqref{eq:varphi}:
	\begin{equation}\label{eq:taut_jbar}
		\R^*(\Jbar_C)\subset\CH^*(\Jbar_C,\QQ)\,.
	\end{equation}
	In forthcoming work \cite{BM}, the tautological ring of compactified Jacobians---and more generally, the logarithmic Picard group---is studied in detail. In this paper, we focus on monomials of divisors. Among the tautological classes of codimension $1$, we highlight the following:
	\[\Th:= -\frac{1}{2}p_*(c_1(\L)^2),\quad \kappa_{0,1}= p_*(c_1(\omega)c_1(\L)), \text{ and }\xi_i:=x_i^*(c_1(\L))\,.\]
	
	We describe the boundary strata of the relative and compactified Jacobians when $B = \Mbar_{g,n}$. For each prestable graph $\Gamma$ and multidegree function $\delta \colon V(\Gamma) \to \ZZ$, let $\pic_{\Gamma_\delta}$ denote the stack parametrizing prestable curves $C_v$ of genus $g(v)$ with $n(v)$ markings for each vertex $v \in V(\Gamma)$, together with a line bundle $L$ on a prestable curve $C$ obtained via the gluing map $\nu : \bigsqcup_{v \in V(\Gamma)} C_v \to C$ associated to $\Gamma$, such that the restriction $\nu^*L|_{C_v}$ has degree $\delta(v)$.
	%We describe boundary strata of relative and compactified Jacobians when $B=\Mbar_{g,n}$. For each prestable graph $\Gamma$ and multidegree $\delta: V(\Gamma)\to \ZZ$, let $\pic_{\Gamma_\delta}$ be the stack parametrizing prestable curves with dual graphs consisting of specializations of $\Gamma$ and with line bundles which have degree $\delta(v)$ restriction to the components corresponding to the vertex $v\in V(\Gamma)$, as in \cite[Section 0.3.2]{BHPSS}.
	%For the log scheme $S$, $\pic_{\Gamma_{\delta}}(S)$ is the groupoid of triple $(C/S,\{\Gamma_s\to\Gamma\}_{s\in S}, L)$, where $C/S$ is the log curve, $\Gamma_s$ is the dual graph of $C_s$, $\Gamma_s\to\Gamma$ is the edge contraction, and $L$ is the line bundle on $C$ whose multidegree is compatible with $\delta$ under the edge contraction. 
	Then $\pic_{\Gamma_\delta}$ admits morphisms
	\begin{equation}\label{eq:picboundary}
		\begin{tikzcd}
			\mathfrak{Pic}_{\Gamma_\delta} \ar[r,"j_{\Gamma_\delta}"]\ar[d,"q"'] & \mathfrak{Pic}_{g,n}\\
			\prod_{v\in V(\Gamma)}\mathfrak{Pic}_{g(v),n(v),\delta(v)}
		\end{tikzcd}
	\end{equation}
	where the morphism $j_{\Gamma_\delta}$ is proper, representable and $q$ is the morphism induced by the partial normalization induced by the edge contraction data. The morphism $q$ is a $\Gm^{|E(\Gamma)|}$-torsor. For a quasi-stable graph $\Gamma$ with $\epsilon$-stable multidegree $\delta$, $\Jbar_{\Gamma_\delta}$ is given by the cartesian diagram
	\begin{equation}\label{eq:stratum}
		\begin{tikzcd}
			\Jbar_{\Gamma_\delta} \ar[r]\ar[d] &\Jbar_{g,n}\ar[d]\\
			\mathfrak{Pic}_{\Gamma_\delta} \ar[r,"j_{\Gamma_\delta}"] & \mathfrak{Pic}_{g,n} \,,
		\end{tikzcd}
	\end{equation}
	where the right vertical morphism is \eqref{eq:varphi}.
	\begin{lemma}\label{lem:torsor}
		Let $\pi:\Jzero_{g,n}\to\Mbar_{g,n}$ be the relative Jacobian.
		Let $j:\Mbar_{g-\ell,n+2\ell}\to\Mbar_{g,n}$ be the gluing map identifying $n+2i-1$-th marking and $n+2i$-th marking for $1\leq i\leq \ell$. Let $q:\Jzero_{g,n}|_{\Mbar_{g-\ell,n+2\ell}}\to\Jzero_{g-\ell,n+2\ell}$ be the morphism given by the pullback of line bundles to partial normalization. Then $q$ can be identified with
		\[\Jzero_{g,n}|_{\Mbar_{g-\ell,n+2\ell}}\cong (L_{n+1}\otimes L_{n+2}^\vee)^{\times}\oplus\cdots\oplus (L_{n+2\ell-1}\otimes L^\vee_{n+2\ell})^{\times}\to \Jzero_{g-\ell,n+2\ell}\,.\]
	\end{lemma}
	\begin{proof}
		By partial normalization, we get a morphism $\Jzero_{g,n}|_{\Mbar_{g-\ell,n+2\ell}}\to\Jzero_{g-\ell,n+2\ell}$. Fibers of the morphism record the identification of the line bundle at $n+2i-1$ and $n+2i'$, $L|_{n+2i-1}\cong L|_{n+2i}$.
	\end{proof}
	%The following description of stratum corresponding to quasistable graphs will be used in Section \ref{sec:todd}.
	\begin{lemma}\label{lem:boundary}
		Let $\Gamma_h$ be the stable graph consisting of a single vertex of genus $g-h$ with $h$ loops and let $\Gamma_h'$ be the quasistable graph obtained by subdividing each edge of $\Gamma_h$. Then there exists a nondegenerate stability condition $\epsilon_h$ of degree $d-h$ together with a map 
		\[\Jbar^{\epsilon_h}_{g-h,n+2h}\to\Jbar_{g,n}^\epsilon\]   
		which identifies $\Jbar^{\epsilon_h}_{g-h,n+2h}\to\Jbar_{g,n}^\epsilon$ with  $\Jbar_{\Gamma'_h} \to\Jbar_{g,n}^\epsilon$.
	\end{lemma}
	\begin{proof}
		We define a nondegenerate stability condition $\epsilon_h$ for $\Mbar_{g-h,n+2h}$ of degree $d-h$ as follows. %There exists a one to one correspondence with a stable graph $\Gamma$ of genus $g-h$ with $n+2h$ legs and a subdivision $\widetilde{\Gamma'_h}$ of $\Gamma_h'$.
		Suppose $\Gamma$ is a stable graph of genus $g-h$ with $n+2h$ legs. Then gluing the last $h$ pairs of legs gives a stable graph of genus $g$ with $n$ legs, so the stability condition $\epsilon$ can be applied to give numbers $\epsilon(v)$ for $v\in V(\Gamma)$ (with sum $d$).
		We then define $\epsilon_h:V(\Gamma)\to \QQ$ by 
		\begin{equation}\label{eq:epsilon_h}
			\epsilon_h(v) = \epsilon(v) -\frac{1}{2}\text{(number of $i$-th legs $n+1\leq i\leq n+2h$ attached to $v$)}\,.
		\end{equation}
		Since $\epsilon$ is a stability condition, \eqref{eq:epsilon_h} is a stability condition for $\Mbar_{g-h,n+2h}$ of degree $d-h$. When $\epsilon$ is nondegenerate, then $\epsilon_h$ is also nondegenerate by \cite[Definition 5.1]{KP19}. 
		
		Over $\Jbar^{\epsilon_h}_{g-h,n+2h}$, there exists a morphism of curves $\nu:\mathcal{C}^\qs_{g-h,n+2h}\to \mathcal{C}_{\Gamma'_h}$ where $\mathcal{C}^\qs_{g-h,n+2h}$ is the universal quasi-stable curve and $\nu$ is the morphism gluing the $(n+2i-1)$-th to the $(n+2i)$-th marking for $1\leq i\leq h$. For the universal line bundle $\L$ on $\mathcal{C}^\qs_{g-h,n+2h}$, the pushforward $\nu_*\L$ is a rank $1$ torsion free sheaf. The universal quasi-stable model of $(\mathcal{C}^\qs_{g-h,n+2h},\nu_*\L)$ defines the desired morphism $\Jbar^{\epsilon_h}_{g-h,n+2h}\to\Jbar_{g,n}^\epsilon$.
	\end{proof}
	
	We end with the connection between the universal Picard stack and the logarithmic Picard group. For the universal Picard stack $\pic_{g,n}$ over $\mathfrak{M}_{g,n}$ and the logarithmic Picard group $\lpic_{g,n}$ over $\Mbar_{g,n}$, in \cite[Proposition 4.4.8]{MW22} a natural morphism
	\begin{equation}\label{eq:pic_lpic}
		\pic_{g,n}\to \lpic_{g,n}
	\end{equation}
	is constructed. When $B=\Mbar_{g,n}$, the morphism \eqref{eq:Jbar_lpic} is in fact defined by the composition
	$\Jbar_{g,n}\to\pic_{g,n}\to\lpic_{g,n}$.
	
	\section{Extending the Poincar\'e line bundle}\label{sec:3}
	\subsection{Poincar\'e line bundle}
	Let $p:C\to B$ be a proper log smooth curve with a section. Let $\epsilon_1$ and $\epsilon_2$ be two nondegenerate stability conditions for $p$. For the compactified Jacobian $\Jbar_C^{\,\epsilon_2}\to B$, let
	\begin{equation}\label{eq:J_C}
		J^{\,\epsilon_2}_C \subset \Jbar^{\,\epsilon_2}_C
	\end{equation}
	be the open substack given by the locus of line bundles. For the compactified Jacobian $\Jbar_C^{\,\epsilon_1}$, choose the universal sheaf $\F$ on the universal curve $C$ which is trivialized along the section. Let $C^\text{qs}\to \Jbar_C^{\,\epsilon_1}$ be the universal quasi-stable curve and $\L_1$ be the universal line bundle on $C^\text{qs}$.  
	\begin{definition}\label{def:Poincare_line}
		Let $p:C^\qs\to\Jbar_C^{\,\epsilon_1}\times_B J_C^{\,\epsilon_2}$ be the pullback of the universal quasi-stable curve, and $\L_1$ (resp. $\L_2$) be the universal line bundle of $\Jbar_C^{\,\epsilon_1}$ (resp. $J_C^{\,\epsilon_2}$). The {\em Poincare line bundle} $\P$ is defined by
		\begin{equation*}
			\P =\langle \L_1,\L_2\rangle \in \mathrm{Pic}(\Jbar_C\times_B J_C)\,,
		\end{equation*}
		where $\langle -, -\rangle$ is the Deligne pairing \cite{DelignePairing}. If we wish to emphasize the universal curve, we may write $\langle -, -\rangle_C$.
	\end{definition}
	Definition \ref{def:Poincare_line} coincides with the definition from \cite{arinkin1,MRV1} because we only consider families of prestable curves.
	The Poincar\'e line bundle in general depends on the choices of universal line bundles, but if both stability conditions $\epsilon_1,\epsilon_2$ are of degree zero, it is independent of the choices.
	We often omit stability conditions in our notation and write $\Jbar_C\times_B J_C:=\Jbar_C^{\,\epsilon_1}\times_B J_C^{\,\epsilon_2}$.
	
	Similarly, the Poincar\'e line bundle exists on $J_C\times_B\Jbar_C$. Since the two line bundles coincide on the intersection, the line bundle $\P$ exists on $J_C\times_B\Jbar_C\cup\Jbar_C\times_B J_C$.
	Let
	\begin{equation}\label{eq:open}
		\iota:\Jbar_C\times_B J_C\cup J_C\times_B\Jbar_C\hookrightarrow \Jbar_C\times_B\Jbar_C
	\end{equation}
	be the open embedding.
	
	By Grothendieck-Riemann-Roch (\cite[(6.6.1)]{DelignePairing}), the first Chern class of the Deligne pairing of two line bundles can be written as 
	\begin{equation}\label{eq:Deligne}
		c_1(\langle \L_1,\L_2\rangle) = p_*(c_1(\L_1)c_1(\L_2))\,.
	\end{equation} 
	\subsection{Birational model of \texorpdfstring{$\Jbar_C\times_B\Jbar_C$}{Jbar}}\label{sec:birational}
	In this section, we set $B=\Mbar_{g,n}$. Let $\epsilon_1,\epsilon_2$ be two nondegenerate stability conditions for $B$. Throughout this section, we abbreviate notation and write $\Jbar_C\times_B\Jbar_C:=\Jbar_C^{\epsilon_1}\times_B\Jbar_C^{\epsilon_2}$. On $\Jbar_C\times_B\Jbar_C$ the Deligne pairing does not extend. Using logarithmic geometry, we construct a birational model of $\Jbar_C\times_B\Jbar_C$ to resolve the indeterminacy of the Deligne pairing \eqref{eq:Deligne}.
	
	A map of log schemes (or algebraic stacks) is \emph{strict} if it induces an isomorphism on log structures. We assume our log schemes come with a tropicalization map 
	\begin{equation}\label{eq:tropical}
		t: B \to \A_B
	\end{equation}
	to some Artin fan $\mathcal{A}_B$, and write $\Sigma_B$ for the corresponding cone stack, see \cite{CCUW}. The tropicalization map is strict.
	Moreover, for any logarithmic morphism $f:X\to S$ between log schemes, there exists an induced morphism $A_X\to \A_S$ between the corresponding Artin fans.
	
	A {\em log alteration} of a log scheme $B$ is a map $\mathcal{A}_B' \times_{\mathcal{A}_B} B \to B$ given by the base change of a proper, Deligne-Mumford type and birational logarithmic morphism $\mathcal{A}_B' \to \mathcal{A}_B$. A log alteration is called a {\em log modification} if it's representable, and a root if it is bijective on geometric points. 
	Under the correspondence $\mathcal{A}_B \leftrightarrow \Sigma_B$, log modifications correspond to subdivisions of $\Sigma_B$, and roots to the choice of a finite index integral substructure of the integral structure of $\Sigma_B$. 
	
	\begin{definition}
		A log alteration $B' \to B$ is called \emph{small} if the map $\mathcal{A}_{B'} \to \mathcal{A}_B$ induces an isomorphism between the groups of Weil divisors of $\mathcal{A}_B$ and $\mathcal{A}_{B'}$.
	\end{definition}

	Let $\mathfrak{Q}$ denote the stack of quasi-stable curves. The stack $\mathfrak{Q}$ is algebraic because it is an open substack of $\M_{g,n}$. By restriction, $\mathfrak{Q}$ carries a natural logarithmic structure. 
	\begin{definition}\label{def:Q2}
		We define a log stack $\mathfrak{Q}^{(2)} : (\textbf{LogSch}/B)^\mathrm{op}\to (\mathrm{Grp})$ by
		\[
		\mathfrak{Q}^{(2)}(S \to B) := \{(C_1 \to S, C_2 \to S, \widetilde{C} \to C_i) \} 
		\]
		where 
		\begin{itemize}
			\item $C_i\to S$ are two quasi-stable models of $C_S/S$,
			\item $\widetilde{C}\to C_i$ is a quasi-stable model such that $\widetilde{C}\to S$ remains a log smooth curve, and
			\item $\widetilde{C}$ is the minimal common log modification of the $C_i$, i.e. \[
			\widetilde{C} = C_1 \times^{\textup{fs}}_{C} C_2.
			\]
		\end{itemize}
	\end{definition} 
	There is a natural projection 
	\[\mathfrak{Q}^{(2)} \to \mathfrak{Q} \times_{B} \mathfrak{Q}\]
	defined by sending 
	\[(C_1 \to S, C_2 \to S, \widetilde{C} \to C_i)\to (C_1/S,C_2/S)\,.\]
	
	The following proposition shows that in fact $\mathfrak{Q}^{(2)}$ is given by a log modification of $\mathfrak{Q}\times_B\mathfrak{Q}$, and is therefore an algebraic stack rather than merely a category fibered in groupoids.
	\begin{proposition}\label{pro:logmod}
		The projection $\mathfrak{Q}^{(2)} \to \mathfrak{Q} \times_{B} \mathfrak{Q}$  is a log modification. 
	\end{proposition}
	\begin{proof}
		We argue tropically. Let $\Sigma_{\mathfrak{Q}}$ denote the stack over rational polyhedral cones that parametrizes, over a cone $\sigma$, a quasi-stable graph $\widetilde{\Gamma}$ with stabilization $\Gamma$, where edge lengths are metrized by $M_{\sigma} := \mathrm{Hom}(\sigma,\NN)$. It is straightforward to see that $\Sigma_{\mathfrak{Q}}$ is a tropicalization of $\mathfrak{Q}$. We define $\Sigma_{\mathfrak{Q}^{(2)}}$ as the subfunctor of $\Sigma_{\mathfrak{Q}} \times_{\Sigma_B} \Sigma_{\mathfrak{Q}}$ that over a cone $\sigma$ parametrizes two quasi-stable models $\Gamma_1$ and $\Gamma_2$ of $\Gamma$ over $\sigma$ for which the fiber product $\Gamma_1 \times_{\Gamma} \Gamma_2$ remains a tropical curve metrized over $\sigma$. Then we have an isomorphism:
		\[
		\mathfrak{Q}^{(2)} \cong \Sigma_{\mathfrak{Q}^{(2)}} \times_{(\Sigma_{\mathfrak{Q}} \times_{\Sigma_{B}} \Sigma_{\mathfrak{Q}})} (\mathfrak{Q} \times_{B} \mathfrak{Q})\,,
		\]
		which indicates that it suffices to verify that $\Sigma_{\mathfrak{Q}^{(2)}} \to \Sigma_{\mathfrak{Q}} \times_{\Sigma_{B}} \Sigma_{\mathfrak{Q}}$ is a subdivision. 
		
		Let $\Gamma_1,\Gamma_2$ be two quasistable models of $\Gamma$ metrized by $M_{\sigma}$. Suppose $e$ is an edge of $\Gamma$ that is subdivided in both $\Gamma_1$ and $\Gamma_2$. The length of the edge $e$ is an element $\ell_e \in M_\sigma$. Since the edge is subdivided in both $\Gamma_1,\Gamma_2$, there exist elements $\ell_{e_i}',\ell_{e_i}''$, $i=1,2$ such that
		\[
		\ell_e = \ell_{e_i}' + \ell_{e_i}''\,.
		\]
		We orient $\ell_{e_1}',\ell_{e_1}''$ so that the terminal point of $\ell_{e_1}'$ and the initial point of $\ell_{e_1}''$ is the quasistable vertex.  Similarly, $\ell_{e_2}',\ell_{e_2}''$ are oriented ``the same way", meaning the initial point of $\ell_{e_2'}$ coincides with the initial point of $\ell_{e_1}'$, and its terminal point is the quasistable vertex. 
		
		Consider now the fiber product 
                	\[
		\widetilde{\Gamma} = \Gamma_1 \times_\Gamma \Gamma_2\,.
		\]
        Combinatorially, it is given by the common subdivision of the $\Gamma_i$, i.e. a model of $\Gamma$ with up to two exceptional vertices. This model, however, is not metrized over  $M_{\sigma}$: put otherwise, the cone complex associated to the fiber product fails to be flat over $\sigma$ on the locus where the functions $\ell_{e_1}'$ and $\ell_{e_2}'$ become equal. Nevertheless, there is a minimal subdivision of $\sigma$ -- the subdivision along the hyperplane
		\[
		\ell_{e_1}' = \ell_{e_2}'
		\]
		for every edge $e$ subdivided in both $\Gamma_1,\Gamma_2$ -- over which $\widetilde{\Gamma}$ is defined. Consequently, in the universal case, it follows that the map $\Sigma_{\mathfrak{Q}^{(2)}} \to \Sigma_{\mathfrak{Q} }\times_{\Sigma_B} \Sigma_\mathfrak{Q}$ is the subdivision along the hyperplanes $\ell_{e_1}' - \ell_{e_2}'$ where $e$ ranges through all the edges of $\Gamma$ that are subdivided in both $\Gamma_1$ and $\Gamma_2$. Therefore, we have proven that $\mathfrak{Q}^{(2)}\to \mathfrak{Q}\times_B \mathfrak{Q}$ is a log modification, and in particular that $\mathfrak{Q}^{(2)}$ is an algebraic stack.
	\end{proof}
	
	\begin{remark}
        The log modification of the base can be understood as resolving the non-flatness of the f.s. fiber product of the universal curves. For two universal quasi-stable curves $C_1 $ and $C_2$ with the same stabilization $C$ over $\mathfrak{Q}\times_B\mathfrak{Q}$, the f.s. fiber product $C_1 \times^{\textup{fs}}_C C_2 \to \mathfrak{Q} \times_B \mathfrak{Q}$ is not a relative curve, as it fails to be flat. By the semistable reduction theorem \cite{Mss}, there exists a canonical log alteration $\mathfrak{Q}'$ of $\mathfrak{Q} \times_B \mathfrak{Q}$ and $C'$ of $C_1 \times^{\textup{fs}}_C C_2$ such that $C' \to \mathfrak{Q}'$ becomes a  curve. Proposition \ref{pro:logmod} provides an alternative perspective: $\mathfrak{Q}^{(2)}$ is exactly $\mathfrak{Q}'$, and further alteration of $C_1 \times^{\textup{fs}}_C C_2$ is unnecessary as it is already flat over $\mathfrak{Q}^{(2)}$. 
	\end{remark}
	
	\begin{corollary}\label{cor:smoothandsmall}
		The stack $\mathfrak{Q}^{(2)}$ is smooth and log smooth. Furthermore, the exceptional locus of $\mathfrak{Q}^{(2)}$ does not contain any divisors.  
	\end{corollary}
	
	\begin{proof}
		Cones of $\Sigma_B$ correspond to stable graphs $\Gamma$; for such a graph, with set of edges $E(\Gamma)$, the associated cone is 
		\[
		\sigma_\Gamma = \mathbb{R}_{\ge 0}^{E(\Gamma)}
		\]
		with its natural integral structure. Similarily, cones of $\Sigma_{\mathcal{Q}}$ correspond to quasistable graphs $\Gamma'$, with corresponding cone 
		\[
		\sigma_{\Gamma'} = \mathbb{R}_{\ge 0}^{E(\Gamma')}. 
		\]
		Let $\Gamma'$ be a quasistable graph with stabilization $\Gamma$. For each edge $e$ of $\Gamma$, there is either exactly one or exactly two edges of $\Gamma'$ that map to $e$. We write $E^{u}(\Gamma)$ for the edges of the first kind, and $E^{s}(\Gamma)$ for the edges of the second kind. If we denote the edges of $\Gamma'$ that map to $e \in E^{s}(\Gamma)$ by $e',e''$, then we have 
		\[
		E(\Gamma') = E^{u}(\Gamma) \bigcup_{e \in E^{s}(\Gamma)}\left\{e',e''\right\}. 
		\]
		The map $\Sigma_{\mathfrak{Q}} \to \Sigma_B$ is defined as follows: on the cone corresponding to $\Gamma'$, with stabilization $\Gamma$, the map 
		\[
		\sigma_{\Gamma'} = \mathbb{R}_{\ge 0}^{E^{u}(\Gamma)} \times_{e \in E^{s}(\Gamma)} \mathbb{R}_{\ge 0}e' \times \mathbb{R}_{\ge 0}e'' \to \sigma_{\Gamma} = \mathbb{R}_{\ge 0}^{E^{u}(\Gamma)} \times_{e \in E^{s}(\Gamma)} \mathbb{R}_{\ge 0}e
		\]
		is the product of the identity on the $\mathbb{R}_{\ge 0}^{E^{u}(\Gamma)}$ factor with the addition maps 
		\begin{align*}
			\mathbb{R}_{\ge 0}e' \times \mathbb{R}_{\ge 0}e'' \to \mathbb{R}_{\ge 0}e\,, (\ell_e',\ell_e'') \mapsto \ell_e' + \ell_e''\,.
		\end{align*}
		
		Similarily, cones of the fiber product $\Sigma_{\mathfrak{Q}} \times_{\Sigma_{B}} \Sigma_{\mathfrak{Q}}
		$ are indexed by pairs of quasistable graphs $\Gamma_1,\Gamma_2$ with mutual stabilization $\Gamma$. Keeping the notation above, we further split the edges of $\Gamma$ as 
		\[
		E^{u}(\Gamma) \cup E^{s_1 \vee s_2}(\Gamma) \cup E^{s_1 \cap s_2}(\Gamma)\,, 
		\]
		where 
		\begin{itemize}
			\item $E^{s_1 \vee s_2}(\Gamma)$ are the edges of $\Gamma$ subdivided in either $\Gamma_1$ or $\Gamma_2$ but not both, and 
			\item $E^{s_1 \cap s_2}(\Gamma)$ are the edges subdivided in both $\Gamma_1$ and $\Gamma_2$. 
		\end{itemize}
		From the description of the maps $\sigma_{\Gamma_i} \to \sigma_{\Gamma}$, it follows that the cone corresponding to $\Gamma_1,\Gamma_2$, which is the fiber product $\sigma_{\Gamma_1}\times_{\sigma_\Gamma} \sigma_{\Gamma_2}$, is explicitly  
		\[
		\mathbb{R}_{\ge 0}^{e \in E^{u}(\Gamma)} \times \prod_{e \in E^{s_1 \vee s_2}(\Gamma)}\mathbb{R}_{\ge 0}^{2} \times \prod_{e \in E^{s_1 \cap s_2}(\Gamma)} \mathbb{R}_{\ge 0}^{2} \times_{\mathbb{R}_{\ge 0}} \mathbb{R}_{\ge 0}^2
		\]
		If we label the two edges in $\Gamma_i$ subdividing $e \in E^{s_1 \cap s_2}(\Gamma)$ by $e_i',e_i''$, as we've done above, the corresponding cone $\mathbb{R}_{\ge 0}^{2} \times_{\mathbb{R}_{\ge 0}} \mathbb{R}_{\ge 0}^2$ is explicitly 
		\[
		\left \{(\ell_{e_1}',\ell_{e_1}'',\ell_{e_2}',\ell_{e_2}''): \ell_{e_i}' + \ell_{e_i}'' = \ell_e \right\}
		\]
		This cone is responsible for an ``$xy = zw$" singularity in $\mathfrak{Q} \times_B \mathfrak{Q}$, and is subdivided in $\Sigma_{\mathfrak{Q}^{(2)}}$ into two unimodular simplices, along the hyperplane 
		\[
		\ell_{e_1}' = \ell_{e_2}'\,.
		\]
		We refer to \eqref{eq:atiyah} for the associated tropical picture.
		Hence, each cone in $\Sigma_{\mathfrak{Q}^{(2)}}$ is smooth. Furthermore, the subdivision is along hyperplanes, which means no additional divisors are introduced. This proves the corollary. 
	\end{proof}
	We return to the compactified Jacobian $\Jbar_C$ and the morphism $\varphi:\Jbar_C\to\mathfrak{Q}$ induced by \eqref{eq:varphi}.
	\begin{definition}\label{def:J2}
		The stack $\Jbar_C^{(2)}$ is defined by the fiber diagram
		\begin{equation}\label{eq:J2}
			\begin{tikzcd}
				\Jbar_C^{(2)}\ar[r]\ar[d,"f"'] & \mathfrak{Q}^{(2)}\ar[d]\\
				\Jbar_C\times_B\Jbar_C\ar[r,"\varphi\times\varphi"] &\mathfrak{Q}\times_B\mathfrak{Q}\,.
			\end{tikzcd}
		\end{equation}
	\end{definition}
	The log modification $f:\Jbar^{(2)}_C\to\Jbar_C\times_B\Jbar_C$ is the one appears in Theorem \ref{thm:Plog=arinkin}.
	The below horizontal arrow is strict by Lemma \ref{lem:smoothness}. Therefore the diagram \eqref{eq:J2} is an f.s. fiber diagram. By Proposition \ref{pro:logmod}, the morphism $f:\Jbar_C^{(2)}\to\Jbar_C\times_B\Jbar_C$ is a log modification. \\
	
	Recall that an algebraic log stack $X$ is called {\em log smooth} if the morphism $X\to \Sigma_X$ to a tropicalization is smooth. 
	\begin{proof}[Proof of Theorem \ref{thm:Plog=arinkin} (a)]
		The morphism $\varphi\times\varphi:\Jbar_C\times_B\Jbar_C\to\mathfrak{Q}\times_B\mathfrak{Q}$ is smooth, strict and log smooth by Lemma \ref{lem:smoothness}. Since the diagram \eqref{eq:J2} is Cartesian, the claim follows from Corollary \ref{cor:smoothandsmall}.
	\end{proof}    
	%    
	%\begin{proposition}\label{pro:model1}
	%    The birational model $\Jbar_C^{(2)}$ is smooth and log smooth. Moreover the complement of $J_C\times_B\Jbar_C\cup \Jbar_C\times_B J_C$ in $\Jbar_C^{(2)}$ has codimension 2.  
	%\end{proposition}
	A morphism $f:X\to Y$ between fine log schemes is called {\em integral} if for any geometric point $x\in X$, the induced morphism between commutative monoids $\overline{\Msf}_{Y,f(x)}\to\overline{\Msf}_{X,x}$ is integral.
	\begin{definition}\label{def:ss}
		A {\em log family} is a log smooth, integral, saturated morphism $f:X \to S$ between log schemes. A log family is {\em semistable} if in addition $X,S$ are smooth and log smooth. 
	\end{definition}
	\begin{proposition}\label{pro:model2}
		Let $f:\Jbar_C^{(2)}\to\Jbar_C\times_B\Jbar_C$ be the log modification in Definition \ref{def:J2}.
		\begin{enumerate}[label=(\alph*)]
			\item The two projections $\pi_1,\pi_2:\Jbar_C^{(2)}\to\Jbar_C$ are semistable. Moreover, the projection $\widetilde{\pi}:\Jbar_C^{(2)}\to B$ is also semistable. 
			\item For any geometric point $x$ of $\Jbar_C\times_B\Jbar_C$, $f^{-1}(x)$ is isomorphic to $\PP^1\times \cdots \times\PP^1$.
		\end{enumerate}
	\end{proposition}
	
	\begin{proof}
		(a) Since $\Jbar_C$ and $\Jbar_C^{(2)}$ are smooth, it suffices to show that the map is integral with reduced fibers. This conclusion follows directly from Corollary \ref{cor:smoothandsmall} and its proof. Specifically, under either projection $\pi_i$, the two cones arising from the subdivision of each square
		\[
		\{\ell_{e_1}',\ell_{e_1}'',\ell_{e_2}',\ell_{e_2}'': \ell_{e_1}' + \ell_{e_1}'' = \ell_{e_2}'+\ell_{e_2}''\} 
		\]
		along $\ell_{e_1}' = \ell_{e_2}'$ map surjectively onto their images in $\Sigma_{\mathfrak{Q}}$, with their integral structures likewise surjecting onto the corresponding images. The second assertion follows from the first because $\Jbar_C \to B$ is semistable. 
		
		(b) This result again follows from the proof of Corollary \ref{cor:smoothandsmall}, as the map $\Sigma_{\mathfrak{Q}^{(2)}} \to \Sigma_{\mathfrak{Q}} \times_{\Sigma_B}  \Sigma_{\mathfrak{Q}}$ is a product of the subdivisions along $\ell_{e_1}' = \ell_{e_2}'$.
		The exceptional fibers arising from these subdivisions are product of $\mathbb{P}^1$'s, ensuring the desired structure.
	\end{proof}
	
	A {\em compactified abelian fibration} $\overline{A}\to B$ is a log alteration of a log abelian scheme $\mathsf{LogAb}\to B$ which is of Deligne-Mumford type over $B$. For the definition of log abelian schemes, we refer to \cite{KKN08b}.
	\begin{proposition}\label{pro:compAb}
		$\widetilde{\pi}: \Jbar_C^{(2)}\to B$ is a compactified abelian fibration.
	\end{proposition}
	\begin{proof}
		Let $\logP_C$ denote the logarithmic Picard group as introduced in \ref{thm:MW}. The map $\widetilde{\pi}$ factors as 
		\[
		\Jbar_C^{(2)} \to \logP_C \times_B \logP_C \to B\,,
		\]
		where the first map is a log modification, and the second map is a log abelian scheme by \cite[Theorem 4.15.7]{MW22}. Therefore $\widetilde{\pi}$ is a compactified abelian fibration. 
	\end{proof}
	
	\subsection{Direct images of line bundles under hyperplane subdivisions} 
	%By Proposition \ref{pro:logmod}, the local model for the morphism $f: \Jbar_C^{(2)} \to \Jbar \times_B \Jbar$ is a product of small resolutions of the threefold ordinary double point singularity, or simply called the Atiyah flop,
	%\[
	%g: \widetilde{X} \to X= \mathrm{Spec}\, \CC[x,y,z,w]/(xy-zw)\,,
	%\]
	%whose exceptional fiber is a $\mathbb{P}^1$. Our goal in this subsection is to study the higher direct images $R^ig_*$ for a class of small resolutions similar to the Atiyah flop. 
	%
	We first consider the local situation.
	Let $X$ be an affine toric variety. We write $N,M$ for the character and cocharacter lattice of its torus respectively, and we write $\Sigma_X$ for its fan, which is a rational polyhedral cone in $N_{\mathbb{R}}$. To simplify the presentation, we assume in the rest of this subsection that $\Sigma_X$ is full dimensional in $N_{\mathbb{R}}$.\footnote{The assumption can be safely omitted, as otherwise $X$ splits as a product $X' \times T$ of a toric variety $X'$ that satisfies this hypothesis with a torus $T$, and any construction $C(X)$ we perform on $X$ in this subsection will also split as $C(X') \times T$.}  
	
	\begin{definition}\label{def:subdivision}
		A toric blowup $g: \widetilde{X} \to X$ is called a \emph{subdivision by a hyperplane} if it is the blowup corresponding to the subdivision of $\Sigma_X$ by a hyperplane: there is a $u \in M$ and $\Sigma_{\widetilde{X}}$ is the union of the two cones $\sigma_+:=\Sigma_{X} \cap \{v \in N_\mathbb{R}:u(v) \ge 0\}$ and $\sigma_:=\Sigma_X \cap \{v \in N_{\mathbb{R}}:u(v) \le 0\}$. 
		Moreover, a subdivision by a hyperplane is called \emph{saturated} if $g: \widetilde{X} \to X$ has reduced fibers. 
	\end{definition}
	
	A subdivision by a hyperplane is, equivalently, the domains of linearity of some convex (down) piecewise linear function $\phi$ whose bend locus is the hyperplane $u=0$. In terms of $\phi$, the subdivision $g:\widetilde{X} \to X$ has reduced fibers if and only if the value of $\phi(v)$ is $0$ or $\pm 1$ on the primitive vector of every ray of $\Sigma_X$ -- see the discussion in \cite[Section 4.4]{MW22}.  
	
	%\begin{definition}
	%    We say that a subdivision by a hyperplane is \emph{saturated} if $g: \widetilde{X} \to X$ has reduced fibers. 
	%\end{definition}
	\begin{example}\label{ex:atiyah}
		Our primary example of interest is the Atiyah flop for $X= \mathrm{Spec}\, \CC[x,y,z,w]/(xy-zw)$. The fan $\Sigma_{\widetilde{X}} :=\widetilde{\Sigma}$ of $\widetilde{X}$ is the cone over a square $\Pi$ subdivided into two unimodular triangles: 
		\begin{equation}\label{eq:atiyah}
			\begin{tikzpicture}
				
				\node[left] at (0,0){$w_{00}=(0,0,1)$};
				\node[draw,circle,inner sep=1pt,fill] at (0,0){};
				\node[right] at (2,0){$w_{10}=(1,0,1)$};
				\node[draw,circle,inner sep=1pt,fill] at (2,0){};
				\node[left] at (0,2){$w_{01}=(0,1,1)$}; 
				\node[draw,circle,inner sep =1pt,fill] at (0,2){};
				\node[right] at (2,2){$w_{11}=(1,1,1)$}; 
				\node[draw,circle,inner sep=1pt,fill] at (2,2){};
				\draw[thick](0,0)--(2,0)--(2,2)--(0,2)--(0,0);
				\draw[thick,blue](0,0)--(2,2);
				\node[right] at (3/4,1/2){$e_2$};
				\node[above] at (3/4,1){$e_1$};
			\end{tikzpicture}    
		\end{equation}
		These triangles are the domains of linearity of the piecewise linear function $\phi :=\min{(e_1,e_2)}$.
	\end{example}
	The fibers of a saturated subdivision by a hyperplane are points, except for the fiber over the fixed point of $X$, which is a $\mathbb{P}^1$. We refer to this $\mathbb{P}^1$ as the \emph{exceptional} fiber. If we translate $\phi$ by a linear function so that its values are $0$ on the separating hyperplane, $\phi$ is naturally an element of $u^{\perp} \cap M$, which is the cocharacter lattice of the exceptional $\mathbb{P}^1$. If $\phi$ is saturated, it becomes a generator of $\mathsf{Pic}(\mathbb{P}^1)$.  
	
	\begin{lemma}\label{lem:restric_pic}
		Let $g: \widetilde{X} \to X$ be a saturated subdivision by a hyperplane with the exceptional fiber $\PP^1\to \widetilde{X}$. Then the restriction map $\mathsf{Pic}(\widetilde{X}) \to \mathsf{Pic}(\mathbb{P}^1)$ is an isomorphism. 
	\end{lemma}

	\begin{proof}
		The Picard group of $\widetilde{X}$ is isomorphic to the quotient of the group of piecewise linear functions on the fan $\widetilde{\Sigma}:=\Sigma_{\widetilde{X}}$ by the group of linear functions, 
		\[
		\mathsf{Pic}(\widetilde{X})=\mathsf{PL}(\widetilde{\Sigma})/M
		\]
		A piecewise linear function on $\widetilde{\Sigma}$ is given by a pair of vectors $u_1,u_2 \in M$, such that $u_1=u_2$ along the separating hyperplane $\{u=0\}$. Translating by $u_2$, the piecewise linear function can be uniquely represented by the function $0,u_1-u_2$. But $u_1-u_2$ is an element of $u^{\perp} \cap M$ which is isomorphic to $\mathbb{Z}$ and naturally identified with $\mathsf{Pic}(\mathbb{P}^1$).
	\end{proof}
	
	In fact, when $g$ is saturated, the convex piecewise linear function $\phi$ when normalized to be $0$ along one of the cones $\sigma_+$ or $\sigma_-$ must have value $-1$ along the primitive vectors of the rays of the other cone that do not belong to the separating hyperplane, i.e. it corresponds to the generator $\mathcal{O}(-1)$ of $\mathsf{Pic}(\mathbb{P}^1$), i.e. $\mathsf{Pic}(\mathbb{P}^1) = \mathbb{Z}\phi$.
	\if0
	Via the isomorphism $\mathsf{Pic}(\widetilde{X})\cong \mathsf{PL}(\widetilde{\Sigma})/M$ (see Definition \ref{def:PP}), 
	Under the correspondence between $\mathsf{Pic}(\widetilde{X})$ and the group $\mathsf{PL}(\widetilde{\Sigma})/M$ of piecewise linear modulo linear functions, 
	%a generator for $\mathsf{Pic}(\widetilde{X})$ is given by the piecewise linear function $\phi$ whose values at vertices of \eqref{eq:atiyah} are
	%\[\phi(w_{00})=0, \, \phi(w_{10})=0, \, \phi(w_{01})=0,\,  \phi(w_{11})=1\,.\]
	%Then $\phi$ restricts to $\CO_{\PP^1}(-1)$ under the restriction map in Lemma \ref{lem:restric_pic}.
	
	%\[
	%  \begin{tikzpicture}
		%  
		%  \node[left] at (0,0){$0$};
		%  \node[draw,circle,inner sep=1pt,fill] at (0,0){};
		%  \node[right] at (2,0){$0$};
		%  \node[draw,circle,inner sep=1pt,fill] at (2,0){};
		%  \node[left] at (0,2){$0$}; 
		%  \node[draw,circle,inner sep =1pt,fill] at (0,2){};
		%  \node[right] at (2,2){$1$}; 
		%  \node[draw,circle,inner sep=1pt,fill] at (2,2){};
		%  \draw[thick](0,0)--(2,0)--(2,2)--(0,2)--(0,0);
		%  \draw[thick,blue](0,0)--(2,2);
		%\end{tikzpicture}
		%\]
		which restricts to $\mathcal{O}(-1)$ on the exceptional $\mathbb{P}^1$. 
		\fi
		\begin{proposition}\label{pro:cohomology_atiyah}
			Let $X$ be an affine toric variety and let $g:\widetilde{X} \to X$ be a separated subdivision by a hyperplane. If $L$ is a line bundle on $\widetilde{X}$ whose restriction to $\mathbb{P}^1$ is $\mathcal{O}_{\PP^1}(m)$ for some $m \ge -1$, then we have $H^i(\widetilde{X},L)=0$ for $i \ge 1$. Moreover, we have $R^ig_*L=0$ for all $i>0$.
		\end{proposition}
		
		\begin{proof}
			Let us write $\Sigma := \Sigma_X, \widetilde{\Sigma} = \Sigma_{\widetilde{X}}$, and recall that there is a vector $u \in M$ such that $\widetilde{\Sigma}$ is the union of the two cones 
			\[
			\sigma_+ = \{v \in \Sigma: \left \langle u,v \right \rangle \ge 0\},
			\sigma_- = \{v \in \Sigma: \left \langle u,v \right \rangle \le 0\}
			\]
			meeting along the separating hyperplane $\{u=0\}$. They are the domains of linearity of the convex down function $\phi$, which we normalize to be $0$ on $\sigma_-$.    Furthermore, the supports of $\Sigma$ and $\widetilde{\Sigma}$ are the same, and we may harmlessly identify them with $\Sigma$ itself. 
			
			By Lemma \ref{lem:restric_pic}, the line bundle $L$ corresponds to the piecewise function $\phi_m:=m\cdot\phi$ for $m\leq 1$. For $w \in M$, set 
			\[
			Z_w = \{v \in \Sigma: \left \langle w,v \right \rangle \le \phi_m(v)\}
			\]
			where $\Sigma$ is the support of $\widetilde{\Sigma}$, or, equivalently, the fan of $X$. As in\footnote{We use the opposite convention regarding the sign of the Weil divisor associated to a piecewise linear function, hence the reverse inequalities.} \cite[Section 3.5]{Fulton_toric}, we have 
			\[
			H^i(\widetilde{X},L) \cong \bigoplus_{w \in M} H^i(\Sigma,\Sigma-Z_w)\,.
			\]
			
			We describe the topology of $\Sigma-Z_w$. Since the origin $v = 0$ is always contained in $Z_w$, the complement of $Z_w$ deformation retracts onto its intersection with the polytopal complex obtained by slicing $\widetilde{\Sigma}$ at ``height 1", i.e. with the intersection of $\widetilde{\Sigma}$ with the level set of a generic linear function $L:\Sigma \to \mathbb{R}_{\ge 0}$ that takes each ray of $\Sigma$ to a positive number. We denote this polytopal complex by $\Pi=\widetilde{\Sigma} \cap L^{-1}(1)$. The possible topological types of the sets $\Pi-Z_w$ are highly constrained. For example, when $g$ is the Atiyah flop, $\Pi$ is a square, and the possible topological types depend on whether $w$ dominates $\phi_m$ or not on the vertices of $\Pi$. Degenerate cases are allowed, where a two-dimensional cell collapses to a lower-dimensional boundary cell. We illustrate the possible topological types below, with the regions $Z_w \cap \Pi$ highlighted in red:
			
			
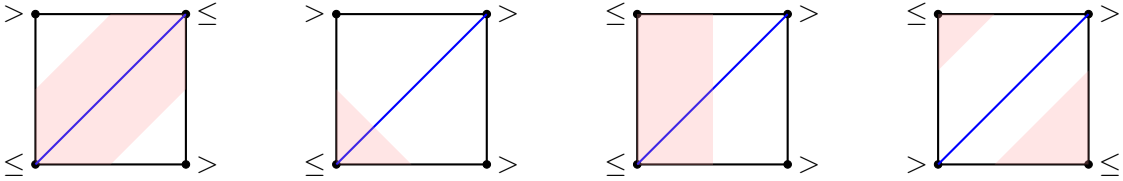
\[
			\begin{tikzpicture}
				
				\node[left] at (0,0){$\le$};
				\node[draw,circle,inner sep=1pt,fill] at (0,0){};
				\node[right] at (2,0){$>$};
				\node[draw,circle,inner sep=1pt,fill] at (2,0){};
				\node[left] at (0,2){$>$}; 
				\node[draw,circle,inner sep =1pt,fill] at (0,2){};
				\node[right] at (2,2){$\le$}; 
				\node[draw,circle,inner sep=1pt,fill] at (2,2){};
				\draw[thick](0,0)--(2,0)--(2,2)--(0,2)--(0,0);
				\draw[thick,blue](0,0)--(2,2);
				\fill[red!40, nearly transparent] (0,0)--(0,1) -- (1,2) -- (2,2) -- (2,1) -- (1,0) -- cycle;
				
				\node[left] at (4,0){$\le$};
				\node[draw,circle,inner sep=1pt,fill] at (4,0){};
				\node[right] at (6,0){$>$};
				\node[draw,circle,inner sep=1pt,fill] at (6,0){};
				\node[left] at (4,2){$>$}; 
				\node[draw,circle,inner sep =1pt,fill] at (4,2){};
				\node[right] at (6,2){$>$}; 
				\node[draw,circle,inner sep=1pt,fill] at (6,2){};
				\draw[thick](4,0)--(6,0)--(6,2)--(4,2)--(4,0);
				\draw[thick,blue](4,0)--(6,2);
				\fill[red!40, nearly transparent] (4,0)--(4,1)--(5,0) -- cycle;
				
				\node[left] at (8,0){$\le$};
				\node[draw,circle,inner sep=1pt,fill] at (8,0){};
				\node[right] at (10,0){$>$};
				\node[draw,circle,inner sep=1pt,fill] at (10,0){};
				\node[left] at (8,2){$\le$}; 
				\node[draw,circle,inner sep =1pt,fill] at (8,2){};
				\node[right] at (10,2){$>$}; 
				\node[draw,circle,inner sep=1pt,fill] at (10,2){};
				\draw[thick](8,0)--(10,0)--(10,2)--(8,2)--(8,0);
				\draw[thick,blue](8,0)--(10,2);
				\fill[red!40, nearly transparent] (8,0) -- (8,2) -- (9,2) -- (9,0) -- cycle;
				
				\node[left] at (12,0){$>$};
				\node[draw,circle,inner sep=1pt,fill] at (12,0){};
				\node[right] at (14,0){$\le$};
				\node[draw,circle,inner sep=1pt,fill] at (14,0){};
				\node[left] at (12,2){$\le$}; 
				\node[draw,circle,inner sep =1pt,fill] at (12,2){};
				\node[right] at (14,2){$>$}; 
				\node[draw,circle,inner sep=1pt,fill] at (14,2){};
				\draw[thick](12,0)--(14,0)--(14,2)--(12,2)--(12,0);
				\draw[thick,blue](12,0)--(14,2);
				\fill[red!40, nearly transparent] (12,2) -- (12,1.25) -- (12.75,2) -- cycle;
				\fill[red!40, nearly transparent] (12.75,0) -- (14,0) -- (14,1.25) -- cycle; 
			\end{tikzpicture}
			\]
			
			In general, for $m \le 0$, the function $m\phi$ is convex up, and hence the sets
			\[
			\Sigma-Z_w = \{v \in \Sigma: \left \langle w, v \right \rangle > \phi_m(v)\}
			\]
			are convex. In the case $m=1$, $\Sigma-Z_w$ is not necessarily convex, but its intersection with either piece $\Pi_{\pm} = \Pi \cap \sigma_{\pm}$ is convex, as it is the intersection of a polyhedron with a half-space. Furthermore, the intersection of 
			\[
			(\Sigma-Z_w) \cap \Pi_+ \cap \Pi_- = (\Sigma-Z_w) \cap \{u=0\} \cap \Pi
			\]
			is the set 
			\[
			\{v \in \Pi \cap \{u=0\}: \left \langle w,v \right \rangle > 0\}
			\]
			which is convex and connected (albeit perhaps empty). Therefore, in any case, the complement of $Z_w$ is contractible. Furthermore, it is connected unless $Z_w$ contains the separating hyperplane $\{u=0\}$, in which case it has exactly two components.  
			
			From the long exact sequence of relative cohomology 
			\[
			\begin{tikzcd}
				0 \ar[r] & H^0(\Sigma,\Sigma-Z_w) \ar[r] & H^0(\Sigma) = \mathbb{C} \ar[r] & H^0(\Sigma-Z_w) \ar[r] & H^1(\Sigma,\Sigma-Z_w) \ar[r] & 0 \,,
			\end{tikzcd}
			\]
			we find that $H^i(\Sigma,\Sigma-Z_w)$ is $0$ for all $i \ge 2$, and $0$ for $i=1$ unless the complement of $Z_w$ has two components, in which case 
			\[
			H^i(\Sigma,\Sigma-Z_w) \cong \mathbb{C}\,.
			\]
			Now, let $w$ be a character for which $Z_w$ contains the separating hyperplane $\{u=0\}$ (which can only happen for $m=1$, as otherwise we've seen that $\Sigma-Z_w$ is convex). By our normalization hypothesis on $\phi$ we have $\phi=0$ on $\sigma_-$ and $\phi \le 0$ on $\sigma_+$. This means that $w$ is 
			\begin{itemize}
				\item $> 0$ on some ray of $\sigma_-$ 
				\item $\le 0$ on the separating hyperplane $\{u=0\}$. 
				\item $> \phi$ on $\sigma_+$.  
			\end{itemize}
			Let $v$ be the primitive vector along a ray of $\sigma_+$ on which $\phi(v) \neq 0$. By the hypothesis that the subdivision is saturated, $\phi(v)=-1$. Now, the linear function $w$ has negative slope in the direction connecting the separating hyperplane with $v$, as it has to rise away from the separating hyperplane towards $\sigma_-$. But $w$ also takes integral values on integral points, so if its value is $\le 0$ along the separating hyperplane, it must be $\le -1$ along $v$. Therefore, for $m=1$, it is impossible to find such a $w$. 
			
			Therefore, if $m\le 1$, for any $w$, $\Sigma-Z_w$ is contractible and connected. 
			%\begin{itemize}
			%    \item For $m \le 1$, and any $w$, $\Sigma-Z_w$ is contractible and connected. 
			%\end{itemize}
			From the long exact sequence, it follows that $H^i(\widetilde{X},L)=0$ for all $i>0$.
			%Write $u=ae_1^*+be_2^*+ce_3^*$ in the standard basis of $M$. We now check for which choices of $u\in M$,  $Z_w$ is of the first type. Since the values of $u - \phi_k$ at the vertices are 
			%\[(u-\phi_k)(w_{00})=c, \, (u-\phi_k)(w_{10})=a+c, (u-\phi_k)(w_{01})=b+c,\, (u-\phi_k)(w_{11})=a+b+c-k\,,\]
			%\[
			%\begin{tikzpicture}
			%    \node[left] at (0,0){$c$};
			%  \node[draw,circle,inner sep=1pt,fill] at (0,0){};
			%  \node[right] at (2,0){$a+c$};
			%  \node[draw,circle,inner sep=1pt,fill] at (2,0){};
			%  \node[left] at (0,2){$b+c$}; 
			%  \node[draw,circle,inner sep =1pt,fill] at (0,2){};
			%  \node[right] at (2,2){$a+b+c-k$}; 
			%  \node[draw,circle,inner sep=1pt,fill] at (2,2){};
			%  \draw[thick](0,0)--(2,0)--(2,2)--(0,2)--(0,0);
			%  \draw[thick,blue](0,0)--(2,2);
			%  \end{tikzpicture}
		%\]
		%we find that $u - \phi_k \le 0$ is of the first type if and only if 
		%\[
		%c \le 0,\, a+c > 0,\, b+c>0,\, \text{ and } a+b+c-k \le 0\,.
		%\]
		%The first and the third inequalities imply that $b \ge 1$, and hence combining with the second, they imply $a+b+c \ge 2$. Therefore, $a + b + c - k \le 0$ requires $k \ge 2$, which proves the proposition. 
		
		For the second claim, since $X$ is affine, Grothendieck's spectral sequence gives $H^0(X,R^ig_*L) \cong H^i(\widetilde{X},L)$. We have just seen that this group vanishes for $i>0$. Since $R^ig_*L$ is a coherent sheaf on the affine variety $X$, hence determined by its global sections, it also vanishes for $i>0$. 
	\end{proof}
	Now, we discuss the global situation. Let $f:Y \to X$ be a log modification. Generalizing Definition \ref{def:subdivision}, we say that $f$ is a {\em saturated subdivision by hyperplanes} if it is \'etale locally pulled back from a saturated subdivision by hyperplanes of toric varieties. 
	%\begin{definition}
	%    Let $f:Y \to X$ be a log modification. We say that $f$ is a saturated subdivision by hyperplanes if it is locally pulled back from a saturated subdivision by hyperplanes of toric varieties. 
	%\end{definition}
	\begin{proposition}
		\label{pro: nefpushforward}
		Let $f: Y \to X$ be a a saturated subdivision by hyperplanes between log smooth varieties, and let $L$ be a line bundle on $Y$. Assume that the restriction of $L$ to any fiber of $f$ has vanishing higher cohomology. Then we have $Rf_*L \cong f_*L$. 
	\end{proposition}
	\begin{proof} 
		We prove the higher direct images $R^if_*L$ vanish for all $i>0$. By Krull's intersection theorem, the statement can be checked formally locally on $X$, so we can assume $X$ is the spectrum of a complete Noetherian local ring. Since $X$ is log smooth, and since $Y$ is a log blowup of $X$, we have a Cartesian diagram 
		\[
		\begin{tikzcd}
			Y \ar[r] \ar[d,"f"] & \widetilde{Z} \ar[d,"g"] \\ 
			X \ar[r,"p"] & Z
		\end{tikzcd}
		\]
		with $Z$ a toric variety, $g$ a toric blowup -- which by hypothesis must be a saturated subdivision along a hyperplane --, and the map $p$ \'etale. By \cite[Lemma 4.4.12.2]{MW22}, the Picard group of $Y$ and the Picard group of $\widetilde{Z}$ are isomorphic, and both coincide with the Picard group of an exceptional $\mathbb{P}^1$. We may thus identify $L$ with a line bundle on $\widetilde{Z}$. By Proposition \ref{pro:cohomology_atiyah}, since the cohomology of $L$ vanishes on the fibers of $f$, it vanishes on the exceptional fibers of $g$ as well. Thus we have $R^ig_*L = 0$ for $i>0$. Since $X$ and $Z$ are formally locally isomorphic, we have 
		\[
		R^if_*L \cong (R^ig_*L)_{\widehat{p(X)}} = 0\,, i>0\,,
		\]
		which gives the result.
	\end{proof}
	For our purposes, we need the following strengthening of the previous result. 
	\begin{proposition}\label{pro:vanishing_higher}
		Let $f:Y \to X$ be a log modification between log smooth varieties which is locally the pullback of a product of saturated subdivisions by hyperplanes. Let $L$ be a line bundle on $Y$ such that for each closed point $x \in X$, $H^i(f^{-1}(x),L|_{f^{-1}(x)}) = 0$ for all $i>0$. Then $Rf_*L \cong f_*L$. 
	\end{proposition}
	
	\begin{proof}
		By the same argument in Proposition \ref{pro: nefpushforward}, we reduce to the toric case. Then the result follows by combining Proposition \ref{pro: nefpushforward} with the K\"unneth formula. 
	\end{proof}
	\subsection{Proof of Theorem \ref{thm:Plog=arinkin}}\label{sec:desingular}
	%\subsection{Desingularization of Arinkin's kernel}\label{sec:desingular}
	We show that the canonical extension of the Poincar\'e line bundle on the birational model constructed in Section \ref{sec:birational} provides a desingularization of Arinkin's kernel (\cite{arinkin2}). %Our construction was motivated from the duality of logarithmic Picard group of log smooth curves, see Remark \ref{rmk:logPic}. 
	In contrast to Arinkin's approach, our construction offers a direct description of the kernel, which will be a crucial input in Section \ref{sec:5}.
	
	The notion of Cohen-Macaulay sheaves can be expressed in terms of the dualizing complex. Denote by $\dbcoh(X)$ the bounded derived category of coherent sheaves on $X$, and let $\dual_X \in \dbcoh(X)$ be a dualizing complex for $X$.  We normalize $\dual_X$ so that its stalk at a generic point of $X$ has nontrivial cohomology only in degree $0$. When $X$ is Gorenstein, the dualizing complex  is an invertible sheaf. 
	Consider the duality functor
	\[\DD:\dbcoh(X)\xrightarrow{\cong}\dbcoh(X), \,  \mathcal{E}\mapsto R\Hom_{\CO_X}(\mathcal{E},\dual_X)\,.\]
	A coherent sheaf $\mathcal{E}$ on $X$ is called {\em Cohen-Macaulay} of codimension $d$ if and only if
	\begin{equation}\label{eq:CM_dual}
		h^i(\DD(\mathcal{E})) = 0, \, i\neq d\,.
	\end{equation}
	A Cohen-Macaulay sheaf of codimension zero is called \emph{maximal}. The following extension property will be used frequently (\cite[Theorem 5.10.5]{EGA_IV_II}).
	\begin{lemma}\label{lem:extension2}
		Let $\mathcal{F}$ be a maximal Cohen-Macaulay sheaf on $X$ and let $Z\subset X$ be a closed subscheme with codimension  $\geq 2$. Let $j: X\setminus Z\to X$ be an open embedding. Then the canonical morphism $\mathcal{F}\to j_*j^*\mathcal{F}$ is an isomorphism.
	\end{lemma}
	%\begin{proof}
	%    The assertion in fact only requires that $\mathcal{F}$ satisfies Serre's $(S_2)$ condition  by \cite[Theorem 5.10.5]{EGA_IV_II}.
	%\end{proof}
	We go back to the construction in Section \ref{sec:birational}. By Theorem \ref{thm:Plog=arinkin} (a), the codimension of the open embedding \eqref{eq:open} is two and the total space $\Jbar^{(2)}_C$ is smooth. Consequently, the Poincare line bundle $\P$ on $J_C\times_B\Jbar_C\cup\Jbar_C\times_B J_C$ extends to the unique line bundle:
	\begin{equation}\label{eq:tildeP}
		\widetilde{\P}\in\mathrm{Pic}(\Jbar^{(2)}_C)\,.
	\end{equation}
	We call $\widetilde{\P}$ the {\em extended Poincar\'e line bundle}.

	\if0
	\begin{theorem}\label{thm:explicit}
		Let $f:\Jbar_C^{(2)}\to\Jbar_C\times_B\Jbar_C$ and let $\widetilde{\P}$ be the line bundle extending $\P$. Then
		\begin{enumerate}[label=(\alph*)]
			\item For $i>0$, we have $R^if_*\widetilde{\P}=0$.
			\item $f_*\widetilde{\P}$  is a maximal Cohen-Macaulay sheaf.
		\end{enumerate}
	\end{theorem}
	\fi
	Before presenting the proof of Theorem \ref{thm:Plog=arinkin}, we explain the connection to Arinkin's work. By Arinkin \cite{arinkin2} (see also \cite{MRV2}), the Poincar\'e line bundle $\P$ admits a unique maximal Cohen-Macaulay extension:
	\begin{equation}\label{eq:arinkin}
		\overline{\P}\in\mathrm{Coh}(\Jbar_C\times_B\Jbar_C)\,.
	\end{equation}
	This extension is obtained in loc.cit. from flat descent of a maximal Cohen-Macaulay sheaf on the isotropic Hilbert scheme of points. Theorem \ref{thm:Plog=arinkin} can be thought of as a desingularization of Arinkin's kernel.
	\begin{corollary}
		Let $\widetilde{\P}$ be the unique extension of the Poincar\'e line bundle \eqref{eq:tildeP} and let $\overline{\P}$ be the maximal Cohen-Macaulay sheaf \eqref{eq:arinkin}. Then we have
		\[Rf_*\widetilde{\P} \cong \overline{\P} \in \dbcoh(\Jbar_C\times_B\Jbar_C)\,.\]
	\end{corollary}
	\begin{proof}
		By Theorem \ref{thm:Plog=arinkin}, $f_*\widetilde{\P}\cong Rf_*\widetilde{\P}$ is a maximal Cohen-Macaulay sheaf which restricts to $\P$ on  $J_C\times_B\Jbar_C\cup\Jbar_C\times_B J_C$. By Theorem \ref{thm:Plog=arinkin} and Lemma \ref{lem:extension2}, we have $f_*\widetilde{\P}\cong\overline{\P}$.
	\end{proof}
	%\begin{proof}
	%    Consider the diagram
	%    \[
	%    \begin{tikzcd}
		%        & \Jbar_C^{(2)}\ar[d,"f"] \\
		%        \Jsm_C\times_B\Jbar_C\cup \Jbar_C\times_B\Jsm_C\ar[r,hookrightarrow,"\iota"]\ar[ur,hookrightarrow,"\iota'"] & \Jbar_C\times_B\Jbar_C\,.
		%    \end{tikzcd}
	%    \]
	%    By Proposition \ref{pro:model1}, the comdimension of complements of $\iota$ and $\iota'$ are two. Therefore by the extension property \eqref{eq:CM} we get
	%    \[
	%    \overline{\P}\cong\iota_*\P = f_*\iota'_*\P \cong f_*\widetilde{\P} \cong Rf_*\widetilde{\P}\,.
	%    \]
	%    The last equivalence follows from \eqref{eq:higher_vanish}.
	%\end{proof}
	In the remainder of this section, we prove Theorem \ref{thm:Plog=arinkin} (c). We begin by providing a modular description of  $\widetilde{\P}$. By pulling back the universal curve over $\mathfrak{Q}^{(2)}$ along the morphism $\Jbar_C^{(2)}\to\mathfrak{Q}^{(2)}$ in \eqref{eq:J2}, we obtain a semi-stable curve:
	\begin{equation}\label{eq:univ_curve}
		\widehat{C}\to\Jbar_C^{(2)}\,.
	\end{equation}
	\begin{proposition}\label{pro:pairing}
		Let $C_1,C_2$ be the two universal quasi-stable curves over $\Jbar_C\times_B\Jbar_C$ and $\mathcal{L}_1,\mathcal{L}_2$ be the two universal line bundles on $C_1,C_2$. Let $\widehat{C}\to \Jbar_C^{(2)}$ be the semi-stable curve \eqref{eq:univ_curve}. Then the unique extension \eqref{eq:tildeP} has the form $\widetilde{\P}\cong\langle\mathcal{L}_1|_{\widehat{C}},\mathcal{L}_2|_{\widehat{C}}\rangle$.
	\end{proposition}
	\begin{proof}
		By Theorem \ref{thm:Plog=arinkin} (a) and Lemma \ref{lem:extension2}, it is enough to show that $\widetilde{\P}$ restricts to the Poincare line bundle $\P$ on $J_C\times_B\Jbar_C\cup\Jbar_C\times_B J_C$ which is immediate from Definition \ref{def:Poincare_line}. 
	\end{proof}
	Theorem \ref{thm:Plog=arinkin} (b) is the consequence of Proposition \ref{pro:pairing}.
	\begin{proof}[Proof of Theorem \ref{thm:Plog=arinkin} (c)]
		A geometric log point $x \in \Jbar_C \times_B \Jbar_C$ corresponds to a quadruple $(C_1,L_1,C_2,L_2)$ of two quasistable log curves with stabilization $C$, and stable line bundles $L_i$ on $C_i$. By Proposition \ref{pro:model2} (b), the fiber of $\Jbar^{(2)}_C$ over $x$ is a product of $\mathbb{P}^1$s, one for each edge $e$ of the dual graph of $C$ that is subdivided simultaneously in the dual graphs of both $C_1$ and $C_2$. It will be clear from our argument that we may work one such edge at a time, so we assume that one edge $e$ has been subdivided, and $f^{-1}(x)=\mathbb{P}^1$. We note that the curves $C_1,C_2$ are not isomorphic as log curves over $x$, but their underlying schemes are isomorphic, obtained by replacing the node of $C$ corresponding to $e$ by a $\mathbb{P}^1$. We may thus write 
		\[
		C_1 = C_2 = C' \cup_{\{n_1,n_2\}} \mathbb{P}^1
		\] 
		with $n_1,n_2$ two nodes, and $\mathbb{P}^1$ the unique unstable component. Since $C_1,C_2$ are curves over a geometric point, we can represent each line bundle $L_i$ as a divisor
		\[
		D'_i + p_i
		\]
		where $D'_i$ is supported on the smooth locus of $C'$ and $p_i$ is a smooth point of the unstable $\mathbb{P}^1$. 
		By construction, the curve $\widehat{C}$ restricted over $f^{-1}(x)=\mathbb{P}^1$ is the blowup of $C_1 \times \mathbb{P}^1$ at $(n_1,0),(n_2,\infty)$, where $n_1,n_2$ are the two nodes of the exceptional $\mathbb{P}^1$ in $C_1$. Now, the pullbacks of $D_i'$ and of $p_i$ remain in the smooth locus $C' \times \mathbb{P}^1$ and its complement respectively. Therefore, $D_i'$ stay disjoint from the $p_j$ and we have 
		\[
		\left \langle D_i', p_j \right \rangle = 0 
		\]
		for $i,j=1,2$. Similarily, $\left \langle D_1',D_2' \right \rangle$ can be calculated inside any proper subvariety of $\widehat{C}$ that contains the smooth locus of $C' \times \mathbb{P}^1$, for example inside $C' \times \mathbb{P}^1$, while $\left \langle p_1,p_2 \right \rangle$ can be calculated inside the complement of the smooth locus of $C' \times \mathbb{P}^1$, which is the blowup $Bl(\mathbb{P}^1 \times \mathbb{P}^1)$ of $
		\mathbb{P}^1 \times \mathbb{P}^1$ at $(n_1,0)$ and $(n_2,\infty)$. Since $C' \times \mathbb{P}^1$ and $D_1',D_2'$ are pulled back from a point, $\left \langle D_1', D_2' \right \rangle = 0$. 
		
		\begin{figure}
			\centering
			
			\tikzset{every picture/.style={line width=0.75pt}} %set default line width to 0.75pt        
			
			\begin{tikzpicture}[x=0.75pt,y=0.75pt,yscale=-.75,xscale=.75]
				
				%Straight Lines [id:da7450771588160526] 
				\draw [color={rgb, 255:red, 0; green, 0; blue, 0 }  ,draw opacity=1 ]   (189.33,123) -- (202,206.67) ;
				\node[left] at (195, 170){$V_0$};
				%Straight Lines [id:da9411523537804128] 
				\draw [color={rgb, 255:red, 0; green, 0; blue, 0 }  ,draw opacity=1 ]   (208.33,72) -- (188,142.67) ;
				\node[left] at (200,100){$E_0$};
				%Straight Lines [id:da8520552533345895] 
				\draw [color={rgb, 255:red, 0; green, 0; blue, 0 }  ,draw opacity=1 ]   (493.33,122) -- (506,205.67) ;
				\node[right] at (500,170){$E_\infty$};
				%Straight Lines [id:da5668861127437494] 
				\draw [color={rgb, 255:red, 0; green, 0; blue, 0 }  ,draw opacity=1 ]   (512.33,71) -- (492,141.67) ;
				\node[right] at (510,105){$V_\infty$};
				%Curve Lines [id:da2203607102133055] 
				\draw [color={rgb, 255:red, 208; green, 2; blue, 27 }  ,draw opacity=1 ]   (201,101.67) .. controls (241,71.67) and (460.33,205) .. (500.33,175) ;
				%Curve Lines [id:da7762344323950362] 
				\draw [color={rgb, 255:red, 74; green, 144; blue, 226 }  ,draw opacity=1 ]   (197,176.67) .. controls (231.33,215) and (478.67,45.33) .. (504.33,95) ;
				%Straight Lines [id:da9501737448499569] 
				\draw    (208.33,72) -- (512.33,71) ;
				\node[above] at (350,71){$H_1$};
				%Straight Lines [id:da9311743242844366] 
				\draw    (202,206.67) -- (506,205.67) ;
				\node[below] at (350,206){$H_2$};
				%Straight Lines [id:da4958301969736413] 
				\draw    (202.33,242) -- (509.33,242) ;

			\end{tikzpicture}
			\caption{The blowup of $\mathbb{P}^1 \times \mathbb{P}^1$ at $(n_1,0),(n_2,\infty)$ and its torus invariant divisors.}
		\end{figure}
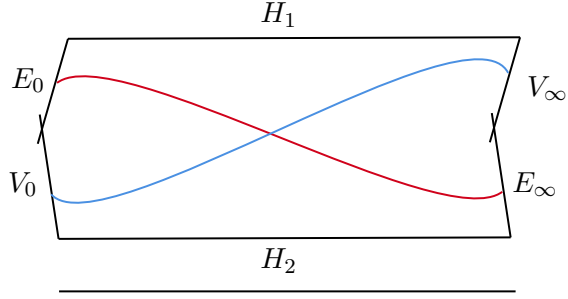
		
		On the other hand, 
		\[
		\mathsf{Pic}(Bl(\mathbb{P}^1 \times \mathbb{P}^1)) = \mathbb{Z}^4
		\]
		is generated by the classes of two horizontal divisors $H_1,H_2$, and four vertical divisors $V_0,V_\infty,E_0,E_\infty$ over $0$ and $\infty$ with relations 
		\[
		H_1 +E_0 = H_2 + E_\infty, V_0 + E_0 = V_\infty + E_\infty\,.
		\]
		Up to renaming, we may arrange so that 
		\[
		H_1\cdot E_\infty = H_2 \cdot V_\infty = 0 \,.
		\]
		In our naming convention, we are thinking of $H_1$ as the strict transform of $n_1 \times \mathbb{P}^1$, $H_2$ as the strict transform of $n_2 \times \mathbb{P}^2$, and of $E_0,E_\infty$ as the exceptional divisors of the blowup $\widehat{C}$ of $C_1 \times \mathbb{P}^1$. By construction, $p_1$ has relative degree $1$ over $\mathbb{P}^1$ and intersects $V_0,V_\infty$ transversely, while $p_2$ has relative degree $1$ and intersects $E_0,E_\infty$ transversely. A simple calculation in $\mathsf{Pic}(Bl(\mathbb{P}^1 \times \mathbb{P}^1))$ shows that these properties uniquely characterize $p_1$ as the line bundle equivalent to $H_2 + E_\infty$ while $p_2$ as $H_1 + V_\infty$. Then, we have
		\[
		\left \langle H_2 + E_\infty, H_1 + V_\infty \right \rangle = f_*(E_\infty \cdot V_\infty) = [\infty]
		\]
		which has degree $1$ on $\mathbb{P}^1$. 
		
		In general, we conclude that
		\begin{equation}\label{eq:restric_P}
			\left \langle \mathcal{L}_1|_{\widehat{C}}, \mathcal{L}_2|_{\widehat{C}} \right \rangle|_{f^{-1}(x)} \cong \CO_{\PP^1}(1) \boxtimes \cdots \boxtimes \CO_{\PP^1}(1)\,,
		\end{equation}
		where we identified $f^{-1}(x) \cong \mathbb{P}^1 \times \cdots \times\mathbb{P}^1$ using Proposition \ref{pro:model2}.

		We may now prove $R^if_*\widetilde{\P}=0$ for $i>0$. By Proposition \ref{pro:vanishing_higher}, it is enough to show that for any $x\in \Jbar_C\times_B\Jbar_C$, we have $H^{i}(f^{-1}(x),\widetilde{\P})=0$ for all $i>0$. But this follows from \eqref{eq:restric_P}.
		
		We show that $f_*\widetilde{\P}$ is a Cohen-Macaulay sheaf using the cohomological criteria \eqref{eq:CM_dual}. By the Grothendieck duality along $f$ (see, e.g., \cite{Neeman}), we have the equivalence $Rf_*\circ \DD_{\Jbar_C^{(2)}} \cong \DD_{\Jbar_C\times_B\Jbar_C} \circ Rf_*$.
		%\[Rf_*R\hom(x,f^!y) \cong R\hom(Rf_*x, y)\,.\]
		%Let $q:\Jbar_C\times_B\Jbar_C\to \spec \, k$ be the structure morphism. Then 
		%\[\omega_{\Jbar_C^{(2)}}^\bullet = q^!\, k [- \text{rel dim}]\,.\]
		Since $f$ is birational, there is no degree shift in this equivalence.
		Applying this and (a), we have
		\[\DD(f_*\widetilde{\P}) \cong \DD(Rf_*\widetilde{\P}) \cong Rf_* \DD(\widetilde{\P}) \,.\]
		The total space $\Jbar_C^{(2)}$ is smooth by Theorem \ref{thm:Plog=arinkin} (a), and hence we can express $\DD(\widetilde{\P}) \cong \widetilde{\P}^\vee\otimes\omega_{\Jbar^{(2)}_C}$. The dualizing line bundle $\omega_{\Jbar^{(2)}_C}$ restricts trivially to each contracted $\PP^1$  by the adjunction formula. Using the K\"unneth formula, for any $x\in \Jbar_C\times_B\Jbar_C$,  we have
		\[H^{i}(f^{-1}(x),\widetilde{\P}^\vee\otimes\omega_{\Jbar_C^{(2)}}) = H^{i}(\PP^1\times\cdots \times\PP^1, \widetilde{\P}^\vee)=0, \, i>0\,,\]
		where the vanishing follows because $\widetilde{\P}^\vee$ has degree $-1$ on each $\PP^1$ by \eqref{eq:restric_P}. Therefore, by  Proposition \ref{pro:vanishing_higher}, the complex $\DD(f_*\widetilde{\P})$ is a coherent sheaf. By \eqref{eq:CM_dual}, the pushforward $f_*\widetilde{\P}$ is a maximal Cohen-Macaulay sheaf on $\Jbar_C\times_B\Jbar_C$.

        We prove the flatness of the coherent sheaf $f_*\widetilde{\P}$. For $i=1,2$, let $\pi_i : \Jbar_C\times_B\Jbar_C \to \Jbar_C$ be projection to $i$-th factor. Since $\Jbar_C\times_B\Jbar_C$ is l.c.i, it is in particular Cohen-Macaulay. We have already shown that the sheaf $f_*\widetilde{\P}$ is maximal Cohen–Macaulay, and therefore its $\pi_i$-flatness follows from Proposition \ref{pro:ca} below.
	\end{proof}
Finally, we are left to prove the following statement from commutative algebra:
\begin{proposition}\label{pro:ca}
    Let $(A,\m_A)$ be a regular local ring and $(B,\m_B)$ be a Cohen-Macaulay local ring. Let $\phi : (A,\m_A)\to (B,\m_B)$ be a flat local ring homomorphism. Let $M$ be a finitely generated maximal Cohen-Macaulay $B$-module. Then $M$ is a flat $A$-module.
\end{proposition}
\begin{proof}
    By the local criterion for flatness \cite[00MK]{stacks-project}, it is enough to show that
    \[
    \mathrm{Tor}^A_1(A/\m_A, M) = 0\,.
    \]
    Since $A$ is a regular local ring, we may write $\m_A= (x_1,\cdots, x_n)$, where $x_1, \cdots, x_n$ is an $A$-regular sequence. Let $K_\bullet(x_1,\cdots, x_n)$ be the Koszul sequence resolving the residue field $A/\m_A$. Let $y_i = \phi(x_i) \in B$ denote the image. Because $\phi$ is flat, the Koszul sequence of $y_1,\cdots, y_n$ is the base change of $K_\bullet(x_1,\cdots, x_n)$ along $\phi : A\to B$, and therefore
    \[
    \mathrm{Tor}^A_1(A/\m_A, M) = H_1 (K_\bullet(x_1,\cdots, x_n)\otimes_A M) = H_1 (K_\bullet(y_1,\cdots, y_n)\otimes_B M)\,.\]
    By \cite[062F]{stacks-project}, it suffices to prove that $y_1,\cdots, y_n$ is an $M$-regular sequence.

    Since $\phi$ is faithfully flat, the images $y_i = \phi(x_i)$ form a $B$-regular sequence \cite[00LM]{stacks-project}. Therefore it is enough to show that any $B$-regular sequence is $M$-regular. When $n=1$, it reduces to proving $\mathrm{Ass}(B)= \mathrm{Ass}(M)$. The inclusion $\mathrm{Ass}(B)\subset \mathrm{Ass}(M)$ is clear. Since $B$ is Cohen-Macaulay, $\mathrm{Ass}(B)= \mathrm{Min} (B)$ \cite[Theorem 2.1.6]{BH}. Conversely, if $\p \in \mathrm{Ass}(M)$, then 
    \[\dim B/\p = \mathrm{depth}\, M= \dim B\]
    where the first equality follows from \cite[Theorem 2.1.2 (a)]{BH}, and the second holds because $M$ is maximal Cohen-Macaulay. Therefore $\p \in\mathrm{Min} (B) = \mathrm{Ass} (B)$. The general case follows from the induction on $n$, after passing to the quotients $M/(y_1,\cdots, y_{n-1})$ over $B/(y_1,\cdots, y_{n-1})$, where the Cohen-Macaulay assumption on the quotients follows from \cite[Theorem 2.1.3]{BH}. Hence $y_1,\cdots,y_n$ is $M$-regular, and consequently $\mathrm{Tor}^A_1(A/\m_A,M) = 0$. Thus $M$ is flat $A$-module.
\end{proof}
	%\begin{remark}\label{rmk:logPic}
	%    We make connection to the duality of logarithmic Picard group. It is expected that there exists a {\em logarithmic Poincar\'e line bundle}
	%    \[\P^{\mathrm{log}}\to \lpic_C\times_B\lpic_C\]
	%    which induces the duality on $\lpic_C$. Theorem \ref{thm:explicit} provides a representation of this logarithmic Poincaré line bundle on a birational model.
	%\end{remark}
	\section{Todd class of residue sheaves}\label{sec:4}
	\subsection{Residue sheaf of a semistable morphism}
	%Semistable morphisms were introduced in  Definition \ref{def:ss}. By \cite[Corollary 4.5]{Kato}, the underlying morphism of a semistable morphism is flat and is a local complete intersection (l.c.i.).
	We recall the local structure of semistable morphisms (Definition \ref{def:ss}). Let $f:X\to S$ be a semistable morphism with $\dim X=n, \dim S=l$. \'Etale locally on the domain, there exists non-negative integers $n$ and $l$, a partition $A$ of $n$ into $l$ parts $n_i$, and a chart for $f$ of the form 
	\begin{equation}\label{eq:ss_loc}
		\prod_{i=1}^l \mathbb{A}^{n_i} \to \prod_{i=1}^l \mathbb{A}^1\,, \text{ given by }
		t_i = \prod_{\alpha \in n_i} x_\alpha
	\end{equation}
	by \cite[Theorem 16]{CMOP} and the proof therein.
	
	%\'Etale locally, $\Omega_{f,\log}^1$ is given by the quotient
	%\[\Omega_{f,\log}^1:=\Omega_f^1\oplus (\CO_X\otimes \Msf_X^\gp)/\sim\]
	%where relations are generated by
	%\[(d\alpha(a),0)-(0,\alpha(a)\otimes a), a\in \Msf_X\]
	%and
	%\[(0,1\otimes a), \,  a\in \text{Im} (f^{-1}\Msf_S\to \Msf_X)\]
	For a semistable morphism $f:X\to S$, the sequence
	\begin{equation}\label{eq:vTan}
		0\to f^*\Omega_Y^1\to \Omega_X^1\to \Omega_f^1\to 0
	\end{equation}
	is exact. Consequently, the Todd class of the virtual tangent bundle associated with $f$ is the dual of the Todd class of $\Omega_f^1$.  For a morphism $f:X\to S$ between log schemes, one can associate the logarithmic cotangent sheaf $\Omega_{f,\log}^1$ (\cite[(1.7)]{Kato}).  
	When $f$ is log smooth, the logarithmic cotangent sheaf $\Omega_{f,\log}^1$ is locally free. 
	\begin{lemma}\label{lem:residue0}
		Let $f:X\to S$ be a semistable morphism. Then the sequence
		\begin{equation}\label{eq:residue}
			0\to \Omega_f^1\to \Omega^1_{f,\log}\to \Omega^1_{f,\log}/\Omega^1_{f}\to 0
		\end{equation}
		is exact. In this case, we call $\mathcal{R}_f := \Omega_{f,\log}^1/\Omega^1_f$ the \emph{residue sheaf} of $f$.
	\end{lemma}
	\begin{proof}
		By the definition of $\Omega_{f,\log}^1$, there exists a canonical morphism $\Omega_f^1\to\Omega_{f,\log}^1$. The injectivity of the above morphism can be checked \'etale locally. % (\cite[Proposition 35.8.9]{stacks-project}).  
		On each \'etale local chart, it follows from \eqref{eq:ss_loc}.
	\end{proof}
	We use the tropicalization map \eqref{eq:tropical} for log schemes.
	%\begin{example}\label{ex:Artin}
	%    Let $X$ be a toric variety with a logarithmic structure induced by the toric boundary. Let $T$  denote the maximal torus acting on  $X$. Then, the associated Artin fan $\A_X$ can be identified with the stack quotient $[X/T]$. When $X$ is the affine space $\mathbb{A}^n$, we write $\A^n:=\A_{\AA^n}$.
	%\end{example}
	For a morphism $f:X\to S$ between log schemes, consider the diagram
	\begin{equation}\label{eq:relArtin}
		\begin{tikzcd}
			X \ar[r,"\phi_f"]\ar[dr,"f"'] & \A_X[S]\ar[d,"g"]\ar[r] & \A_X\ar[d]\\
			& S\ar[r,"t"] & \A_S
		\end{tikzcd}
	\end{equation}
	where the middle square is the fiber diagram. We call $\A_X[S]$ the {\em relative Artin fan}.
	\begin{lemma}\label{lem:residue}
		Let $f:X\to S$ be a semistable morphism. In the diagram \eqref{eq:relArtin}, the natural morphism $\phi_f^*\cR_g\to\cR_f$ is an isomorphism.
	\end{lemma}
	\begin{proof}
		Since $f$ is flat (\cite[Corollary 4.5]{Kato}) and log smooth, the morphism $\phi_f$ is strict and smooth. Thus, the natural morphism $\Omega_{\phi_f}^1\to\Omega^1_{\phi_f,\log}$ is an isomorphism and the result follows from Lemma \ref{lem:residue0}.% and the diagram chasing.
	\end{proof}
	
	We work with piecewise linear and piecewise polynomial functions on smooth and log smooth schemes; see \cite[Section 2.5]{HMPPS}. For an algebraic log stack $X$, a {\em (strict) piecewise linear function} on $X$ is a global section $H^0(X,\overline{\Msf}_X^\gp)\otimes_\ZZ\QQ$ and a {\em (strict) piecewise polynomial} on $X$ is a global section $H^0(X,\, \mathrm{Sym}^\bullet_\QQ\overline{\mathsf{M}}^\gp_X)$.
	%\begin{definition} \label{def:PP}
	%Let $X$ be an algebraic log stack. \Scomment{phrase also as functions on tropicalization}
	%\begin{enumerate}[label=(\alph*)]
	%    \item A {\em (strict) piecewise linear function} on $X$ is a global section $H^0(X,\overline{\Msf}_X^\gp)\otimes_\ZZ\QQ$.
	%    \item A {\em (strict) piecewise polynomial} on $X$ is a global section $H^0(X,\, \mathrm{Sym}^\bullet_\QQ\overline{\mathsf{M}}^\gp_X)$.
	%\end{enumerate}
	%\end{definition}
	For $\alpha\in H^0(X,\Msf_X^\gp)$, the $\CO^\times_X$-torsor $\CO_X^\times(\alpha)$ is defined as the preimage of  $\alpha$ under the natural map $\Msf_X^\gp\to \overline{\Msf}_X^\gp$. The associated line bundle, denoted by $\CO_X(\alpha)$, is obtained by gluing along the infinity section.
	
	We now proceed to calculate the Todd class of the residue sheaf of a semistable morphism. From Lemma \ref{lem:residue}, it follows that to compute $\cR_f$, it is sufficient to perform the calculation for the map $\mathcal{A}_X \to \mathcal{A}_S$. By \cite[Theorem 3.3.1]{MR24}, we have a graded ring isomorphism 
	\[
	\CH^*(\mathcal{A}_X) \xrightarrow{\cong} \mathsf{PP}(\Sigma_X)\,,
	\]
	which reflects the fact that $\CH^*(\mathcal{A}_X)$ satisfies \'etale descent. As a consequence, the calculation of the Todd class can be carried out stratum by stratum. 
	
	For a Chern root $x$, we denote by $\Td(x)$ the associated Todd class and the dual of the Todd class by 
	\[\Td^\vee(x):= \frac{x}{e^x-1}=\sum_{k=0}^\infty\frac{B_k}{k!}x^k\]
	where $B_k$ is the $k$-th {\em Bernoulli number}. 
	\begin{lemma}\label{lem:td_local}
		Let $n,l$ be positive integers, and $A$ be a partition of $n$ into $l$ parts $n_i$. Consider the map $f: \mathcal{A}^n \to \mathcal{A}^l$ induced by the semistable map $\AA^n \to \AA^l$ defined according to the partition $A$, as described in \eqref{eq:ss_loc}. The Todd class of the residue sheaf of $f$ is expressed as
		\[
		(\Td^{\vee}(\Rels_f))^{-1}= \prod_{i=1}^{l} \frac{\prod_{\alpha \in n_i}\Td(\delta_\alpha)}{\Td(\sum_{\alpha \in n_i} \delta_\alpha)}\in\mathsf{PP}(\Sigma_{\AA^n})
		\]
		where $\delta_{\alpha}$ is the piecewise linear function corresponding to the normal bundle of the $\alpha$-th coordinate hyperplane. 
	\end{lemma}
	
	\begin{proof}
		We begin by analyzing the case $l=1$. Let $\mathcal{D}$ denote the origin in $\mathcal{A}^1$, and $\mathcal{D}_i$ denote the coordinate axes in $\mathcal{A}^n$. Denote by $\delta$, $\delta_i$ the Chern classes of their respective normal bundles. These classes correspond to piecewise linear functions, where $\delta_i$ has slope $1$ along the ray associated with $\mathcal{D}_i$ and slope $0$ elsewhere. The pullback relationship for $f$ gives:
		\[
		f^*\delta = \sum_{i} \delta_i',.
		\]
		%We consider a diagram
		%\[
		%\begin{tikzcd}
		%& \ar[d] 0 & \ar[d] 0 & \ar[d] 0 & \\
		%    0 \ar[r] & f^*\Omega^1_{\mathcal{A}^1} \ar[d] \ar[r] & f^*\Omega^1_{\mathcal{A}^1,\log} \ar[r] \ar[d] & f^*\mathcal{O}_{\mathcal{D}} \ar[r] \ar[d] & 0 \\ 
		%    0 \ar[r] & \Omega^1_{\mathcal{A}^n} \ar[d] \ar[r] & \Omega^1_{\mathcal{A}^n,\log} \ar[r] \ar[d] & \oplus \mathcal{O}_{\mathcal{D}_i} \ar[r] \ar[d] & 0  \\
		%    0 \ar[r] & \Omega^1_{\mathcal{A}^n/\mathcal{A}^1} \ar[r] \ar[d]  & \Omega^1_{\mathcal{A}^n/\mathcal{A}^1,\log} \ar[r] \ar[d] & \Rels_f \ar[r]  \ar[d] & 0\\
		%    & 0 & 0 & 0 &
		%\end{tikzcd}
		%\]
		By \eqref{eq:vTan} and Lemma \ref{lem:residue0}, we get a short exact sequence
		\[0\to f^*\CO_{\mathcal{D}}\to\bigoplus_{i=1}^n \CO_{\mathcal{D}_i}\to\Rels_f\to 0\,.\]
		Therefore, we deduce:
		\begin{equation}
			\label{eq:specialtodd}
			\Td(\Rels_f) = \frac{\Td(\oplus\CO_{\mathcal{D}_i})}{\Td(f^*\CO_{\mathcal{D}})} =  \frac{\prod \Td^{\vee}(\delta_i)^{-1}}{\Td^{\vee} (\sum \delta_i)^{-1}}\,. 
		\end{equation}
		where the last equality follows from the usual exact sequence $0\to\CO(-D)\to\CO\to\CO_D\to 0$ for Cartier divisors $D$.
		For the general case, the map $f$ decomposes as a product of $l$ maps, $f_i:\mathcal{A}^{n_i} \to \mathcal{A}^1$. Thus, the Todd class of the relative cotangent bundle $\Td(\Rels_f)$ is given as the product of the formula for the special case \eqref{eq:specialtodd} applied to each factor.
	\end{proof}
	Now we state the general formula for the Todd class of the residue sheaf. To do so, we first introduce some notation (see \cite[Section 4]{PRSS}). For a cone $\sigma \in \Sigma_S$, let $j_\sigma: S_\sigma \to S$ denote the corresponding monodromy torsor. Similarly, for a cone in $\Sigma_X$, corresponding to a cell $c$ in the generic fiber of $\Sigma_X \to \Sigma_S$, let $i_c: X_c \to X$ denote the corresponding monodromy torsor,  and let
	\[f_c: X_c \to S_\sigma\]
	the induced map. 
	For a ray $\rho$ in $\Sigma_S$ or $\Sigma_X$, we write $\delta_\rho$ for the piecewise linear function with slope $1$ along $\rho$ and $0$ along the other rays; we identify $\delta_\rho$ implicitly with its image in the Chow ring of $S$ or $X$, which is the first Chern class of the normal bundle of $D_\rho$ in $S$ or $X$. For a cone $\sigma$, we write $\sigma(1)$ for its set of rays. 
	
	To express the Todd class in terms of the strata algebra, we proceed as follows. Given a polynomial $P_c$ on a cone $c$ of $\Sigma_X$, we define $\mathsf{prin}_c[P_c]$ to be the sum of terms in $P_c$ which are formally divisible by
	\[\delta_c:= \prod_{r \in c(1)} \delta_r\,.\]
	where $c(1)$ is the set of rays in the cone $c$. Then, we define: 
	\[
	\floor{P}_c = \frac{\mathsf{prin}_c[P_c]}{\delta_c}\,. 
	\]
	Additionally, we write $G_c$ for the monodromy group of $c$. Finally, write $\Sigma_f:\Sigma_X \to \Sigma_S$ for the induced map of cone stacks. 
	\begin{proposition}\label{pro:td_ss}
		Let $f:X \to S$ be a semistable morphism. The Todd class of the residue sheaf of $f$ is given by 
		\[
		\Td^{\vee}(\mathcal{R}_f)^{-1} = \sum_{c \in \Sigma_X} \frac{(i_c)_*}{|G_c|} \prod_{\rho \in \Sigma_f(c)(1)} \floor{\frac{\prod_{r \in c(1), \Sigma_{f}(r)=\rho} \Td(\delta_r)}{\Td(\delta_\rho)}}_c\,.
		\]
	\end{proposition}
	\begin{proof}
		We prove the formula  in the ring of piecewise polynomials:
		\begin{equation}\label{eq:td_pp}
			\Td^{\vee}(\Rels_f)^{-1}=\left( \prod_{\rho \in \Sigma_f(c)(1)} \frac{\prod_{r \in c(1), \Sigma_{f}(r)=\rho} \Td (\delta_r)}{\Td (\delta_\rho)} \right)_{c}\in\mathsf{PP}(\Sigma_X)\,.
		\end{equation}
		Here, $(-)_c$ indicates that the piecewise polynomial function is determined on the cone $c$ of $\Sigma_X$ by truncating the power series
		\[
		\prod_{\rho \in \Sigma_f(c)(1)} \frac{\prod_{r \in c(1), \Sigma_{f}(r)=\rho} \Td (\delta_r)}{\Td (\delta_\rho)}.
		\]
		at codimension greater than $\dim X$. It is implicitly understood that these functions must agree on intersections of cones $c \cap c'$. 
		%It suffices to verify $\eqref{eq:td_pp}$ on each \'etale local chart \eqref{eq:ss_loc} and confirm that the formula glues  because piecewise polynomials satisfies \'etale local descent \cite{MR24}. 
		The formula agrees with Lemma \ref{lem:td_local} on each cone because
		\[\delta_\rho = \sum_{r \mapsto \rho} \delta_r\,.\]
		This ensures that the Todd class satisfies the statement of the proposition \'etale locally on $X$. But the formula \eqref{eq:td_pp} is invariant under automorphisms of $\mathbb{R}_{\ge 0}^{n}$, and, by direct computation, the formula defined on each cone agrees with its restriction to any face, $\mathbb{R}_{\ge 0}^{k} \subset \mathbb{R}_{\ge 0}^n$. Since piecewise polynomials satisfy \'etale local descent \cite{MR24}, the formula descends to $\mathsf{PP}(\Sigma_X)$.
		
		Once the formula \eqref{eq:td_pp} is established in the ring of piecewise polynomials, translating it into a statement about normally decorated strata classes becomes straightforward, see \cite[Lemma 40]{HMPPS}. 
	\end{proof}

	\begin{example} \label{ex:mumford}
		For the morphism $\mathcal{A}^2 \to \mathcal{A}^1$ corresponding to $(x,y) \to xy$, the formula above reduces  to Mumford's formula for the Todd class of a node (\cite[Section 5]{mumford83}). 
	\end{example}
	\subsection{Residue sheaf for compactified Jacobians}\label{sec:todd}    
	The Todd class of $\Rels_\pi$ for compactified Jacobians can be expressed in terms of tautological classes. Let $B:=\Mbar_{g,n}$ and $\pi: \Jbar_{g,n}\to \Mbar_{g,n}$ be a relative compactified Jacobian for some non-degenerate stability condition. For $1\leq h\leq g$, denote by $\Gamma_h$ the stable graph consisting of a single vertex of genus $g-h$ with $h$ loops. Define 
	\[i_h :\Mbar_{g-h,n+2h}\to \Mbar_{g,n}\]
	as the morphism induced by gluing the $n+i$-th marking with the $n+h+i$-th marking for $1\leq i\leq h$. Let $\Gamma_h'$ be the quasi-stable graph obtained by subdividing each edge of $\Gamma_h$. Consider the following diagram:
	\begin{equation}\label{eq:bdy}
		\begin{tikzcd}
			\Jbar_{g-h,n+2h}\ar[r]\ar[dr] & \Jbar_{\Gamma_h} \ar[r]\ar[d] &\Jbar_{g,n}\ar[d,"\pi"]\\
			& \Mbar_{g-h,n+2h}\ar[r,"i_{h}"] & \Mbar_{g,n}\,.
		\end{tikzcd}
	\end{equation}
	Here, the right square is a Cartesian diagram, and the morphism $\Jbar_{g-h,n+2h}\to\Jbar_{\Gamma_h}$ is induced by Lemma \ref{lem:boundary}. 
	
	\begin{proposition}\label{pro:Todd}
		Let $\pi:\Jbar_C\to B=\Mbar_{g,n}$ be the projection. Let $\jmath_{h}:\Jbar_{g-h,n+2h}\to\Jbar_{g,n}$ be the composition in \eqref{eq:bdy}. Then we have
		\begin{equation*}
			\Td^\vee(\Rels_\pi)^{-1}-1= \sum_{h=1}^g \frac{(-1)^h}{h!2^h} \jmath_{h*}\Bigg[\prod_{i=1}^h \left(\sum_{k_i=1}^\infty\frac{(-1)^{k_i}B_{2k_i}}{(2k_i)!}\frac{\alpha_i^{2k_i-1}+\beta_i^{2k_i-1}}{\alpha_i+\beta_i}\right)\Bigg]
		\end{equation*}
		with $\alpha_i=\psi_{n+i}-\xi_{n+i}+\xi_{n+h+i}$ and $\beta_i=\psi_{n+h+i} + \xi_{n+i}-\xi_{n+h+i}$.
	\end{proposition}
	\begin{proof}
		The morphism $\pi:\Jbar_{g,n}\to\Mbar_{g,n}$ satisfies hypothesis in Lemma \ref{lem:residue}. By Lemma \ref{lem:residue} we compute the sheaf $\Omega^1_{\A_{\Jbar[\Mbar_g]/\Mbar_g},\log}/\Omega^1_{\A_{\Jbar[\Mbar_g]/\Mbar_g}}$ on each chart and show that they glue. 
		
		We consider the morphism \eqref{eq:varphi}. Let $j_{\Gamma_\ell'}:\pic_{\Gamma'_\ell}\to\pic_{g,n}$ be the finite morphism. Then the following diagram is cartesian
		\[
		\begin{tikzcd}
			\Jbar_{g-\ell,2\ell} \ar[r,"\jmath_{g-\ell}"] \ar[d] & \Jbar_{g,n} \ar[d,"\varphi"]\\
			\mathfrak{Pic}_{\Gamma_\ell'} \ar[r,"j_{\Gamma_\ell'}"] & \mathfrak{Pic}_{g,n}.
		\end{tikzcd}
		\]
		For each $i$, there exists a quasi-stable vertex $v_i$ and two edges $e=(h_i,\overline{h}_i)$ and $e=(h_i',\overline{h}'_i)$ connecting $v_i$ and the unique stable vertex $v_0$ where $h_i,h_i'$ are half-edges at $v_0$.
		By Lemma \ref{lem:smoothness}, $\varphi$ is smooth, so the normal bundle of $j_{g-\ell}$ is isomorphic to the pullback of the normal bundle of $j_{\Gamma_\ell'}$. Therefore, the Chern roots of the normal bundle of $j_{\Gamma_\ell'}$ are given by $\psi_{h_i}+\psi_{\overline{h}_i}$, $\psi_{h_i'}+\psi_{\overline{h}_i'}$ for $1\leq i\leq \ell$. By the genus $0$ relation \cite{BHPSS}, we have 
		\begin{equation}\label{eq:g=0}
			\psi_{\overline{h}_i} = -\xi_{h_i} + \xi_{h_i'} \text{ and } \psi_{\overline{h}_i'} = \xi_{h_i} -\xi_{h_i'}\,.
		\end{equation}
		The result now follows from Example \ref{ex:mumford}.
	\end{proof}
	Let $T_\pi^\vir$ be the virtual tangent bundle of $\pi:\Jbar_C\to B$. To relate the Todd class of $T_\pi^\vir$ and the Todd class of $\Rels_\pi$, we use the following adaptation of Faltings-Chai \cite{FaltingsChai}.
	\begin{proposition}\label{pro:Logcotangent}
		Let $\pi:\Jbar_C\to B$ be a compactified Jacobian. Then we have  $\Omega^1_{\pi,\log}\cong \pi^*\EE$ where $\EE=\pi_*\Omega^1_{\pi,\log}$ is the Hodge bundle. 
	\end{proposition}
	\begin{proof}
		By \cite[Chapter VI, Theorem 1.1]{FaltingsChai} (also \cite[Lemma 4.3.10]{Olsson}), the corresponding identity was obtained for certain compactified abelian schemes. After pulling back to $B$, $\Jbar_C$ and the compactified abelian scheme have a common refinement. Since log modifications are log \'etale, $\Omega^1_{\pi,\log}$ does not change under log modification. Hence we get the result.
	\end{proof} 
	By combining \eqref{eq:vTan}, \eqref{eq:residue} and Proposition \ref{pro:Logcotangent} we have
	\begin{equation}\label{eq:td_Tvir}
		\Td(T_\pi^\vir) = \pi^*\Td^\vee(\EE)\cdot \Td^\vee(\Rels_\pi)^{-1}\,.
	\end{equation}
	\subsection{Todd class for the birational model}
	We consider the projection
	\[\widetilde{\pi}:\Jbar_C^{(2)}\to B\]
	where $\Jbar_C^{(2)}\to \Jbar_C\times_B\Jbar_C$ is a log modification constructed in Definition \ref{def:J2}. By Proposition \ref{pro:model2} (a), $\widetilde{\pi}$ is semistable and by Proposition \ref{pro:compAb}, it is a compactified abelian fibration. 
	%Consider the short exact sequence \eqref{eq:residue}:
	%\[0\to \Omega_{\Jbar_C^{(2)}/B}\to \Omega_{\Jbar_C^{(2)}/B}^{\log}\to\cR_{\widetilde{\pi}}\to 0 \,.\]
	We generalize the result in Section \ref{sec:todd}. %and prove that the Todd class of $\cR_{\widetilde{\pi}}$ can be written by a combinatorial data associated with the tropicalization of $\Jbar_C^{(2)}$.
	By the same argument in Proposition \ref{pro:Logcotangent} we have
	\begin{equation}\label{eq:Td2}
		\Omega_{\widetilde{\pi},\log}^1\cong \widetilde{\pi}^*(\EE\oplus\EE)\,.
	\end{equation}
	
	\begin{proposition}\label{pro:diffTd}
		Let $\pi:\Jbar_C\to B$ denote the projection, and let $\pi_1,\pi_2:\Jbar_C\times_B\Jbar_C\to\Jbar_C$ represent the two projections. Then the following holds:
		\begin{equation}\label{eq:td_J2}
			\Td(T_{\Jbar^{(2)}_C/B})\cdot f^*\Td(-T_{\Jbar_C\times_B\Jbar_C/B}) = \Td^\vee(\cR_{\widetilde{\pi}})\cdot\pi_1^*\Td^\vee(\cR_\pi)^{-1}\cdot\pi_2^*\Td^\vee(\cR_\pi)^{-1}\,.
		\end{equation}
	\end{proposition}
	\begin{proof}
		Consider the fiber diagram
		\[
		\begin{tikzcd}
			\Jbar_C\times_B\Jbar_C\ar[r]\ar[d] & \Jbar_C\times \Jbar_C\ar[d,"\pi\times\pi"]\\
			B\ar[r,"\Delta_B"] & B\times B\,,   
		\end{tikzcd}
		\]
		where $\Delta_B:B\to B\times B$ is the diagonal.
		Since $\Jbar_C, B$ are smooth and $\pi$ is flat, we have
		\[ T_{\Jbar_C\times_B\Jbar_C/B} = \pi_1^*T_{\Jbar_C}+\pi_2^*T_{\Jbar_C}-2\,\pi^*T_B=\pi_1^*T_\pi^\vir + \pi_2^*T_\pi^\vir\]
		in $K^0(\Jbar_C\times_B\Jbar_C)$.
		Using \eqref{eq:td_Tvir}, we obtain
		\[\Td(T_{\Jbar_C\times_B\Jbar_C/B}) = \Td(\EE^{\oplus 2})^\vee\cdot\pi_1^*\Td^\vee(\cR_\pi)^{-1}\cdot \pi_2^*\Td^\vee(\cR_\pi)^{-1}\,. \]
		Finally, by applying Lemma \ref{lem:residue} and \eqref{eq:Td2}, we derive the desired result.
	\end{proof}
	%\begin{proposition}
	%    Let $\varphi:\Jbar_C^{(2)}\to\M_{g,n}$ be the morphism induced by \eqref{eq:univ_curve}. Then 
	%    \[\mathsf{PP}(\Jbar_C^{(2)})=\varphi^*\mathsf{PP}(\M_{g,n})\,.\]
	%\end{proposition}
	%\begin{proof}
	%    By Proposition \ref{pro:model1}, $\Jbar_C^{(2)}$ is smooth and log smooth. Hence it is enough to show that the universal curve $\widetilde{C}$ is tropically smooth and $\widetilde{C}\to\Jbar_C^{(2)}$ is weakly versal (in the sense of \cite{BM}).
	%\end{proof}
	Let now $\kappa$ be a cone of $\Jbar_C \times_B \Jbar_C$ parametrizing two quasistable graphs $\Gamma_1,\Gamma_2$ with stabilization $\Gamma$. Keeping the notations of Corollary \ref{cor:smoothandsmall}, a cone $c$ of $\widetilde{J}^{(2)}_C$ lying over $\kappa$ has the form 
	\begin{equation}\label{eq:cone}
		c = \mathbb{R}_{\ge 0}^{e \in E^{u}(\Gamma)} \times \prod_{e \in E^{s_1 \vee s_2}(\Gamma)}\mathbb{R}_{\ge 0}^{2} \times \prod_{e \in E^{s_1 \cap s_2(\Gamma)}} \mathbb{R}_{\ge 0}^3    
	\end{equation}
	where, for $e \in E^{s_1 \cap s_2(\Gamma)}$, the component $\mathbb{R}_{\ge 0}^3$ is either the $\ell_{e_1}' \le \ell_{e_2}'$ or the $\ell_{e_2}' \le \ell_{e_1}'$ piece of the resolution of 
	\[
	\mathbb{R}_{\ge 0}^{2} \times_{\mathbb{R}_{\ge 0}} \mathbb{R}_{\ge 0}^{2} = \{(\ell_{e_1}',\ell_{e_1}'',\ell_{e_2}',\ell_{e_2}''): \ell_{e_1}'+\ell_{e_1}'' = \ell_{e_2}' + \ell_{e_2}''\}. 
	\]
	
	\begin{proposition} 
		\label{prop:tdbirationalmodel}
		As a piecewise polynomial function on $\Jbar_C^{(2)}$, the restriction of \eqref{eq:td_J2} to the cone \eqref{eq:cone} is given by 
		\[
		\prod_{e \in E^{s_1 \cap s_2}(\Gamma)} \frac{\Td(\max{(\ell_{e_1}',\ell_{e_2}')}-\min{(\ell_{e_1}',\ell_{e_2}')})\Td(\ell_e)}{\Td(\max{(\ell_{e_1}',\ell_{e_2}'))}\Td(\ell_{e}-\min{(\ell_{e_1}',\ell_{e_2}')})}
		\]
		where $\ell_e = \ell_{e_1}'+\ell_{e_2}'$ is the total length of the edge.
	\end{proposition}
	
	\begin{proof}
		The cone $c$ is the product of cones $\mathbb{R}_{\ge 0},\mathbb{R}_{\ge 0}^2$ and $\mathbb{R}_{\ge 0}^3$. Cones of the first two kinds correspond to strata of $\Jbar_C^{(2)}$ on which the map to $\Jbar_C \times_B \Jbar_C$ is an isomorphism, so the relative Todd class is trivial on them. To ease the notation, we assume $\ell_{e_1}' \ge \ell_{e_2}'$. Then, in the third case, the cone and its two projections look as follows: 
		
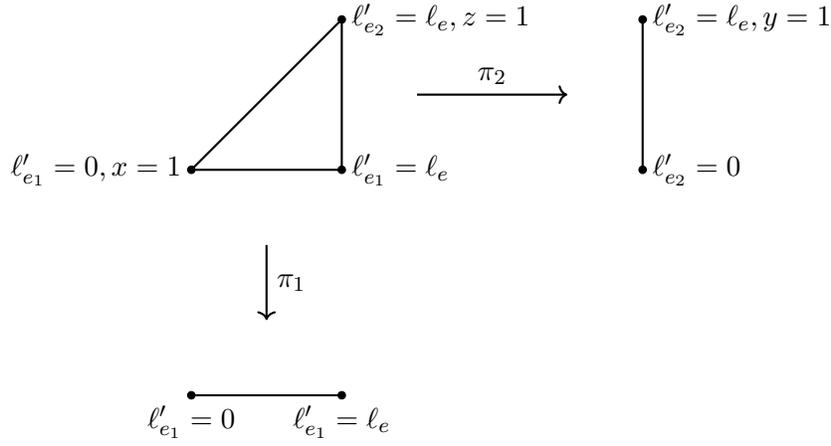
\[
		\begin{tikzpicture}
			
			\node[left] at (0,0){$\ell_{e_1}'=0,x=1$};
			\node[draw,circle,inner sep=1pt,fill] at (0,0){};
			\node[right] at (2,0){$\ell_{e_1}'=\ell_e$};
			\node[draw,circle,inner sep=1pt,fill] at (2,0){};
			\node[right] at (2,2){$\ell_{e_2}'=\ell_e,z=1$}; 
			\node[draw,circle,inner sep=1pt,fill] at (2,2){};
			\draw[thick](0,0)--(2,0)--(2,2)--(0,0);
			\node[right] at (6,0){$\ell_{e_2}'=0$};
			\node[draw,circle,inner sep=1pt,fill] at (6,0){};
			\node[right] at (6,2){$\ell_{e_2}'=\ell_e,y=1$}; 
			\node[draw,circle,inner sep=1pt,fill] at (6,2){};
			\draw[thick](6,0)--(6,2);
			\node[below] at (0,-3){$\ell_{e_1}'=0$}; 
			\node[draw,circle,inner sep=1pt,fill] at (0,-3){};
			\node[below] at (2,-3){$\ell_{e_1}'=\ell_e$}; 
			\node[draw,circle,inner sep=1pt,fill] at (2,-3){};
			\draw[thick](0,-3)--(2,-3);
			
			\draw[->, thick] (3,1)--(5,1); 
			\node[above] at (4,1){$\pi_2$};
			\draw[->,thick] (1,-1)--(1,-2);
			\node[right] at (1,-1.5){$\pi_1$};
		\end{tikzpicture}
		\]
		Let $x,y,z$ be the piecewise linear functions on $\mathbb{R}_{\ge 0}^3$ with slope $1$ along the rays indicated in the diagram, and slope $0$ on the other rays. On the cone $c$, formula \eqref{eq:td_pp} reads
		\begin{align*}
			\Td^\vee(\cR_{\widetilde{\pi}})\cdot\pi_1^*\Td^\vee(\cR_\pi)^{-1}\cdot\pi_2^*\Td^\vee(\cR_\pi)^{-1} & =  \frac{\Td(x)\Td(y)\Td(z)}{\Td(x+y+z)}\frac{\Td(x+y+z)}{\Td(x)\Td(y+z)}\frac{\Td(x+y+z)}{\Td(x+y)\Td(z)} \\
			& = \frac{\Td(y)\Td(x+y+z)}{\Td(x+y)\Td(y+z)}.
		\end{align*}
		Given that $x = \ell_{e} - \ell_{e_1}', \, y = \ell_{e_1}' - \ell_{e_2}'$ and $z = \ell_{e_2}'$,
		the result follows. It is easy to treat the analogous case where $\ell_{e_2}' \ge \ell_{e_1}'$, which simply interchanges the roles of $\ell_{e_2}'$ and $\ell_{e_1'}$, leading to the claimed formula. 
	\end{proof}
	
	\section{Fourier transform}\label{sec:5}
	\subsection{Derived equivalence of compactified Jacobians}\label{sec:arinkin}
	In Section \ref{sec:3}, we constructed an extension of Poincar\'e line bundle. We now observe that it gives a derived equivalence of compactified Jacobians, following  Arinkin's argument in \cite{arinkin2} (see also \cite{MRV2}).
	\begin{theorem}[\cite{arinkin2}] \label{thm:autoequiv}
		Let $B=\Mbar_{g,n}$ and let $\epsilon_1,\epsilon_2$ be two non-degenerate stability conditions. Let $\overline{\P}:=f_*\widetilde{\P}\in\mathrm{Coh}(\Jbar_C^{\epsilon_1}\times_B\Jbar_C^{\epsilon_2})$ be pushforward of the extended Poincar\'e line bundle. Then $\overline{\P}$ induces a derived equivalence
		\[\fm:=-\otimes\overline{\P}: \dbcoh(\Jbar^{\epsilon_1}_C)\xrightarrow{\cong} \dbcoh(\Jbar^{\epsilon_2}_C)\]
		which is linear over $\dbcoh(B)$ with the inverse kernel given by $\overline{\P}^{-1}:=\hom (\overline{\P},\CO)\otimes \pi_2^*\omega_\pi[g]$.
	\end{theorem}
	\begin{proof}
		Over $B=\Mbar_{g,n}$, the relative Jacobian $\Jzero_C\to B$ is a $\delta$-regular family of semi-abelian schemes (\cite[Proposition 3.3]{arinkin1}). By Theorem \ref{thm:Plog=arinkin}, the kernel $\overline{\P}=f_*\widetilde{\P}\cong Rf_*\widetilde{\P}$ is a  maximal Cohen-Macaulay extension of Poincar\'e line bundle. Therefore, the argument in \cite[Section 7.3]{arinkin2} applies to the family $\Jbar_C\to B$ and shows that that $\fm$ is a derived equivalence.
	\end{proof}
	
	For a separated scheme $X$, let $K_0(X)$ denote the Grothendieck group of coherent sheaves on $X$. We consider the Baum-Fulton-MacPherson homomorphism (\cite[Chapter 18]{Fulton1984Intersection-th}):
	\begin{equation}\label{eq:GRR}
		\tau: K_0(X)\xrightarrow{} \CH_*(X,\QQ)\,.
	\end{equation} 
	When $X$ is smooth, we have $\tau(F)=\ch(F)\cdot \Td(T_X)$. This homomorphism is functorial with respect to proper pushforward and l.c.i. pullback \cite[Theorem 18.2]{Fulton1984Intersection-th}.
	
	Since $B=\Mbar_{g,n}$ is a Deligne–Mumford stack, \eqref{eq:GRR} requires a modification. One can work with the Baum-Fulton-MacPherson homomorphism for quotient stacks \cite{GilletSoule}. For our purposes, it suffices to take a finite \'etale morphism $B'\to B$ where $B'$ is a scheme (e.g., the moduli space of curves with level structure). After base change to $B'$, the calculations involving $\tau_X$ hold in the Chow group with $\QQ$-coefficients.
	
	We fix our notation for relative correspondences (\cite{CortiHanamura}). Let $\pi_1:M_1\to B, \pi_2: M_2\to B$ be two proper morphisms between smooth Deligne-Mumford stacks. Any relative correspondence $Z\in \CH_*(M_1\times_B M_2)$ then defines a map
	\[ Z : \CH^*(M_1)\to\CH^*(M_2), \, \alpha\mapsto \pi_{2*}(\pi_1^!\alpha\cap [Z])\,,\]
	where $\pi_1^!$ is the l.c.i. pullback. We use the same notation for $Z$ and its induced map.
	
	Following \cite[Section 2.3]{MSY23}, we consider the Chow-theoretic Fourier transform induced by Theorem \ref{thm:autoequiv}. Since $\Jbar_C\times_B\Jbar_C$ is l.c.i., its virtual tangent bundle has a well-defined Todd class:
	\[\Td(T_{\Jbar_C\times_B\Jbar_C})\in\CHop^*(\Jbar_C\times_B\Jbar_C)_\QQ\,.\] 
	The (Chow-theoretic) {\em Fourier transform} is defined by the relative correspondence
	\begin{equation}\label{eq:FM}
		\fm:=\Td(-T_{\Jbar_C\times_B\Jbar_C})\cap\tau(\overline{\P})\in\CH_*(\Jbar_C\times_B\Jbar_C)_\QQ\,.
	\end{equation}
	The inverse Fourier transform is given by
	\begin{equation}\label{eq:fminv}
		\fm^{-1}:=\Td(-(p_1\times_B p_2)^*T_B)\cap\tau(\overline{\P}^{-1})\in\CH^*(\Jbar_C\times_B\Jbar_C)_\QQ\,.
	\end{equation}
	The choice of the asymmetric Fourier transform  will be useful for understanding the Fourier transform of tautological classes.  By Theorem \ref{thm:autoequiv}, we have
	\begin{equation}\label{eq:FMchow}
		\fm\circ\fm^{-1} \cong \id\,, \, \fm^{-1}\circ \fm\cong\id\,.
	\end{equation}
	
	The Fourier transform of the class of a section which factors through the locus of line bundles $J_C\subset \Jbar_C$ \eqref{eq:J_C} can be easily computed:
	\begin{proposition}\label{pro:section}
		Let $\epsilon_1,\epsilon_2$ be two non-degenerate stability conditions for $p:C\to B$.
		Let $s:B\to\Jbar_C^{\,\epsilon_1}$ be a section which factors through $J_C^{\,\epsilon_1}$. Let $i_s:=(s,\id): \Jbar_C^{\epsilon_2}\to\Jbar_C^{\,\epsilon_1}\times_B\Jbar_C^{\epsilon_2}$ be a closed embedding. Then we have
		\[\fm^{-1}\big(\ch \,(i_s^*\Pbar)\big)=[s]\,.\]
	\end{proposition}
	\begin{proof}
		On the open locus $J_C^{\,\epsilon_1}\times_B\Jbar_C^{\,\epsilon_2}\subset \Jbar^{\,\epsilon_1}_C\times_B\Jbar_C^{\epsilon_2}$, the sheaf $\Pbar$ is the usual Poincar\'e line bundle. Since $i_s$ factors through the open locus, the pullback $i_s^*\Pbar$ is a line bundle. For $j=1,2$, let $\pi_j:\Jbar^{\,\epsilon_1}_C\times_B\Jbar_C^{\epsilon_2}\to \Jbar_C^{\,\epsilon_j}$ be the $j$-th projection. 
		By the functoriality of  \eqref{eq:GRR}, we have
		\begin{align*}
			\fm([s]) &=\pi_{2*}\big(i_{s*}1\cdot\Td(-T_{\Jbar_C\times_B\Jbar_C})\cdot\tau(\overline{\P})\big)\\
			&= (\pi_{2}\circ i_s)_*\big(\Td(T_{i_s})\cdot\Td(-i_s^*T_{\Jbar_C\times_B\Jbar_C})i_s^*\tau(\overline{P})\Td(-T_{\Jbar_C^{\,\epsilon_2}})\big)\\
			&=\tau(i_s^*\Pbar)\Td(-T_{\Jbar_C^{\,\epsilon_2}})\\
			&= \ch \, (i_s^*\Pbar)\,.
		\end{align*}
		By \eqref{eq:FMchow}, we get the result by applying $\fm^{-1}$.
	\end{proof}
	\subsection{Logarithmic Abel-Jacobi theory}\label{sec:AJ}
	We compute the Fourier transform of the closures of rational Abel-Jacobi sections. The main result is Proposition \ref{pro:FM_aj}, which extends Proposition \ref{pro:section}. A key ingredient in our approach is the desingularization of Arinkin's kernel constructed in Theorem \ref{thm:Plog=arinkin}.
	
	Let $\Jbar_C$ be a compactified Jacobian with respect to a non-degenerate stability condition $\epsilon$ of degree $0$.
	For  $b\in\ZZ$ and $\sfa\in\ZZ^n$ with $\sum_i a_i=(2g-2+n)b$, the rational Abel-Jacobi section $\aj_{b;\sfa}: B \dashrightarrow \Jbar_C$ is given by $t\in B\mapsto\CO_{C_t}(\sum a_i x_i)\otimes(\omega_{C_t,\log}^{\otimes -b})$. Here, $\omega_{C,\log}:=\omega_C(\sum_i x_i)$ is the log dualizing line bundle. The section is well-defined on the locus where the underlying curve is smooth. 
	
	We resolve the indeterminacy of the rational Abel-Jacobi section by taking the following f.s. fiber product in the category of algebraic log stacks (see Section \ref{sec:jbar}):
	\[
	\begin{tikzcd}
		B_{b;\sfa} \ar[r]\ar[d] & \Jbar_C\ar[d]\\
		B\ar[r,"\aj_{b;\sfa}"] & \lpic_C.
	\end{tikzcd}
	\]
	Since $\Jbar_C\to\lpic_C$ is a log modification, the morphism $B_{b;\sfa}\to B$ is also a log modification, and the map $B_{b;\sfa}\to\Jbar_C$ extends the rational Abel-Jacobi section. 
	
	As the stack $B_{b;\sfa}$ is not smooth, we work with the smooth modular compactification constructed by the second author \cite{Molcho23}, and refer reader to loc. cit. for the necessary background. 
    %As we only need the properties of the construction, we do not review the terminology and rather refer the interested reader to loc.cit.
	\begin{theorem}[\cite{Molcho23}]\label{thm:B_v}
		An algebraic logarithmic stack $B_{b;\sfa}^\epsilon : (\textbf{LogSch}/B)^\mathrm{op}\to (\mathrm{Grp})$ is defined by $(S\to B)\mapsto (C'\to C_S,\alpha)$,
		where $C'\to C_S$ is a quasi-stable model and $\alpha$ is an equidimensional piecewise linear function on $C'$ where $\CO_{C'}(\sum_i a_ix_i)\otimes(\omega_{C_t,\log}^{\otimes -b})\otimes\CO_{C'}(\alpha)$ is $\epsilon$-stable. Then $B_{b;\sfa}^\epsilon$ is a Deligne-Mumford stack which is connected and smooth. 
	\end{theorem}
	%By construction, the natural morphism $B_\sfa^\epsilon\to B$ is a log modification, and there exists a {\em logarithmic Abel-Jacobi section}
	The {\em resolved Abel-Jacobi section} is defined by
	\begin{equation}\label{eq:AJ}
		\aj_{b;\sfa} :B_{b;\sfa}^\epsilon\to \Jbar_C, (C'\to C,\alpha)\mapsto \CO_{C'}(\sum a_i x_i)\otimes(\omega_{C',\log}^{\otimes -b})\otimes\CO_{C'}(\alpha)\,.
	\end{equation}
	
	We aim to compute the Fourier transform of \eqref{eq:AJ}. By pulling back the diagram \eqref{eq:J2} along the resolved Abel-Jacobi-section \eqref{eq:AJ}, we obtain the following fiber diagram
	\begin{equation}\label{eq:model}
		\begin{tikzcd}
			&\widetilde{J_{b;\sfa}^\epsilon}\ar[r,"\widetilde{\aj}_{b;\sfa}"]\ar[d,"g"] & \Jbar_C^{(2)}\ar[d,"f"]\\
			\Jbar_C\ar[d,"\pi"] &\Jbar_{b;\sfa}^\epsilon\ar[r,"\aj_{b;\sfa}'"]\ar[l,"\pi_2"']\ar[d] & \Jbar_C\times_B\Jbar_C\ar[d,"\pi_1"]\\
			B & B_{b;\sfa}^\epsilon\ar[r,"\aj_{b;\sfa}"]\ar[l] & \Jbar_C\,.
		\end{tikzcd}   
	\end{equation}
	By Theorem \ref{thm:B_v}, the resolved Abel-Jacobi section $\aj_{b;\sfa}$ is a morphism between smooth stacks and hence $\aj_{b;\sfa}$ is l.c.i. Since $\pi_1$ is flat, $\aj_{b;\sfa}'$ is also l.c.i.
	\begin{lemma}\label{lem:fssq}
		In the diagram \eqref{eq:model}, we have $(\aj_{b;\sfa}')^![\Jbar_C^{(2)}] = [\widetilde{J_{b;\sfa}^\epsilon}]$ in $\CH_*(\widetilde{J_{b;\sfa}^\epsilon})$.
	\end{lemma}
	\begin{proof}
		We first show that $\widetilde{J_{b;\sfa}^\epsilon}$ is reduced and irreducible. By Proposition \ref{pro:model2} (a), the composition $\pi_1\circ f$ is a saturated morphism. Therefore, the outer square is the f.s. fiber product. Similarly, the bottom square is an f.s. fiber product, and thus, the top square is also an f.s. fiber product. The morphism $g$ is a log modification because the top square is an f.s. fiber product and $f$ is a log modification. This shows that $\widetilde{J_{b;\sfa}^\epsilon}$ is reduced and irreducible. 
		
		By dimension considerations, the class $(\aj_{b;\sfa}')^![\Jbar_C^{(2)}]$ is some multiple of the fundamental class of $\widetilde{J_{b;\sfa}^\epsilon}$ since it is irreducible. The multiplicity can be verified on the open locus where the underlying curve is smooth. Over this locus, the map $B_{b;\sfa}^\epsilon\to B$ is an isomorphism and $g$ is also an isomorphism. Therefore the multiplicity is one.
	\end{proof}
	We describe the pullback of the extended Poincar\'e line bundle $\widetilde{\P}$ along the resolved Abel-Jacobi section. From Theorem \ref{thm:B_v},  $k$-points of $\widetilde{J_{b;\sfa}^\epsilon}$  are tuples
	\begin{equation}\label{eq:J_sfv}
		(C_1,C_2, C'\to C_i,\alpha,L_2)\,,
	\end{equation}
	where $C'\to C_1,C_2$ is a common semistable model, $\alpha$ is an equidimensional piecewise linear function on $C_1$ where $\CO_{C_1}(\sum_i a_i x_i)\otimes(\omega_{C_1,\log}^{\otimes -b})\otimes\CO_{C_1}(\alpha)$ is $\epsilon$-stable, $L_2$ is a $\epsilon$-stable line bundle on $C_2$.
	
	%For a log curve $C\to B$, we have the formula of the Deligne pairing of an arbitrary line bundle and a line bundle associated to a  piecewise linear function. 
	\begin{lemma}\label{lem:deligne_pl}
		%\Scomment{Take a look to see if there is shorter argument} 
		Let $p: C\to B$ be a log smooth curve over a log smooth base $B$. For a piecewise linear function $\alpha$ on $C$ and a line bundle $L$ on $C$, we have
		\begin{equation}\label{eq:deligne_pl}
			\langle \CO_C(\alpha), L\rangle_C = \sum_{w\in V(\Gamma_{C_b})}\deg (L|_w)\cdot\alpha(w) \in \mathsf{PL}(B)\,.
		\end{equation}
	\end{lemma}
	\begin{proof}
		We explain the right hand side of \eqref{eq:deligne_pl}. For each generic point $b\in B$ of a maximal cone, there exists a map $\mathsf{PL}(C_b)\to \ZZ^{V(\Gamma_{C_b})}$ induced by taking the slope and composing with the divisor map as  in \cite[Section 3.3]{HMPPS}. The right-hand side of \eqref{eq:deligne_pl} gives the value at each maximal cone.
		
		%We apply reduction steps. 
		Both sides of \eqref{eq:deligne_pl} is stable under log modification of the base $B$ and taking semistable modification of the log curve $C\to B$, as the degree of $L$ is $0$ on newly introduced exceptional components. Therefore it is enough to show \eqref{eq:deligne_pl} when $B$ and $C$ are smooth and logarithmically smooth. The log curve $p$ is smooth over the nonempty open locus of $B$, where the left hand side of \eqref{eq:deligne_pl} is trivial, it is enough to check the multiplicity of left hand side on each boundary of $B$. %  The log curve $p$ is smooth over the nonempty open locus of $B$, and hence the left hand side of \eqref{eq:deligne_pl} is trivial on this open locus. Since $B$ is  smooth, the left hand side can be written a linear combination of the boundary divisors. Therefore it is enough to check the multiplicity on each boundary divisor $D\subset B$. 
		
		We represent the line bundle $L$ as the difference of two very ample divisors $L\cong \CO(D_1- D_2)$. From the above reduction step, it is enough to prove \eqref{eq:deligne_pl} when $\alpha$ is $\CO_C(E)$ where $E$ is an irreducible component of $p^{-1}D$ and $L$ is an effective divisor intersecting $E$ transversely. When $p^{-1}D$ is irreducible $L$ can be written as a linear combination of a divisor flat over $B$ and $p^{-1}D$, Since $p^{-1}D$ is the divisor pulled back from the base, it does not contribute. When $p^{-1}D$ has several irreducible components, we first consider the case when $L$ is  a linear combination of a divisor flat over $B$ and a vertical divisor not equal to $E$.  Then $\langle \CO_C(\alpha), L\rangle_C \cong \CO(mD)$ where $m$ is the multiplicity of $E$ and $E$ which is the same as the right hand side. When $L$ contains contribution of $E$, $E$ can be written as $p^{-1}D$ minus other vertical divisors transverse to $E$, hence we get the result.     
	\end{proof}
	
	\begin{remark}
		In fact, Lemma~\ref{lem:deligne_pl} holds without assuming that $B$ is log smooth. However, the argument in this more general setting is more involved, and since we do not need it here, we omit it.
	\end{remark}
	From \eqref{eq:J_sfv}, we consider the universal curve $C'\to\widetilde{J^\epsilon_{b;\sfa}}$, the equidimensional piecewise linear function $\alpha$ on $C'$ pulled back from $C_1$, and the universal line bundle $L_2$ on $C'$ pulled back from $C_2$.
	\begin{corollary}\label{cor:pullback_P}
		For the piecewise linear function $\beta$ on $\widetilde{J^\epsilon_{b;\sfa}}$ given by $\langle\CO_{C'}(\alpha),L_2\rangle_{C'}$, we have
		\[
		\widetilde{\aj}_{b;\sfa}^*(c_1(\widetilde{\P})) =  g^*\pi_2^*\big(-b\kappa_{0,1}+\sum_{i=1}^n a_i \xi_i \big)+ \beta\,.
		\]
	\end{corollary}
	\begin{proof}
		By the proof of Lemma \ref{lem:fssq}, the upper square of \eqref{eq:model} is an f.s. fiber diagram. Since $f$ is log modification, $\widetilde{J_{b;\sfa}^\epsilon} \to \Jbar_C$ is a log modification. By Proposition \ref{pro:pairing}, 
		%$\widetilde{\P}$ has an expression of Deligne pairing, and hence, 
		we obtain:
		\begin{align*}
			\widetilde{\aj}_{b;\sfa}^*(c_1(\widetilde{\P})) &\cong \left \langle L_2|_{C'},\CO_{C'}(\sum a_ix_i)\otimes(\omega_{C_1,\log}^{\otimes -b}) \otimes \CO_{C'}(\alpha) \right \rangle \\
			&\cong g^*\pi_2^*\big(-b\kappa_{0,1}+\sum_{i=1}^n a_i \xi_i\big) + \beta\,.
		\end{align*}
		The last equality follows from the linearity of Deligne pairing, together with Lemma \ref{lem:deligne_pl} and the property that $\left \langle L, \CO_C(s)\right\rangle = s^*L$ for any smooth section $s$ of the universal curve. 
	\end{proof}
	We say a class $\gamma$ in $\mathsf{PP}(\Jbar_C)$ is {\em supported away from the integral locus} if it is a linear combination of normally decorated strata whose stabilizations correspond to reducible (i.e. non-integral) curves.
	\begin{proposition}\label{pro:FM_aj}
		For  $b\in\ZZ$ and $\sfa\in\ZZ^n$ with $\sum_i a_i=(2g-2+n)b$, there exists a class $\gamma_{b;\sfa}$ in $\mathsf{PP}(\Jbar_C)$ supported away from the integral locus such that
		\[\fm(\aj_{b;\sfa}) = \exp(-b\kappa_{0,1}+\sum_{i=1}^n a_i\xi_i)(1+\gamma_{b;\sfa})\,.\]
	\end{proposition}
	\begin{proof}
		We consider the diagram \eqref{eq:model}. Let $\aj_{b;\sfa}': \Jbar_{b;\sfa}^\epsilon\to\Jbar_C\times_B\Jbar_C$ be the pullback of the resolved Abel-Jacobi section. Then we have
		\begin{align}
			(\aj_v')^!(\Td(-T_{\Jbar_C\times_B\Jbar_C})\cap\tau(\overline{\P})) &=(\aj_{b;\sfa}')^!(\Td(-T_{\Jbar_C\times_B\Jbar_C})\cap\tau(Rf_*\widetilde{\P})) \nonumber\\
			&=(\aj_{b;\sfa}')^!(\Td(-T_{\Jbar_C\times_B\Jbar_C})\cap f_*\tau(\widetilde{\P})) \nonumber\\
			&=g_*(\widetilde{\aj}_{b;\sfa})^!\big(\ch(\widetilde{\P})\Td(T_{\Jbar_C^{(2)}})f^*\Td(-T_{\Jbar_C\times_B\Jbar_C})\cap [\Jbar_C^{(2)}]\big) \nonumber\\
			&=g_*\big(\widetilde{\aj}_{b;\sfa}^*(\ch(\widetilde{\P})\Td^\vee(\cR_{\widetilde{\pi}})\cdot\pi_1^*\Td^\vee(\cR_\pi)^{-1}\cdot\pi_2^*\Td^\vee(\cR_\pi)^{-1})\cap [\widetilde{J_{b;\sfa}^\epsilon}]\big)\,.\label{eq:fm_aj1}
		\end{align}
		The first equality follows from Theorem \ref{thm:Plog=arinkin}, the second from the functoriality of $\tau$ under proper pushforward, the third from the base change formula for l.c.i. pullback, and the fourth from \eqref{eq:Td2} and Lemma \ref{lem:fssq}.

		We now compute the Fourier transform of $\aj_{b;\sfa}$. From the definition \eqref{eq:FM}, we have 
		\[
		\fm(\aj_{b;\sfa}) = \pi_{2*}(\aj'_{b;\sfa})^!\big(\Td(-T_{\Jbar_C\times_B\Jbar_C})\cap\tau(\overline{\P})\big)\,.
		\]
		Therefore, applying \eqref{eq:fm_aj1}, we get
		\begin{align*}
			\fm(\aj_{b;\sfa}) = \pi_{2*}g_*\big(\widetilde{\aj}_{b;\sfa}^*(\ch(\widetilde{\P})\Td^\vee(\cR_{\widetilde{\pi}})\cdot\pi_1^*\Td^\vee(\cR_\pi)^{-1}\cdot\pi_2^*\Td^\vee(\cR_\pi)^{-1})\cap [\widetilde{J_{b;\sfa}^\epsilon}]\big)\,.
		\end{align*}
		The pullback of the extended Poincar\'e line bundle along $\widetilde{\aj}_{b;\sfa}$ is calculated in Corollary \ref{cor:pullback_P}:
		\[\widetilde{\aj}_{b;\sfa}^*\ch(\widetilde{\P})= \exp(g^*\pi_2^*(-b\kappa_{0,1}+\sum_{i=1}^n a_i\xi_i))\exp(\beta)\,.\]
		Since the pullback along log maps preserves piecewise polynomials, by applying Proposition \ref{pro:td_ss} and Corollary \ref{cor:pullback_P} the class
		\begin{equation}\label{eq:error}
			\widetilde{\aj}_{b;\sfa}^* \big(\Td^\vee(\cR_{\widetilde{\pi}})\cdot\pi_1^*\Td^\vee(\cR_\pi)^{-1}\cdot\pi_2^*\Td^\vee(\cR_\pi)^{-1}\big)
		\end{equation}
		lies in $\mathsf{PP}(\widetilde{J_{\sfa}^\epsilon})$.
		The pullback of a piecewise polynomial along a log map is also piecewise polynomial. Moreover, the composition $\pi_2\circ g$ is a sequence of log modifications. Since the pushforward of a piecewise polynomial is a piecewise polynomial \cite[Proposition 67]{PRSS}, we have
		\[\gamma_{b;\sfa}:=\pi_{2*}g_*\big(\exp(\beta)\cdot \widetilde{\aj}_{b;\sfa}^* \big(\Td^\vee(\cR_{\widetilde{\pi}})\cdot\pi_1^*\Td^\vee(\cR_\pi)^{-1}\cdot\pi_2^*\Td^\vee(\cR_\pi)^{-1}\big)\big) -1 \in \mathsf{PP}(\Jbar_C)\,.\]
		
		Lastly, we show that $\gamma_{b;\sfa}$ is supported away from the integral locus. Over the locus where the underlying curve is integral, the Abel-Jacobi section is already well-defined, so the indeterminacy of $\aj_{b;\sfa}$ occurs only on reducible curves. Therefore, over the integral locus,
		\begin{enumerate}
			\item[(i)] the function $\beta$ is $0$, and
			\item[(ii)] the section $\widetilde{\aj}_{b;\sfa}^*$ lands in the locus of $\Jbar_C^{(2)}$ which is isomorphic to $\Jbar_C \times_B \Jbar_C$.
		\end{enumerate}
		From (i), we find that $\exp(\beta)=1$. Combining (ii) with the Todd class formula of Proposition \ref{prop:tdbirationalmodel} we see that the smaller of the coordinates $\ell_{e_1}',\ell_{e_2}'$ is zero, reducing \eqref{eq:error} to 1 over the integral locus. Combining these two, we conclude that the class $\gamma_{b;\sfa}$ vanishes over the integral locus. 
	\end{proof}
	\subsection{Partial Fourier transform}
	Let $\pi:\Jbar_C\to B$ be the compactified Jacobian associated with a nondegenerate stability condition.  Let $J_C\subset\Jbar_C$ be the open substack corresponding to line bundles. We consider the Fourier transform restricted to the product  $\Jbar_C\times_B J_C$.
	\begin{proposition}\label{pro:F0}
		Let $\fm^{-1}\in\CH_*(\Jbar_C\times_B\Jbar_C)$ be the inverse Fourier kernel \eqref{eq:fminv}. Then we have
		\[
		\fm^{-1}|_{\Jbar_C\times_B J_C} = (-1)^g\ch(\P^\vee)\cdot \pi_1^*\,\Td^\vee(\mathcal{R}_\pi)^{-1}\,,
		\]
		where $\mathcal{R}_\pi$ is the residue sheaf for $\pi$ defined in Lemma \ref{lem:residue}.
	\end{proposition}
	\begin{proof}
		Let $\EE:=\pi_*\Omega_p^1\to B$ be the Hodge bundle on $B$. Consider the projection $\pi_1\times_B\pi_2:\Jbar_C\times_B\Jbar_C\to B$. On $\Jbar_C\times_BJ_C$, the inverse Fourier transform is expressed as
		\[\fm^{-1}|_{\Jbar_C\times_B J_C} = (-1)^g\cdot \ch(\P^\vee)\cdot (\pi_1\times_B \pi_2)^*\ch(\det \EE)\cdot\Td(T_{\Jbar_C\times_B J_C})\cdot\Td(-(\pi_1\times_B \pi_2)^*T_B)\,,\]
		where the sign comes from the shift $[-g]$ in $\Pbar^{-1}$. 
		Applying Lemma \ref{lem:residue}, this simplifies to
		\[\Td(T_{\Jbar_C\times_B J_C})\cdot\Td(-(\pi_1\times_B \pi_2)^*T_B) = \Td^\vee(\Rels_\pi)^{-1}\cdot\Td^\vee((\pi_1\times_B \pi_2)^*\EE^{\oplus 2})\,.\]
		From \cite[Corollary 5.3]{mumford83}, it follows that $c(\EE) c(\EE^\vee)=1$, which implies that all Pontryagin classes of  $\EE$ vanish. Consequently, the identity
		\begin{equation}\label{eq:toddE}
			\Td^\vee(\EE)\ch\left(\frac{1}{2}\det (\EE)\right)=1\,.
		\end{equation}
		holds. With this, the desired result follows.
	\end{proof}
	Now we consider a partial Fourier transform
	\begin{equation}\label{eq:fm_circ}
		\fm^\circ :=\ch(\P^\vee)\in\CH_*(\Jbar_C\times_B J_C)\,.
	\end{equation}
	When the genus is important, we denote it by $\fm^\circ_g$. 
	
	We prove an elementary property of $\fm^\circ$. Let $\Jzero_C\to B$ be the relative Jacobian of multidegree zero line bundles. When the stability condition for the second factor is small, then we can further restrict \eqref{eq:fm_circ} to $\Jbar_C\times_B\Jzero_C$. The morphism $\mu$ is the action defined in   \eqref{eq:action}.
	\begin{lemma}\label{lem:c1}
		On $\Jbar_C\times_B\Jzero_C$, we have $c_1(\P) = -\mu^*\Th +\pi_1^*\Th +\pi_2^*\Th$.
	\end{lemma}
	\begin{proof}
		Let $\hat{p}: \widehat{C}\to \Jbar_C\times_B J_C$ be the universal quasi-stable curve. By the Deligne pairing formula \eqref{eq:Deligne}, we have $c_1(\P) = \hat{p}_*(c_1(\L_1)c_1(\L_2))$.
		Therefore we get the desired formula.
	\end{proof}
	\subsection{Recursive structure of Fourier transform}
	
	We prove that the image of the Fourier transform, when restricted to $\Jzero_{g,n}$, has a recursive structure. For a nondegenerate stability condition $\epsilon$, we consider the partial Fourier transform \eqref{eq:fm_circ}, given by $\fm^\circ_g :\CH^*(\Jbar^\epsilon_{g,n})\to\CH^*(\Jzero_{g,n})$. The following is the main result of this section:
	\begin{theorem}\label{thm:inductive}
		Let $\gamma$ be a piecewise polynomial class on $\Jbar_{g,n}^\epsilon$ and let $\Xi_g$ be a monomial in the $\xi$-classes and the $\kappa_{0,1}$-class on $\Jbar_{g,n}^\epsilon$. There exists an effective algorithm to compute $\fm_g^\circ(\gamma\cdot\Xi_g)$ in terms of $\fm_h^\circ (\Xi_h)$ over $\Mbar_{h,m}$ for all pairs $(h,m)$, where either $h=g,m \le n$ or $h< g$ with any $m\geq 1$.
	\end{theorem}
	The precise algorithm in Theorem \ref{thm:inductive} will be clear from the proof below. The basic idea is straightforward - a piecewise polynomial class on $\Jbar_{g,n}^\epsilon$ is supported on the boundary. Each boundary stratum (of positive codimension) corresponds to some stable graph $\Gamma$ along with a quasi-stable subdivision $\widetilde{\Gamma}$ and an $\epsilon$-stable multidegree $\delta$ on $\widetilde{\Gamma}$. We show that the Fourier transform reduces to the transforms on the ''smaller"  $\Jbar_{h,m}$ corresponding to the vertices of $\Gamma$, via an analysis of the boundary strata of the logarithmic Picard group.
	
	Before proving Theorem \ref{thm:inductive}, we study two extreme cases: a class supported on a maximally subdivided stable graph \eqref{eq:bdy} and a class supported on a stable graph. We begin with the first case. Consider the following diagram where the right square is a fiber product
	\begin{equation}\label{eq:Jzero_tor}
		\begin{tikzcd}
			& J^{\underline{0}, \Gm^h}\ar[r,"r_h"]\ar[d]\ar[dl,"q"'] & \Jzero_{g,n}\ar[d]\\
			\Jzero_{g-h,n+2h}\ar[r] & \Mbar_{g-h,n+2h}\ar[r] & \Mbar_{g,n}\,.
		\end{tikzcd}
	\end{equation}
	\begin{proposition}\label{pro:fm_max}
		For $1\leq h\leq g$, let $\jmath_h: \Jbar_{g-h,n+2h}\to\Jbar_{g,n}$ be the morphism defined in \eqref{eq:bdy}. For a class $\alpha$ in $\CH^*(\Jbar_{g-h,n+2h})$, we have
		\[\fm_g^\circ ((\jmath_h)_*\alpha) = (r_{h})_*q^*\fm_h^\circ(\alpha)\,.\]
	\end{proposition}
	\begin{proof}
		To simplify notation, we denote by $C=\overline{\mathcal{C}}_{g,n} \to B = \Mbar_{g,n}$ and $C'=\overline{\mathcal{C}}_{g-h,n+2h} \to \partial B = \Mbar_{g-h,n+2h}$ the two universal curves.  Consider the diagram:
		\begin{equation}\label{eq:fm_maximal}
			\begin{tikzcd}
				\Jbar_{\partial C}\times_{\partial B} J^{\underline{0},\Gm^h}_{\partial C} \ar[d,"\pi_1'"]\ar[r,"\jmath_h'"] & \Jbar_C\times_{B}\Jzero_C\ar[r,"\pi_2"]\ar[d,"\pi_1"] &\Jzero_C \\
				\Jbar_{\partial C}\ar[r,"\jmath_h"] & \Jbar_C & 
			\end{tikzcd}
		\end{equation}
		where $\jmath_h$ is the map defined in Lemma \ref{lem:boundary} and the square is Cartesian. Let $q: J^{\underline{0}, \Gm^h}_{C'} \to \Jzero_{C'}$ be the projection as described in \eqref{eq:Jzero_tor}. Let $\nu: C' \to \overline{C'}$ be the morphism that results from gluing the $(n+2i-1)$-th and $(n+2i)$-th markings for all $1 \leq i \leq h$. Let $\L^{\Gm}$ denote the restriction of the universal line bundle on $C \to \Jzero_C$ to $J^{\underline{0}, \Gm^h}_{C'}$, and let $\L$ be the universal line bundle on $C' \to \Jzero_{C'}$.
		
		We first compare two Poincaré line bundles. Let $\P_g$ be the Poincaré line bundle on $\Jbar_C \times_B \Jzero_C$ and let $\P_{g-h}$ be the Poincaré line bundle on $\Jbar_{C'} \times_{\partial B} \Jzero_{C'}$. It is easier to work with the universal sheaf described by a rank $1$ torsion free sheaf. Let $\F_g$ be the universal sheaf on $C\to \Jbar_C$ and $\F_{g-h}$ be the universal sheaf on $C'\to \Jbar_{C'}$. For the gluing morphism $\nu: C'\to \overline{C'}$ over $\Jbar_{C'}$, Lemma \ref{lem:boundary} implies that
		\[\jmath_h^*\F_g\cong \nu_*\F_{g-h}\,.\]
		The Deligne pairing naturally extends when one of the factors is a rank $1$ torsion-free sheaf via an appropriate determinant line bundle (\cite[eq. (3)]{arinkin1}). 
		For $\id\times q :\Jbar_{\partial C}\times_{\partial B} J^{\underline{0},\Gm^h}_{\partial C}\to \Jbar_{\partial C}\times_{\partial B}\Jzero_{\partial C}$, we have
		\begin{align}\label{eq:P_compare}
			(\id\times q)^*\P_{g-h} &\cong (\id\times q)^*\langle \F_{g-h}, \L\rangle_{C'} \cong \langle\F_{g-h},\nu^*\L^{\Gm}\rangle_{C'} \nonumber\\
			& \cong \langle\nu_*\F_{g-h},\L^{\Gm}\rangle_{\overline{C'}} \cong (\jmath_h')^*\P_g\,,
		\end{align}
		where the third isomorphism follows from $\nu_*$ being exact. 
		
		We apply the above computation to \eqref{eq:fm_maximal}. For any $\alpha\in \CH^*(\Jbar_{g-h,n+2h})$, we have
		\begin{align*}
			\fm_g^\circ((\jmath_{h})_*\alpha) &=\pi_{2*}(\pi_1^*(\jmath_{h})_*\alpha\cup\ch(\P_g)) \\
			&=\pi_{2*}(\jmath'_{h})_*((\pi_1')^*\alpha\cup (j_h')^*\ch(\P_g))\\
			&= \pi_{2*}(\jmath'_{h})_*((\id\times q)^*\pi_1^*\alpha\cup (\id\times q)^*\ch(\P_{g-h}))\\
			&= \pi_{2*}(\jmath'_{h})_*(\id\times q)^*(\pi_1^*\alpha\cup\ch(\P_{g-h}))\\
			&= (r_{h})_*\pi_{2*}(\id\times q)^*(\pi_1^*\alpha\cup\ch(\P_{g-h}))\\
			&= (r_{h})_*q^*\pi_{2*}(\pi_1^*\alpha\cup\ch(\P_{g-h}))=(r_{h})_*q^*\fm_{g-h}^\circ(\alpha)\,,
		\end{align*}
		where the third equality follows from \eqref{eq:P_compare} and the other equalities are from the projection formula. This completes the proof.
	\end{proof}
	
	Now we describe the strata of $\Jbar_{g,n}^\epsilon$ corresponding to stable graphs. For a stable graph $\Gamma$ of genus $g$ with $n$ markings, let
	\[\Mbar_\Gamma := \prod_{v\in V(\Gamma)} \Mbar_{g(v),n(v)}\to\Mbar_{g,n}\]
	denote the gluing map. Let $\delta: V(\Gamma)\to\ZZ$ be an $\epsilon$-stable multidegree. Let $\Jbar_{\Gamma_\delta}$ be the compactified Jacobian defined as in \eqref{eq:stratum}. Then $\Jbar_{\Gamma_\delta}$ is smooth.
	
	Let $\Gamma$ be a stable graph with $h^1(\Gamma)=0$, and let $\delta:V(\Gamma)\to \ZZ$ be the unique $\epsilon$-stable multidegree. Then there exists a nondegenerate stability condition $\epsilon_v$ for $\Mbar_{g(v),n(v)}$ of degree $\delta(v)$ such that the stratum $\Jbar_{\Gamma_\delta}$ is isomorphic to  $\prod_{v\in V(\Gamma)}\Jbar_{g(v),n(v)}^{\epsilon_v}$.
	\begin{proposition}\label{pro:tree}
		Let $\Gamma$ be a stable graph with $h^1(\Gamma)=0$. For each vertex $v\in V(\Gamma)$,  choose  a class $\alpha_v\in \CH^*(\Jbar^{\epsilon_v}_{g(v),n(v)})$. Then, the Fourier transform satisfies the factorization property:
		\[\fm_g^\circ\big(\prod_{v\in V(\Gamma)}\alpha_v\big) = \prod_{v\in V(\Gamma)}\fm_{g(v)}^\circ(\alpha_v)\,.\]
	\end{proposition}
	\begin{proof}
		Let $\P_\Gamma$ be the restriction of the Poincar\'e line bundle on $\Jbar^{\epsilon}_{g,n}\times_{\Mbar_{g,n}}\Jzero_{g,n}$ to $\Jbar_{\Gamma_\delta}\times_{\Mbar_{\Gamma}}\Jzero_\Gamma$. By \cite[Lemma 5.5]{MRV1} we have
		\[\P_\Gamma \cong \boxtimes_{v\in V(\Gamma)} \P_{v}\,,\]
		where $\P_v$ are Poincar\'e line bundles on $\Jbar_{g(v),n(v)}\times_{\Mbar_{g(v),n(v)}}\Jzero_{g(v),n(v)}$. This gives the result. 
	\end{proof}
	
	Let $\Gamma$ be a stable graph with $h^1(\Gamma)>0$, and let $\delta:V(\Gamma)\to \ZZ$ be an $\epsilon$-stable multidegree. We describe the boundary strata $\Jbar_{\Gamma_\delta}$, up to birational equivalence. Consider the composition
	\begin{equation}\label{eq:stratum_lpic}
		\Jbar_{\Gamma_\delta}\to \pic_{\Gamma_\delta}\xrightarrow{q}\prod_{v\in V(\Gamma)}\pic_{g(v),n(v),\delta(v)}\to \prod_{v\in V(\Gamma)}\lpic_{g(v),n(v)}\,,
	\end{equation}
	where the first morphism is from \eqref{eq:stratum}, the second morphism is from \eqref{eq:picboundary}, and the third morphism is from \eqref{eq:pic_lpic}. For each vertex $v$, choose any nondegenerate stability condition $\epsilon_v$ for $\Mbar_{g(v),n(v)}$ of degree $\delta(v)$, and set
	\[\Jbar^{\Gamma_\delta} := \prod_{v\in V(\Gamma)}\Jbar_{g(v),n(v)}^{\epsilon_v}\,.\] 
	Combining \eqref{eq:stratum_lpic} and \eqref{eq:Jbar_lpic} yields an f.s. fiber diagram:
	\begin{equation}\label{eq:tildeJ}
		\begin{tikzcd}
			\widetilde{J}_{\Gamma_\delta}\ar[r,"s"]\ar[d,"r"] & \Jbar^{\Gamma_\delta}\ar[d,"t"]\\
			\Jbar_{\Gamma_\delta}\ar[r,"\rho"] &\prod_{v\in V(\Gamma)}\lpic_{g(v),n(v)}\,.
		\end{tikzcd}
	\end{equation}
	Since the right vertical map is a log modification, $r:\widetilde{J}_{\Gamma_\delta}\to\Jbar_{\Gamma_\delta}$ is also a log modification.
	
	To understand the space $\widetilde{J}_{\Gamma_\delta}$, it remains to analyze the composed map $\rho$ from (\ref{eq:stratum_lpic}). 
	\begin{lemma}\label{lem:tildeJ}
		The stack $\widetilde{J}_{\Gamma_\delta}$ parametrizes tuples
		\[
		(C,L,\tau,C_v',\alpha_v)\,,
		\] 
		where 
		\begin{itemize}
			\item $C$ is a quasi-stable curve, and $L$ is a $\epsilon$-stable line bundle. 
			\item $\tau$ is a specialization $\Gamma_C \to \Gamma$ satisfying \eqref{eq:mdeg}.
			\item $C_v'$ is a quasi-stable model of $C_v$, and $\alpha_v$ is a piecewise linear function on $C_v'$ such that 
			\[
			L_v(\alpha_v) \textup{ is } \epsilon_v\textup{-stable}
			\] 
		\end{itemize}
		%Furthermore, the map $s: \widetilde{J}_{\Gamma_\delta} \to \Jbar^{\Gamma_\delta}$ is a log blowup of a $(\mathbb{G}_m^h)^{\log}$-torsor, where $h = h^1(\Gamma)$. 
	\end{lemma}
    \begin{proof}
    The moduli space $\Jbar_{\Gamma_\delta}$ parametrizes triples $(C,L,\tau)$, where $C$ is a quasi-stable curve, $L$ is an $\epsilon$-stable line bundle, and $\tau$ is a specialization from the dual graph $\Gamma_C$ of $C$ to $\Gamma$, so that for each $v \in V(\Gamma)$, 
    \begin{equation}\label{eq:mdeg}
        \sum_{w \in \tau^{-1}(v)} \deg L|_{C_w} = \delta(v).
    \end{equation}
	Let $\nu: C^\nu \to C$ be the partial normalization of $C$ determined by normalizing along the nodes corresponding to the edges $\tau^{-1}(e), e \in E(\Gamma)$. Write $C^\nu_v$ for the union of components of $C^{\nu}$ corresponding to the subgraph $\tau^{-1}(v)$ for $v \in V(\Gamma)$. The morphism $q$ sends $(C,L,\tau)$ to 
	\[
	(C^\nu_v,\, \nu^*L_v)_{v \in V(\Gamma)}\,,
	\]
	where $\nu^*L_v$ is the restriction of $\nu^*L$ to $C^\nu_v$. The quasistable components of $C$ lying over edges $e \in E(\Gamma)$ become rational tails of length $1$ attached to the components $C^{\nu}_v$, and $\nu^*L_v$ has degree $1$ on such tails. There is a unique piecewise linear function $\beta_v$ on $C^{\nu}_v$ which has value $0$ at the vertices which are not rational tails, and slope $1$ on the oriented edges emanating away from the rational tail vertices. The line bundle $\nu^*L_v(\beta_v)$ then has degree $0$ on rational tails, and hence is induced by a line bundle $L_v$ on the curve $C_v$ obtained by contracting the rational tails of $C^{\nu}_v$. The composed map $\rho$ takes $(C,L,\tau)$ to 
	\[
	(C_v,[L_v])_{v \in V(\Gamma)}
	\]
	where $L_v$ is the associated log line bundle of $L_v$.
    \end{proof}
One difficulty in understanding \eqref{eq:tildeJ} is that $s: \widetilde{J}_{\Gamma_\delta} \to \Jbar^{\Gamma_\delta}$ has positive relative dimension. To address this, we introduce an algebraic log stack that sits in between, which provides a genus-recursive structure for the Fourier transform.

The logarithmic multiplicative group $\Glog$ is an algebraic log stack whose functor of points on log schemes $S$ is given by $\Glog(S) := H^0(S,\Ms^\gp_S)$ (\cite[Definition~2.2.7]{MW22}). Since $\Gm\subset \Glog$,  any $(\Gm)^h$-torsor $T$ on a log scheme extends canonically to a $(\Glog)^h$-torsor $T_{\log}$ via extension of scalars.
\begin{proposition}\label{prop:Glog_torsor}
    For a stable graph $\Gamma$, choose a spanning tree of $\Gamma$ with complementary edges labeled by $e_i = (h_i,h_i')$ for $1 \leq i \leq h^1(\Gamma)$. Let $T$ be the $(\Gm)^{h^1(\Gamma)}$-torsor over $\Jbar^{\Gamma_\delta}$ defined by
    \begin{equation}\label{eq:torsor}
        T = \bigoplus_{i=1}^{h^1(\Gamma)} (L_{h_i}^{\vee} \otimes L_{h'_i})^\times\,.
    \end{equation}
    Let $T_{\log}$ as the associated $(\Glog)^{h^1(\Gamma)}$-torsor. Then the morphism $s: \widetilde{J}_{\Gamma_\delta}\to \Jbar^{\Gamma_\delta}$ defined in \eqref{eq:tildeJ} is a log modification of $T_{\log}$.
\end{proposition}
	\begin{proof}
		We use the explicit description of $\widetilde{J}_{\Gamma_\delta}$ given in Lemma \ref{lem:tildeJ}. Given data $(C,L,\tau)$ as above and an edge $e=\{h,h'\}$ consisting of two half edges, we write $C_h,C_{h'}$ for the two preimages of the node in the stabilization of $C$ corresponding to the half edges $\tau^{-1}(h),\tau^{-1}(h') \in \Gamma_C$. Similarly, let $[L]|_h, [L]|_{h'}$ denote the restrictions of an element of $\prod_{v \in V(\Gamma)} \lpic_{g(v),n(v)}$ to $C_h$ and $C_{h'}$, respectively.
        
        Let $\mathcal{T}$ denote the category fibered in groupoids over the (unrigidified) stack $\prod_{v \in V(\Gamma)} \mathbf{LogPic}_{g(v),n(v)}$ parametrizing an element of $\prod_{v \in V(\Gamma)} \mathbf{LogPic}_{g(v),n(v)}$ together with isomorphims of log line bundles at half edges 
        \[
        [L]|_{h} \xrightarrow{\cong} [L]|_{h'}\,.
        \]
		The space of isomorphism classes $\pi_0(\mathcal{T})$ is a $(\Glog)^{h^1(\Gamma)}$-torsor. Indeed,  choosing a spanning tree in $\Gamma$ with complementary edges $e_i=\{h_i,h_i'\}$, for $i=1,\cdots,h_1(\Gamma)$ yields an identification 
        \begin{equation}\label{eq:torsor2}
            \pi_0(\mathcal{T})\cong \bigoplus_{i=1}^{h^1(\Gamma)} \mathsf{Iso}([L]|_{h_i},[L]|_{h_i'}) \cong \bigoplus_{i=1}^{h^1(\Gamma)} ([L]|_{h_i}^{\vee} \otimes [L]|_{h'_i})^\times\,.
        \end{equation}

        The map $\rho$ in \eqref{eq:tildeJ} lifts to a map to $\pi_0(\mathcal{T})$. Since the log Picard group $\lpic_{g,n}$ satisfies the log N\'eron mapping property \cite[Theorem~6.11]{HMOP}, and the same holds for $\Glog$--in particular $\pi_0(\mathcal{T})$--also satisfies the log N\'eron mapping property. Moreover, the normalization $\Jbar_{\Gamma_\delta}$ is toroidal, and there is a natural rational map $\Jbar_{\Gamma_\delta} \dashrightarrow{} \pi_0(\mathcal{T})$ over the locus of line bundles. Thus, the map $\rho$ lifts to
		\[
		\rho' : \Jbar_{\Gamma_\delta} \to \pi_0(\mathcal{T}).
		\]
        
        In fact, $\rho'$ admits the following explicit description. Using the notation of Lemma \ref{lem:tildeJ}, a triple $(C,L,\tau)$ determines the line bundle $\nu^*L_v$ on $C_v^{\nu}$, together with isomorphisms $\phi_{h,h'}:\nu^*L|_{h} \cong \nu^*L|_{h'}$ for each pair of half edges $(h,h')$ corresponding to a normalized node. Here, for a vertex $v$ adjacent to $h$, $\nu^*L|_{h}$ denotes the restriction of $\nu^*L_v$ to the preimage of the normalized node corresponding to $h$ in $C^{\nu}_v$. As noted above, the line bundle $\nu^*L_v(\beta_v)$ descends to a line bundle $L_v$ on the contraction $C_v$ of $C_v^{\nu}$, and hence defines a log line bundle $[L_v]$ on $C_v$. Likewise, each $\phi_{h,h'}$ induces an isomorphism of log line bundles $[\phi_{h,h'}]:[\nu^*L|_{h}] \to [\nu^*L|_{h'}]$. Since $[\nu^*L_v] = [\nu^*L_v(\beta_v)]$, it induces an isomorphism 
        \[
        [\phi_{h,h'}]:[L]|_h \xrightarrow{\cong} [L]|_{h'}\,.
        \]
        Thus, the  map $\rho'$ takes $(C,L,\tau)$ to the data $(C_v,[L_v],[\phi_{h,h'}])$. 

        The map $\rho'$ is a log modification, by a standard deformation theoretic argument. Let $(C,L,\tau)$ be a point of $\Jbar_{\Gamma}$ with image $(C_v,[L_v],[\phi_{h,h'}])$. %Write $\Gamma_C$ for the dual graph of $C$ and $\Gamma_{C^{\nu}}$ for that of its partial normalization $C^{\nu}$. 
        The relative log tangent space of $\Jbar_{\Gamma_\delta}$ over $\Mbar_{\Gamma}$ at $(C,L,\tau)$ is $H^1(C,\CO_C)$ which is an extension of $H^1(C^\nu,\CO_{C^\nu})$ by $\mathbb{G}_{a}^{h^1(\Gamma)}$.
        %, parametrizing $h^1(\Gamma_{C^{\nu}})$ deformations of the line bundles $\nu^*L_{v}$ and $h^1(\Gamma)$ deformations of the isomorphisms $\phi_{h,h'}$. 
        By \cite[Lemma 4.13.2]{MW22}, deformations of $\nu^*L_v$ agree with those of the log line bundles $[L_v]$, and the log tangent space of both $\mathbb{G}_m$ and $\mathbb{G}_m^{\log}$ is $\mathbb{G}_a$, so deformations of $\phi_{h,h'}$ and $[\phi_{h,h'}]$ coincide as well. Thus $\rho'$ induces an isomorphism on log tangent spaces, so it is log \'etale. It is also proper and birational, and hence it is a log modification. Therefore, diagram \eqref{eq:tildeJ} in fact factors as 
		\begin{equation*}
			\begin{tikzcd}
				\widetilde{J}_{\Gamma_\delta}\ar[r]\ar[d] & \pi_0(\mathcal{T})|_{\Jbar^{\Gamma_\delta}} \ar[r] \ar[d] & \Jbar^{\Gamma_\delta}\ar[d,"t"]\\
				\Jbar_{\Gamma_\delta}\ar[r,"\rho'"] & \pi_0(\mathcal{T}) \ar[r] & \prod_{v\in V(\Gamma)}\lpic_{g(v),n(v)}\,.
			\end{tikzcd}
		\end{equation*}
		with the top right arrow a $(\Glog)^{h^1(\Gamma)}$-torsor and the top left arrow a log modification. 
        
        The pullback of the $(\Glog)^{h^1(\Gamma)}$-torsor $\pi_0(\mathcal{T})$ to $\Jbar^{\Gamma_\delta}$ along the morphism $t$ is induced by the $(\Gm)^{h^1(\Gamma)}$-torsor, since the universal log line bundle on $\Jbar^{\Gamma_\delta}$ is represented by an actual line bundle. The explicit formula \eqref{eq:torsor} follows from \eqref{eq:torsor2}. This concludes the argument. 
		%To see the second statement, we abstract the situation. Suppose $B$ is an idealized log smooth, reduced and irreducible base, and $\pi: C \to B$ is a family of log curves, with dual graph $\Gamma$ over the generic point. Suppose that on $B$, the partial normalization $\nu: C^\nu \to C$ along the nodes corresponding to edges of $\Gamma$ can be defined. The curve $C^\nu$ has a log structure making it log smooth over $B$, where we do not put log structure on the markings. On $C$, we can find a log structure $N_C$ which agrees with $M_C$ away from the normalized nodes, and is generated by $\pi^*M_B$ around the normalized nodes. Then we have a roof 
	%	\[
	%	\begin{tikzcd}
	%		& N_C \ar[ld] \ar[rd] & \\
	%		M_C & & \nu_*M_{C^{\nu}}
	%	\end{tikzcd}
	%	\]
	%	The bounded monodromy subsheaf of $R^1\pi_*\nu_*M_{C^{\nu}}^{\gp}$ is $\lpic_{C^\nu}$. A calculation with the long exact sequence associated to $\pi$ gives a short exact sequence.  
	%	\[
	%	\begin{tikzcd}
	%		0 \ar[r] & \Hom(H_1(\Gamma),\mathbb{G}_m^{\log}) \ar[r] & (R^1\pi_*N_C^\gp)^{\dagger} \ar[r] & \lpic_{C^{\nu}} \to 0
	%	\end{tikzcd}
	%	\]
	%	We can now apply this construction to $B= \Jbar_{\Gamma_\delta}$ with its universal curve $C_\Gamma$. Since the map $\Jbar_{\Gamma_\delta}$ to $\lpic_C$ is induced by taking the associated $M^{\gp}$ torsor to the universal $\mathcal{O}^*$-torsor $\mathcal{L}$ on $C_\Gamma$, it is induced by first taking the associated $N^{\gp}$-torsor.  
	\end{proof}

For each $v$, denote by $\kappa_{0,1}[v] \in \CH^1(\Jbar_{\Gamma_\delta})$ the class pulled back from $\pic_{g(v),n(v)}$.
	\begin{corollary}\label{cor:xi}
		Let $r:\widetilde{J}_{\Gamma_\delta}\to\Jbar_{\Gamma_\delta}$ and $s:\widetilde{J}_{\Gamma_\delta}\to \Jbar^{\Gamma_\delta}$ be the projections defined in \eqref{eq:tildeJ}. For $1\le i\le n$,  $r^*\xi_i=s^*\xi_i+\alpha_i$ for some $\alpha_i\in\mathsf{PL}(\widetilde{J}_{\Gamma_\delta})$. Furthermore, for each $v \in V(\Gamma_\delta)$, $r^*\kappa_{0,1}(v)=s^*\kappa_{0,1} [v]+\alpha[v]$ for some $\alpha[v] \in\mathsf{PL}(\widetilde{J}_{\Gamma_\delta})$.
	\end{corollary}
	\begin{proof}
	    Let $r^*C$ and $s^*C$ denote the universal curves on $\widetilde{J}_{\Gamma_\delta}$, pulled back respectively from $\Jbar_{\Gamma_\delta}$ and $\Jbar^{\Gamma_\delta}$. Consider their components in the respective normalizations corresponding to a vertex $v \in V(\Gamma)$. Following the notation in Lemma~\ref{lem:tildeJ}, these components are $C_v'$ and $C_v$, with a log blowup $C_v' \to C_v$. The universal bundles $r^*L$ and $s^*L$, pulled back to $C_v'$ then become 
        \[
        L_v(\alpha_v), \quad L_v
        \]
        respectively, where $\alpha_v$ is the piecewise linear function defined in Lemma \ref{lem:tildeJ}. The log canonical lines bundles pulls back to the log canonical bundles on each $C'_v$ and $C_v$. Therefore, 
        \[
        r^*\xi_i = x_i^*(c_1(L_v(\alpha_v))) = x_i^*c_1(L_v) + x_i^*\alpha_v = s^*\xi_i + \alpha_i 
        \]
        for the piecewise linear function $\alpha_i = x_i^*\alpha_v$. 
        
        For $\kappa_{0,1}[v]$, let $p_v:C_v' \to \widetilde{J}_{\Gamma_\delta}$ be the projection. Then 
        \begin{align*}
            r^*\kappa_{0,1}[v] &= (p_v)_*(c_1(L_v(\alpha_v)) \cdot c_1(\omega_{\log})) = (p_v)_*(c_1(L_v)\cdot c_1(\omega_{\log}))+(p_v)_*(c_1(L_v)\cdot \alpha_v)\\ 
            &= (p_v)_*(c_1(L_v)\cdot c_1(\omega_{\log}))+\alpha[v]
        \end{align*}
        for the piecewise linear function $\alpha[v]$ defined by
        \[
        \alpha[v] = (p_v)_*(\alpha_v\cdot c_1(\omega_{\log})) = \left \langle \alpha_v, c_1(\omega_{\log}) \right \rangle
        \]
        as in Lemma \ref{lem:deligne_pl}. 
	\end{proof}

	We compare Poincar\'e line bundles up to birational equivalence.  For the relative Jacobian $\Jzero_{g,n}\to\Mbar_{g,n}$ and a stable graph $\Gamma$ of genus $g$ with $n$ markings, we denote
	\[J_\Gamma := \prod_{v\in V(\Gamma)}\Jzero_{g(v),n(v)}\]
	and
	\[J^{\Gm}_\Gamma := \Jzero_{g,n}|_{\Mbar_\Gamma}\,.\]
	Similar to the diagram \eqref{eq:Jzero_tor}, we have a commutative diagram
	\begin{equation}\label{eq:J_Gamma}
		\begin{tikzcd}
			& J^{\Gm}_\Gamma\ar[r,"r_\Gamma"]\ar[d]\ar[dl,"q"'] & \Jzero_{g,n}\ar[d]\\
			\Jzero_{\Gamma}\ar[r] & \Mbar_{\Gamma}\ar[r] & \Mbar_{g,n}\,,
		\end{tikzcd}
	\end{equation}
	where the morphism $q:J^{\Gm}_\Gamma\to J_\Gamma$ induced by the partial normalization of the underlying curve is a $\Gm^{h^1(\Gamma)}$-torsor by Lemma \ref{lem:torsor}. 
	For each $v\in V(\Gamma)$, let $\P_{g(v),n(v)}$ be the Poincar\'e line bundle on $\Jbar_{g(v),n(v)}\times_{\Mbar_{g(v),n(v)}}\Jzero_{g(v),n(v)}$ and denote
	\[\P^\Gamma :=\boxtimes_{v\in V(\Gamma)}\P_{g(v),n(v)}\in \mathrm{Pic}(\Jbar^\Gamma\times_{\Mbar_\Gamma} J_\Gamma)\,.\]
	Let $\P_\Gamma$ be the restriction of the Poincar\'e line bundle on $\Jbar^\epsilon_{g,n}\times_{\Mbar_\Gamma}\Jzero_{g,n}$ to $\Jbar_{\Gamma_\delta}\times_{\Mbar_\Gamma} J^{\Gm}_\Gamma$.
	From \eqref{eq:tildeJ}, we obtain a diagram
	\begin{equation}\label{eq:compare}
		\begin{tikzcd}
			\P_\Gamma \ar[d]& \widetilde{J}_{\Gamma_\delta}\times_{\Mbar_\Gamma}J^{\Gm}_{\Gamma}\ar[dl,"r"']\ar[dr,"s"] & \P^\Gamma\ar[d]\\
			\Jbar_{\Gamma_\delta}\times_{\Mbar_{\Gamma}}J^{\Gm}_{\Gamma} & & \Jbar^\Gamma\times_{\Mbar_\Gamma}J^{\Gm}_{\Gamma}\,. 
		\end{tikzcd}
	\end{equation}
	Here, we abbreviate $r=r\times\id, s=s\times\id$.
	\begin{proposition}\label{pro:Poincare_bir}
		In \eqref{eq:compare}, we have $r^*\P_\Gamma \cong s^*\P^\Gamma$ in $\mathrm{Pic}(\widetilde{J}_{\Gamma_\delta}\times_{\Mbar_\Gamma}J^{\Gm}_\Gamma)$.
	\end{proposition}
	\begin{proof}
		We first write $r^*\P_\Gamma$ and $s^*\P^\Gamma$ as Deligne pairings. Let $C_\Gamma\to \Jbar_{\Gamma_\delta}$ be the universal quasi-stable curve induced by the morphism $\Jbar_{\Gamma_\delta}\to\pic_{g,n}$. The partial normalization $\nu:C^\nu_\Gamma\to C_\Gamma$ arises from the morphism $q$ in \eqref{eq:picboundary}. We pullback $\nu$ along $\widetilde{J}_{\Gamma_\delta}\times_{\Mbar_\Gamma}J_\Gamma^{\Gm}$. Denoting by $L_1\to C_\Gamma$ the restriction of the admissible model of the universal sheaf from $\Jbar_C$ and by $M$ the pullback of the universal multidegree zero line bundle from $\Jzero_C$, the projection formula along  $\nu$ gives:
		\begin{equation}\label{eq:rP}
			r^*\P_\Gamma = r^*\langle L_1,M \rangle_{C_\Gamma} \cong \langle \nu^*L_1,\nu^* M\rangle_{C^\nu_\Gamma}\,.
		\end{equation}
		To compute $s^*\P^\Gamma$, introduce the contraction morphism $\mu: C^\nu_{\Gamma}\to C^s_\Gamma$ which contracts rational $\PP^1$ components with a single special point. Let $L_2\to C^s_\Gamma$ be the pullback of the admissible line bundle from $\Jbar^{\Gamma_\delta}$. By Lemma \ref{lem:tildeJ}, there exists a piecewise linear function $\alpha$ on $C^\nu_\Gamma$ which induces an isomorphism
		\begin{equation}\label{eq:L1L2}
			\mu^*L_2 \cong \nu^*L_1 \otimes \CO_{C^\nu_\Gamma}(\alpha)\,.
		\end{equation}
		
		We now compare the two line bundles $r^*\P_\Gamma$ and $s^*\P^\Gamma$ using the Deligne pairing. From the established relationships, we can write: 
		\begin{equation*}\label{eq:sP}
			r^*\P_\Gamma\cong \langle \nu^*L_1, \nu^*M\rangle\cong \langle\nu^*L_1\otimes \CO_{C^\nu_\Gamma}(\alpha), \nu^*M\rangle \cong \langle\mu^*L_2,\mu^*M\rangle_{C^\mu_\Gamma} \cong s^*\P^\Gamma\,.
		\end{equation*}
		The first isomorphism follows from \eqref{eq:rP}, and the second from \eqref{eq:L1L2}. For the third, note that $M$ is a line bundle of multidegree zero, so Lemma~\ref{lem:deligne_pl} applies and the correction factor $\CO_{C^\nu_\Gamma}(\alpha)$ has trivial contribution to the Deligne pairing. Hence we obtain the desired isomorphism $r^*\P_\Gamma \cong s^*\P^\Gamma$.
	\end{proof}
	
	\begin{lemma}\label{lem:ph}
		Let $B$ be a log scheme with smooth underlying scheme, and let $L_1 \oplus \cdots \oplus L_h$ be a split vector bundle over $B$. Denote by $T$ the underlying $(\Gm)^h$-torsor. Let $p:M\to B$ be a log modification of $T_{\log}$. For any $\alpha\in \mathsf{PP}(M)$, the pushforward $\pi_*(\alpha)$ lies in the subring of $\CH^*(B)$ generated by $\mathsf{PP}(B)$ and $c_1(L_1),\cdots, c_1(L_h)$.
	\end{lemma}
	\begin{proof}
		Consider the projective bundle $q:\PP^h_B:=\PP(L_1\oplus\cdots \oplus L_h\oplus\CO_B)\to B$, which admits a morphism $\PP^h_B\to T_{\log}$ over $B$ that is a log modification. Choose a log modification $M'$  with proper surjective log modifications $f:M'\to M$ and $g:M'\to\PP^h_B$. Since $f$ is a proper, surjective, and a log modification, there exists $\alpha'\in\mathsf{PP}(M')$ where $f_*(\alpha')=\alpha$.  By \cite[Proposition 67]{PRSS}, we have $g_*(\alpha')\in\mathsf{PP}(\PP^h_B)$. For any $\beta\in\mathsf{PP}(\PP^h_B)$, the projective bundle formula implies that $\pi_*(\beta)$ lies in the subring generated by $\mathsf{PP}(B)$ and $c_1(L_1),\cdots, c_1(L_h)$. Since $p_*(\alpha)=p_*f_*(\alpha')=p_*g_*(\alpha')$, the result follows. 
	\end{proof}

    We compute the Fourier transform of monomials supported on a stratum associated to a stable graph $\Gamma$  via lower genus Fourier transforms.
	\begin{corollary}\label{cor:fm_st}
		Let $\Gamma$ be a stable graph of genus $g$ with $n$ markings, and let $\delta:V(\Gamma)\to \ZZ$ be an $\epsilon$-stable multidegree. For any monomial in $\xi_1,\ldots, \xi_n$ and the $\kappa_{0,1}$-class $\Xi$ in $\CH^*(\Jbar_{\Gamma_\delta})$, there exists polynomials $\Xi_v$ in the $\xi$ and $\kappa_{0,1}$ classes on $\Jbar_{g(v),n(v)}^{\epsilon(v)}$ and $\alpha_v\in\mathsf{PP}(\Jbar_{g(v),n(v)}^{\epsilon(v)})$ such that
		\[\fm_g^\circ \Big( (j_{\Gamma_\delta})_*(\Xi) \Big)= (r_{\Gamma})_*q^*\Big(\prod_{v\in V(\Gamma)}\fm_{g(v)}^\circ(\Xi_v \cdot\alpha_v)\Big)\,.\]
		Here, $r_\Gamma$ and $q$ are defined in \eqref{eq:J_Gamma}.
	\end{corollary}
	\begin{proof}
		We consider the diagram \eqref{eq:compare}. Consider a monomial $\Xi=\kappa_{0,1}^m\cdot\prod\xi_i^{k_i}$ on $\Jbar_{\Gamma_\delta}$. By Corollary \ref{cor:xi}, there exists $\alpha_0,\alpha_1,\cdots,\alpha_n\in\mathsf{PL}(\widetilde{J}_{\Gamma_\delta})$ such that
		\begin{equation}\label{eq:xi}
			r^*\pi_1^* (\Xi) = s^*\pi_1^*\Big((\kappa_{0,1}+\alpha_0)^m\cdot\prod_i(\xi_i+\alpha_i)^{k_i}\Big)\,.
		\end{equation}

        We consider the $(\Gm)^{h^1(\Gamma)}$-torsor $T$ over $\Jbar_{\Gamma_\delta}$ defined in \eqref{eq:torsor}. By Proposition \ref{prop:Glog_torsor}, the morphism $s:\widetilde{J}_{\Gamma_\delta}\to \Jbar^{\Gamma_\delta}$ is a log modification of $T_{\log}$. Therefore, we get
		\begin{align*}
			\fm_g^\circ(j_{\Gamma_\delta*}(\Xi)) &=(r_\Gamma)_*q^*\pi_{2*}(\pi_1^*(\Xi) \cup\ch(\P_\Gamma))\\
			&=(r_\Gamma)_*q^*\pi_{2*}s_*(r^*\pi_1^*\Xi\cup s^*\ch(\P^\Gamma))\\
			&=(r_\Gamma)_*q^*\pi_{2*}s_*(s^*\pi_1^*((\kappa_{0,1}+\alpha_0)^m\cdot\prod_i(\xi_i+\alpha_i)^{k_i})\cup s^*\ch(\P^\Gamma))\\
			&=(r_\Gamma)_*q^*\Big(\prod_{v\in V(\Gamma)}\fm_{g(v)}^\circ(\Xi_v \cdot\alpha_v)\Big)\,,
		\end{align*}
		where the first equality follows from the projection formula, the second equality follows from Proposition \ref{pro:Poincare_bir}, and the third equality follows from \eqref{eq:xi} and the last equality follows from Proposition \ref{prop:Glog_torsor} and Lemma \ref{lem:ph} together with the construction of $\P^\Gamma$.
	\end{proof}
	We now complete the proof of Theorem \ref{thm:inductive}.
	\begin{proof}[Proof of Theorem \ref{thm:inductive}]
		Let $\gamma$ be a piecewise polynomial class on $\Jbar_{g,n}^\epsilon$. Such a class can be expressed as a linear combination of terms of the form $[\widetilde{\Gamma}_{\widetilde{\delta}}, \gamma_0]$, where $\widetilde{\Gamma}_{\widetilde{\delta}}$ is an $\epsilon$-stable quasi-stable graph and $\gamma_0$ is a monomial of balanced $\psi$-classes. We may assume that the graph is nontrivial.
		
		To describe the boundary stratum $\Jbar_{\widetilde{\Gamma}_{\widetilde{\delta}}}$more concretely, consider that the graph $\widetilde{\Gamma}$ arises from subdividing a stable graph $\Gamma$ at $h$ edges for some $h\ge 0$. By Lemma \ref{lem:boundary}, we can find a stability condition $\epsilon_h$ on $\Mbar_{g-h,n+2h}$ with a map 
		\[
		\Jbar_{g-h,n+2h}^{\epsilon_h} \to \Jbar_{g,n}
		\]
		whose image is the closure of the locus of curves with $h$ self nodes. We can then lift the data of the $\epsilon$-stable multidegree $\widetilde{\delta}: V(\widetilde{\Gamma}) \to \mathbb{Z}$ to an $\epsilon_h$-stable multidegree $\delta$ on a graph $\Gamma_\delta$ obtained from $\Gamma$ by cutting the $h$ edges that are subdivided in $\widetilde{\Gamma}$, creating $2h$ additional markings, and choosing some labeling for them. Using this multidegree, we find a diagram   
		%By Lemma \ref{lem:boundary}, there is a stability condition  $\epsilon_h$ on $\Mbar_{g-h,n+2h}$. Using this stability condition, there exists an $\epsilon_h$-stable multidegree $\delta: V(\Gamma)\to\ZZ$, such that the class of the boundary stratum $\Jbar_{\Gamma_\delta}$ in $\Jbar^\epsilon_{g,n}$ coincides with $\Jbar_{\widetilde{\Gamma}_{\widetilde{\delta}}}$ via the maps
		\begin{equation}\label{eq:inductive}
			\begin{tikzcd}
				\Jbar_{\Gamma_\delta} \ar[r,"j_\Gamma"]\ar[d] & \Jbar_{g-h,n+2h}^{\epsilon_h} \ar[r,"\jmath_h"]\ar[d]& \Jbar_{g,n}^\epsilon\ar[d] \\
				\Mbar_\Gamma\ar[r] & \Mbar_{g-h,n+2h}\ar[r] & \Mbar_{g,n}\,.
			\end{tikzcd}
		\end{equation}
		with $\Jbar_{\Gamma_\delta}$ being finite over the normalization of $\Jbar_{\widetilde{\Gamma}_{\widetilde{\delta}}}$. 
		
		For the monomial $\gamma_0$, and the $\psi$-classes on unstable vertices, we apply the genus $0$ relation \eqref{eq:g=0} on such unstable vertices. For a stable vertex $v$, let $\st : \M_{g(v),n(v)}\to \Mbar_{g(v),n(v)}$ denote the stabilization morphism. By \cite[Proposition 3.14]{BS1}, 
        \begin{equation}\label{eq:st}
            \st^* \psi_i = \psi_i - \delta_i\,,
        \end{equation}
        where $\delta_i \in \CH^1(\M_{g(v),n(v)})$ is the divisor corresponding to the prestable graph with two vertices of genus $0$ and $g(v)$, connected by an edge, with the $i$-th marking attached to the genus-$0$ vertex. Using \eqref{eq:st}, we may rewrite $\gamma_0$ as a polynomial involving $\psi$-classes pulled back from $\Mbar_{\Gamma}$ and $\xi$-classes at the last $2h$ markings on $\Jbar_{\Gamma_\delta}$. In particular, no $\xi$-classes are introduced along the edges of $\Gamma$. 
		
		We consider the class $\pi^*\Psi\cdot\Xi$, where $\Psi$ is a monomial in $\psi$-classes pulled back from $\Mbar_{\Gamma}$ in \eqref{eq:inductive} and $\Xi$ is a monomial in $\xi$-classes and $\kappa_{0,1}$ on $\Jbar_{\Gamma_\delta}$. By Proposition \ref{pro:fm_max}, we have
		\begin{equation}\label{eq:g_induct}
			\fm_g^\circ(\jmath_{h*}j_{\Gamma*}(\pi^*\Psi\cdot \Xi)) =  r_{h*}q^*\fm_{g-h}^\circ(j_{\Gamma*}(\pi^*(\Psi)\cdot\Xi))= \pi^*(\Psi)\cdot r_{h*}q^*\fm_{g-h}^\circ(j_{\Gamma*}(\Xi))\,,
		\end{equation}
		where the second equality follows from the linearity of $\fm^\circ_{g-h}$. If $h^1(\Gamma)>0$, the Fourier transform of \eqref{eq:g_induct} reduces to computing lower genus Fourier transforms by Corollary \ref{cor:fm_st}. If $h^1(\Gamma)=0$,  the Fourier transform of \eqref{eq:g_induct} reduces to a computation involving Fourier transforms $\fm^\circ_{g_v}$ corresponding to the vertices of $\Gamma$ by Proposition \ref{pro:tree}. Since $\Gamma$ is a nontrivial stable graph, each vertex $v$ either has smaller genus $g_v < g$ or satisfies $g_v=g$ and $n_v<n$. This establishes the recursive structure of the Fourier transform.
	\end{proof}

	\section{The top degree part of the \texorpdfstring{$\DR$}{DR} formula}
	\subsection{The double ramification cycle formula}\label{subsec:DRintro}
	Suppose we have a genus $g\ge 0$, codimension $c\ge 0$, integer $b$, and vector $\sfa = (a_1,\ldots,a_n)$ of integers with sum $(2g-2+n)b$. Then the double ramification cycle ($\DR$) formula \cite{JPPZ1} gives a cycle
	\[
	\DR_g^c(b;\sfa)\in\CH^c(\Mbar_{g,n}).
	\]
	In the special case $c=g$, this formula gives the ($b$-twisted) $\DR$ cycle, which encodes the divisorial condition $\mathcal{O}_C(a_1p_1+\cdots+a_np_n)\cong (\omega_{C,\log})^{\otimes b}$. But the formula gives a well-defined cycle for all values of $c$, and the value $c=g$ will not be special for anything we do in this section.
	
	The formula given in \cite[Section~1.1]{JPPZ1} for $\DR_g^c(b;\sfa)$ is
	\[
	\left[
	\exp\left(-\frac{b^2}{2}\kappa_1+\sum_{i=1}^n\frac{a_i^2}{2}\psi_i\right)\sum_{
		\substack{\Gamma\in \mathsf{G}_{g,n} \\
			w\in \mathsf{W}_{\Gamma,r,b}}
	}
	\frac{r^{-h^1(\Gamma)}}{|\aut(\Gamma)| }
	\;
	(j_{\Gamma})_*\Bigg[
	\prod_{e=(h,h')\in E(\Gamma)}
	\frac{1-\exp\left(-\frac{w(h)w(h')}2(\psi_h+\psi_{h'})\right)}{\psi_h + \psi_{h'}} \Bigg]\right]_{\substack{\text{codim $c$,}\\  r=0}}. 
	\]
	Here $\mathsf{G}_{g,n}$ is the set of stable graphs for $\Mbar_{g,n}$ and $\mathsf{W}_{\Gamma,r,b}$ is the set of functions $w$ assigning to each half-edge $h\in H(\Gamma)$ an integer $w(h)\in \{0,1,\ldots,r-1\}$ satisfying the congruence conditions:
	\begin{itemize}
		\item If $h_i$ is the leg with marking $i$, then $w(h_i) \equiv a_i$ (mod $r$).
		\item If $e = (h,h')\in E(\Gamma)$ is an edge, then $w(h) + w(h')\equiv 0$ (mod $r$).
		\item If $v\in V(\Gamma)$ is a vertex, then
		\[
		\sum_{h\text{ at }v}w(h) \equiv (2g_v-2+n_v)b\pmod{r}.
		\]
	\end{itemize}
	The expression inside the brackets then has the property that its codimension $c$ part is a polynomial in the integer parameter $r$ for $r\gg 0$, and we take the constant term of that polynomial.
	
	It will be convenient for us to work with ``total $\DR$", i.e.
	\[
	\DR_g(b;\sfa) = \sum_{c\ge 0}\DR_g^c(b;\sfa).
	\]
	One reason is that we can then factor out the exponential factor at the beginning of the formula and write
	\begin{equation}\label{eq:DRfactor}
		\DR_g(b;\sfa) = \exp(\DRD_g(b;\sfa))\DRP_g(b;\sfa),
	\end{equation}
	where
	\[
	\DRD_g(b;\sfa) = -\frac{b^2}{2}\kappa_1+\sum_{i=1}^n\frac{a_i^2}{2}\psi_i
	\]
	is a divisor class and
	\begin{equation}\label{eq:DRP}
		\DRP_g(b;\sfa) = \left[\sum_{
			\substack{\Gamma\in \mathsf{G}_{g,n} \\
				w\in \mathsf{W}_{\Gamma,r,b}}
		}
		\frac{r^{-h^1(\Gamma)}}{|\aut(\Gamma)| }
		\;
		(j_{\Gamma})_*\Bigg[
		\prod_{e=(h,h')\in E(\Gamma)}
		\frac{1-\exp\left(-\frac{w(h)w(h')}2(\psi_h+\psi_{h'})\right)}{\psi_h + \psi_{h'}} \Bigg]\right]_{r=0}
	\end{equation}
	is the ``piecewise polynomial part" of the $\DR$ formula. Thus the complexity of the $\DR$ formula is contained in the $\DRP$ factor.
	
	It will be convenient for us to isolate a single coefficient in $\DRP_g(b;\sfa)$ and describe it in a more raw combinatorial form. Let $G$ be a connected finite graph (if it is a stable graph with genus assignments and legs, we ignore those features) and let $a_v$ be integers indexed by the vertices $v\in V(G)$, subject to the relation $\sum_va_v = 0$. Then define the {\it double ramification cycle graph invariant} $C(G)$ by
	\begin{equation}\label{eq:DRcoeff1}
		C(G)(\{a_v\}):= \left[r^{-h^1(G)}\sum_{w\in \mathsf{W}_{G,r}}\prod_{e=(h,h')\in E(G)}\frac{1}{2}w(h)w(h')\right]_{r=0}.
	\end{equation}
	Here $\mathsf{W}_{G,r}$ is analogous to $\mathsf{W}_{\Gamma,r,b}$, but $G$ has no legs and the conditions on $w:H(G)\to\{0,1,\ldots,r-1\}$ are simply:
	\begin{itemize}
		\item If $e = (h,h')\in E(\Gamma)$ is an edge, then $w(h) + w(h')\equiv 0$ (mod $r$).
		\item If $v\in V(\Gamma)$ is a vertex, then
		\[
		\sum_{h\text{ at }v}w(h) \equiv -a_v\pmod{r}.
		\]
	\end{itemize}
	
	Then for a stable graph $\Gamma\in \mathsf{G}_{g,n}$, the coefficient of the pure boundary stratum class
	\[
	\frac{1}{|\aut(\Gamma)|}(j_{\Gamma})_*1
	\]
	in $\DRP_g(b;\sfa)$ (or in $\DR_g(b;\sfa)$) is precisely given by $C(\Gamma)(\{a_v\})$ after taking
	\begin{equation}\label{eq:specialization}
		a_v := \left(\sum_{\text{leg $i$ at vertex $v$}}a_i\right) - (2g_v-2+n_v)b.
	\end{equation}
	
	It turns out that $C(G)(\{a_v\})$ is a polynomial in the integers $a_v$ of degree at most $2|E(G)|$ \cite{Spelier,Pixton_Poly}. We think of $C(G)$ as an element
	\[
	C(G) \in \QQ[\{a_v\}]\Big/\left(\sum_v a_v\right)
	\]
	where the $a_v$ are now formal variables.
	
	The full $\DRP$ formula includes tautological classes other than pure boundary strata classes - there can be insertions along edges of polynomials in $\psi+\psi'$, the sum of the two $\psi$ classes corresponding to the two halves of the edge. However, the coefficients of these more general tautological classes in the $\DR$ formula are still easily expressed in terms of the graph invariants $C(G)$. This is done by modifying the graph slightly - if edge $e$ has an insertion of $(-\psi-\psi')^d/(d+1)!$, then modify the stable graph $\Gamma$ by subdividing the edge $e$ into $d+1$ edges (adding $d$ new semistable vertices of genus $0$). If the resulting semistable graph is $\widehat{\Gamma}$, then the coefficient in the $\DR$ formula is now given by $C(\widehat{\Gamma})$ (after the same specialization of variables \eqref{eq:specialization}, which gives $a_v := 0$ for all semistable vertices). This perspective is explained further in \cite{PixAbel} and will be used in Section~\ref{subsubsec:topdeg_proof_general}.
	
	\subsection{The top degree part}
	The total $\DR$ cycle
	\[
	\DR_g(b;\sfa) = \exp(\DRD_g(b;\sfa))\DRP_g(b;\sfa)
	\]
	depends polynomially on the inputs $b,a_i$ \cite{Spelier,Pixton_Poly}. More precisely, the codimension $c$ part is a polynomial of degree at most $2c$:
	\[
	\DR_g^c(b;\sfa)\in\CH^c(\Mbar_{g,n})\otimes_\QQ\left[\QQ[b,a_1,\ldots,a_n]/(a_1+\cdots+a_n - (2g-2+n)b)\right]_{\text{deg $\le 2c$}}.
	\]
	We can take the top degree part of this polynomial to define
	\[
	\widetilde{\DR}_g^c(b;\sfa)\in\CH^c(\Mbar_{g,n})\otimes_\QQ\left[\QQ[b,a_1,\ldots,a_n]/(a_1+\cdots+a_n - (2g-2+n)b)\right]_{\text{deg $2c$}}.
	\]
	We also let $\widetilde{\DR}_g(b;\sfa)$ be the part of the total $\DR$ formula $\DR_g(b;\sfa)$ in which the polynomial degree is exactly twice the codimension. Clearly we can factor out the exponential factor:
	\[
	\widetilde{\DR}_g(b;\sfa) = \exp(\DRD_g(b;\sfa))\widetilde{\DRP}_g(b;\sfa).
	\]
	
	The main result of this section is a correspondence between the total $\DR$ formula and its top degree part. The correspondence is most naturally stated using the negative zeta value regularization convention
	\begin{equation}\label{eq:zeta}
		\sum_{k=1}^{\infty}k^{d+1} := \zeta(-d-1) = -\frac{B_{d+2}}{d+2}\quad\text{ for $d\ge 0$},
	\end{equation}
	where $B_n$ are the Bernoulli numbers.
	
	\begin{theorem}\label{thm:topdeg_correspondence}
		Let $g,n\ge 0$. For each $0 \le m \le g$, let $j_m:\Mbar_{g-m,n+2m}\to\Mbar_{g,n}$ be the gluing map gluing the last $m$ pairs of markings. Then
		\[
		\DR_g(b;\sfa) = \sum_{m=0}^{g}\frac{1}{2^mm!}(j_m)_*\left[\sum_{k_1,\ldots,k_m > 0}\left(\prod_{i=1}^mk_i\right)\widetilde{\DR}_{g-m}(b;\sfa,k_1,-k_1,k_2,-k_2,\ldots,k_m,-k_m)\right],
		\]
		where the infinite sums over $k_i$ of polynomials in $k_i$ are evaluated via \eqref{eq:zeta}.
	\end{theorem}
	
	Note that we can take the codimension $c$ part of both sides to express $\DR_g^c$ in terms of $\widetilde{\DR}_{g-m}^{c-m}$. The correspondence can also be formally inverted by adding a sign factor:
	\begin{corollary}
		Let $g,n\ge 0$. Then
		\[
		\widetilde{\DR}_g(b;\sfa) = \sum_{m=0}^{g}\frac{(-1)^m}{2^mm!}(j_m)_*\left[\sum_{k_1,\ldots,k_m > 0}\left(\prod_{i=1}^mk_i\right)\DR_{g-m}(b;\sfa,k_1,-k_1,k_2,-k_2,\ldots,k_m,-k_m)\right].
		\]
	\end{corollary}
	
	We will prove Theorem~\ref{thm:topdeg_correspondence} in Section~\ref{subsec:topdeg_correspondence_proof} after first reviewing a formula of Zagier in Section~\ref{subsec:zagier}. We will then generalize to the universal Picard stack in Theorem~\ref{thm:top} in Section~\ref{subsec:topdeg_uniDR}.
	
	\subsection{Zagier's formula for a $\DR$ coefficient}\label{subsec:zagier}
	Zagier gave an alternative expression for the graph invariant $C(G)$ appearing in the $\DR$ formula - a proof can be found in the notes \cite{Pixton_Poly}. We briefly recall Zagier's formula here. It has the advantage of being visibly polynomial in the $a_v$ variables.
	
	To describe the formula, it is convenient to fix an orientation of every edge of the graph $G$ (it will be easy to see that this choice does not affect the following formula for $C(G)$). Also, for each edge $e\in E(G)$, let $z_e$ be a formal variable. The shape of Zagier's formula is a sum over spanning trees $T$ of $G$. Given such a tree $T$, we need two auxiliary definitions before we can state the formula.
	
	First, suppose $e\in E(G)$ is an edge that does not belong to the spanning tree $T$. Then there is a unique cycle in the subgraph $T\cup\{e\}$; let $z_{e,T}$ be the signed sum of $z_f$ over edges $f$ in that cycle, with signs given by comparing the orientation of $f$ with the orientation of $e$ as you go around the cycle (which must contain $e$). In other words, $z_e$ always has positive sign in $z_{e,T}$.
	
	Second, suppose $e\in E(G)$ is an edge that does belong to the spanning tree $T$. Then the subgraph $T\setminus\{e\}$ (given by cutting the edge $e$) has two connected components. Take the connected component containing the head of $e$, and let $a_{e,T}$ be the sum of $a_v$ for all vertices $v$ in that connected component.
	
	Then Zagier's formula states that
	\begin{equation}\label{eq:DRcoeff2}
		C(G) = (-1)^{|E(G)|}\left[\sum_{\text{$T$ spanning tree}}\prod_{e\in T}\exp(a_{e,T}z_e)\prod_{e\notin T}\frac{z_e}{\exp(z_{e,T})-1}\right]_{\text{coeff of }\prod_{e\in E(G)}z_e^2}.
	\end{equation}
	Here the final subscript indicates to take the coefficient of $\prod_{e\in E(G)}z_e^2$ when the bracketed expression is expanded as a power series in the $z_e$ variables. Although the individual terms in the sum are not power series in the $z_e$ variables due to dividing by $z_{e,T}$, it can be checked that these poles cancel in the overall sum, yielding an analytic function near the origin $z_e = 0$. If a type of multivariate Laurent series expansion is chosen that is compatible with regular power series expansion (e.g. sequentially expanding as Laurent series in the individual variables $z_e$ in some fixed order), then taking this coefficient can be moved inside the sum over $T$ without changing the answer.
	
	To write down a formula for the top degree part $\widetilde{\DR}$, we need to take the top degree part of the $\DR$ coefficient graph invariant $C(G)$. Let
	\[
	\widetilde{C}(G) \in \QQ[\{a_v\}]\Big/\left(\sum_{v\in V(G)}a_v\right)
	\]
	be the degree $2|E(G)|$ part of $C(G)$. It is easy to extract this from Zagier's formula. The final factor is the only one in which the degrees might be different in the $z_e$ and $a_v$ variables, so we just have that
	\begin{equation}\label{eq:DRcoeff2top}
		\widetilde{C}(G) = (-1)^{|E(G)|}\left[\sum_{\text{$T$ spanning tree}}\prod_{e\in T}\exp(a_{e,T}z_e)\prod_{e\notin T}\frac{z_e}{z_{e,T}}\right]_{\text{coeff of }\prod_{e\in E(G)}z_e^2}.
	\end{equation}
	
	\subsection{Proof of Theorem~\ref{thm:topdeg_correspondence}}\label{subsec:topdeg_correspondence_proof}
	We begin by canceling some of the exponential factors $\exp(\DRD)$ appearing on both sides of Theorem~\ref{thm:topdeg_correspondence}. Since $j_m^*\kappa_1 = \kappa_1$ and $j_m^*\psi_i = \psi_i$, the projection formula lets us cancel all the factors of $\kappa_1$ and $\psi_1,\ldots,\psi_n$ on both sides, so it suffices to prove that
	\begin{align*}
		\DRP_g(b;\sfa) = \sum_{m=0}^{g}\frac{1}{2^mm!}(j_m)_*\Bigg[\sum_{k_1,\ldots,k_m > 0}\left(\prod_{i=1}^mk_i\right)&\exp\left(\frac{k_1^2}{2}(\psi_{n+1}+\psi_{n+2})+\cdots+\frac{k_m^2}{2}(\psi_{n+2m-1}+\psi_{n+2m})\right)\\
		&\cdot\widetilde{\DRP}_{g-m}(b;\sfa,k_1,-k_1,k_2,-k_2,\ldots,k_m,-k_m)\Bigg].
	\end{align*}
	We can see that both sides have the property that they only have insertions of powers of $\psi+\psi'$ (along edges), which is a good sign. But we will begin by checking that for any stable graph $\Gamma$ with no automorphisms, the coefficient of $(j_\Gamma)_*1$ is equal on both sides. We will then discuss in Section~\ref{subsubsec:topdeg_proof_general} how to modify things to check that things still work when $\Gamma$ has automorphisms or there are $\psi$ insertions.
	
	On the left side, the coefficient (which we will call $L(\Gamma)$) is simply $C(\Gamma)$ with the usual specialization of variables \eqref{eq:specialization} applied. On the right side, we have many terms that all produce a multiple of $(j_\Gamma)_*1$ - we get one term for each stable graph of genus $g-m$ with $n+2m$ legs such that gluing the last $m$ pairs of legs recovers $\Gamma$. This is the same as taking a subset of the edges $E(\Gamma)$ of size $m$ such that the complement $\Gamma'$ is still connected, and then attaching legs to vertices of $\Gamma'$ where deleted edges were attached. There are $2^mm!$ different ways to label these legs with $n+1,n+2,\ldots,n+2m$ such that gluing them in pairs recovers $\Gamma$, but they all give the same contribution. So we can think of the coefficient on the right side (which we will call $R(\Gamma)$) as summing over all subsets of $E(\Gamma)$ such that the complement $\Gamma'$ is still connected, then taking the $\DR$ coefficients $\widetilde{C}(\Gamma')$ with appropriate specializations of variables, and finally evaluating the negative zeta value regularization.
	
	We evaluate the $C(\Gamma)$ and $\widetilde{C}(\Gamma')$ that appear using Zagier's formula \eqref{eq:DRcoeff2} and its top degree variant \eqref{eq:DRcoeff2top}. As explained in Section~\ref{subsec:zagier}, this requires choosing an orientation for every edge in $\Gamma$ (and we use the same orientations on the edges of each $\Gamma'$). We now write out what this gives for the coefficients $L(\Gamma),R(\Gamma)$ on the two sides. We begin with the simpler left side to introduce some shorthand:
	\begin{equation}\label{eq:topdegLHS}
		L(\Gamma) = (-1)^{|E(\Gamma)|}\left[\sum_{T\subseteq\Gamma}\prod_{e\in T}\exp(a_{e,T}z_e)\prod_{e\notin T}\frac{z_e}{\exp(z_{e,T})-1}\right]_{\substack{\text{coeff of }\prod z_e^2,\\ \{a_v\} \mapsto \{b,a_i\}}}
	\end{equation}
	Here the subscripts at the end mean that first we take the coefficient of $\prod_{e\in E(\Gamma)}z_e^2$ and then we do the specialization \eqref{eq:specialization} replacing the $a_v$ with the $b$ and $a_i$ variables (using the data of the genus and leg assignments to vertices in $\Gamma$).
	
	Then the right side can be written as
	\begin{equation}\label{eq:topdegRHS}
		R(\Gamma) = \sum_{m\ge 0}\sum_{\substack{S\subseteq E(\Gamma)\\ |S| = m \\ \Gamma' := \Gamma\setminus S\\ \Gamma'\text{ is connected}}}(-1)^{|E(\Gamma')|}\left[\sum_{T\subseteq\Gamma'}\prod_{e\in T}\exp(a_{e,T}z_e)\prod_{\substack{e\notin T\\ e\notin S}}\frac{z_e}{z_{e,T}}\prod_{e\in S}z_e^2\right]_{\substack{\text{coeff of }\prod z_e^2,\\ \{a_v\} \mapsto \{b,a_i,\pm k_i\},\\ k_i^d\mapsto \zeta(-d-1)}}.
	\end{equation}
	We make a few notes to explain various things about the above formula:
	\begin{itemize}
		\item The factor $1/(2^mm!)$ in the statement of Theorem~\ref{thm:topdeg_correspondence} has been cancelled out by a factor coming from the choice of labels for the $2m$ legs that would correspond to the $m$ edges in $S$, as discussed above.
		\item We include the extra factor $\prod_{e\in S}z_e^2$ because those are the variables for edges that are in $\Gamma$ but not $\Gamma'$, and at the end we want to take the coefficient of $z_e^2$ with respect to all of those variables too.
		\item The specialization of variables $\{a_v\} \mapsto \{b,a_i,\pm k_i\}$ indicates that we need to adjust the endpoints of the edges in $S$ by $\pm k_i$. Although the signs and indices of the $k_i$ here depend on the choice of labels described in the first note, it is easy to see from the following note that this choice does not affect the final expression.
		\item The final subscript at the end, $k_i^d\mapsto \zeta(-d-1)$, combines multiplication by $\prod k_i$ and summing over $k_i > 0$ using negative zeta regularization. Note that $\zeta(-d-1) = -B_{d+2}/(d+2)$ is zero unless $d$ is even.
	\end{itemize}
	
	We simplify this expression by pulling out the sum over spanning trees of $\Gamma'$ - the spanning trees of $\Gamma'$ are precisely the spanning trees of $\Gamma$ with edges disjoint from $S$, and conveniently if $S$ is disjoint from the edges of some spanning tree of $\Gamma$ then $\Gamma'$ is automatically connected. We also multiply by the sign factor $(-1)^{|E(\Gamma)|}$ that also appeared on the left side. This leaves a remaining sign of $(-1)^m$, which we incorporate into the product over edges in $S$. The result is
	\begin{equation}\label{eq:topdegRHS2}
		(-1)^{|E(\Gamma)|}R(\Gamma) = \sum_{T\subseteq\Gamma}\sum_{m\ge 0}\sum_{\substack{S \subseteq E(\Gamma)\setminus E(T) \\ |S| = m}}\left[\prod_{e\in T}\exp(a_{e,T}z_e)\prod_{\substack{e\notin T\\ e\notin S}}\frac{z_e}{z_{e,T}}\prod_{e\in S}\left(-z_e^2\right)\right]_{\substack{\text{coeff of }\prod z_e^2,\\ \{a_v\} \mapsto \{b,a_i,\pm k_i\},\\ k_i^d\mapsto \zeta(-d-1)}}.
	\end{equation}
	Note that taking the coefficient of $\prod z_e^2$ of an individual term like this requires choosing a multivariate Laurent series expansion - we can do this for instance by choosing an ordering on the variables and expanding as a Laurent series in each of them in turn. We do this in a consistent way, e.g. by fixing an ordering of $E(\Gamma)$.
	
	To continue simplifying this expression, we need to process the steps involving the auxiliary variables $k_i$. We can do this before taking the coefficient of $\prod z_e^2$ by interpreting the bracketed expression as a formal Laurent series in the $z_e$ variables with coefficients that are polynomials in the $a_v$, and applying these operations coefficient by coefficient. We first consider the specialization of variables $\{a_v\} \mapsto \{b,a_i,\pm k_i\}$. Again, performing this specialization requires choosing an ordering $S = \{e_1,\ldots,e_m\}$ and choosing one endpoint of each edge to add $k_i$ and one to subtract $k_i$ - we use the fixed orientation on edges of $\Gamma$ for this. We can then check (recalling the definitions of $a_{e,T}$ and $z_{e,T}$) that
	\[
	\left[\sum_{e\in T} a_{e,T}z_e\right]_{\{a_v\} \mapsto \{b,a_i,\pm k_i\}} = \left[\sum_{e\in T} a_{e,T}z_e\right]_{\{a_v\} \mapsto \{b,a_i\}}+\sum_{i=1}^mk_i(z_{e_i,T}-z_{e_i}).
	\]
	Using this, our expression for $(-1)^{|E(\Gamma)|}R(\Gamma)$ becomes
	\[
	\sum_{T\subseteq\Gamma}\sum_{m\ge 0}\sum_{\substack{S \subseteq E(\Gamma)\setminus E(T) \\ |S| = m}}\left[\prod_{e\in T}\exp(a_{e,T}z_e)\prod_{\substack{e\notin T\\ e\notin S}}\frac{z_e}{z_{e,T}}\prod_{e\in S}\left(-z_e^2\right)\prod_{i=1}^m\exp\left(k_i(z_{e_i,T}-z_{e_i})\right)\right]_{\substack{\text{coeff of }\prod z_e^2,\\ \{a_v\} \mapsto \{b,a_i\},\\ k_i^d\mapsto \zeta(-d-1)}},
	\]
	where again $S = \{e_1,\ldots,e_m\}$.
	
	We can now directly perform the negative zeta regularization. Note that
	\begin{align*}
		\left[\exp(kU)\right]_{k^d\mapsto \zeta(-d-1)} &= \sum_{d = 0}^{\infty}\frac{-B_{d+2}}{(d+2)d!}U^d\\
		&= \frac{d}{dU}\left(\frac{1}{U}\left(\frac{-U}{e^U-1} + 1\right)\right)\\
		&= \left(\frac{e^U}{(e^U-1)^2} - \frac{1}{U^2}\right),
	\end{align*}
	using the standard exponential generating function for the Bernoulli numbers.
	
	Taking $U = z_{e_i,T}-z_{e_i}$ and substituting this in, our expression for $(-1)^{|E(\Gamma)|}R(\Gamma)$ is now
	\[
	\sum_{T\subseteq\Gamma}\sum_{m\ge 0}\sum_{\substack{S \subseteq E(\Gamma)\setminus E(T) \\ |S| = m}}\left[\prod_{e\in T}\exp(a_{e,T}z_e)\prod_{\substack{e\notin T\\ e\notin S}}\frac{z_e}{z_{e,T}}\prod_{e\in S}\left(-z_e^2\left(\frac{\exp(z_{e,T}-z_e)}{(\exp(z_{e,T}-z_e)-1)^2}-\frac{1}{(z_{e,T}-z_e)^2}\right)\right)\right]_{\substack{\text{coeff of }\prod z_e^2,\\ \{a_v\} \mapsto \{b,a_i\}}},
	\]
	where we've combined our two products over edges in $S$ now that the $e_1,\ldots,e_m$ labels are unnecessary.
	
	But now we can freely move the sums over $m$ and $S$ inside the bracket and evaluate them as the usual product of binomials. The result is
	\[
	\sum_{T\subseteq\Gamma}\left[\prod_{e\in T}\exp(a_{e,T}z_e)\prod_{e\notin T}\left(\frac{z_e}{z_{e,T}} - z_e^2\left(\frac{\exp(z_{e,T}-z_e)}{(\exp(z_{e,T}-z_e)-1)^2}-\frac{1}{(z_{e,T}-z_e)^2}\right)\right)\right]_{\substack{\text{coeff of }\prod z_e^2,\\ \{a_v\} \mapsto \{b,a_i\}}}
	\]
	This now looks a lot like \eqref{eq:topdegLHS} (recall we've multiplied through by the sign factor $(-1)^{|E(\Gamma)|}$), but we have a different meromorphic function of $z_e$ and $z_{e,T}$ appearing inside the product over edges $e\notin T$. Note that for a fixed edge $f\notin T$, the variable $z_f$ only appears in this one factor in the expression. We claim that (when expanded as a Laurent series in the $z_e$ in the fixed order previously chosen) the two functions
	\[
	\frac{z_f}{\exp(z_{f,T})-1}\quad\text{ and }\quad \left(\frac{z_f}{z_{f,T}} - z_f^2\left(\frac{\exp(z_{f,T}-z_f)}{(\exp(z_{f,T}-z_f)-1)^2}-\frac{1}{(z_{f,T}-z_f)^2}\right)\right)
	\]
	have the same coefficient of any monomial $\prod_e z_e^{d_e}$ with exponent $d_f = 2$. This will complete the proof of the desired identity of coefficients $L(\Gamma) = R(\Gamma)$.
	
	To see this claim, note that the difference between these two functions is analytic near the origin $z_e = 0$, so it suffices to expand that as a power series and check that the desired coefficients vanish. Let $X = z_f$ and $Y = z_{f,T}-z_f$ (which does not use the variable $z_f$); then it suffices to check that the coefficient of $X^2Y^i$ in
	\[
	\frac{X}{e^{X+Y}-1} - \frac{X}{X+Y} + X^2\left(\frac{e^Y}{\left(e^Y-1\right)^2}-\frac{1}{Y^2}\right)
	\]
	is $0$ for all $i \ge 0$. Equivalently,
	\[
	\left[\frac{X}{e^{X+Y}-1} - \frac{X}{X+Y}\right]_{\text{coeff of $X^2$}} = -\left(\frac{e^Y}{\left(e^Y-1\right)^2}-\frac{1}{Y^2}\right),
	\]
	which is easily verified by dividing by $X$, differentiating with respect to $X$, and setting $X=0$.
	
	\subsubsection{General tautological classes}\label{subsubsec:topdeg_proof_general}
	
	We are not yet finished with the proof of Theorem~\ref{thm:topdeg_correspondence} - so far, we have only verified that for pure boundary strata corresponding to stable graphs $\Gamma$ with no automorphisms, the coefficients match up on both sides.
	
	First we discuss automorphisms. Let $\Gamma$, $S\subseteq E(\Gamma)$ with $|S| = m$, and $\Gamma' = \Gamma\setminus S$ be as above. Let $M$ be the set of $2^g\cdot g!$ different ways to add and label $2m$ legs to $\Gamma'$ at the endpoints of the edges in $S$. Previously the resulting $2^g\cdot g!$ stable graphs were all nonisomorphic (and all had trivial automorphism group), but this is no longer true if $\aut(\Gamma)$ is nontrivial. There is a natural action of $\aut(\Gamma)$ on $M$, and elements of $M$ in the same orbit will correspond to isomorphic stable graphs. Moreover, if $\Gamma''$ is one such stable graph then its automorphism group is isomorphic to the subgroup of $\aut(\Gamma)$ stabilizing the corresponding element of $M$. We can also see that $|\aut(\Gamma'')|$ does not depend on the choice of $\Gamma''$. By the orbit-stabilizer theorem, the number of isomorphism classes of stable graphs corresponding to elements of $M$ is
	\[
	\frac{|\aut(\Gamma'')|}{|\aut(\Gamma)|}2^g\cdot g!.
	\]
	This modified number precisely cancels the automorphism factors on the two sides of the desired identity.
	
	Finally, we need to discuss how to handle insertions of powers of $\psi+\psi'$ along edges. To check that coefficients match up for tautological classes with these internal $\psi$ classes, we replace $\Gamma$ with semistable $\widehat{\Gamma}$ as described at the end of Section~\ref{subsec:DRintro}. The proof that $L(\widehat{\Gamma}) = R(\widehat{\Gamma})$ (expressions defined using \eqref{eq:topdegLHS} and \eqref{eq:topdegRHS}) goes through unchanged, so it remains to check that this identity is indeed the equality of coefficients we want to prove to obtain Theorem~\ref{thm:topdeg_correspondence}. We have that $L(\widehat{\Gamma})$ is the coefficient of
	\[
	\frac{1}{|\aut(\Gamma)|}(j_{\Gamma})_*\prod_{e = (h,h')\in E(\Gamma)}\frac{(-\psi_h-\psi_{h'})^{d_e}}{(d_e+1)!}
	\]
	for some exponents $d_e$. The other side requires more discussion - what happens when the set of edges $S\subseteq E(\widehat{\Gamma})$ contains edges along the semistable portions of $\widehat{\Gamma}$?
	
	Suppose a given edge $e\in E(\Gamma)$ was subdivided into $m+1$ edges $e_1,\ldots,e_{m+1}\in E(\widehat{\Gamma})$, i.e. $m = d_e$. Then the condition that the complement of $S$ is connected implies that $S$ contains at most one of the $e_i$. Also, it is easily checked that altering $S$ by replacing one of the $e_i$ by a different one does not change the value of a term in \eqref{eq:topdegRHS}. If we combine terms with $S$ that are related in this way, we get a factor of $m+1$ and can reduce to a sum over $S\subseteq E(\Gamma)$, which are the actual edges appearing in the gluing. We then get to choose powers of $\psi$ on the two sides of the glued edge. The factor of $m+1$ then corresponds to the identity
	\[
	\sum_{i = 0}^{m}\frac{(-\psi)^i}{i!}\cdot\frac{(-\psi')^{m-i}}{(m-i)!} = (m+1)\cdot\frac{(-\psi-\psi')^m}{m!}.
	\]
	
	\subsubsection{A related lemma}
	We state and prove an identity that is closely related to Theorem~\ref{thm:topdeg_correspondence} here. It will be used later in Section~\ref{sec:forgetful}.
	\begin{lemma}\label{lemma:codim_minus_deg}
		Let $g,n\ge 0$. Let $j:\Mbar_{g-1,n+2}\to\Mbar_{g,n}$ be the gluing map gluing the last two markings. Then
		\[
		(2[\text{codim}]-[\text{deg}])\DR_g(b;\sfa) = j_*\left[-\frac{k^2}{2}\DR_{g-1}(b;\sfa,k,-k)\right]_{k^d\mapsto B_d}.
		\]
		On the left side, $[\text{codim}]$ and $[\text{deg}]$ are the derivations that act (respectively) by multiplication by codimension and by multiplication by polynomial degree (in $b,a_i$) on the appropriate graded pieces. On the right side, the subscript indicates that powers of the formal variable $k$ are replaced by the corresponding Bernoulli numbers.
	\end{lemma}
	\begin{proof}
		We follow the same proof strategy as for Theorem~\ref{thm:topdeg_correspondence}. Things are mostly directly analogous, so we will only write down how the details change in a few key points.
		
		Suppose we are checking that the coefficient of the pure boundary stratum $(j_\Gamma)_*1$ is equal on the two sides, where $\Gamma$ is a stable graph with $c$ legs. On the left side, we have to apply the derivation $(2[\text{codim}]-[\text{deg}])$ to $C(\Gamma)$. Using Zagier's formula for $C(\Gamma)$ as in \eqref{eq:topdegLHS}, we note that we can move $2[\text{codim}]$ inside the brackets and replace it there with $[\text{deg}_z]$, the polynomial degree in the $z_e$ variables. Then $[\text{deg}_z] - [\text{deg}_a]$ annihilates the $\exp(a_{e,T}z_e)$ factors, so we just need to let it act on one of the other types of factor. In other words, the left side becomes
		\[
		(-1)^{|E(\Gamma)|}\left[\sum_{T\subseteq\Gamma}\sum_{f\notin T}\prod_{e\in T}\exp(a_{e,T}z_e)\left(\prod_{\substack{e\notin T\\e\ne f}}\frac{z_e}{\exp(z_{e,T})-1}\right)\left([\text{deg}_z]\frac{z_{f}}{\exp(z_{f,T})-1}\right)\right]_{\substack{\text{coeff of }\prod z_e^2,\\ \{a_v\} \mapsto \{b,a_i\}}}
		\]
		
		Meanwhile, the right side of the identity to be proved is similar to \eqref{eq:topdegRHS2}, but we make the following changes:
		\begin{itemize}
			\item We only have $|S| = m = 1$, so we will write $S = {f}$ and replace the sums over $m$ and $S$ with a sum over edges $f$, and write $k$ in place of $k_1$;
			\item We did not use the top degree version of $\DR$, so the $z_{e,T}$ in the denominator should be $\exp(z_{e,T}-1)$;
			\item At the end we apply $k^d\mapsto -B_{d+2}$ instead of $k^d\mapsto \zeta(-d-1)$.
		\end{itemize}
		The result is
		\[
		\sum_{T\subseteq\Gamma}\sum_{f\in E(\Gamma)\setminus E(T)}\left[\prod_{e\in T}\exp(a_{e,T}z_e)\left(\prod_{\substack{e\notin T\\ e\ne f}}\frac{z_e}{\exp(z_{e,T})-1}\right)\left(-z_{f}^2\right)\right]_{\substack{\text{coeff of }\prod z_e^2,\\ \{a_v\} \mapsto \{b,a_i,\pm k\},\\ k^d\mapsto -B_{d+2}}}.
		\]
		We then proceed as before and process the steps involving $k$. This involves the new computation
		\begin{align*}
			\left[\exp(kT)\right]_{k^d\mapsto -B_{d+2}} &= -\sum_{d = 0}^{\infty}\frac{B_{d+2}}{d!}T^d\\
			&= -\left(\frac{d}{dT}\right)^2\frac{T}{e^T-1}.
		\end{align*}
		After this the two sides of the identity are of the same shape but they just use different meromorphic functions of $z_f$ and $z_{f,T}$, as before. If we let $X = z_{f}$ and $Y = z_{f,T}-z_{f}$, then we end up just needing to check that the coefficient of $X^2Y^i$ in
		\[
		\left(X\frac{d}{dX}+Y\frac{d}{dY}\right)\frac{X}{e^{X+Y}-1}-X^2\left(\left(\frac{d}{dY}\right)^2\frac{Y}{e^Y-1}\right)
		\]
		is $0$ for all $i\ge 0$. This is easily done with a symbolic algebra package.
	\end{proof}
	
	\subsection{Top degree part of the universal $\DR$ formula}\label{subsec:topdeg_uniDR}
	The universal double ramification formula is used to compute the closure of the Abel-Jacobi section on the universal Picard stack \cite{BHPSS}. 
	Let $\pic_{g,n,0}$ denote the universal Picard stack of total degree $0$ line bundles for $\C_{g,n}\to \M_{g,n}$ (see Section \ref{sec:taut_class}). Let $b\in\ZZ$ and let $\sfa=(a_1,\ldots, a_n)\in\ZZ^n$ be a vector of integers with $\sum_{i=1}^n a_i=b(2g-2+n)$. For $r\in\ZZ_{\geq 0}$, we denote by $\uniDR_{g,r}^{c}(b; \sfa)$ the codimension $c$ component of the tautological class 
	\begin{align}
		\sum_{
			\substack{\Gamma_\delta\in \mathsf{G}_{g,n,0} \nonumber\\
				w\in \mathsf{W}_{\Gamma_\delta,r}}
		}
		\frac{r^{-h^1(\Gamma_\delta)}}{|\aut(\Gamma_\delta)| }
		\;
		j_{\Gamma_{\delta*}}\Bigg[
		\prod_{i=1}^n \exp\left(\frac{1}{2} a_i^2 \psi_i + a_i \xi_i \right)
		\prod_{v \in V(\Gamma_\delta)} \exp\left(-\frac{1}{2} \kappa_{-1,2}(v)-b\kappa_{0,1}-\frac{b^2}{2}\kappa_1 \right)
		\\ 
		\prod_{e=(h,h')\in E(\Gamma_\delta)}
		\frac{1-\exp\left(-\frac{w(h)w(h')}2(\psi_h+\psi_{h'})\right)}{\psi_h + \psi_{h'}} \Bigg]\label{eq:unidr}
	\end{align} 
	in the operational Chow ring $\CHop^*(\pic_{g,n,0})$. Here $\Gamma_\delta$ runs over all prestable graphs $\Gamma$ equipped with arbitrary multidegrees $\delta:V(\Gamma)\to\ZZ$ with sum $0$. The weighting set $\mathsf{W}_{\Gamma_\delta,r}$ is defined in the same way as the set $\mathsf{W}_{\Gamma,r}$ used in the $\DR$ formula (Section~\ref{subsec:DRintro}), except that $\delta(v)$ is added to the right side of the vertex balancing condition.
	
	For sufficiently large $r$, this expression is polynomial in $r$. Let $\uniDR^c_{g}(b;\sfa)$ be the constant term of this polynomial. We formally sum over $c$ to define the total $\uniDR$ cycle
	\[
	\uniDR_g(b;\sfa) := \sum_{c\ge 0}\uniDR^c_{g}(b;\sfa)
	\]
	As in \eqref{eq:DRfactor}, we can factor this as $\uniDR_g(b,\sfa) = \exp(\uniDRD_g(b;\sfa))\uniDRP_g(b;\sfa)$, where $\uniDRD_g(b;\sfa)$ is a divisor given as a linear combination of the $\psi_i,\xi_i$, and $\kappa_{i,j}$ with $i+j = 1$.
	
	For a detailed discussion of the invariant properties of the universal double ramification cycle formula, see \cite[Section 7]{BHPSS}.
	
	We consider the ``multiplication by $N$" map
	\begin{equation}\label{eq:Nmap}
		[N] : \pic_{g,n,0}\to\pic_{g,n,0}\,.
	\end{equation}
	A class $\alpha\in\CH^*(\pic_{g,n,0})$ has {\em weight $w$} if $[N]^*\alpha  = N^w\alpha$ for all $N\in\ZZ$.
	We define the subring
	\begin{equation*}
		\CH^*_{\pure}(\pic_{g,n,0}):= \bigoplus_{w\geq 0} \CH^*_{(w)}(\pic_{g,n,0}) \subset\CH^*(\pic_{g,n,0})_\QQ
	\end{equation*}
	generated by weight $w$ classes. Compared to finite-type commutative group schemes (see Theorem \ref{thm:wtdec}), this inclusion is strict because $\pic_{g,n,0}$ is not finite type over $\M_{g,n}$. Nevertheless, the universal double ramification formula lies in this subring.
	\begin{proposition}\label{pro:P_wt}
		For any $b\in\ZZ$ and $\sfa\in\ZZ^n$ with $\sum_{i=1}^na_i=(2g-2+n)b$,  the class $\uniDR^c_g(b;\sfa)$ lies in $\CH^c_{\pure}(\pic_{g,n,0})$.
	\end{proposition}
	\begin{proof}
		This is easily checked for the divisor class $\uniDRD_g(b;\sfa)$ - we have that $\psi_i$ has weight $0$, $\xi_i$ has weight $1$, and $\kappa_{i,j}$ has weight $j$. For the ``piecewise polynomial part" $\uniDRP_g(b;\sfa)$, we just need to check that the coefficient on the boundary stratum corresponding to a graph $\Gamma$ with multidegree $\delta$ (possibly with $-\psi-\psi'$ decorations) depends polynomially on $\delta$. But this coefficient is just one of the $\DR$ graph invariants $C(G)$ of \eqref{eq:DRcoeff1}, evaluated at $a_v$ equal to $-\delta(v)$ plus a linear combination of $b$ and the $a_i$. Since $C(G)$ is polynomial in these inputs \cite{Spelier,Pixton_Poly}, we are done.
	\end{proof}
	A simple case of Proposition \ref{pro:P_wt} first appears in \cite[Proposition 4.2]{BMSY}. By Proposition \ref{pro:P_wt}, we can consider the following definition:
	\begin{definition}
		We define the {top degree part} $\widetilde{\uniDR}^{c}_g(b;\sfa)$ as the sum over $m$ of the weight $2c-m$ part of the coefficient of monomials in $b,a_i$ of degree $m$.
	\end{definition}
	In other words, we use the sum of the weight grading and the polynomial codimension grading in $b,a_i$ to define the top degree part. It is easily checked that since $C(G)$ has degree at most twice the number of edges of $G$, the highest ``degree" we can get in this way in codimension $c$ is $2c$ (just as it was for the top degree part of regular $\DR$). We define the total top degree part $\widetilde{\uniDR}_g(b;\sfa)$ by summing over $c$ as usual.
	
	For $1\leq m\leq g$, we consider the diagram
	\[
	\begin{tikzcd}
		& \mathfrak{Pic}^{\Gm^m}\ar[r,"j_m"]\ar[d]\ar[dl,"q"'] & \mathfrak{Pic}_{g,n,0}\ar[d]\\
		\mathfrak{Pic}_{g-m,n+2m,0}\ar[r] & \M_{g-m,n+2m}\ar[r] & \M_{g,n}
	\end{tikzcd}
	\]
	by gluing the pairs of markings $(n+1,n+2), \ldots, (n+2m-1,n+2m)$. The right square is Cartesian.
	
	\begin{theorem}\label{thm:top}
		Let $b\in\ZZ$ and $\sfa\in\ZZ^n$ with $\sum_{i=1}^na_i=(2g-2+n)b$. Then we have 
		\[
		\uniDR_g(b;\sfa) = \sum_{m=0}^g\sum_{k_1,\ldots,k_h\in\ZZ}\frac{1}{2^mm!}\Big(\prod_{i=1}^m k_i\Big)(j_m)_*q^* \widetilde{\uniDR}_{g-m}(b;\sfa,k_1,-k_1,\ldots,k_m,-k_m)\,.
		\]
	\end{theorem}
	Here we use the negative zeta value regularization \eqref{eq:zeta} on the right hand side, as in Theorem~\ref{thm:topdeg_correspondence}.
	\begin{proof}
		The proof is nearly identical to that of Theorem~\ref{thm:topdeg_correspondence} in the previous section. As before, the exponential factor $\exp(\uniDRD_g(b;\sfa))$ is of pure top degree (relative to its codimension) so we can write
		\[
		\widetilde{\uniDR}_g(b;\sfa) = \exp(\uniDRD_g(b;\sfa))\widetilde{\uniDRP}_g(b;\sfa).
		\]
		We can again use the projection formula to cancel most of the exponential factors on both sides, leaving only $\psi_{n+1},\ldots,\psi_{n+2m}$ (as before) and now also $\xi_{n+1},\ldots,\xi_{n+2m}$. But the $\xi$ classes appear as factors of the form
		\[
		\exp(k_i(\xi_{n+2i-1}-\xi_{n+2i})).
		\]
		Since these two $\xi$ classes become equal on applying $q^*$, these factors vanish. The only change in the remainder of the proof is that we subtract $\delta(v)$ from the right side of the variable specialization \eqref{eq:specialization}.
	\end{proof}
	
	We can lift Lemma~\ref{lemma:codim_minus_deg} to $\uniDR$ in the same way. The only change required is that we must replace $[\text{deg}]$ with $[\text{deg}]+[\text{weight}]$.
	\begin{lemma}\label{lemma:unicodim_minus_deg}
		Let $g,n\ge 0$. Then
		\[
		(2[\text{codim}]-[\text{deg}] - [\text{weight}])\uniDR_g(b;\sfa) = (j_1)_*q^*\left[-\frac{k^2}{2}\uniDR_{g-1}(b;\sfa,k,-k)\right]_{k^d\mapsto B_d}.
		\]
	\end{lemma}
	
	\section{Forgetful pushforward of the \texorpdfstring{$\DR$}{DR} formula}\label{sec:forgetful}
	\subsection{Statement of result}
	In this section we prove an identity satisfied by the pushforward of the $\DR$ formula along the forgetful map $\pi:\Mbar_{g,n+1}\to\Mbar_{g,n}$ forgetting the last marking.
	\begin{theorem}\label{thm:DR_push}
		Let $g,c,n\ge 0$. Let 
		\[F = \pi_*\DR_g^c(b;a_1,\ldots,a_{n+1}) \in \CH^{c-1}(\Mbar_{g,n})\otimes_\QQ\QQ[b,a_1,\ldots,a_{n+1}]/(a_1+\cdots+a_{n+1}-(2g-1+n)b)\,.\]
		%$F = \pi_*\DR_g^c(b;a_1,\ldots,a_{n+1}) \in \CH^{c-1}(\Mbar_{g,n})\otimes_\QQ\QQ[b,a_1,\ldots,a_{n+1}]/(a_1+\cdots+a_{n+1}-(2g-1+n)b)$.
		\begin{enumerate}[label=(\alph*)]
			\item $F$ is a multiple of $(a_{n+1}-b)^2$.
			\item We have the identity
			\[
			\left[\frac{F}{(a_{n+1}-b)^2}\right]_{a_{n+1}:=b} = (g+1-c)\DR_g^{c-1}(b;a_1,\ldots,a_n).
			\]
		\end{enumerate}
	\end{theorem}
	To see that the statement of part (b) makes sense, note that setting $a_{n+1} := b$ (i.e. quotienting by $a_{n+1}-b$) naturally induces a ring homomorphism
	\[
	\QQ[b,a_1,\ldots,a_{n+1}]/(a_1+\cdots+a_{n+1}-(2g-1+n)b) \to \QQ[b,a_1,\ldots,a_{n}]/(a_1+\cdots+a_{n}-(2g-2+n)b).
	\]
	
	When $c > g$, Theorem~\ref{thm:DR_push} is an easy consequence of the $\DR$ relations proved in \cite{claderjanda}, since then both $\DR_g^c(b;a_1,\ldots,a_{n+1})$ and $(g+1-c)\DR_g^{c-1}(b;a_1,\ldots,a_n)$ vanish. But for $c \le g$, it gives a new interpretation of the lower codimension parts of the $\DR$ formula, and this is the case we need in the proof of Theorem~\ref{thm:pushforward}. Of course, on the level of the strata algebra Theorem~\ref{thm:DR_push} is new for all $c$.
	
	The proof of Theorem~\ref{thm:DR_push} is essentially just a lengthy computation. We break it up into subsections as follows. In Section~\ref{subsec:pushforward_setup} we begin the proof and describe how to break the computation of $F$ into three parts. In the following three subsections we handle each of these parts in turn. Finally, in Section~\ref{subsec:pushforward_uniDR} we state and prove a generalization of Theorem~\ref{thm:DR_push} to $\uniDR$.
	
	\subsection{Setting up the computation}\label{subsec:pushforward_setup}
	Although Theorem~\ref{thm:DR_push} is stated in pure codimension, it will be convenient to sum over $c$ and prove that version instead. In other words, we want to prove that the total $\DR$ pushforward
	\[
	\pi_*\DR_g(b;a_1,\ldots,a_{n+1})
	\]
	is a multiple of $(a_{n+1}-b)^2$ and that we have the identity
	\begin{equation}\label{eq:DR_push_goal_old}
		\left[\frac{\pi_*\DR_g(b;a_1,\ldots,a_{n+1})}{(a_{n+1}-b)^2}\right]_{a_{n+1}:=b} = (g-[\text{codim}])\DR_g(b;a_1,\ldots,a_n),
	\end{equation}
	where the linear operator $[\text{codim}]$ acts by multiplying pure-dimensional cycles by their codimension.
	
	We can merge these two goals into the single congruence
	\begin{equation}\label{eq:DR_push_goal}
		\pi_*\DR_g(b;a_1,\ldots,a_{n+1}) \equiv (a_{n+1}-b)^2(g-[\text{codim}])\DR_g(b;a_1,\ldots,a_n)\pmod{(a_{n+1}-b)^3}.
	\end{equation}
	Here we are abusing notation in a subtle way which we will now explain. The left side of the congruence naturally belongs to the ring (let us call it $R$) of polynomials in $b,a_1,\ldots,a_{n+1}$ modulo the single relation $(2g-2+n+1)b = a_1+\cdots+a_{n+1}$. But it does not make sense to write $\DR_g(b;a_1,\ldots,a_n)$ in $R$, since the inputs to $\DR_g$ then do not satisfy the relation $(2g-2+n)b = a_1+\cdots+a_n$ in $R$. Note that this was not an issue with \eqref{eq:DR_push_goal_old} because there we have set $a_{n+1} := b$. However, the right side of \eqref{eq:DR_push_goal} still makes sense if we use the facts that we are working mod $(a_{n+1}-b)^3$ and that the problematic $\DR_g$ factor is being multiplied by $(a_{n+1}-b)^2$. In other words, we think of $\DR_g(b;a_1,\ldots,a_n)$ as belonging to $R/(a_{n+1}-b)$ and treat multiplication by $(a_{n+1}-b)^2$ as an operator $R/(a_{n+1}-b)\to R/(a_{n+1}-b)^3$.
	
	Whenever we write $\DR_g(b;a_1,\ldots,a_n)$ inside a congruence mod $(a_{n+1}-b)^3$ for the remainder of this section, it will occur along with a factor of $(a_{n+1}-b)^2$ and should be interpreted via the above procedure. Note that we have
	\[
	(a_{n+1}-b)^2\DR_g(b;a_1,\ldots,a_n)\equiv (a_{n+1}-b)^2\DR_g(b;a_1+a_{n+1}-b,a_2,\ldots,a_n)\pmod{(a_{n+1}-b)^3}.
	\]
	This is another way to interpret the notation; the choice of which $a_i$ to increase by $a_{n+1}-b$ does not affect the value of the expression mod $(a_{n+1}-b)^3$.
	
	We begin by rewriting both sides of \eqref{eq:DR_push_goal} using the factorization $\DR = \exp(\DRD)\DRP$ and manipulating the exponential factors. Recall that
	\[
	\DRD_g(b;a_1,\ldots,a_{n+1}) = -\frac{b^2}{2}\kappa_1+\sum_{i=1}^{n+1}\frac{a_i^2}{2}\psi_i.
	\]
	Although the terms here are not equal to the pullbacks under $\pi$ of the corresponding divisor classes on $\Mbar_{g,n}$, we can write down the error terms:
	\[
	\DRD_g(b;a_1,\ldots,a_{n+1}) = \pi^*\left(-\frac{b^2}{2}\kappa_1+\sum_{i=1}^{n}\frac{a_i^2}{2}\psi_i\right) + \frac{a_{n+1}^2-b^2}{2}\psi_{n+1} + \sum_{i=1}^n\frac{a_i^2}{2}\delta_{i,n+1},
	\]
	where $\delta_{i,n+1}$ is the class of the boundary divisor where markings $i$ and $n+1$ come together and bubble off.
	
	The projection formula then tells us that
	\begin{align*}
		\pi_*\DR_g(b;a_1,\ldots,a_{n+1}) &= \exp\left(-\frac{b^2}{2}\kappa_1+\sum_{i=1}^{n}\frac{a_i^2}{2}\psi_i\right)\\
		&\cdot\quad\pi_*\left[\exp\left(\frac{a_{n+1}^2-b^2}{2}\psi_{n+1} + \sum_{i=1}^n\frac{a_i^2}{2}\delta_{i,n+1}\right)\DRP_g(b;a_1,\ldots,a_{n+1})\right].
	\end{align*}
	
	Meanwhile, on the right side of \eqref{eq:DR_push_goal} we can use the fact that $[\text{codim}]$ is a derivation to compute that
	\begin{align*}
		&(g-[\text{codim}])\DR_g(b;a_1,\ldots,a_n) = \\ &\quad\exp(\DRD_g(b;a_1,\ldots,a_n))(g-[\text{codim}]-\DRD_g(b;a_1,\ldots,a_n))\DRP_g(b;a_1,\ldots,a_n)
	\end{align*}
	
	Using both of the last two equations, we can cancel a factor of $\exp(-\frac{b^2}{2}\kappa_1+\sum_{i=1}^{n}\frac{a_i^2}{2}\psi_i)$ from both sides of \eqref{eq:DR_push_goal} to get the equivalent congruence
	\begin{align}
		&\pi_*\left[\exp\left(\frac{a_{n+1}^2-b^2}{2}\psi_{n+1} + \sum_{i=1}^n\frac{a_i^2}{2}\delta_{i,n+1}\right)\DRP_g(b;a_1,\ldots,a_{n+1})\right] \equiv\label{eq:DR_push_F_goal}\\
		&\quad (a_{n+1}-b)^2\left(g-[\text{codim}]-\DRD_g(b;a_1,\ldots,a_n)\right)\DRP_g(b;a_1,\ldots,a_n)\pmod{(a_{n+1}-b)^3},\nonumber
	\end{align}
	where again we should note that the $\DRD$ and $\DRP$ components of the expression on the right side only make sense because we are multiplying them by $(a_{n+1}-b)^2$.
	
	The proof now is a matter of verifying \eqref{eq:DR_push_F_goal} by carefully computing this pushforward modulo $(a_{n+1}-b)^3$. This can be expressed as a sum of three different types of terms via the identity
	\[
	\exp\left(\frac{a_{n+1}^2-b^2}{2}\psi_{n+1} + \sum_{i=1}^n\frac{a_i^2}{2}\delta_{i,n+1}\right) = \left(\exp\left(\frac{a_{n+1}^2-b^2}{2}\psi_{n+1}\right)-1\right) + \sum_{i=1}^n\left(\exp\left(\frac{a_i^2}{2}\delta_{i,n+1}\right)-1\right) + 1,
	\]
	which follows from the observation that the divisors appearing in the formula all have trivial intersection with each other. In other words, we can either have a power of $\psi_{n+1}$, a power of $\delta_{i,n+1}$ for some $1\le i \le n$, or neither.
	
	It turns out that it is natural to group part of the pushforward $\pi_*\DRP_g(b;a_1,\ldots,a_{n+1})$ along with the terms with a power of $\delta_{i,n+1}$. For each $1\le i \le n$, let $\DRP_g^{(i,n+1)}(b;a_1,\ldots,a_{n+1})$ be the sum of those terms in the $\DR$ formula sum \eqref{eq:DRP} in which the graph $\Gamma$ has a genus $0$ vertex $v$ with both legs $i,n+1$, no other legs, and a single incident edge.
	
	We now state the contributions of these three types of terms:
	\begin{lemma}\label{lem:pushforwardA}
		We have the congruence
		\begin{align*}
			&\pi_*\left[\left(\exp\left(\frac{a_{n+1}^2-b^2}{2}\psi_{n+1}\right)-1\right)\DRP_g(b;a_1,\ldots,a_{n+1})\right] \equiv\\
			&\quad\quad\quad (a_{n+1}-b)^2\left(g+\frac{n}{2}-[\text{deg}]+\frac{b^2}{2}\kappa_1\right)\DRP_g(b;a_1,\ldots,a_n)\\
			&\quad\quad\quad +\quad (a_{n+1}-b)\sum_{i=1}^na_i\DRP_g(b;a_1,\ldots,a_i+a_{n+1}-b,\ldots,a_n) \pmod{(a_{n+1}-b)^3}.
		\end{align*}
	\end{lemma}
	\begin{lemma}\label{lem:pushforwardB}
		For each $1\le i\le n$, we have the congruence
		\begin{align*}
			&\pi_*\left[\left(\exp\left(\frac{a_i^2}{2}\delta_{i,n+1}\right)-1\right)\DRP_g(b;a_1,\ldots,a_{n+1})+\DRP_g^{(i,n+1)}(b;a_1,\ldots,a_{n+1})\right] \equiv\\
			&\quad\quad\quad (a_{n+1}-b)^2\left(-\frac{1}{2}-\frac{a_i^2}{2}\psi_i\right)\DRP_g(b;a_1,\ldots,a_n)\\
			&\quad\quad\quad +\quad (a_{n+1}-b)(-a_i)\DRP_g(b;a_1,\ldots,a_i+a_{n+1}-b,\ldots,a_n) \pmod{(a_{n+1}-b)^3}.
		\end{align*}
	\end{lemma}
	\begin{lemma}\label{lem:pushforwardC}
		We have the congruence
		\begin{align*}
			&\pi_*\left[\DRP_g(b;a_1,\ldots,a_{n+1})-\sum_{i=1}^n\DRP_g^{(i,n+1)}(b;a_1,\ldots,a_{n+1})\right] \equiv\\
			&\quad\quad\quad (a_{n+1}-b)^2\left([\text{deg}]-[\text{codim}]\right)\DRP_g(b;a_1,\ldots,a_n)\pmod{(a_{n+1}-b)^3}.
		\end{align*}
	\end{lemma}
	
	In the next three subsections we will prove these three lemmas. Adding Lemma~\ref{lem:pushforwardA}, Lemma~\ref{lem:pushforwardB} (summing over $i$), and Lemma~\ref{lem:pushforwardC} then yields \eqref{eq:DR_push_F_goal} and completes the proof of Theorem~\ref{thm:DR_push}.
	
	\subsection{Proof of Lemma~\ref{lem:pushforwardA}}\label{subsec:pushforwardA}
	We begin with the terms with a power of $\psi_{n+1}$. Since $a_{n+1}^2-b^2$ is a multiple of $a_{n+1}-b$ and we are working modulo $(a_{n+1}-b)^3$, we only need to consider terms with an exponent of $1$ or $2$. The pushforward then turns this power of $\psi_{n+1}$ into $\kappa_0=2g_v-2+n_v$ or $\kappa_1$ (on the vertex where the $(n+1)$-th leg was attached).
	
	It is tempting to conclude that this kappa class factors out (since the leg could be placed on any vertex) and that we are left with a kappa class times $\DRP_g(b;a_1,\ldots,a_{n})$. This will turn out to be true when we started with $\psi_{n+1}^2$, but with $\psi_{n+1}$ it doesn't even make sense because we do not have the factor of $(a_{n+1}-b)^2$ needed to use $\DRP_g(b;a_1,\ldots,a_{n})$ (see the discussion after equation \eqref{eq:DR_push_goal}).
	
	So we need to be more careful with this type of computation. Let $\Gamma$ be a stable graph for $\Mbar_{g,n}$. For simplicity we will assume that $\Gamma$ has no automorphisms and we do not worry about internal $\psi$ insertions - those details can be handled as in Section~\ref{subsubsec:topdeg_proof_general} (though things are even simpler here since we do not have gluing maps to worry about). Recall that the coefficient of the boundary stratum $(j_\Gamma)_*1$ in $\DRP_g(b;a_1,\ldots,a_n)$ is given by the graph invariant polynomial $C(\Gamma)(\{a_v\})$ after applying the specialization of variables \eqref{eq:specialization}.
	
	For each vertex $w\in V(\Gamma)$, let $\Gamma_w$ be the stable graph for $\Mbar_{g,n+1}$ given by attaching leg $n+1$ at $w$. Then we have that the coefficient of $(j_{\Gamma_w})_*1$ in $\DRP_g(b;a_1,\ldots,a_{n+1})$ is again given by a specialization of the same $C(\Gamma)(\{a_v\})$. One way to describe the altered specialization in $\Gamma_w$ as opposed to $\Gamma$ is that first we replace $a_w$ with $a_w+a_{n+1}-b$ and then we perform the previous specialization. This motivates the following notation (to be used only in this section): let $T$ be a formal variable (later to be set to $a_{n+1}-b$) and then given a vertex $w\in V(\Gamma)$, let
	\begin{equation}\label{eq:Pdef}
		\mathsf{P}_w = \mathsf{P}_w(T,\{a_v\}) := \left[C(\Gamma)\right]_{a_w \mapsto a_w + T} \in \QQ[T,\{a_v\}]\Big/\left(T + \sum_{v\in V(G)}a_v\right).
	\end{equation}
	Also, for $1\le i \le n$, let $\mathsf{P}_i := \mathsf{P}_{w_i}$ where $w_i$ is the vertex where leg $i$ is located. All the different $\mathsf{P}_w$ become equal to $C(\Gamma)$ on setting $T:=0$, so in particular they are all congruent mod $T$.
	
	We now return to the question of computing the pushforward of the $\psi_{n+1}^i$ terms. For $\psi_{n+1}^2$, we have that the coefficient of the class corresponding to $\Gamma$ with $\kappa_1$ at vertex $w$ is
	\[
	\text{coeff}_{(j_{\Gamma})_*\kappa_1[w]} \left(\pi_*\left[\frac{1}{2}\left(\frac{a_{n+1}^2-b^2}{2}\psi_{n+1}\right)^2\DRP_g(b;a_1,\ldots,a_{n+1})\right]\right) = \frac{(a_{n+1}^2-b^2)^2}{8}[\mathsf{P}_w]_{T,\{a_v\}\mapsto b,\{a_i\}},
	\]
	where the subscript indicates that we apply $T := a_{n+1}-b$ as well as the standard specialization \eqref{eq:specialization}.
	
	Because this term is divisible by $(a_{n+1}-b)^2$, the dependence on $w$ disappears modulo $(a_{n+1}-b)^3$. We can collect these terms for different $\Gamma$ and $w$ and factor out the $\kappa_1$ to get (mod $(a_{n+1}-b)^3$) a contribution of
	\begin{equation}\label{eq:pushA1}
		(a_{n+1}-b)^2\frac{b^2}{2}\kappa_1\DRP_g(b;a_1,\ldots,a_n).
	\end{equation}
	This is part of the right side of Lemma~\ref{lem:pushforwardA}.
	
	For $\psi_{n+1}^1$, things are very similar. The differences are that we do not have divisibility by $(a_{n+1}-b)^2$, but we do have that $\kappa_0[w] = 2g_w-2+n_w$ is a scalar. We get
	\begin{align}
		&\text{coeff}_{(j_{\Gamma})_*1} \left(\pi_*\left[\left(\frac{a_{n+1}^2-b^2}{2}\psi_{n+1}\right)\DRP_g(b;a_1,\ldots,a_{n+1})\right]\right) =\nonumber\\
		&\quad\quad\quad\quad\quad\quad\frac{a_{n+1}^2-b^2}{2}\sum_{w\in V(\Gamma)}(2g_w-2+n_w)[\mathsf{P}_w]_{T,\{a_v\}\mapsto b,\{a_i\}}.\label{eq:psi1_contrib}
	\end{align}
	
	It turns out that we want to break this into two terms using
	\[
	\frac{a_{n+1}^2-b^2}{2} = \frac{(a_{n+1}-b)^2}{2} + (a_{n+1}-b)b.
	\]
	The first term created in this way is then divisible by $(a_{n+1}-b)^2$, so again we can combine terms with different $w$. Since
	\[
	\sum_{w\in V(\Gamma)}(2g_w-2+n_w) = 2g-2+n
	\]
	is independent of $w$, we end up with a contribution of
	\begin{equation}\label{eq:pushA2}
		\frac{(a_{n+1}-b)^2}{2}(2g-2+n)\DRP_g(b;a_1,\ldots,a_n)
	\end{equation}
	towards the right side of Lemma~\ref{lem:pushforwardA}.
	
	The second term is more complicated because we only have a factor of $a_{n+1}-b$, not $(a_{n+1}-b)^2$. We use the following lemma:
	\begin{lemma}\label{lem:pushforwardAaux}
		Let $\mathsf{P}_w, \mathsf{P}_i$ be defined in terms of $C(\Gamma)$ by \eqref{eq:Pdef}. Then
		\begin{align*}
			&(a_{n+1}-b)\left[\sum_{w\in V(\Gamma)}(2g_w-2+n_w)b\cdot\mathsf{P}_w - \sum_{i=1}^n a_i\mathsf{P}_i\right]_{T,\{a_v\}\mapsto b,\{a_1,\ldots,a_n\}}\equiv\\
			&\quad\quad\quad (a_{n+1}-b)^2(1-[\text{deg}])[C(\Gamma)]_{\{a_v\}\mapsto b,\{a_1,\ldots,a_n\}}\quad\pmod{(a_{n+1}-b)^3}.
		\end{align*}
	\end{lemma}
	\begin{proof}
		Recall the formula for $a_w$ in the specialization \eqref{eq:specialization}. If we combine the sum over $w$ with the sum over $i$ and pull out a negative sign, we can rewrite the left side as
		\begin{equation}\label{eq:part2_intermediate}
			-(a_{n+1}-b)\left[\sum_{w\in V(\Gamma)}a_w\mathsf{P}_w\right]_{T,\{a_v\}\mapsto b,\{a_1,\ldots,a_n\}}.
		\end{equation}
		
		We need to do a little algebra with Taylor polynomials now. Note that the formal identity
		\[
		f(x,y+t) \equiv f(x+t,y) + t\left(\left(\frac{d}{dy}-\frac{d}{dx}\right)f\right)(x+t,y)\pmod{t^2}
		\]
		holds for any polynomial $f(x,y)$ by expanding both sides as Taylor series in $t$. Applying this to the polynomial $f = C(\Gamma)$ with $x=a_{w_1}, y=a_w, t=T$ gives the identity
		\[
		\mathsf{P}_w \equiv \mathsf{P}_1 + T\left[\left(\frac{d}{da_w}-\frac{d}{da_{w_1}}\right)C(\Gamma)\right]_{a_{w_1}\mapsto a_{w_1}+T}\pmod{T^2}.
		\]
		Applying this identity to replace all the $\mathsf{P}_w$ with $\mathsf{P}_1$ in \eqref{eq:part2_intermediate}, we get (mod $(a_{n+1}-b)^3)$):
		\begin{align*}
			&-(a_{n+1}-b)\left[\left(\sum_{w\in V(\Gamma)}a_w\right)\mathsf{P}_1\right]_{T,\{a_v\}\mapsto b,\{a_1,\ldots,a_n\}}\\
			&-(a_{n+1}-b)^2\left[\sum_{w\in V(\Gamma)}a_w\left(\frac{d}{da_w}-\frac{d}{da_{w_1}}\right)C(\Gamma)\right]_{\substack{a_{w_1}\mapsto a_{w_1}+T,\\T,\{a_v\}\mapsto b,\{a_1,\ldots,a_n\}}}.
		\end{align*}
		
		On the first line, the sum of $a_w$ is equal to $-T$ so this line simplifies to
		\begin{equation}\label{eq:part2a}
			(a_{n+1}-b)^2\left[\mathsf{P}_1\right]_{T,\{a_v\}\mapsto b,\{a_1,\ldots,a_n\}}.
		\end{equation}
		On the second line, the sum of $a_w$ inside the brackets is equal to $0$ (since it is before the change of variables introducing $T$) so the terms with $(d/da_{w_1})$ cancel. Also, we can write
		\[
		\sum_{w\in V(\Gamma)}a_w\frac{dC(\Gamma)}{da_w} = [\text{deg}]C(\Gamma),
		\]
		where $[\text{deg}]$ is the operator that multiplies a homogeneous polynomial by its degree, so the second line becomes simply
		\begin{equation}\label{eq:part2b}
			-(a_{n+1}-b)^2[\text{deg}]\left[\mathsf{P}_1\right]_{T,\{a_v\}\mapsto b,\{a_1,\ldots,a_n\}}.
		\end{equation}
		Since both \eqref{eq:part2a} and \eqref{eq:part2b} have a factor of $(a_{n+1}-b)^2$, we can replace $\mathsf{P}_1$ with $C(\Gamma)$ and add them to get the right side of the lemma.
	\end{proof}
	
	We now apply Lemma~\ref{lem:pushforwardAaux} to complete the proof of Lemma~\ref{lem:pushforwardA}. The remaining contribution we had to analyze had coefficient
	\[
	(a_{n+1}-b)b\sum_{w\in V(\Gamma)}(2g_w-2+n_w)[\mathsf{P}_w]_{T,\{a_v\}\mapsto b,\{a_i\}}
	\]
	(on the class of the boundary stratum corresponding to $\Gamma$). By Lemma~\ref{lem:pushforwardAaux}, we can replace this with
	\[
	(a_{n+1}-b)\sum_{w\in V(\Gamma)}a_i[\mathsf{P}_i]_{T,\{a_v\}\mapsto b,\{a_i\}} + (a_{n+1}-b)^2(1-[\text{deg}])[C(\Gamma)]_{\{a_v\}\mapsto b,\{a_1,\ldots,a_n\}}.
	\]
	But these terms can easily be summed over $\Gamma$ and rewritten in terms of $\DRP_g$. The result is
	\[
	(a_{n+1}-b)\sum_{i=1}^na_i\DRP_g(b;a_1,\ldots,a_i+a_{n+1}-b,\ldots,a_n) + (a_{n+1}-b)^2(1-[\text{deg}])\DRP_g(b;a_1,\ldots,a_n).
	\]
	Adding this to \eqref{eq:pushA1} and \eqref{eq:pushA2} gives the right side of Lemma~\ref{lem:pushforwardA}.
	
	\subsection{Proof of Lemma~\ref{lem:pushforwardB}}\label{subsec:pushforwardB}
	
	We now move on to the term with a power of $\delta_{i,n+1}$ (for a single fixed $1\le i \le n$): we wish to compute
	\[
	\pi_*\left[\left(\exp\left(\frac{a_i^2}{2}\delta_{i,n+1}\right)-1\right)\DRP_g(b;a_1,\ldots,a_{n+1})\right].
	\]
	Let $\iota:\Mbar_{g,n}\to\Mbar_{g,n+1}$ be the map corresponding to the boundary divisor $\delta_{i,n+1}$, i.e. we glue a rational bubble to the $i$th marking and place markings $i$ and $n+1$ on the bubble. The projection formula (along with $\pi\circ\iota=\id$ and the standard formula for the normal bundle of a boundary divisor) then gives that
	\[
	\pi_*(\delta_{i,n+1}^{k+1}\alpha) = (-\psi_i)^k\iota^*\alpha
	\]
	for any class $\alpha\in\CH^*(\Mbar_{g,n+1})$.
	
	Since $\DR$ cycles pull back to $\DR$ cycles along gluing maps at separating nodes, we can compute
	\begin{align*}
		\iota^*\DRP_g(b;a_1,\ldots,a_{n+1}) &= \exp(\DRD_g(b; a_1,\ldots,a_i+a_{n+1}-b,\ldots,a_n)-\iota^*\DRD_g(b;a_1,\ldots,a_{n+1}))\\
		&\quad\quad\cdot\,\DRP_g(b; a_1,\ldots,a_i+a_{n+1}-b,\ldots,a_n)\\
		&= \exp\left(\frac{(a_i+a_{n+1}-b)^2}{2}\psi_i\right)\DRP_g(b; a_1,\ldots,a_i+a_{n+1}-b,\ldots,a_n).
	\end{align*}
	Combining these equations, we have that
	\begin{align*}
		&\pi_*\left[\left(\exp\left(\frac{a_i^2}{2}\delta_{i,n+1}\right)-1\right)\DRP_g(b;a_1,\ldots,a_{n+1})\right]\\
		&= \frac{\exp\left(-\frac{a_i^2}{2}\psi_i\right)-1}{-\psi_i}\exp\left(\frac{(a_i+a_{n+1}-b)^2}{2}\psi_i\right)\DRP_g(b; a_1,\ldots,a_i+a_{n+1}-b,\ldots,a_n).
	\end{align*}
	
	This looks a little worrying - we have high powers of $\psi_i$, but the right side of Lemma~\ref{lem:pushforwardB} only allows a single $\psi_i$. However, this is why we include in Lemma~\ref{lem:pushforwardB} the additional terms with no $\delta_{i,n+1}$ but where the $\DRP_g$ factor uses a graph with a similar shape, with legs $i$ and $n+1$ bubbled off.
	
	Such stable graphs for $\Mbar_{g,n+1}$ are in bijection with all stable graphs for $\Mbar_{g,n}$. Let $\Gamma$ be the stable graph on $\Mbar_{g,n}$, i.e. the one obtained after forgetting leg $n+1$ and contracting the resulting unstable component. Then the $\DR$ formula gives that the pushforward of these terms is
	\[
	\frac{1-\exp\left(\frac{(a_i+a_{n+1}-b)^2}{2}\psi_i\right)}{\psi_i}\DRP_g(b; a_1,\ldots,a_i+a_{n+1}-b,\ldots,a_n).
	\]
	Adding this to the previous expression, we get
	\[
	\frac{1-\exp\left(\frac{(a_i+a_{n+1}-b)^2-a_i^2}{2}\psi_i\right)}{\psi_i}\DRP_g(b; a_1,\ldots,a_i+a_{n+1}-b,\ldots,a_n).
	\]
	Since $(a_i+a_{n+1}-b)^2-a_i^2$ is divisible by $a_{n+1}-b$, terms with a high power of $\psi_i$ will vanish modulo $(a_{n+1}-b)^3$. We are left with
	\begin{align*}
		&\frac{a_i^2-(a_i+a_{n+1}-b)^2}{2}\DRP_g(b; a_1,\ldots,a_i+a_{n+1}-b,\ldots,a_n)\\
		&+\frac{-((a_i+a_{n+1}-b)^2-a_i^2)^2}{8}\psi_i\DRP_g(b; a_1,\ldots,a_i+a_{n+1}-b,\ldots,a_n).
	\end{align*}
	The second line is a multiple of $(a_{n+1}-b)^2$. After reducing modulo $(a_{n+1}-b)^3$ it becomes
	\[
	(a_{n+1}-b)^2\left(\frac{-a_i^2}{2}\psi_i\right)\DRP_g(b; a_1,\ldots,a_n),
	\]
	which is part of the right side of Lemma~\ref{lem:pushforwardB}.
	
	We split the first line into a sum of two terms using
	\[
	\frac{a_i^2-(a_i+a_{n+1}-b)^2}{2} = -\frac{(a_{n+1}-b)^2}{2}-(a_{n+1}-b)a_i.
	\]
	This gives the remaining components of the right side of Lemma~\ref{lem:pushforwardB} (since in the first term we have a factor of $(a_{n+1}-b)^2$ so can change $a_i+a_{n+1}-b$ back to $a_i$).

	\subsection{Proof of Lemma~\ref{lem:pushforwardC}}\label{subsec:pushforwardC}
	
	What remains is to compute the part of
	\[
	\pi_*\DRP_g(b;a_1,\ldots,a_{n+1})
	\]
	that does not involve popping a rational bubble containing leg $n+1$ and one other leg.
	
	There are two different things that can happen with this pushforward. The first is that leg $n+1$ is on a vertex that becomes unstable without that leg. Since we've already considered the case in which this happens with another leg present, this means it has two incident edges. In other words, we are popping a rational bubble formed by leg $n+1$ moving into a node. In the other case, the graph remains stable without leg $n+1$. The pushforward is then computed by decreasing the exponent of $\psi$ along one incident half-edge.
	
	We group terms together based on the stable graph $\Gamma$ for $\Mbar_{g,n}$ (given after possibly stabilizing by popping a rational bubble) along with a chosen edge $f\in E(\Gamma)$ (either the node where the rational bubble was located or the edge along which the exponent of a $\psi$ class was decremented. Thus after picking $\Gamma$ and $f$ there are three possibilities for the starting stable graph $\Gamma'$ for $\Mbar_{g,n+1}$ that we are grouping together - either the edge is subdivided with a genus $0$ vertex and leg $n+1$ is placed there, or leg $n+1$ is placed on one of the two endpoints of the edge.
	
	We now consider the $\DR$ formula for these terms. Let $r>0$ be a positive integer. Then we want to sum over half-edge weightings $w:H(\Gamma')\to\{0,\ldots,r-1\}$ satisfying certain congruence conditions mod $r$. Although the graph $\Gamma'$ and the precise congruence conditions vary in the three cases described above, there are bijections between the weighting sets $W$ in the different cases that keep things unchanged away from the edge $f$ (or its subdivision). We show the effects of these bijections on the weights near the edge $f$ below. Here $T := a_{n+1}-b$, $A$ is an arbitrary integer, and the weights should all be interpreted mod $r$.
	
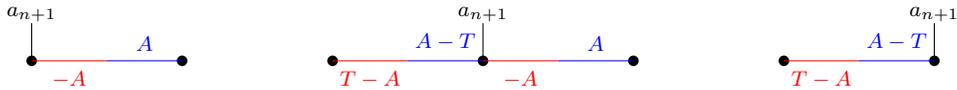
	\begin{center}
		\begin{tikzpicture}[>=latex,
			lab/.style={font=\scriptsize},              
			halflab/.style={font=\scriptsize,yshift=4pt}
			]
			
			% Graph 1
			\fill (0,0) circle (2pt) (2,0) circle (2pt);
			\draw[red]  (0,0) -- (1,0)
			node[midway, below,  lab, red] {$-A$};
			\draw[blue] (1,0) -- (2,0)
			node[midway, above,  lab, blue] {$A$};
			\draw (0,0) -- ++(0,0.5)
			node[midway, above, halflab] {$a_{n+1}$};
			
			% Graph 2
			\fill (4,0) circle (2pt) (6,0) circle (2pt) (8,0) circle (2pt);
			\draw[red]  (4,0) -- (5,0)
			node[midway, below,  lab, red] {$T-A$};
			\draw[blue] (5,0) -- (6,0)
			node[midway, above,  lab, blue] {$A-T$};
			\draw[red]  (6,0) -- (7,0)
			node[midway, below,  lab, red] {$-A$};
			\draw[blue] (7,0) -- (8,0)
			node[midway, above,  lab, blue] {$A$};
			\draw (6,0) -- ++(0,0.5)
			node[midway, above, halflab] {$a_{n+1}$};
			
			% Graph 3
			\fill (10,0) circle (2pt) (12,0) circle (2pt);
			\draw[red]  (10,0) -- (11,0)
			node[midway, below,  lab, red] {$T-A$};
			\draw[blue] (11,0) -- (12,0)
			node[midway, above,  lab, blue] {$A-T$};
			\draw (12,0) -- ++(0,0.5)
			node[midway, above, halflab] {$a_{n+1}$};
			
		\end{tikzpicture}
	\end{center}
	
	We now simply add up the pushforwards of these three types of term, with corresponding weights grouped together. The result is something that looks a lot like the formula for $\DRP_g(b;a_1,\ldots,a_n)$, but with the factor for edge $f$ changed:
	\[
	\sum_{\substack{\Gamma\in \mathsf{G}_{g,n} \\ f \in E(\Gamma)}}\left[\frac{r^{-h^1(\Gamma)}}{|\aut(\Gamma,f)|}(j_{\Gamma})_*\Bigg[\sum_{w\in \mathsf{W}}\left(
	\prod_{e=(h,h')\ne f\in E(\Gamma)}
	\frac{1-\exp\left(-\frac{w(h)w(h')}2(\psi_h+\psi_{h'})\right)}{\psi_h + \psi_{h'}}\right)Q(A,T,r)\Bigg]\right]_{r=0},
	\]
	where $\aut(\Gamma,f)$ is the stabilizer of $f$ in $\aut(\Gamma)$, $A$ is the value of $w$ on one of the two half-edges in $f$ (chosen arbitrarily), $\mathsf{W}$ is a suitable set of weightings, and $Q(A,T,r)$ is a somewhat messy function of $A,T,r$ (with values that are power series in $\psi,\psi'$ on the two halves of $f$).
	
	To write down $Q$, it will be convenient to let $[x]_r$ for an integer $x$ denote the unique integer between $0$ and $r-1$ congruent to $x$ (mod $r$). Also, to save writing we will replace $w(h')$ with $-w(h)$ in the $\DR$ formula - this will not affect the result at the end after taking the $r^0$ coefficient. We can then take
	\[
	Q(A,T,r) = \frac{-\psi'\exp\left(\frac{A^2}{2}(\psi+\psi')\right)+(\psi+\psi')\exp\left(\frac{[A-T]_r^2}{2}\psi+\frac{A^2}{2}\psi'\right)-\psi\exp\left(\frac{[A-T]_r^2}{2}(\psi+\psi')\right)}{\psi\psi'(\psi+\psi')}.
	\]
	Here the $\psi$ and $\psi'$ in the denominator come from the terms where the pushforward decrements the exponent of $\psi$ or $\psi'$.
	
	Let $\widehat{Q}(A,T,r)$ be the same expression with the instances of $[A-T]_r$ replaced by $A-T$. Then expanding in $T$ gives (after a straightforward calculus computation) that
	\[
	\widehat{Q}(A,T,r) = T^2\left(\frac{-A^2}{2}\exp\left(\frac{A^2}{2}(\psi+\psi')\right)\right) + O(T^3).
	\]
	Since we will set $T := a_{n+1}-b$ and work modulo $(a_{n+1}-b)^3$, this means that if we were using $\widehat{Q}$ instead of $Q$, we would get a contribution divisible by $(a_{n+1}-b)^2$. Moreover, note that
	\[
	\frac{-A^2}{2}\exp\left(\frac{A^2}{2}(\psi+\psi')\right) = (1 + [\text{codim}])\frac{1 - \exp\left(\frac{A^2}{2}(\psi+\psi')\right)}{\psi+\psi'}.
	\]
	Since each edge also contributes $1$ to the codim after applying $(j_\Gamma)_*$, this means that we would get precisely
	\begin{equation}\label{eq:codim_only}
		(a_{n+1}-b)^2\cdot[\text{codim}]\DRP_g(b;a_1,\ldots,a_n).
	\end{equation}
	
	It remains to compute the above contribution using $Q(A,T,r)-\widehat{Q}(A,T,r)$ instead of $Q(A,T,r)$. We may assume that $r > T > 0$; then $Q(A,T,r)-\widehat{Q}(A,T,r)$ vanishes for $T\le A< r$ and agrees with a formal power series in $\psi,\psi'$ with coefficients that are polynomials in $A,T,r$ for $0 \le A < T$. Call this power series $\overline{Q}$. Explicitly, we have
	\begin{align*}
		\overline{Q}(A,T,r) &= -\frac{1}{\psi'(\psi+\psi')}\left(\exp\left(\frac{(A-T+r)^2}{2}(\psi+\psi')\right)-\exp\left(\frac{(A-T)^2}{2}(\psi+\psi')\right)\right)\\
		&+\frac{\exp\left(\frac{A^2}{2}\psi'\right)}{\psi\psi'}\left(\exp\left(\frac{(A-T+r)^2}{2}\psi\right)-\exp\left(\frac{(A-T)^2}{2}\psi\right)\right).
	\end{align*}
	
	We wish to compute
	\[
	\sum_{\substack{\Gamma\in \mathsf{G}_{g,n} \\ f \in E(\Gamma)}}\left[\frac{r^{-h^1(\Gamma)}}{|\aut(\Gamma,f)|}(j_{\Gamma})_*\Bigg[\sum_{\substack{w\in \mathsf{W}\\ 0 \le A < T}}\left(
	\prod_{e=(h,h')\ne f\in E(\Gamma)}
	\frac{1-\exp\left(-\frac{w(h)w(h')}2(\psi_h+\psi_{h'})\right)}{\psi_h + \psi_{h'}}\right)\overline{Q}(A,T,r)\Bigg]\right]_{r=0}.
	\]
	If $f$ is a separating edge, then $A$ is uniquely determined by the balancing conditions on $w$ and the factor $\overline{Q}$ comes out of the sum over $w$. The sum over $w$ is then just the product of the sums appearing in the $\DR$ formula for the two connected components formed by cutting the separating edge $f$. In particular, it is polynomial in $r$ for $r$ sufficiently large and divisible by $r^{h^1(\Gamma)}$. Since $\overline{Q}$ is divisible by $r$, the $r=0$ specialization then vanishes.
	
	So we can replace the sum over $f$ with a sum over non-separating edges $f$. In this case, we can factor the sum over $w$ by first summing over $A$:
	\[
	\sum_{\substack{\Gamma\in \mathsf{G}_{g,n} \\ f \in E(\Gamma)_{\text{nonsep}}}}\left[\sum_{k=0}^{T-1}
	\frac{r^{-h^1(\Gamma)}}{|\aut(\Gamma,f)|}(j_{\Gamma})_*\Bigg[\sum_{\substack{w\in \mathsf{W}\\ A = k}}\left(
	\prod_{e=(h,h')\ne f\in E(\Gamma)}
	\frac{1-\exp\left(-\frac{w(h)w(h')}2(\psi_h+\psi_{h'})\right)}{\psi_h + \psi_{h'}}\right)\overline{Q}(k,T,r)\Bigg]\right]_{r=0}.
	\]
	But the subset $\{w\in \mathsf{W}\mid A=k\}$ is naturally in bijection with the set of balanced weightings for the graph $\Gamma'$ formed by deleting edge $f$, attaching a new leg at each endpoint of $f$, and appending weights $k,-k$ to the vector $\sfa$. So we have the $\DRP$ formula for $\Gamma'$, but we have the extra factor $\overline{Q}(k,T,r)$. We also have an extra factor of $r^{-1}$ since $h^1(\Gamma') = h^1(\Gamma)-1$. Fortunately, $\overline{Q}(k,T,r)$ is divisible by $r$, so we get that the expression inside the sum over $k$ is at least a polynomial in $r$ (for $r$ sufficiently large).
	
	We can modify $\overline{Q}/r$ by any multiple of $r$ without changing the resulting $r\ge 0$ specialization (the taking of which commutes with summing over $k$). We can also modify it by any multiple of $T^2$ without changing the final result (mod $(a_{n+1}-b)^3$), since the sum over $k$ gives an extra factor of $T$. A simple calculus computation gives that
	\[
	\frac{\overline{Q}(k,T,r)}{r}\equiv k^2\exp\left(\frac{k^2}{2}(\psi+\psi')\right)T\pmod{(r,T^2)}.
	\]
	
	Note that
	\[
	\frac{k^2}{2}(\psi_{n+1}+\psi_{n+2}) = \DRD_{g-1}(b;a_1,\ldots,a_n,k,-k) - j^*\DRD_g(b;a_1,\ldots,a_n),
	\]
	where $j$ is the map gluing the last two markings. We can then rewrite our expression (using the projection formula) as
	\[
	T^2\exp(-\DRD_g(b;a_1,\ldots,a_n))\frac{1}{T}\sum_{k=0}^{T-1}\frac{k^2}{2}j_*\DR_{g-1}(b;a_1,\ldots,a_n,k,-k).
	\]
	Here the factor of $\frac{1}{2}$ comes from the choice of labels of legs $n+1$ and $n+2$ after cutting $\Gamma$ at $f$.
	
	Since the coefficient of $T^1$ in $\sum_{k=0}^{T-1}k^{d+2}$ is precisely the Bernoulli number $B_{d+2}$, this gives the negative of the right side of Lemma~\ref{lemma:codim_minus_deg} (times $T^2$, mod $T^3$, with the exponential factor removed). Therefore it is equal to
	\[
	(a_{n+1}-b)^2([\text{deg}]-2[\text{codim}])\DRP_g(b;\sfa).
	\]
	When combined with the earlier contribution \eqref{eq:codim_only} in this section, we get the right side of Lemma~\ref{lem:pushforwardC}. This completes the proof of Lemma~\ref{lem:pushforwardC} (and thus also Theorem~\ref{thm:DR_push}).
	
	\subsection{Extension to $\uniDR$}\label{subsec:pushforward_uniDR}
	As in Section~\ref{subsec:topdeg_uniDR}, we generalize the pushforward identity Theorem~\ref{thm:DR_push} to the context of $\uniDR$. One complication that arises here is that compared with $\Mbar_{g,n}$, the forgetful morphism $\pic_{g,n+1}\to\pic_{g,n}$ is not proper. There are actually two separate things that go wrong. The first is that the forgetful morphism for stacks of prestable curves $\mathfrak{M}_{g,n+1}\to\mathfrak{M}_{g,n}$ is already not proper. We address this by using curves with valuation in a semigroup, following \cite[Section 2.2]{BS1}.
	
	Consider the semi-group $\mathcal{A}=\{\mathbf{0},\mathbf{1}\}$ and consider the moduli stack of $\A$-valued prestable curves $\M_{g,n,\mathbf{1}}$. These are prestable curves with an additional choice of element $a_v\in\mathcal{A}$ for every vertex $v$, such that every vertex with $a_v=0$ satisfies the usual stability condition. The universal curve is denoted $p:\C_{g,n,\mathbf{1}}\to\M_{g,n,\mathbf{1}}$. By \cite[Proposition 2.6]{BS1}, the morphism $\M_{g,n+1,\mathbf{1}}\to\M_{g,n,\mathbf{1}}$ forgetting the last marking can be identified with the universal curve. Let $F:\M_{g,n,\mathbf{1}}\to\M_{g,n}$ be the morphism forgetting the $\A$-valuation.
	
	Let $\pic_{g,n,\mathbf{1},0}\to\M_{g,n,\mathbf1}$ be the universal degree $0$ Picard stack for the universal curve $\C_{g,n,\mathbf{1}}\to \M_{g,n,\mathbf{1}}$. Let $p:\C_{g,n,\mathbf{1}}\to\pic_{g,n,\mathbf{1},0}$ be the pullback of the universal curve and let $\L$ be the universal line bundle on $\C_{g,n,\mathbf{1}}$. Then tuple $(\C_{g,n,\mathbf{1}}/\pic_{g,n,\mathbf{1},0},\L)$ defines a $\uniDR$ formula $\uniDR_{g,\mathbf{1}}(b;\sfa)\in\CH^*(\pic_{g,n,\mathbf{1},0})$. This $\uniDR$ formula looks just like the standard $\uniDR$ formula, i.e. it is blind to the $\mathcal{A}$-valuation in the following sense:
	\begin{lemma}\label{lem:unidr_A}
		Let $F: \pic_{g,n,\mathbf{1},0}\to\pic_{g,n,0}$ be the morphism forgetting the $\A$-valuation. Then $F^*\uniDR_g(b;\sfa)= \uniDR_{g,\mathbf{1}}(b;\sfa)$.
	\end{lemma}
	\begin{proof}
		For the forgetful morphism $F:\M_{g,n,\mathbf{1}}\to\M_{g,n}$, the pullback of the universal curve $\C_{g,n}\to\M_{g,n}$ along $F$ is the universal curve $\C_{g,n,\mathbf{1}}$. Therefore, the same holds for $F:\pic_{g,n,\mathbf{1},0}\to\pic_{g,n,0}$. Moreover, the universal line bundle on $\C_{g,n,\mathbf{1}}$ is the pullback of the universal line bundle on $\C_{g,n}$ so the $\uniDR$ cycle formula on $\pic_{g,n,0}$ pulls back to the $\uniDR$ cycle formula for $\pic_{g,n,\mathbf{1},0}$.
	\end{proof}
	
	However, now we come to our second problem: there is no forgetful map from $\pic_{g,n+1,\mathbf{1},0}$ to $\pic_{g,n,\mathbf{1},0}$. The issue is that we might need to contract a component of the curve on which we have a nontrivial line bundle. We address this by using an open substack $\pic'_{g,n+1,\mathbf{1},0}$ that does admit a proper forgetful map to $\pic_{g,n,\mathbf{1},0}$. This is simply the open substack determined by the condition that if the underlying curve is such that forgetting marking $n+1$ would make a component unstable, then the line bundle must be trivial on that component (which is necessarily genus $0$). Equivalently, we can construct $\pic'_{g,n+1,\mathbf{1},0}$ by pulling back the forgetful map on $\A$-valued prestable curves, as follows.
	
	We consider the fiber product
	\[
	\begin{tikzcd}
		\mathfrak{Pic}'_{g,n+1,\mathbf{1},0} \ar[r,"p"] \ar[d] & \mathfrak{Pic}_{g,n,\mathbf{1},0} \ar[d] \\
		\M_{g,n+1,\mathbf{1}} \ar[r] & \M_{g,n,\mathbf{1}}.
	\end{tikzcd}
	\]
	By \cite[eq. (22)]{BS1}, there exists a diagram
	\[
	\begin{tikzcd}
		\C_{g,n+1,\mathbf{1}}\ar[r,"r"]\ar[dr] & \C'_{g,n,\mathbf{1}}\ar[r]\ar[d] & \C_{g,n,\mathbf{1}}\ar[d]\\
		& \mathfrak{Pic}_{g,n+1,\mathbf{1},0}'\ar[r,"p"] & \mathfrak{Pic}_{g,n,\mathbf{1},0}
	\end{tikzcd}
	\]
	where the square is the fiber product and $r$ is the canonical map which is proper birational. The tuple  $(\C_{g,n+1,\mathbf{1}}/\pic'_{g,n+1,\mathbf{1},0},r^*\L)$ defines a $\uniDR$ formula
	\[\uniDR'_{g,\mathbf{1}}(b;\sfa)\in \CH^*(\pic'_{g,n+1,\mathbf{1},0}).\]
	
	We can now push forward $\uniDR'_{g,\mathbf{1}}$ along $p$ and compare the result with $\uniDR_{g,\mathbf{1}}$, lifting
	Theorem \ref{thm:DR_push}:
	\begin{theorem}\label{thm:pushforward_uniDR}
		Let $g,c,n\ge 0$. Let 
		\[
		F = p_*\uniDR'^c_{g,\mathbf{1}}(b;a_1,\ldots,a_{n+1}) \in \CH^{c-1}(\pic_{g,n,\mathbf{1},0})\otimes_\QQ\QQ[b,a_1,\ldots,a_{n+1}]/(a_1+\cdots+a_{n+1}-(2g-1+n)b)\,.
		\]
		\begin{enumerate}[label=(\alph*)]
			\item $F$ is a multiple of $(a_{n+1}-b)^2$.
			\item We have the identity
			\[
			\left[\frac{F}{(a_{n+1}-b)^2}\right]_{a_{n+1}:=b} = (g+1-c)\uniDR_{g,\mathbf{1}}^{c-1}(b;a_1,\ldots,a_n).
			\]
		\end{enumerate}
	\end{theorem}
	
	\begin{proof}
		We use the same proof strategy as before, noting only the changes that happen with $\uniDR$. As before, we begin by canceling exponential factors as much as possible using the projection formula -- this step requires pulling back $\uniDRD$ along the forgetful map. The exponential factor in \eqref{eq:DR_push_F_goal} now includes an additional term:
		\[
		(a_{n+1}-b)\xi_{n+1}.
		\]
		There are no changes to the right side other than having extra terms $\xi_i,\kappa_{0,1},\kappa_{-1,2}$ inside the divisor $\uniDRD$.
		
		We again divide the computation into the same three pieces as before, with the only difference being that each now carries an additional factor of $\exp((a_{n+1}-b)\xi_{n+1})$.
		
		\subsubsection{Changes in Section~\ref{subsec:pushforwardA}}\label{subsubsec:changesA}
		Here, where we have a power of $\psi_{n+1}$, there are two places where things change. The first is that we get an additional term with pushing forward $\psi_{n+1}\xi_{n+1}$ (which produces a $\kappa_{0,1}[v]$ on the vertex where leg $n+1$ is attached). Since this term is divisible by $(a_{n+1}+b)^2$, computing its contribution is easy - we get precisely
		\[
		(a_{n+1}-b)^2b\kappa_{0,1}\uniDRP_g(b;a_1,\ldots,a_n),
		\]
		which accounts for the new $\kappa_{0,1}$ term in the $\uniDRD$ factor.
		
		Second, Lemma~\ref{lem:pushforwardAaux} requires changes because the specialization \eqref{eq:specialization} used to write $a_w$ in terms of $b,a_i$ now also includes the term $-\delta(w)$. This means that the left side of Lemma~\ref{lem:pushforwardAaux} must include $+\sum_{w}\delta(w)\mathsf{P}_w$ inside the brackets, while on the right side $[\text{deg}]$ must be replaced by $[\text{deg}]+[\text{weight}]$ (since the degree of $C(\Gamma)$ in the $a_v$ variables corresponds to the mixed degree of $\uniDRP$ in the $b,a_i,\delta(v)$ variables and the weight in $\uniDRP$ is the degree in the $\delta(v)$ variables). Both of these changes produce additional contributions which will be exactly canceled out by changes described in Section~\ref{subsubsec:changesC} below.
		
		\subsubsection{Changes in Section~\ref{subsec:pushforwardB}}
		The first thing to note here is that the correct analogue of $\DRP_g^{(i,n+1)}(b;a_1,\ldots,a_{n+1})$ for $\uniDR'$ is to require that the rational bubble containing legs $i,n+1$ must also have $\mathcal{A}$-valuation $0$ (which then also implies that it has $\delta(v) = 0$). The boundary strata of $\pic'_{g,n+1,\mathbf{1},0}$ with this property are in bijection with the boundary strata of $\pic_{g,n,\mathbf{1},0}$, as desired.
		
		Then in both types of terms in this section, the $\xi_{n+1}$ exponential just factors out to become an exponential of $\xi_i$ after the pushforward. Since positive powers of $\xi_{n+1}$ come with factors of $a_{n+1}-b$, the only new thing that we get (mod $(a_{n+1}-b)^3$) comes from the part of the previous result that was not yet divisible by $(a_{n+1}-b)^2$ (the final line of Lemma~\ref{lem:pushforwardB}). So the new contribution is
		\[
		(a_{n+1}-b)(-a_i)(\exp((a_{n+1}-b)\xi_i)-1)\uniDRP_g(b; a_1,\ldots,a_i+a_{n+1}-b,\ldots,a_n).
		\]
		Reducing mod $(a_{n+1}-b)^3$ gives us
		\[
		-(a_{n+1}-b)^2\sum_{i=1}^na_i\xi_i\uniDRP_g(b;a_1,\ldots,a_n),
		\]
		which accounts for the new $\xi_i$ term in the $\uniDRD$ factor.
		
		\subsubsection{Changes in Section~\ref{subsec:pushforwardC}}\label{subsubsec:changesC}
		Here there are multiple new things that happen. One thing that happens is that we are no longer just decrementing an incident $\psi$ exponent when the vertex remains stable - we also need to handle the $\xi$ exponent. In other words, we need to know the formula for
		\[
		\pi_*\big(\xi_{n+1}^i\psi_1^{e_1}\cdots\psi_n^{e_n}\big).
		\]
		Using the relation $\psi_i = \pi^* \psi_i + \delta_{i, n+1}$ and applying the projection formula, we find that two types of terms arise: either we replace the $\xi_{n+1}^i$ by $\kappa_{-1,i}[v]$, or we decrement one of the $\psi$ exponents and convert $\xi_{n+1}$ into a corresponding $\xi_i$.
		
		In the first case, we must have $i \leq 2$ in order to obtain a nontrivial contribution, since for $i > 2$ the expression is divisible by $(a_{n+1}-b)^3$. When $i = 1$, we use the identity $\kappa_{-1,1}[v] = \delta(v)$, which precisely cancels one of the changes involving $\delta(w)\mathsf{P}_w$ from Section~\ref{subsubsec:changesA}. When $i=2$, the extra factor of $(a_{n+1} - b)$ allows us to combine terms across different vertices $v$, and factor out a global $\kappa_{-1,2}$ term. This yields
		\[
		(a_{n+1}-b)^2\frac{\kappa_{-1,2}}{2}\uniDRP_g(b;a_1,\ldots,a_n),
		\]
		which accounts for the new $\kappa_{-1,2}$ term in the $\uniDRD$ factor.
		
		In the second case, we just get the same thing as before except with an extra power of $\xi$ along the edge $f$. We also just get an extra power of $\xi$ when pushing forward a class on a graph that requires restabilization. Thus, the $\xi$ factors out of all such terms, and any term involving a positive power of $\xi$ vanishes due to the presence of an extra factor of $(a_{n+1}-b)$.
		
		The final change that happens is that we replace Lemma~\ref{lemma:codim_minus_deg} with Lemma~\ref{lemma:unicodim_minus_deg}. This just means that we are replacing $[\text{deg}]$ with $[\text{deg}] + [\text{weight}]$, which cancels out with the previous time when we did this (in Section~\ref{subsubsec:changesA}).
	\end{proof}
	
	\section{Fourier transform and pushforward}\label{sec:7}
	\subsection{Leading term of the Fourier transform}
	Let $\epsilon_0$ be a small stability condition (i.e. such that the trivial line bundle is $\epsilon_0$-stable) and let $\epsilon$ be any nondegenerate stability condition. By Theorem \ref{thm:autoequiv}, we have a Fourier transform
	\begin{equation}\label{eq:fmtozero}
		\fm_g :\CH^*(\Jbar_{g,n}^\epsilon) \xrightarrow{\cong} \CH^*(\Jbar_{g,n}^{\epsilon_0})\,.
	\end{equation}
	Since $\epsilon_0$ is small, we have the unit section $e: B\to\Jbar_{g,n}^{\epsilon_0}$.
	\begin{proposition}\label{pro:IntFM}
		For a nondegenerate stability condition $\epsilon$ and a small nondegenerate stability condition $\epsilon_0$, let $\fm$ be the Fourier transform \eqref{eq:fmtozero}. For any $\alpha\in \CH^*(\Jbar_{g,n}^\epsilon)$, we have
		\[\pi_*(\alpha)=e^*\fm_g(\alpha)\]
	\end{proposition}
	\begin{proof}
		Let $\pi_1:\Jbar_C^{(2)}\to \Jbar_C^{\epsilon}$ and $\pi_2:\Jbar_C^{(2)}\to \Jbar_C^{\epsilon_0}$ be the two projections.
		Consider the following diagram
		\begin{equation}\label{eq:basechange}
			\begin{tikzcd}
				\Jbar_C^{\epsilon} \ar[r,"\Tilde{e}"]\ar[d,"\pi"] & \Jbar_C^{(2)}\ar[d,"\pi_2"]\\
				B\ar[r,"e"] & \Jbar_C^{\epsilon_0},
			\end{tikzcd}
		\end{equation}
		where $\pi_2,\pi$ are flat by Proposition \ref{pro:model2}, and $\Tilde{e}$ is the map induced by $\textup{id} \times e$. Since $e:B\to \Jbar_C^{\epsilon_0}$ factors through $J^{\epsilon_0}_C$, the fiber product is isomorphic to $\Jbar_C^{\epsilon}$. Therefore, since $e$ is a regular embedding, we have
		\begin{align*}
			e^*\fm_g(\alpha) &= e^*(\pi_{2})_*(\pi_1^*(\alpha)\Td(-T_{\Jbar_C\times_B\Jbar_C})\tau(\pbar)) \\
			& = \pi_*\widetilde{e}^*(\pi_1^*(\alpha)\Td(-T_{\Jbar_C\times_B\Jbar_C})\tau(\pbar))\\
			& = \pi_*(\alpha \cup \ch(\widetilde{e}^*\widetilde{\P})) = \pi_*(\alpha)
		\end{align*}
		where the second follows from the base change formula applied to \eqref{eq:basechange} and the last equality follows from the Todd class calculation in the proof of  Proposition \ref{pro:section} and $\widetilde{e}^*\widetilde{\P}\cong \CO$.
	\end{proof}
	The same argument implies $\pi_*(\alpha)=e^*\fm_g^\circ(\alpha)$, which will be used later.
	
	We state a consequence of \cite{BHPSS}. This formula will serve as the first step in computing the Fourier transform using the universal $\DR$ cycle formula. Let $\varphi:\Jbar_{g,n}^{\epsilon}\to\pic_{g,n}$ be the morphism   \eqref{eq:varphi}. 
	\begin{proposition}\label{pro:aj_P}
		Let $b\in\ZZ$ and $\sfa\in\ZZ^n$ with $\sum_{i=1}^n a_i=(2g-2+n)b$. For any nondegenerate stability condition $\epsilon$ of degree $0$, we have
		\[(\aj_{b;\sfa})_* [B_{b;\sfa}^\epsilon] = \uniDR^g_g(b;\sfa)|_{\Jbar_{g,n}^\epsilon} \in \CH^g(\Jbar_{g,n}^\epsilon)\,.\]
	\end{proposition}
	\begin{proof}
		We consider the Abel-Jacobi map for the universal Picard stack \cite[(3.1)]{BHPSS}:
		\[\ajp_{b;\sfa}: \M_{g,n}\to\pic_{g,n}, (C,x_1,\ldots,x_n)\mapsto\omega_{C,\log}^{\otimes-b}(\sum_{i=1}^na_ix_i)\,.\]
		Consider the schematic image $\overline{\ajp_{b;\sfa}}\subset\pic_{g,n}$ which is the smallest closed reduced substack through which $\ajp_{b;\sfa}$ factors. Since $B_{b;\sfa}^\epsilon$ is irreducible and proper, the image $\mathrm{im} (\aj_{b;\sfa})$ is a closed irreducible subset of $\Jbar_{g,n}^{\epsilon}$. Both $\aj_{b;\sfa}$ and $\ajp_{b;\sfa}$ agree over the locus where the underlying curve is smooth, so we have $\aj_{b;\sfa} = \overline{\varphi^{-1}(\ajp_{b;\sfa})}$. Since $\varphi$ is smooth, $\overline{\varphi^{-1}(\ajp_{b;\sfa})} = \varphi^{-1}(\overline{\ajp_{b;\sfa}})$ by \cite[081I]{stacks-project}. Therefore we get
		\begin{equation}\label{eq:seteq}
			\mathrm{im} (\aj_{b;\sfa})=\varphi^{-1}(\overline{\ajp_{b;\sfa}})\,.
		\end{equation}
		
		We compare \eqref{eq:seteq} with the universal $\mathsf{DR}$ cycle formula. 
		Since $\varphi$ is smooth (Lemma \ref{lem:smoothness}), by \cite[Theorem 0.7]{BHPSS}, we have 
		\[
		\uniDR^g_g(b;\sfa)|_{\Jbar_{g,n}^\epsilon} = [\varphi^{-1}(\overline{\ajp_{b;\sfa}})]\,.
		\]
		Since morphism $\aj_{b;\sfa}:B_{b;\sfa}^{\epsilon}\to \mathrm{im}(\aj_{b;\sfa})$  is birational,  $(\aj_{b;\sfa})_*[B_{b;\sfa}^\epsilon] = [\mathrm{im}(\aj_{b;\sfa})]$. Therefore, we have $(\aj_{b;\sfa})_*[B_{b;\sfa}^\epsilon] = [\mathrm{im}(\aj_{b;\sfa})] = [\varphi^{-1}(\overline{\ajp_{b;\sfa}})] =  \uniDR^g_g(b;\sfa)|_{\Jbar_{g,n}^\epsilon}$.
	\end{proof}
	When $\epsilon$ is small, Proposition \ref{pro:aj_P} also follows from \cite[Theorem b]{CM24}.
	
	We compute the lowest codimension nontrivial term of the Fourier transform up to codimension $g$, which will be the main input to prove Theorem \ref{thm:pushforward}.
	For a class $x \in \CH^*(\Jzero_{g,n})$, $[x]_{\codim=c}$ denotes the codimension $c$ part.
	\begin{theorem}\label{thm:leading}
		For a nondegenerate stability condition $\epsilon$ on $\Mbar_{g,n}$, we consider the partial Fourier transform $\fm_g^\circ:=\ch(\P^\vee) :\CH^*(\Jbar_{g,n}^\epsilon)\to\CH^*(\Jzero_{g,n})$. Let $b\in\ZZ$ and $\sfa\in\ZZ^n$ with $\sum_{i=1}^na_i=(2g-2+n)b$.
		\begin{enumerate}
			\item[(a)]  If $c<g$, then $\big[\fm_g^\circ\big(\exp(-b\kappa_{0,1}+\sum_{i=1}^n a_i\xi_i)\big)\big]_{\codim=c} = 0$.
			\item[(b)] If $c=g$, then $\big[\fm_g^\circ\big(\exp(-b\kappa_{0,1}+\sum_{i=1}^n a_i\xi_i)\big)\big]_{\codim=g} =  (-1)^g\cdot\widetilde{\uniDR}^{g}_g(b;\sfa)|_{\Jzero_{g,n}}$.
		\end{enumerate}
	\end{theorem}
	\begin{proof}
		First, we collect the relevant results from previous sections. Let $\epsilon_0$ be a small stability condition and let $\fm_g:\CH^*(\Jbar_{g,n}^{\epsilon_0})\to\CH^*(\Jbar_{g,n}^\epsilon)$ be the Fourier transform. By Proposition \ref{pro:FM_aj}, the Fourier transform of the resolved Abel-Jacobi section has the form
		\[\fm_g\big([\aj_{b;\sfa}]\big) = \exp(-b\kappa_{0,1}+\sum_{i=1}^n a_i\xi_i)(1+\gamma_{b;\sfa})\,,\]
		where $\gamma_{b;\sfa}\in\mathsf{PP}(\Jbar_{g,n}^\epsilon)$ is a class supported away from the integral locus. By taking the inverse Fourier transform $\fm_g^{-1}$ and restricting to $\Jzero_{g,n}$, we obtain 
		\begin{equation}\label{eq:fm_1}
			[\aj_{b;\sfa}]  = \fm^{-1}_g\Big(\exp(-b\kappa_{0,1}+\sum_{i=1}^n a_i\xi_i)(1+\gamma_{b;\sfa})\Big)\,.
		\end{equation}
		Using Propositions \ref{pro:Todd}, \ref{pro:F0} and \ref{pro:fm_max}, the right-hand side of \eqref{eq:fm_1} can be expressed as
		\begin{align}
			& (-1)^g\cdot\fm_g^\circ\Big(\exp(-b\kappa_{0,1}+\sum_i a_i\xi_i)(1+\gamma_{b;\sfa})\Big) \nonumber \\
			&\hspace{5mm} + (-1)^g\cdot\sum_{h=1}^g\frac{(-1)^h}{h!2^h}j_{h*}q_h^* \fm_{g-h}^\circ\Big(\exp(-b\kappa_{0,1}+\sum_{i=1}^na_i\xi_i)(1+\gamma_{b;\sfa})\prod_{i=1}^h \sum_{k_i=1}^\infty\frac{(-1)^{k_i}B_{2k_i}}{(2k_i)!}\frac{\alpha_i^{2k_i-1}+\beta_i^{2k_i-1}}{\alpha_i+\beta_i}\Big)\label{eq:fm_2}\,.
		\end{align}
		Below, we use \eqref{eq:fm_2} to match the recursive structure of $\uniDR^g_g(b;\sfa)$ (Theorem \ref{thm:top}) and the recursive structure of Fourier transform (Theorem \ref{thm:inductive}).
		
		(a) We prove the statement by double induction on the genus $g$ and the codimension $c$. For the base case $g = 0$, there is nothing to prove. Now assume the statement holds for all genera less than $g$, and proceed by (finite) induction on the codimension $c$. Let $c < g$, and assume that the statement holds for all codimensions less than $c$.
		
		Taking the codimension $c$ part of \eqref{eq:fm_1}, the left hand side is zero. On the right hand side, we apply Theorem~\ref{thm:inductive} to simplify the term $\gamma_{b;\sfa}$ as follows. Write $\gamma_{b;\sfa}$ as a linear combination of tautological classes of the form $[\widetilde{\Gamma}_{\widetilde{\delta}}, \gamma_0]$. As in the proof of Theorem~\ref{thm:inductive}, we express $\gamma_{b;\sfa}$ as the pushforward of tautological classes from compactified Jacobians of the form $\Jbar_{\Gamma_\delta} \to \Mbar_{\Gamma}$ (see \eqref{eq:inductive}). By Proposition~\ref{pro:FM_aj}, the class $\gamma_{b;\sfa}$ is supported away from the integral locus. Therefore, the stratum $\Mbar_\Gamma \to \Mbar_{g-h, n+2h}$ is nontrivial, and the Chern classes of the normal bundle can be written in terms of $\psi$- and $\xi$-classes at the additional markings on $\Jbar_{\Gamma_\delta}$. Hence, we have
		\begin{equation}\label{eq:leading_gamma}
			\Big[\fm_g^\circ\Big(\exp(-b\kappa_{0,1}+\sum_{i=1}^n a_i\xi_i)\cdot [\widetilde{\Gamma}_{\widetilde{\delta}},\gamma_0]\Big)\Big]_{\codim=c}  = j_{h*}q_h^*\Big[\fm^\circ_{g-h}\Big(\exp(-b\kappa_{0,1}+\sum_{i=1}^n a_i\xi_i)(j_{\Gamma_\delta})_*(\widetilde{\gamma}_0)\Big)\Big]_{\codim=c-h}
		\end{equation}
		for some $h\leq g$ and some $\widetilde{\gamma}_0\in \CH^*(\Jbar_{\Gamma_\delta})$.
		
		We show that \eqref{eq:leading_gamma} vanishes when $\Gamma$ is nontrivial. The class $\widetilde{\gamma}_0$ can be expressed as a polynomial in $\psi$-classes at half-edges and $\xi$-classes only at the markings. For monomials in $\psi$-classes, we apply \eqref{eq:st} to rewrite them as $\psi$-monomials pulled back from $\Mbar_{\Gamma}$, plus additional terms supported on boundary strata of $\Jbar_{\Gamma_\delta}$ corresponding to unstable vertices. By the argument above, the contribution of these additional terms can be written as Fourier transforms of lower genus. Hence, by the induction hypothesis, their contribution to \eqref{eq:leading_gamma} vanishes. 
        
        For the leading term, we write $\widetilde{\gamma_0} = \gamma_0' \cdot \pi^*(\gamma_0'')$, where $\gamma_0'$ is a monomial in $\xi$-classes and $\gamma_0''$ is a monomial in $\psi$-classes on $\Mbar_\Gamma$. By Proposition~\ref{pro:tree} in the case $h^1(\Gamma)=0$, and by Corollary~\ref{cor:fm_st} when $h^1(\Gamma)>0$, together with the projection formula, we obtain
        \[
        \fm^\circ_{g-h}\Big(\exp(-b\kappa_{0,1}+\sum_{i=1}^n a_i\xi_i)(j_{\Gamma_\delta})_*(\widetilde{\gamma}_0)\Big)=(r_{\Gamma})_*q^*\Big(\gamma_0''\cdot \prod_{v\in V(\Gamma)}\fm_{g(v)}^\circ(\Xi_v \cdot\alpha_v)\Big)
        \]
        where $\Xi_v$ is some monomials in $\xi$ and $\kappa_{0,1}$ classes and $\alpha_v$ is some class in $\mathsf{PP}(\Jbar^{\epsilon(v)}_{g(v),n(v)})$. 
        
        If $\alpha_v$ is nontrivial, then by the induction hypothesis the associated contribution to \eqref{eq:leading_gamma} vanishes. If instead $\alpha_v=1$ for all $v\in V(\Gamma)$, the class inside $\fm_{g(v)}^\circ$ remains a polynomial in $\xi$- and $\kappa_{0,1}$-classes. In this case, the codimension $c-h-\deg(\gamma_0'')-|E(\Gamma)|$ part of $\prod_{v\in V(\Gamma)} \fm_{g(v)}^\circ(\Xi_v)$ contributes to the codimension $c-h$ part of \eqref{eq:leading_gamma}. Since $\Gamma$ is nontrivial, we have $|E(\Gamma)|>0$, and thus the codimension $c-h$ part vanishes by the induction hypothesis. Therefore, \eqref{eq:leading_gamma} vanishes.
		
		By the above argument, the codimension $c$ part of the right hand side of \eqref{eq:fm_1} reduces to the codimension $c$ part of \eqref{eq:fm_2}, with the following simplification: 
		\begin{align*}
			& \Big[(-1)^g\fm_g^\circ\Big(\exp(-b\kappa_{0,1}+\sum_{i=1}^n a_i\xi_i)\Big)\Big]_{\codim=c}  \\
			& + \sum_{h=1}^g\frac{1}{h!2^h}j_{h*}q_h^* \Big[(-1)^{g-h}\fm_{g-h}^\circ\Big(\exp(-b\kappa_{0,1}+\sum_{i=1}^na_i\xi_i)(1+\gamma_{b;\sfa})\prod_{i=1}^h \sum_{k_i=1}^\infty\frac{(-1)^{k_i}B_{2k_i}}{(2k_i)!}\frac{\alpha_i^{2k_i-1}+\beta_i^{2k_i-1}}{\alpha_i+\beta_i}\Big)\Big]_{\codim={c-h}}\,.
		\end{align*}
		By Proposition \ref{pro:Todd} and \ref{pro:FM_aj}, $\gamma_{b;\sfa}$ and $\alpha_i,\beta_i$ can be expressed as polynomials in the  $\psi_i$ and $\xi_i$ classes. Since $\fm^\circ_{g-h}$ is linear on the base, we can factor out the $\psi$-classes. By the induction hypothesis, the second term vanishes. Thus, for $c<g$, we conclude:
		\[ [\fm_g^\circ(\exp(-b\kappa_{0,1}+\sum_{i=1}^n a_i\xi_i))]_{\codim =c} = 0\,, \]
		which completes the proof of part (a).
		
		(b) We prove part (b) by induction on the genus. Assume the statement holds for all genera less than $g$. Taking the codimension $g$ part of \eqref{eq:fm_1}, Proposition \ref{pro:aj_P} implies:
		\begin{equation}\label{eq:fm_3}
			[\aj_{b;\sfa}]|_{\Jzero_{g,n}} = \uniDR^g_g(b;\sfa)|_{\Jzero_{g,n}}\,.
		\end{equation}
		
		By the argument in part~(a), the contribution involving the class $\gamma_{b;\sfa}$ lies in codimension  greater than $g$. Therefore, by the vanishing result of part (a), the codimension $g$ part of the right hand side of equation~\eqref{eq:fm_2} simplifies to: 
		\begin{align*}
			& \Big[(-1)^g\fm_g^\circ\Big(\exp(-b\kappa_{0,1}+\sum_{i=1}^n a_i\xi_i)\Big)\Big]_{\codim=g}  \\
			&+ \sum_{h=1}^g\frac{1}{h!2^h}j_{h*}q_h^* \Big[(-1)^{g-h}\fm_{g-h}^\circ\Big(\exp(-b\kappa_{0,1}+\sum_{i=1}^na_i\xi_i)\prod_{i=1}^h \sum_{k_i=1}^\infty\frac{(-1)^{k_i}B_{2k_i}}{(2k_i)!}(2k_i-1)(\xi_i-\xi_i')^{2k_i-2}\Big)\Big]_{\codim={g-h}}\,.
		\end{align*}
		By the induction hypothesis, each lower-genus term becomes:
		\begin{align*}
			&\Big[(-1)^{g-h}\fm_{g-h}^\circ\Big(\exp(-b\kappa_{0,1}+\sum_{i=1}^na_i\xi_i)\prod_{i=1}^h \sum_{k_i=1}^\infty\frac{(-1)^{k_i}B_{2k_i}}{(2k_i)!}(2k_i-1)(\xi_i-\xi_i')^{2k_i-2}\Big)\Big]_{\codim={g-h}}\\
			& \hspace{5mm} = \sum_{k_1,\ldots, k_h\in\ZZ}(\prod_{i=1}^h k_i)\widetilde{\uniDR}^{g-h}_{g-h,n+2h}(b;\sfa, k_1,-k_1,\ldots,k_h,-k_h)\,.
		\end{align*}
		Here, the infinite sum is interpreted using the $\zeta$-function renormalization defined in \eqref{eq:zeta}.
		Finally, by Theorem \ref{thm:top}, the codimension $g$ part of the leading term of \eqref{eq:fm_2} becomes:
		\[\big[(-1)^g\cdot\fm_g^\circ(\exp(-b\kappa_{0,1}+\sum_{i=1}^n a_i\xi_i))\big]_{\codim=g} = \widetilde{\uniDR}^g_g(b;\sfa)\]
		which completes the proof of (b).
	\end{proof}
	\subsection{Proof of Theorem \ref{thm:pushforward}}\label{sec:pf}
	We finally finish the proof of our main theorem.
	\begin{proof}[Proof of Theorem \ref{thm:pushforward}]
		We finally finish the proof of our main theorem. We will induct on $\ell$, the exponent of $\Theta$. We save the base case $\ell = 0$ for the end, so first assume $\ell > 0$. We are given $\epsilon$, a nondegenerate stability condition for $\Mbar_{g,n}$. Let $p:\Mbar_{g,n+1}\to\Mbar_{g,n}$ be the morphism forgetting the last marking. The pullback of $\epsilon$ to $\Mbar_{g,n+1}$ is degenerate. However, there exists a small perturbation $\epsilon'$ which is a nondegenerate stability condition for $\Mbar_{g,n+1}$ such that there is a morphism $q:\Jbar_{g,n+1}^{\epsilon'}\to\Jbar_{g,n}^\epsilon$ and the following commutative diagram holds:
		\[
		\begin{tikzcd}
			\Jbar_{g,n+1}^{\epsilon'}\ar[r,"q"]\ar[d,"\pi_1"] & \Jbar_{g,n}^\epsilon\ar[d,"\pi"]\\
			\Mbar_{g,n+1}\ar[r,"p"] & \Mbar_{g,n}\,.
		\end{tikzcd}
		\]
		
		Consider the class (on $\Mbar_{g,n+1}$)
		\[
		\mathcal{P}^c(b;\sfa,a_{n+1}) := (\pi_1)_*\left[\Theta^{l-1}\exp\left(-b\kappa_{0,1}+\sum_{i=2}^{n+1}a_i\xi_i\right)\right]_{\text{codim } g+c},
		\]
		viewed as a polynomial in $b,a_2,\ldots,a_{n+1}$.
		
		By the inductive hypothesis, $\mathcal{P}^c(b;\sfa,a_{n+1}) = 0$ for $c < g-\ell+1$, is tautological for all $c$, and for $c=g-\ell+1$ we have the formula
		\begin{equation}\label{eq:induct}
			\mathcal{P}^{g-\ell+1}(b;\sfa,a_{n+1}) = (-1)^{g-\ell+1}(\ell-1)!\cdot\DR_g^{g-\ell+1}(b,\sfa,a_{n+1}).
		\end{equation}
		Now consider applying $p_*$. Applying $p_*\circ (\pi_1)_* = \pi_*\circ q_*$, the projection formula for $q$, as well as the identity $p^*\kappa_{0,1} = \kappa_{0,1}-\xi_{n+1}$, we get
		\[
		p_* \mathcal{P}^c(b;\sfa,a_{n+1}) = \pi_*\left[\Theta^{l-1}\exp\left(-b\kappa_{0,1}+\sum_{i=2}^{n}a_i\xi_i\right)\cdot p_*\exp((a_{n+1}-b)\xi_{n+1})\right]_{\text{codim }g+c-1}.
		\]
		Since $p_*(\xi_{n+1}^i) = 0$ for $i < 2$, this pushforward is divisible by $(a_{n+1}-b)^2$. Moreover, if we divide by $(a_{n+1}-b)^2$ and then set $a_{n+1} := b$, the formula $p_*(\xi_{n+1}^2/2) = -\Theta$ yields
		\[
		\left[\frac{p_* \mathcal{P}^c(b;\sfa,a_{n+1})}{(a_{n+1}-b)^2}\right]_{a_{n+1} := b} = -\pi_*\left[\Theta^{l}\exp\left(-b\kappa_{0,1}+\sum_{i=2}^{n}a_i\xi_i\right)\right]_{\text{codim }g+c-1}.
		\]
		This immediately gives the induction step for parts (a) and (c) of the theorem. For part (b), we also apply Theorem~\ref{thm:DR_push} to see that the steps we have taken (apply $p_*$, divide by $(a_{n+1}-b)^2$, set $a_{n+1} := b$) indeed act on the right side of \eqref{eq:induct} by replacing $l-1$ with $l$ and $n+1$ with $n$.
		
		For the remainder of the proof we assume $\ell = 0$. We prove (a) and (b) by induction on the genus. When $g=0$, there is nothing to prove. Consider the map 
		\[\fm_g^\circ:\CH^*(\Jbar_{g,n}^\epsilon)\to\CH^*(\Jzero_{g,n})\]
		as defined in \eqref{eq:fm_circ}. Let $\sfa\in\ZZ^n$ with $\sum_i a_i=0$. By Proposition \ref{pro:IntFM}, it is sufficient to prove that
		\[e^*\Big[\fm^\circ_g(\exp \big(\sum_i a_i\xi_i)\big)\Big]_{\codim=g} = (-1)^g e^*\widetilde{\uniDR}^g_g(\sfa)=(-1)^g\widetilde{\DR}_g^g(\sfa)\,,\]
		which is proven in Theorem \ref{thm:leading}. This proves (a) and (b).
		
		We now prove (c). By Proposition \ref{pro:IntFM} and the above argument, it suffices to show that 
		\begin{equation}\label{eq:fm_taut}
			\Big[\fm^\circ_g(\exp \big(-b\kappa_{0,1}+\sum_i a_i\xi_i)\big)\Big]_{\codim=c}\in\R^c(\Jzero_{g,n})
		\end{equation}
		for all codimension $c$, where $\R^*(\Jzero_{g,n})$ is the restriction of tautological ring \eqref{eq:taut_jbar} to $\Jzero_{g,n}$.
		
		We proceed by induction, first on the genus and then on the codimension. For the base case $g=0$, there is nothing to prove. Assume the statement holds for all genera less than $g$. For codimenions less than or equal to $g$,  \eqref{eq:fm_taut} was already established in (a) and (b). Now assume \eqref{eq:fm_taut} holds in codimension less than $c$, where taking the codimension $c$ part of \eqref{eq:fm_1}, the left-hand side is zero. The right-hand side reduces to 
		\begin{align*}
			& \Big[(-1)^g\fm_g^\circ\Big(\exp(-b\kappa_{0,1}+\sum_i a_i\xi_i)\Big)\Big]_{\codim=c} + \Big[(-1)^g\fm_g^\circ\Big(\exp(-b\kappa_{0,1}+\sum_i a_i\xi_i)\gamma_{b;\sfa}\Big)\Big]_{\codim=c}\\
			&\hspace{10mm} + \sum_{h=1}^g\frac{1}{h!2^h}j_{h*}q_h^* \Big[(-1)^{g-h}\fm_{g-h}^\circ\Big(\exp(-b\kappa_{0,1}+\sum_{i=1}a_i\xi_i)(1+\gamma_{b;\sfa})\\
			&\hspace{40mm} \cdot\prod_{i=1}^h \sum_{k_i=1}^\infty\frac{(-1)^{k_i}B_{2k_i}}{(2k_i)!}\frac{\alpha_i^{2k_i-1}+\beta_i^{2k_i-1}}{\alpha_i+\beta_i}\Big)\Big]_{\codim={c-h}}\,.
		\end{align*}
		By the induction hypothesis, the second term and the third term lie in the tautological subring. Therefore \eqref{eq:fm_taut} holds in codimension $c$, completing the proof of part (c).
	\end{proof}
	For the pushforward of monomials of divisors with $2\ell+m+\sum_{i=1}^nk_i\leq 2g$, an explicit formula of the class $\gamma_{b;\sfa}$ does not play a role. When $2\ell +m +\sum_{i=2}^nk_i>2g$, the pushforward will depend on the stability condition -- this can be easily checked over the locus $\Mct_{g,n}\subset\Mbar_{g,n}$ of compact type curves. 
	\begin{remark}
		Let $q:\Jbar_{g,n+1}^{\epsilon'}\to\Jbar^\epsilon_{g,n}$ be the projection used in the proof of Theorem \ref{thm:pushforward}. By Lemma \ref{lem:unidr_A} the same pushforward formula as in Theorem \ref{thm:pushforward_uniDR} holds for $q_*$. Therefore, a similar argument to that used in the proof of Theorem \ref{thm:pushforward} (c)  generalizes Theorem \ref{thm:leading} for monomials with $\Theta$ divisors.
	\end{remark}
	\subsection{Connection to the perverse filtration}\label{sec:vanish}
	Let $\pi: M\to B$ be a proper morphism from a smooth variety $M$. Assume that $\pi$ has equidimensional fibers whose dimension is $g$. Let ${}^{p}\tau_{\le k} (-)$ be the $k$-th perverse truncation functor. By the Decomposition Theorem \cite{BBDG}, the natural map ${}^{\p}\tau_{\leq k  + \dim B}(R\pi_*\QQ) \to R\pi_*\QQ$  induces the canonical increasing {\em perverse filtration}
	\[P_0H^*(M,\QQ)\subset P_1H^*(M,\QQ)\subset \cdots\subset P_{2g}H^*(M,\QQ)=H^*(M,\QQ)\,.\]
	\begin{lemma}\label{lem:low_filt}
		Let $\pi: M\to B$ be a proper morphism with equidimensional fibers of dimension $g$, where $M$ and $B$ are smooth. For a class $\alpha\in P_{2g-1}H^*(M,\QQ)$, we have $\pi_*(\alpha)=0$.
	\end{lemma}
	\begin{proof}
        By the adjunction $\pi_! \dashv \pi^!$ and the purity $\pi^!(-) \cong \pi^*(-)[2g]$, we obtain a trace map 
        \[
        \mathrm{Tr} : R\pi_*\QQ \to \QQ[-2g]\,,
        \]
        which is a sheaf theoretic analogue of the Gysin pushforward in cohomology. For $k<2g$, it suffices to show that the composition
        \[
            {}^{\p}\tau_{\le k + \dim B}(R\pi_*\QQ) \to R\pi_*\QQ \xrightarrow{\mathrm{Tr}} \QQ[-2g]
        \]
        is zero. Since $B$ is smooth, we have $\QQ[-2g]\in {}^{\p}D^{b,\geq 2g+\dim B}_{c}(B,\QQ)$. Therefore the composition is zero by the perverse $t$-structure.
	\end{proof}
	We give a  proof of Theorem \ref{thm:pushforward} (a) in cohomology following the idea of Maulik-Shen \cite{MS_pw}.
	\begin{proposition}\label{pro:vanishing}
		Theorem \ref{thm:pushforward} (a) holds in $H^*(\Mbar_{g,n},\QQ)$.
	\end{proposition}
	\begin{proof}
		Let $\pi:\Jbar_{g,n}^\epsilon\to\Mbar_{g,n}$ denote the projection. We use the notion of {\em strong perversity}  as defined in \cite{MS_pw} and show that $\xi_i$ and $\kappa_{0,1}$ have strong perversity 1.  By \cite{MSV}, building upon \cite{Ngo}, the complex $R\pi_*\QQ$ has the full support property, meaning that any non-trivial simple perverse summand has support equal to $\Mbar_{g,n}$. Since $\xi_i,\kappa_{0,1}\in H^2(\Jbar_{g,n}^\epsilon,\QQ)$, it suffices to verify strong perversity being $1$ over the open locus $\mathcal{M}_{g,n}\subset\Mbar_{g,n}$. Over $\mathcal{M}_{g,n}$, the morphism $\pi$ is a torsor under the abelian variety $\Jzero_{g,n}$. Since $n\geq 1$, we can identify it with $\Jzero_{g,n}$ via twisting by a section. Under the twisting, $\xi_i$ (resp. $\kappa_{0,1}$) maps to $\xi_i$ (resp. $\kappa_{0,1}$) modulo a $\psi$-class. Since $\psi$-classes are pulled back from $\mathcal{M}_{g,n}$, they have perversity zero, so the twisting does not change the strong perversity. By \cite[Section 3]{DM91}, $\xi_k$ and $\kappa_{0,1}$ are weight $1$ classes, so they have strong perversity $1$ over $\mathcal{M}_{g,n}$. Therefore, $\xi_k$ and $\kappa_{0,1}$ are strong perversity $1$ classes over $\Mbar_{g,n}$. 
		
		The $\Theta$ divisor, by codimension considerations, has strong perversity $2$. By \cite[Lemma 1.3]{MS_pw}, the strong perversity is multiplicative and hence the class $\Theta^\ell(\kappa_{0,1})^m\prod \xi_i^{k_i}$ has strong perversity $2\ell+m+\sum k_i$. Applying \cite[Lemma 1.2]{MS_pw}, we conclude that the monomial lies in $(2\ell+m+\sum k_i)$-th perverse filtration of $H^*(\Jbar_{g,n},\QQ)$.
		Since $2\ell+m+\sum_{i=1}^n k_i<2g$, the pushforward to $H^*(\Mbar_{g,n},\QQ)$ vanishes by Lemma \ref{lem:low_filt}.
	\end{proof}
    \begin{remark}
        In \cite{BMSY2}, the first author together with D.~Maulik, J.~Shen, and Q.~Yin showed that the graded ring structure on the cohomology group $H^*(\Jbar_{g,n}^\epsilon,\QQ)$ depends on the stability condition $\epsilon$. On the other hand, the bigraded $H^*(\Mbar_{g,n},\QQ)$-algebra structure on the associated graded with respect to the perverse filtration,
        \[ \mathbb{H}_{g,n}^\epsilon :=\Big(\bigoplus_{k=0}^{2g}\mathrm{Gr}^P_k H^*(\Jbar_{g,n}^\epsilon,\QQ), \, \overline{\cup}\Big)\]
        is independent of the non-degenerate stability condition.

        This ``intrinsic cohomology structure" on $\Jbar_{g,n}^\epsilon$ naturally leads to study the structure of intersection numbers on the universal compactified Jacobians for all $(g,n)$. Let
        \[
        \pi_* : \mathbb{H}^\epsilon_{g,n} \to H^*(\Mbar_{g,n},\QQ)
        \]
        denote the proper pushforward. By Theorem \ref{thm:pushforward}, we obtain the formula \eqref{eq:intrinsic} in the introduction.
        
        The generating series of the intersection numbers arising from double ramification formulae with $\psi$-monomials satisfy a noncommutative Korteweg-de Vries (KdV) hierarchy, as shown by Buryak and Rossi \cite{BR23}. Since the top degree part of $\DR$ formula recovers the full $\DR$ formula by Theorem \ref{thm:topdeg_correspondence}, the identity \eqref{eq:intrinsic} suggests that the generating series of intersection numbers in the intrinsic cohomology ring underlies an integrable system. A connection to a conjecture of Goulden, Jackson, and Vakil \cite{GJV} is an interesting question. 
    \end{remark}
	\section{Weight decomposition and Fourier transforms}\label{sec:9}
	
	\subsection{Weight decomposition}\label{sec:weight}
	We discuss the generalization of Beauville \cite{Beauville86} and Deninger-Murre \cite{DM91}'s weight decomposition for semi-abelian schemes. Any semi-abelian variety $G$ can be written as an extension
	\[0\to T \to G\to A\to 0\]
	of an abelian variety $A$ by a split torus $T$. For a semi-abelian scheme $\pi:G\to B$,  the {\em Kimura dimension} of $\pi$ is defined by 
	\begin{equation}\label{eq:kd}
		\kd (\pi) := \max \{2\dim (A_b) + \dim (T_b) : b\in B\}\,.
	\end{equation}
	For $N\in \ZZ$, consider the ``multiplication by $N$ map" 
	\begin{equation}\label{eq:Nmap2}
		[N]: G\to G\,.
	\end{equation}
	The weight $w$ piece is defined by
	\begin{equation}\label{eq:wt_dec}
		\CH_{(w)}^*(G) := \{\alpha\in\CH^*(G,\QQ) : [N]^*\alpha= N^w\alpha,\text{ for all } N\in\ZZ\}\,.
	\end{equation}
	
	The multiplicative splitting is proven by Ancona-Huber-Pepin Lehalleur \cite[Theorem 4.9]{AHL16}.
	\begin{theorem}[\cite{AHL16}]\label{thm:wtdec}
		The rational Chow group of $G$ has the multiplicative decomposition
		\begin{equation*}
			\CH^*(\Jzero_C,\QQ) = \bigoplus_{w=0}^{\kd(\pi)}\CH^*_{(w)}(\Jzero_C)\,.
		\end{equation*}
		In particular, if the weight of a class is greater than the Kimura dimension $\kd(\pi)$, it vanishes. 
	\end{theorem}
	\subsection{Leading term of partial Fourier transform}
	In Section \ref{sec:7}, we studied the connection between $\fm^\circ$ \eqref{eq:fm_circ} and $\fm^{-1}$ both from $\CH^*(\Jbar_C^{\,\epsilon})$ to $\CH^*(\Jzero_C)$. Here, we compute the leading term of $\fm^\circ(1)$ using Theorem \ref{thm:wtdec}. $\fm^\circ$ has a nice compatibility property with the ``multiplication by $N$" map \eqref{eq:Nmap2}.
	\begin{lemma}\label{lem:wt1}
		Let $x\in \CH^p(\Jbar_C,\QQ)$. If we write $\fm^\circ(x)=\sum_q y_q$ with $y_q\in\CH^q(\Jzero_C,\QQ)$, then 
		\[[N]^*y_q = N^{g+q-p}y_q,\, N\in\NN\,.\]
		
	\end{lemma}
	\begin{proof}
		By \eqref{eq:Deligne}, we have 
		\begin{equation}\label{eq:PandN}
			(1\times [N])^*c_1(\P) = N\cdot c_1(\P)\,.
		\end{equation}
		Then the result follows from the projection formula. See \cite{Beauville86} and \cite[Proposition 2.16]{DM91}.
		%Let $\pi_1:\Jbar_C\times_B\Jzero_C\to\Jbar_C$ and $\pi_2:\Jbar_C\times_B\Jzero_C\to\Jzero_C$ be two projections. By codimension constraint, we have
		%\begin{align*}
		%    [N]^*y_q &=[N]^*\pi_{2*}\big(\pi_1^*(x)\cup\frac{c_1(\P)^{g+q-p}}{(g+q-p)!}\big)\\
		%    & = \pi_{2*}\big(\pi_1^*(x)\cup\frac{(1\times [N])^*c_1(\P)^{g+q-p}}{(g+q-p)!}\big)\\
		%    &= N^{g+q-p}y_q\,,
		%\end{align*}
		%where the third equality follows from \eqref{eq:PandN}.
	\end{proof}
	We present the main calculation of this section.
	\begin{proposition}\label{pro:fm1}
		Let $B=\Mbar_{g,n}$ and $\pi:\Jzero_C\to B$ be the relative Jacobian. Then  we have
		\[[\fm_g^\circ(1)]_{\codim=g} = \frac{\Th^g}{g!} \in\CH^g(\Jzero_g,\QQ)\,.\]
	\end{proposition}
	\begin{proof}
		We first compute $\fm_g^\circ(\exp(-\Th))$. By Lemma \ref{lem:c1} and the projection formula, we have
		\begin{equation}\label{eq:fm_mumford}
			\fm_g^\circ(\exp (-\Th)) =\exp(\Th)\cdot p_{2*}\mu^*\exp(-\Th)\,.
		\end{equation}
		Let $\partial B:=\Mbar_{g-1,n+2}$ and $j: \partial B \to B$ be the gluing map. We also denote $j:\Jzero_C|_{\partial B}\to \Jzero_C$ the associated finite morphism. If $w>2g$, we have $\CH^*_{(w)}(\Jzero_C)=0$ by Theorem \ref{thm:wtdec}. For $\partial\Jbar_C:=\Jbar_C\setminus\Jzero_C$, we consider the excision sequence
		\[\CH^*(\partial\Jbar_C,\QQ)\xrightarrow{j_*} \CH^*(\Jbar_C,\QQ)\to\CH^*(\Jzero_C,\QQ)\to 0\,.\]
		Since $\Theta$ has weight $2$, if $c>g$, then $\Th^c$ restricted to $\Jzero_C$ vanishes. By the above excision sequence, $\Theta^c$ is supported on $\partial\Jbar_C$. The pullback of \eqref{eq:action} to $\partial\Jbar_C$ yields
		\[\mu:\partial\Jbar_C\times_{\partial B}\Jzero_C|_{\partial B} \to \partial\Jbar_C\,.\]
		If $c>g$, using the projection formula, the class  $p_{2*}\mu^*(\Th^c)=0$ lies in the image of $j_*$. Therefore, \eqref{eq:fm_mumford} can be written by
		\begin{equation}\label{eq:fm_mumford2}
			\fm^\circ_g(-\exp(\Th))=\exp(\Th)\cdot(1+ \im j_*)\,.
		\end{equation}
		
		By Theorem \ref{thm:wtdec}, one can take the pure weight pieces of \eqref{eq:fm_mumford}. Taking the codimension $g$ part of \eqref{eq:fm_mumford2}, we get 
		\[\sum_{k=0}^\infty \Big[\fm_g^\circ(\frac{(-\Th)^k}{k!})\Big]_{\codim =g} = \frac{\Th^g}{g!} + \big[\im j_*\big]_{\codim=g}\,.\]
		By Lemma \ref{lem:wt1}, we have
		\[\big[\fm_g^\circ(\Theta^k)\big]_{\codim=g}\in\CH^g_{(2g-k)}(\Jzero_C)\,.\]
		The weight $2g$ part of the left hand side is $[\fm_g^\circ(1)]_{\codim=g}$. Since the Kimura dimension \eqref{eq:kd} of the boundary $\Jzero_C|_{\partial B}$ drops to $2g-1$, the weight $2g$ piece of $\im j_*$ vanishes by Theorem \ref{thm:wtdec}. Therefore the weight $2g$ part of the right hand side is $\frac{\Theta^g}{g!}$ and we get the equality.
		%is a $\Gm^\ell$-torsor over $\Jzero_{g-\ell}$ for some $\ell>0$. Since $j_*[N]^*=[N]^*j_*$, the pushforward $j_*$ preserves weights. This implies that the image of $j_*:\CH^*_{(2g)}(\Jzero_g|_{\partial B})\to\CH^*_{(2g)}(\Jzero_g)$ is zero. Therefore the weight $2g$ part of the right hand side is $\frac{\Theta^g}{g!}$. Therefore we get the equality.
	\end{proof}
	\subsection{Universal double ramification cycle formula over the treelike locus}\label{sec:treelike}
	Let $\Mint_{g,n}\subset\Mbar_{g,n}$ be the locus of {\em treelike} curves--these are stable curves
	whose graph is a tree with any number of self-loops attached, which is an open substack of $\Mbar_{g,n}$. It can be uniquely characterized as the largest open substack of $\Mbar_{g,n}$ where every Abel-Jacobi section extends \cite{HPS}. We give a  proof of the universal double ramification formula (Proposition \ref{pro:aj_P}) on $\Jzero_{g,n}\to\Mint_{g,n}$, independent from \cite{BHPSS}. We use compatibility of Fourier transform and the group structure of $\Jzero_{g,n}$.
	
	Over $\Mint_{g,n}$, the compactified Jacobians are isomorphic to the relative moduli space of multidegree zero, rank 1 torsion free sheaves, and we denote the compactified Jacobian of degree $0$ by $\pi:\Jbar^\text{\,tl}_{g,n}\to\Mint_{g,n}$. For $b\in\ZZ$ and $\sfa\in\ZZ^n$ with $\sum_{i=1}^n a_i=(2g-2+n)b$, the twisted $\Theta$-divisor is defined by
	\begin{equation}\label{eq:twist_theta}
		\Theta_{b;\sfa} := \Theta + b\kappa_{1} + \frac{b^2}{2}\kappa_{1} + \frac{1}{2}\sum_{i=1}^n a_i^2\psi_i + \sum_{i=1}^n a_i\xi_i + \pi^*D_{b;\sfa}\,,
	\end{equation}
	Here $D_{b;\sfa}$ is a divisor class given by
	\[D_{b;\sfa}=\sum_\Gamma c(\Gamma,b;\sfa)[\Gamma]\,,\]
	where the sum is over all stable graphs $\Gamma$ with exactly two vertices connected
	by a single edge described by partition $(g_1, I_1|g_2, I_2)$ and $c(\Gamma,b;\sfa)=-(\sum_{i\in I_1} a_i-b(2g_1-2+|I_1|+1))^2$. We refer \cite[Section 4.1]{BHPSS} for the further explanations.
	
	The combinatorics of the universal double ramification formula \eqref{eq:unidr} gets simplified over $\Mint_{g,n}$. For $h>0$, let $j_h:\Mint_{g-\ell,n;2h}\to\Mint_{g,n}$ be the restriction of the gluing map to $\Mint_{g,n}$. In general, $\Mint_{g-h,n;2h}$ is an open substack of $\Mint_{g-h,n+2h}$.
	\begin{lemma}\label{lem:unidr_int}
		Let $\Jzero_{g,n}\to\Mint_{g,n}$ be the relative Jacobian. For $b\in\ZZ$ and $\sfa\in\ZZ^n$ with $\sum_{i=1}^n a_i=(2g-2+n)b$, let $\mathsf{P}^c_g(b;\sfa)$ be the universal double ramification formula \eqref{eq:unidr} restricted to $\Jzero_{g,n}$.
		\begin{enumerate}[label=(\alph*)]
			\item For all $c$, we have
			\begin{equation*}
				\mathsf{P}^c_{g}(b;\sfa) = \left[\exp(\Theta_{b;\sfa})\cdot\sum_{h=0}^g\frac{(-1)^h}{2^hh!}(j_{h})_*\prod_{i=1}^h\left(\sum_{k_i=1}^\infty \frac{B_{2k_i}}{k_i!}\frac{(\psi_{h_i}+\psi_{h_i'})^{k_i-1}}{2^{k_i}}\right)\right]_{\text{codim}=c}\in \CH^c(\Jzero_{g,n})\,.
			\end{equation*}
			%where $j_h:\Mint_{g-h,n;2h}\to\Mint_{g,n}$ is the gluing morphism.
			\item For all $c,$ we have $\widetilde{\mathsf{P}}^c
			_{g}(b;\sfa) = \frac{(\Theta_{b;\sfa})^c}{c!}$.    
		\end{enumerate}
	\end{lemma}
	\begin{proof}
		We apply \cite[Proposition 4.1]{BHPSS}. The divisor inside the exponential in \cite[eq (4.2)]{BHPSS} coincides with \eqref{eq:twist_theta}. Both (a) and (b) follows from Faulhaber's formula applied to the contribution from the $r$-weights.
	\end{proof}
	We replace Theorem \ref{thm:leading} by the below:
	\begin{proposition}\label{pro:Tdcirc}
		For any $b\in\ZZ$ and $\sfa\in\ZZ^n$ with $\sum_{i=1}^na_i=(2g-2+n)b$, suppose that we have $\fm^{-1}_h(\exp(-b\kappa_{0,1}+\sum_{i=1}^n a_i\xi_i))=\mathsf{P}_h^h(b;\sfa)$ for all genus $h$ less than $g$. Let $\Xi\in\CH^*(\Jbar^{\tl}_{h,n})$ be any monomial of $\xi$  and $\kappa_{0,1}$-classes. For any genus $h$ less than $g$, we have
		\begin{enumerate}
			\item[(a)]  $\fm^\circ_h(\Xi)\in \CH^{\geq h}(\Jzero_{h,n})$, and
			\item[(b)]  $[(-1)^h\fm^\circ_h(\exp(-b\kappa_{0,1}+\sum_{i=1}^n a_i\xi_i))]_{\codim=h}= \frac{(\Theta_{b;\sfa})^h}{h!}$
		\end{enumerate}
	\end{proposition}
	\begin{proof}
		We first prove (a) by induction on $h$. Suppose the statement holds for all genera less than $h$. By assumption, we have $\fm^{-1}_h(\exp(-b\kappa_{0,1}+\sum_{i=1}^n a_i\xi_i))=\mathsf{P}^h_g(b;\sfa)\in\CH^h(\Jzero_{h,n})$.
		From Lemma \ref{lem:unidr_int} (a) the polynomiality of $\mathsf{P}_h^h(b;\sfa)$ in $b,\sfa$ follows from the polynomiality of \eqref{eq:twist_theta}. Therefore, for any monomial $\Xi$ of $\xi$ and $\kappa_{0,1}$-classes, we have $\fm^{-1}_h(\Xi)\in \CH^h(\Jzero_h)$.  By Proposition \ref{pro:Todd} and Proposition \ref{pro:F0}, we get
		\begin{equation}\label{eq:FXi}
			\fm^{-1}_h(\Xi)= (-1)^h\fm_h^\circ(\Xi) + \sum_{f=1}^h \frac{1}{2^ff!}(j_f)_*(-1)^{h-f}\fm_{h-f}^\circ \Big(\Xi\cdot\prod_{i=1}^f \sum_{k_i=1}^{\infty}\frac{(-1)^{k_i}B_{2k_i}}{(2k_i)!}\frac{\alpha_i^{2k_i-1}+\beta_i^{2k_i-1}}{\alpha_i+\beta_i}\Big)\,.
		\end{equation}
		where $\alpha_i=\psi_{n+i}-\xi_{n+i}+\xi_{n+f+i}$ and $\beta_i=\psi_{n+f+i}+\xi_{n+i}-\xi_{n+f+i}$. The term inside $\fm^\circ_{h-f}$ is a polynomial in $\psi_i,\xi_i$ and $\kappa_{0,1}$-classes.  By linearity of $\fm^\circ_{g-h}$, $\psi$-classes can be factored out. By the induction hypothesis, these terms has codimension at least $h-g$, so  $(j_f)_*\fm^\circ_{h-f}(\cdots)$ has codimension at least $h$. Therefore, we conclude that $\fm^\circ_h(\Xi)\in\CH^{\geq h}(\Jzero_{h,n})$. 
		
		Similarly, we prove (b) by induction on $h$. Suppose the statement holds for all genera less than $h$. In \eqref{eq:FXi}, we substitute  $\Xi_{b;\sfa} :=\exp(-b\kappa_{0,1}+\sum_{i=1}^na_i\xi_i)$:
		\begin{equation}\label{eq:FXi2}
			\fm^{-1}_h(\Xi_{b;\sfa}) = (-1)^h\fm_h^\circ(\Xi_{b;\sfa}) +\sum_{f=1}^h \frac{1}{2^ff!}(j_f)_*(-1)^{h-f}\fm_{h-f}^\circ \Big(\Xi_{b;\sfa}\cdot\prod_{i=1}^f \sum_{k_i=1}^{\infty}\frac{(-1)^{k_i}B_{2k_i}}{(2k_i)!}\frac{\alpha_i^{2k_i-1}+\beta_i^{2k_i-1}}{\alpha_i+\beta_i}\Big)\,.
		\end{equation}
		Take codimension $h$ part of each term in \eqref{eq:FXi2}. By (a), within $\fm^\circ_{h-f}(\cdots)$, only pure monomials of $\xi$ and $\kappa_{0,1}$ classes will contribute. By the induction hypothesis, the boundary term takes the form
		\[
		\big[(-1)^{h-f}\fm^\circ_{h-f}\Big(\Xi_{b;\sfa}\cdot\prod_{i=1}^f(\xi_{h_i}-\xi_{h_i'})^{2k_i-2}\Big)\big]_{\codim=h-f} = \frac{(\Theta_{b;\sfa})^{h-k}}{(h-k)!}\prod_{i=1}^f\frac{(2k_i-1)!}{(k_i-1)!}\Big(\frac{\psi_{h_i}+\psi_{h_i'}}{2}\Big)^{k_i-1}
		\]
		modulo relations $\xi_{h_1}-\xi_{h_1'}=\cdots=\xi_{h_f}-\xi_{h_f'}=0$. Here, $k:=\sum_i k_i$ and the relations are from Lemma \ref{lem:torsor}. By Lemma \ref{lem:unidr_int} (a), the boundary term of the right hand side \eqref{eq:FXi2} aligns, leading to  $[(-1)^h\fm^\circ_h(\Xi_{b;\sfa})]_{\codim=h} = \frac{(\Th_{b;\sfa})^h}{h!}$.
		%the boundary terms in \eqref{eq:FXi2} sum up to
		%\[\left[\exp(\Th_{b;\sfa})\cdot\sum_{f=1}^h\frac{1}{2^f f!}j_{f*}\prod_{i=1}^f\Big(\sum_{k_i=1}^\infty\frac{B_{2k_i}}{(2k_i)!}\frac{(2k_i-2)!(2k_i-1)}{(k_i-1)!}\frac{(\psi_{h_i}+\psi_{h_i'})^{k_i-1}}{2^{k_i-1}}\Big)\right]_h\]
		%Therefore we get the relation.
	\end{proof}
	We prove the main result of this section. Over $\Mint_{g,n}$, there exists a unique piecewise linear function $\alpha$ on the universal curve such that the line bundle $\omega_{C,\log}^{\otimes-b}(\sum_{i=1}^na_ix_i)\otimes\CO_C(\alpha)$ has multidegree zero for all underlying curves in $\Mint_{g,n}$. After twisting by $\alpha$, the Abel-Jacobi section $\aj_{b;\sfa}:\Mint_{g,n}\to\Jbar^{\,\text{tl}}_{g,n}$ is well-defined and factors through $\Jzero_{g,n}$.
	\begin{theorem}\label{thm:main2} Let $\Jzero_{g,n}\to\Mint_{g,n}$ be the relative Jacobian. For any $b\in\ZZ$ and $\sfa\in\ZZ^n$ with $\sum_{i=1}^na_i=(2g-2+n)b$, we have 
		\begin{equation}\label{eq:aj_tl}
			[\aj_{b;\sfa}] = \mathsf{P}^g_g(b;\sfa)\in \CH^g(\Jzero_{g,n})\,.
		\end{equation}
	\end{theorem}
	\begin{proof}
		We prove this by induction on the genus. For $g=0$ the statement is trivial. Suppose we know \eqref{eq:aj_tl} for genus $h$ less than $g$. We first prove \eqref{eq:aj_tl} for $g=0$ and $\sfa=0$. By Proposition \ref{pro:section}, we get $[e] = \fm^{-1}_g(1)$. 
		By Proposition \ref{pro:Tdcirc} (a), after taking the codimension $g$ part, we get
		\[[e] = [(-1)^g\fm_g^\circ(1)]_{\codim = g} + \sum_{h=1}^g \frac{1}{2^hh!}(j_h)_*\Big[(-1)^{g-h}\fm_{g-h}^\circ \Big(\prod_{i=1}^h \sum_{k_i=1}^{\infty}\frac{ B_{2k_i}(2k_i-1)}{(2k_i)!}(\xi_{h_i}-\xi_{h_i}')^{2k_i-2}\Big)\Big]_{\codim = g-h}\,.\]
		By Proposition \ref{pro:Tdcirc} (b), the right hand side matches with Lemma \ref{lem:unidr_int}. This finishes the proof for genus $g$ with $\sfa=0$.
		
		We prove for arbitrary $b\in\ZZ$ and $\sfa\in\ZZ^n$ with $\sum_{i=1}^na_i=(2g-2+n)b$. Let $\alpha$ be the unique piecewise linear function on the universal curve necessary to construct the Abel-Jacobi section. Then, the Abel-Jacobi section can be written as
		\[\aj_{b;\sfa}=\widetilde{\phi}_{b;\sfa}\circ e : \Mint_{g,n}\to\Jbar^{\tl}_{g,n}\,,\]
		where $e: \Mint_{g,n}\to\Jbar_{g,n}^{\tl}$ is the unit section and $\widetilde{\phi}_{b;\sfa}$ is an automorphism of $\Jbar^{\tl}_{g,n}$ induced by $L\mapsto L\otimes\omega_{C,\log}^{\otimes-b}(\sum_{i=1}^na_ix_i)\otimes\CO_C(\alpha)$. 
		By \eqref{eq:twist_theta} and Lemma \ref{lem:unidr_int}, we have $\widetilde{\phi}_{-b;-\sfa}^*\mathsf{P}^c_g(0;0)=\mathsf{P}^c_{g}(b;\sfa)$. Therefore, by twisting \eqref{eq:aj_tl} for $\sfa=0$ with respect to $\widetilde{\phi}_{b;\sfa}$, we have
		\[[\aj_{b;\sfa}]=\widetilde{\phi}^*_{-b;-\sfa}[e]=\widetilde{\phi}^*_{-b;-\sfa}\mathsf{P}^g_g(0;0)=\mathsf{P}^g_g(b;\sfa)\,.\]
		Thus we obtain the result from the induction.
	\end{proof}
	By the argument in Section \ref{sec:pf}, Theorem \ref{thm:main2} gives an independent proof of Theorem \ref{thm:pushforward} over $B=\Mint_{g,n}$ for $\ell=0$.
	
	\subsection{Fourier transform for principally polarized abelian schemes}
	%Our argument for $\Mint_{g,n}$ is based on the existence of markings of underlying curves. We generalize the argument to  abelian schemes.
	We extend our argument in the previous sections to abelian schemes. Let $\pi:A\to B$ be a principally polarized abelian scheme (p.p.a.s.) of relative dimension $g$ over a smooth scheme $B$. Let $e:B\to A$ be the unit section and let $\mu: A\times_BA\to A$ be the addition. For $n\geq 1$, we denote
	\[ A^n:=A\times_B\cdots \times_B A \text{\hspace{2mm} ($n$-times)}\,,\]
	and  $A^0:=B$. For $1\leq i\leq n$, let $p_i:A^n\to A$ be the $i$-th projection and for $1\leq i< j\leq n$, let $p_{ij}:A^n\to A\times_BA$ be the projection to $i,j$-th factor. Let $\lambda_m:=c_m(\EE)\in \CH^m(B)$ be the $m$-th Chern class of the Hodge bundle $\EE:=\pi_*\Omega_\pi$.
	
	For an abelian scheme $A\to B$, let $A^\vee := \mathrm{Pic}^0_{A/B}$ denote the dual abelian scheme and let $\P$ be the Poincar\'e line bundle on $A\times_B A^\vee$ which is trivialized along the unit sections of $A$ and $A^\vee$. Assume that $A\to B$ is principally polarized; that is, there exists a polarization $\lambda: A\to A^\vee$ which is an isomorphism (\cite[Definition 2.1]{Yin}). We identify $(\id,\lambda): A\times_B A\xrightarrow{\cong} A\times_B A^\vee$. Consider a $\QQ$-line bundle
	\[L_\lambda := ((e,\lambda)^*\P)^{-1/2}\in\mathrm{Pic}(A)\otimes_\ZZ\QQ\]
	which is relatively ample, symmetric and trivialized along the unit section.  
	Let $\Theta := c_1(L_\lambda)\in\CH^1(A)$.  For $1\leq i\leq n$, let $\Theta_i:=p_i^*\Theta$ and for $1\leq i\neq j\leq n$, let  $\ell_{ij}:=p_{ij}^*c_1(\P)$ in $\CH^1(A^n)$.
	\begin{definition}
		The tautological ring $\R^*(A^n)$ is the subring of $\CH^*(A^n,\QQ)$ generated by classes
		\[\{\Theta_i\}_{1\leq i\leq n},\, \{\ell_{ij}\}_{1\leq i\neq j\leq n},\, \{\lambda_m\}_{1\leq m\leq g}\,.\]
	\end{definition}
	
	For $n\geq 1$, we consider a p.p.a.s. $\pi:A^n\to A^{n-1}$ where $\pi$ is the projection along the first $n-1$ factors. Under the weight decomposition \eqref{eq:wt_dec}, the weights of tautological classes are given by
	\begin{equation}\label{eq:weight_A}
		\Theta_n\in \CH^1_{(2)}(A^n), \text{  and }\, \ell_{i,n}\in \CH^1_{(1)}(A^n), 1\leq i\leq n-1\,.
	\end{equation}
	The first follows from the fact that $\Theta_n$ is a symmetric ample class and the second follows from the Theorem of Square. Other generators of tautological ring has weight zero. Since the weight is multiplicative, \eqref{eq:weight_A} determines the weight of all tautological classes. 
	
	Given a vector $\sfa\in\ZZ^{n-1}$, we define the $\pi$-relative {\em translation automorphism} by
	\begin{equation}\label{eq:translate}
		\tau_\sfa : A^n\to A^n, (x_1,\ldots, x_{n})\mapsto (x_1,\ldots,x_{n-1},x_{n}+ a_1x_1+\cdots + a_{n-1}x_{n-1})\,.
	\end{equation}
	\begin{lemma}\label{lem:twist_theta}
		For $\sfa\in\ZZ^{n-1}$, the pullback of $\Theta_n$ along $\tau_\sfa$ is given by
        \[
        \tau_\sfa^*(\Theta_n)=\Theta_n - \sum_{i=1}^{n-1}a_i\ell_{in} - \sum_{1\leq i < j\leq n-1} a_ia_j\ell_{ij} +\sum_{i=1}^{n-1}a_i^2\Theta_i\,.
        \]
	\end{lemma}
	\begin{proof}
		We write \eqref{eq:translate} as follows. Consider the composition $\mu_\sfa: A^n\xrightarrow{(a_1,\ldots,a_{n-1},1)}A^n\xrightarrow{\mu_n} A$ where the first morphism is the ``multiplication by $N$" map with the prescribed multiplicities, and $\mu_n$ is the addition map. For $\mu:A\times_BA \to A$, consider the Mumford formula\footnote{Since we use the opposite convention for the polarization $\lambda:A\to A^\vee$, the sign is opposite to that in \cite{Mumford_Abelian}.}
		\begin{equation*}
			\mu^*(\Theta)=-\ell_{12}+p_1^*\Theta+p_2^*\Theta\,.
		\end{equation*}
		For the morphism $\mu\times\id:A\times_BA\times_BA\to A\times_BA$, the Theorem of the Cube gives
        \begin{equation}\label{eq:cube}
            (\mu\times\id)^*\ell_{12} = \ell_{13}+\ell_{23}\,.
        \end{equation}
        Therefore, writing $\mu_n$ by a composition of $n-1$ addition maps, the Mumford formula and \eqref{eq:cube} yields $\mu_n^*(\Theta) = -\sum_{1\leq i< j\leq n}\ell_{ij} +\sum_{i=1}^n \Theta_i$.
        
		From the weight of divisors \eqref{eq:weight_A}, the pullback of tautological classes along the map $(a_1,\ldots, a_{n-1},1)$ is clear. Since $\tau_\sfa^*(\Theta_n) = \tau_\sfa^*p_n^*(\Theta) = \mu_\sfa^*(\Theta)$, the result follows.
	\end{proof}
	We consider the Fourier-Mukai transform (\cite{Mukai81,DM91})
	\begin{equation}\label{eq:fm_A}
		\fm:=\ch(\P)=\exp (\ell_{nn+1}) :\CH^*(A^n)\to \CH^*(A^n)\,.
	\end{equation}
    \begin{lemma}\label{lem:Ftau}
        For $\alpha\in \CH^*(A^n)$ and $\sfa\in\ZZ^{n-1}$, we have $\tau_\sfa^*\fm(\alpha) = \fm(\alpha\cdot \exp(\sum_{i=1}^{n-1}a_i\ell_{in}))$.
    \end{lemma}
    \begin{proof}
        For $\sfa\in\ZZ^{n-1}$, consider the map $\nu_\sfa: A^{n+1}\to A^2$ sending
        \[
        (x_1,\cdots,x_{n+1}) \mapsto (x_n, x_{n+1} + a_1x_1+\cdots +a_{n-1}x_{n-1})\,.
        \]
        Using \eqref{eq:cube} and \eqref{eq:weight_A}, we obtain $\nu_\sfa^*\ell_{12} = \ell_{nn+1} +\sum_{i=1}^{n-1}a_i\ell_{in}$.

        Next, consider the commutative diagram
        \[
        \begin{tikzcd}
            A^n\ar[d,equal] & A^n\times_{A^{n-1}}A^n \ar[l,"\pi_1"]\ar[r,"\pi_2"]\ar[d,"\id\times\tau_\sfa"] & A^n\ar[d,"\tau_\sfa"]\\ 
            A^n & A^n\times_{A^{n-1}}A^n\ar[l,"\pi_1"]\ar[r,"\pi_2"] & A^n,
        \end{tikzcd}
        \]
        where the right square is a fiber diagram. After identifying $A^n\times_{A^{n-1}}A^n \cong A^{n+1}$, we compute:
        \begin{align*}
            \tau_\sfa^*\fm(\alpha) &= \tau_\sfa^*\pi_{2*} \big(\pi_1^*(\alpha)\cdot\exp(\ell_{nn+1})\big)\\
            &=\pi_{2*}(\id\times\tau_\sfa)^*\big(\pi_1^*(\alpha)\cdot\exp(\ell_{nn+1})\big)\\
            &=\pi_{2*}\big(\pi_1^*(\alpha)\cdot \exp(\nu_\sfa^*\ell_{12})\big)\\
            &=\fm\big(\alpha\cdot\exp(\sum_{i=1}^{n-1}a_i\ell_{in})\big),
        \end{align*}
        The third identity uses the fact that
        $(\id\times\tau_\sfa)^*\ell_{nn+1}=(\id\times\tau_\sfa)^*p_{nn+1}^*\ell_{12} = \nu_\sfa^*\ell_{12}$, and the fourth uses the computation above. This proves the result.
    \end{proof}
    \if0
	\begin{proposition}\label{pro:twist_theta}
		Let $\sfa\in\ZZ^{n-1}$ be a vector of integers. Then we have
		\begin{equation}\label{eq:fm_taut_Ag}
			\fm^{-1}\big(\exp(\sum_{i=1}^{n-1} a_i\ell_{in})\big)= \frac{\big(\Theta_n-\sum_{i=1}^{n-1}a_i\ell_{in}-\sum_{1\leq i<j\leq n-1} a_{ij}\ell_{ij}+\sum_{i=1}^{n-1}a_i^2\Theta_i\big)^g}{g!}\,.
		\end{equation}
	\end{proposition}
	\begin{proof}
		By the argument in Proposition \ref{pro:fm1}, we have $\fm(\exp(\Theta_n)) = \exp(-\Theta_n)$. By Proposition \ref{pro:section} and \cite[Lemma 2.18]{DM91}, we have
		\begin{equation}\label{eq:Aunit}
			[e]=\fm^{-1}(1) = \frac{(\Theta_n)
				^g}{g!}\,.
		\end{equation}
		The pullback of \eqref{eq:Aunit} along \eqref{eq:translate} follows from Lemma \ref{lem:twist_theta}. By the argument in the proof of Theorem \ref{thm:main2}, we get the result.
	\end{proof}
    \fi
	\begin{theorem}\label{thm:Apush}
		The Fourier-Mukai transform \eqref{eq:fm_A} preserves tautological classes. Moreover, proper pushforward along the projection $\pi:A^n\to A^{n-1}$ preserves tautological classes. 
	\end{theorem}
    \begin{proof}
        By the argument in Proposition \ref{pro:fm1}, we have $\fm(\exp(-\Theta_n)) = \exp(\Theta_n)$. Applying Lemma \ref{lem:Ftau} to this identity, we obtain
        \begin{equation}\label{eq:Fclose}
            \fm\big(\exp(-\Theta_n+\sum_{i=1}^{n-1}a_i\ell_{in})\big) = \exp \big( \Theta_n - \sum_{i=1}^{n-1}a_i\ell_{in} - \sum_{1\leq i < j\leq n-1} a_ia_j\ell_{ij} + \sum_{i=1}^{n-1}a_i^2\Theta_i\big)\,.
        \end{equation}
        The identity \eqref{eq:Fclose} further decompose as follows. By the polynomiality of \eqref{eq:Fclose}, for fixed monomial in the classes $\ell_{in}$, say $\Xi$, the expression $\fm(\exp(\Theta_n)\cdot\Xi)$ has a closed formula in tautological calsses. Since weight $w$ class maps to weight $2g-w$ class under Fourier-Mukai transform \cite[Lemma 2.18]{DM91}, the Fourier-Mukai image of each monomial of $\Theta_n,\ell_{1n},\cdots, \ell_{n-1n}$ has a closed form in tautological classes. Because $\fm$ is linear over the base, this proves that $\fm$ preserves tautological classes. 
        
        Using a similar argument as in Proposition \ref{pro:IntFM}, we find $\pi_*(\alpha)=e^*\fm(\alpha)$ for all $\alpha\in\CH^*(A_n)$. Since $e^*:\CH^*(A^n)\to\CH^*(A^{n-1})$ preserves tautological classes, so does $\pi_*$.
    \end{proof}
    By pulling back \eqref{eq:Fclose} along the unit section, we obtain:
    \begin{equation}\label{eq:Apush}
        \pi_*\Big( \exp \big( -\Theta_n+\sum_{i=1}^{n-1}a_i\ell_{in} \big) \Big) = \exp\Big( -\sum_{1\leq i<j\leq n-1}a_ia_j\ell_{ij} +\sum_{i=1}^{n-1}a_i^2\Theta_i\Big)\,.
    \end{equation}
    Since the pushforward of a class vanishes whenever its weight is not equal to $2g$, \eqref{eq:Apush} provides a closed formula for the pushforward of all tautological classes.
	\subsection{Torus rank at most one locus}\label{sec:trk1}
	Let $\A_g\subset\Abarone_g$ be the canonical partial compactification of rank $1$ degenerations \cite{mumford_kodaira}. Let $\pi: \Xbarone_g\to \Abarone_g$ be the universal family and let $\X^\circ_g\to \Abarone_g$ be the universal semi-abelian scheme. We adapt the argument from the previous sections to study the intersection theory of this family.  Let $\mu:\X^\circ_g\times_{\Abarone_g}\Xbarone_g\to\Xbarone_g$ be the natural action. Then the tuple
	\begin{equation}\label{eq:DAF}
		(\pi:\X'_g\to\A'_g, \X^\circ_g\to\Abarone_g, \mu:\X^\circ_g\times_{\Abarone_g}\Xbarone_g\to\Xbarone_g)
	\end{equation}
	is a {\em degenerate abelian scheme} in the sense of Arinkin-Fedorov \cite[Definition 2.1]{AF16}.
	
	We show that the auto-equivalence of $\X_g$ extends to $\X'_g$. Let $\Theta$ be the principal polarization of $\mathcal{X}_g\to\A_g$ which is trivialized along the unit section. By slight abuse of notation, let $\Theta$ also denote a relatively ample class on $\Xbarone_g$ that extends this polarization. Following \cite{AF16}, the Poincar\'e line bundle $\P$ on $\X^\circ_g\times_{\Abarone_g}\Xbarone_g$ defined by
	\begin{equation}\label{eq:Xmumford}
		\P=\mu^*\Theta\otimes p_1^*\Theta^{\vee}\otimes p_2^*\Theta^\vee
	\end{equation}
	which extends to a line bundle $\P$ on $\X^\circ\times_{\A_g'}\X_g'\cup \X'_g\times_{\A'_g}\X^\circ_g$. For compactified Jacobians, the failure to extending the Poincar\'e line bundle to the product $\Jbar_C \times_B \Jbar_C$ was attributed to the lack of a common universal curve on the two factors. For $\X_g'$, the failure to extend $\P$ is that the multiplication map $\mu$ does not extend to $\X_g' \times_{\A_g'} \X_g'$. 
	
	We can resolve the indeterminacy of the action $\mu$ for any toroidal compactification of $\X_g\to\A_g$. Let $\Atrop_g:=\textup{Sym}_{rc}^2(\ZZ^g)/\textup{GL}_{g}(\ZZ)$ be the moduli space of principally polarized tropical abelian varieties and let $\Xtrop_g \to \Atrop_g$ be the the universal tropical abelian variety. Any toroidal compactification $\Xbar_g\to\Abar_g$ corresponds to a subdivision $\Sigma_A$ of $\Atrop_g$ and a subdivision $\Sigma_X$ of $\Sigma_A|_{\Xtrop_g}$.  Let
	\[\mu: \Xtrop_g\times_{\Atrop_g}\Xtrop_g\to \Xtrop_g\]
	be the multiplication map. Let $\Sigma_X^\mu\to\Sigma_X$ be the base change of $\mu: \Sigma_X \times_{\Sigma_A} \Sigma_X\to\Xtrop_g|_{\Sigma_A}$ along $\Sigma_X\to\Xtrop_g|_{\Sigma_A}$.
	%\[
	%\begin{tikzcd}
	%    \Sigma_{X}^\mu \ar[r] \ar[d] & \Sigma_X \ar[d] \\ 
	%    \Sigma_X \times_{\Sigma_A} \Sigma_X \ar[r,"\mu"] & \Xtrop_g|_{\Sigma_A}.
	%\end{tikzcd}
	%\]
	This subdivision defines a  log modification $f:\Xbar^\mu_g\to \Xbar_g\times_{\Abar_g}\Xbar_g$ together with a map $\mu: \Xbar^{\mu}_g\to \Xbar_g$ extending the multiplication map on $\X^\circ\times_{\A_g'}\X_g'\cup \X'_g\times_{\A'_g}\X^\circ_g$.
	
	We specialize to $\X'_g\to \A'_g$. Inside any $\Abar_g$, $\A'_g$  sits as the inverse image of $\A_g\sqcup \A_{g-1}$ from Satake compactification and $\X_g'$ is the unique minimal family that contains the identity section in its smooth locus. The restriction of $\Xbar^\mu_g$ to the locus $\A_g'$ is denoted by $\X_g^{\mu}$.
	The fan $\Sigma_A$ restricts to the ray $\mathbb{R}_{\ge 0}$, and $\Xtrop_g|_{\Sigma_A}$ is the torus with fiber $\mathbb{R}/\ell \ZZ$ over $\ell \in \mathbb{R}_{> 0}$ (and a point over $0$). The subdivision $\Sigma_X$ is the subdivision of $\mathbb{R}/\ell \ZZ$ adding the lattice points $\ell \ZZ$. Therefore, $\Sigma_X^{\mu}$ over the torus-rank $1$ locus is the subdivision of the product $\mathbb{R}/\ell \ZZ \times_{\mathbb{R}_{\ge 0}} \mathbb{R}/\ell \ZZ$ along the diagonal 
	\[
	x + y = 0 \,(\textup{mod } \ell \ZZ)
	\]
	By the same analysis as in Theorem \ref{thm:Plog=arinkin} (a), $\X_g' \times_{\A_g'} \X_g'$ has ordinary double point singularities, and $\X_g^{\mu}$ is a small resolution, i.e. it is smooth and the complement of $\X^\circ\times_{\A_g'}\X_g'\cup \X'_g\times_{\A'_g}\X^\circ_g$ has codimension $2$. Moreover, nontrivial fibers of $f: \X_g^{\mu} \to \X_g' \times_{\A_g'} \X_g'$ are $\mathbb{P}^1$s.
	
	\begin{proposition}
		\label{prop:mumdeg}
		Let $\widetilde{\P}$ be the unique line bundle on $\X^\mu_g$ which extends $\P$. Then the degree of $\widetilde{\P}$ on the exceptional fibers of $f: \X_g^{\mu} \to \X_g' \times_{\A_g'} \X_g'$ is $1$. 
	\end{proposition}
	
	\begin{proof}
		The exceptional locus of $f$ is a $\PP^1$-bundle over the boundary of $\A_g'$, so we check the degree on any given fiber. Under the Torelli map, it is enough to check the degree over a point corresponding to a $\Jbar_C$ where $C$ is an integral curve with one one self-node. Two resolutions $\Jbar_C^{(2)}$ and the resolution $\Jbar^\mu_C$ given by $\X_g^\mu$  of $\Jbar_C\times\Jbar_C$ corresponds to two different Atiyah flops of ordinary double point singularity. To compare two, we form the common refinement $\Jbar^b_C$ of $\Jbar_C^{(2)}$ and $\Jbar_C^\mu$ as follows. The product $\Jbar_C\times\Jbar_C$ is locally modeled on the toric variety
		\[\left \{(\ell_{e_1}',\ell_{e_1}'',\ell_{e_2}',\ell_{e_2}''): \ell_{e_i}' + \ell_{e_i}'' = \ell_e \right\}\]
		and $\Jbar_C^b$ corresponds to the subdivision $\min (\ell_{e_1}',\ell_{e_1}'',\ell_{e_2}',\ell_{e_2}'')$.
				Following Definition \ref{def:Q2} and Definition \ref{def:J2}, the multiplication map on $\Jbar_C^b$ has the following modular description:
				\[
				\mu(\widetilde{C} \to C_i, L_i) = L_1|_{\widetilde{C}} \otimes L_2|_{\widetilde{C}} (\alpha)
				\]
				where $\alpha$ is the unique piecewise linear function up to pullback from the base such that the line bundle $L_1|_{\widetilde{C}} \otimes L_2|_{\widetilde{C}} (\alpha)$ is stable on $\widetilde{C}$. 
				
				%The calculation of $\alpha$ is a standard calculation in tropical Abel-Jacobi theory. We do not recall the theory here, and refer the reader to \cite[Section 3, Section 4.3]{HMPPS} and \cite[Section 8]{Molcho23} for further details. The tropical curve $\Gamma$ associated to $C$ is a loop with a single vertex $v$. We fix some orientation of $\Gamma$. The tropical curves $\Gamma_i$ associated to $C_i$ have one additional vertex $w_i$ at distance $\ell_{e_1}',\ell_{e_2}'$ from $v$ respectively, measured with respect to the given orientation; the tropical divisors associated to $L_1,L_2$ have degree $-1$ at $v$ and $1$ at the exceptional vertex. The tensor product $L_1|_{\widetilde{C}} \otimes L_2|_{\widetilde{C}}$ has associated tropical divisor supported on the tropical curve $\widetilde{\Gamma}$ which has two exceptional vertices, and has degree $-2$ at $v$ and $1$ on the exceptional vertices -- and is thus unstable.     
				
				We compute $\alpha$. Let $\Gamma$ be the tropical curve associated to $C$ which has a single vertex $v$ with a loop, where we fix some orientation. The tropical curves $\Gamma_i$ associated to $C_i$ have one additional vertex $w_i$ at distance $\ell_{e_1}',\ell_{e_2}'$ from $v$ respectively, measured with respect to the given orientation; the tropical divisors associated to $L_1,L_2$ have degree $-1$ at $v$ and $1$ at the exceptional vertex.
				%To calculate $\alpha$, we must look at the four cones of the subdivision defining $\Jbar^b$. 
				We look first at the cone of $\Jbar_C^b$ where $\ell_{e_2}'$ is minimal. We form the tropical curve which is the subdivision of $\Gamma$ at a vertex $u$ at distance $\ell_{e_1}'+\ell_{e_2}'$ from $v$ in the given orientation. Since $\ell_{e_2}'$ is minimal, $\ell_{e_1}'+\ell_{e_2}'<\ell_{e_1}'+\ell_{e_1}'' = \ell_e$, and hence in the given orientation we have an oriented path $v,w_2,w_1,u,v$. We form the tropical $D_u$ divisor with degree $-1$ at $v$ and $1$ at $u$. Then, we have
				\[
				D_1 + D_2 - D_u = \textup{div}(\alpha)
				\]
				where $\alpha$ has slope $1$ from $v$ to $w_2$, slope $1$ from $u$ to $w_1$, and slope $0$ elsewhere\footnote{We are using here the convention that the multidegree of $\mathcal{O}(\alpha)$ is $-\textup{div}(\alpha)$.}.  
				If we normalize $\alpha$ so that $\alpha(v)=0$, it follows that $\alpha(w_1)=\alpha(w_2)=\ell_{e_2}'$, and $\alpha(u)=0$. A similar calculation at the other cones of $\Jbar^b$ shows that in general $\alpha(w_1)=\alpha(w_2) = \min(\ell_{e_1}',\ell_{e_2}',\ell_{e_1}'',\ell_{e_2}'')$.  
				
				We now compare two models  of extended Poincar\'e line bundles. By the bilinearity of the Deligne pairing, we have
				%We find 
				%\[
				%\mu^*\Theta = \frac{-1}{2}c_1(\left \langle L_1 \otimes L_2 \otimes \mathcal{O}(\alpha), L_1 \otimes L_2 \otimes \mathcal{O}(\alpha) \right \rangle)
				%\]
				%and thus, by bilinearity of the Deligne pairing, 
				\begin{equation}\label{eq:MD}
					\mu^*\Theta - p_1^*\Theta - p_2^*\Theta = -c_1(\left \langle L_1,L_2 \right \rangle) - c_1(\left \langle L_1,\mathcal{O}(\alpha) \right \rangle) - c_1(\left \langle L_2, \mathcal{O}(\alpha) \right \rangle) - \frac{1}{2}c_1(\left \langle \mathcal{O}(\alpha),\mathcal{O}(\alpha) \right \rangle)\,.
				\end{equation}
				
				Since $\alpha(w_i)=\min(\ell_{e_1}',\ell_{e_2}',\ell_{e_1}'',\ell_{e_2}'')$, using Lemma \ref{lem:deligne_pl}, the last three terms of \eqref{eq:MD} is
				\[
				-\min(\ell_{e_1}',\ell_{e_2}',\ell_{e_1}'',\ell_{e_2}'')-\min(\ell_{e_1}',\ell_{e_2}',\ell_{e_1}'',\ell_{e_2}'')+\min(\ell_{e_1}',\ell_{e_2}',\ell_{e_1}'',\ell_{e_2}'') = -\min(\ell_{e_1}',\ell_{e_2}',\ell_{e_1}'',\ell_{e_2}'')
				\]
				which is the negative of the class of the exceptional divisor $E$ of $\Jbar_C^b \to \Jbar_C \times \Jbar_C$. The degree of $-\left \langle L_1,L_2 \right \rangle$ has been calculated to be $-1$ on the exceptional $\PP^1$s of $\Jbar_C^{(2)}$ in Theorem \ref{thm:Plog=arinkin}. On the blowup of the ordinary double point the line bundle $\mu^*\Theta-p_1^*\Theta-p_2^*\Theta$ is thus represented by $-E-D$, where $D$ is the strict transform of one of the two Weil divisors that are not blown up in the construction of the flop $\Jbar_C^{(2)} \to \Jbar_C \times \Jbar_C$. If $F$ denotes the strict transform of one of the adjacent Weil divisors to $D$, the divisor $-E-D-F$ is rationally equivalent to $0$\footnote{For instance, if we take $(0,0,1),(1,0,1),(0,1,1),(1,1,1)$ as the four vertices of the square, $D$ the bottom right vertex, $F$ the bottom left, then $-E-D-F$ is the divisor of the character $e_2^*-e_3^*$, where $e_i^*$ are the standard coordinates of $\mathbb{Z}^3$.}; hence we have 
				\[
				c_1(\widetilde{\P}) = -E-D \sim F
				\]
				on the exceptional fibers of $\Jbar^b_C \to \Jbar_C \times \Jbar_C$. The divisor $F$ does not get blown up in the complementary flop $\Jbar^{\mu}_C \to \Jbar_C \times \Jbar_C$, and so descends to a Cartier divisor on $\Jbar^{\mu}_C$, where it has degree $1$. 
			\end{proof}
			\begin{proof}[Proof of Theorem \ref{thm:Xdual}]
				We follow Arinkin's the proof of auto-duality in \cite{arinkin2}. The argument in loc. cit. is divided into three parts: (i) computing the the cohomology of the Poincar\'e line bundle $\P$ along the projection $\pi_2:\X'_g\times_{\A'_g}\X^\circ_g\to \X^\circ_g$, (ii) proving that $\P$ admits a flat maximal Cohen-Macaulay extension $\overline{\P}$, and (iii) using the $\delta$-regularity to show that $\overline{\P}$ and $\overline{\P}^{-1}$ are inverse to each other.
                The first part follows from \cite[Theorem 1]{AF16}. By Proposition \ref{prop:mumdeg} and following the argument in Theorem \ref{thm:Plog=arinkin}, the derived pushforward $\Pbar:=Rf_*\widetilde{\P}$ is a maximal Cohen-Macaulay sheaf extending $\P$ which is flat relative to each projections $\pi_i:\X'_g\times_{\A'_g}\X'_g\to\X'_g$. For the semi-abelian scheme $\pi:\X_g^\circ\to\A'_g$, the locus where the torus rank increases clearly has codimension $1$, and therefore forms a $\delta$-regular family. This completes the proof of the theorem.
			\end{proof}

			We adapt \eqref{eq:bdy}. The boundary $\Abarone_g\setminus \A_g$ admits a natural two-to-one cover $i:\X_{g-1}\to \Abarone_g\setminus \A_g$.
			Consider the diagram
			\begin{equation}\label{eq:bdy_Ag}
				\begin{tikzcd}
					\mathcal{X}_{g-1}\times_{\A_{g-1}}\X_{g-1}\ar[r]\ar[dr,"p_1"'] & \Xbarone_g|_{\X_{g-1}} \ar[r]\ar[d] &\Xbarone_{g}\ar[d,"\pi"]\\
					& \X_{g-1}\ar[r,"i"] & \Abarone_g\,,
				\end{tikzcd}
			\end{equation}
			where $p_1$ is the first projection. By \cite{EvG}, the boundary of $\A_g'$ is naturally identified with the dual abelian fibration $\X_{g-1}^\vee$. Here, we identified it with $\X_{g-1}$ using the principal polarization. 
			
			If we restrict $\X_g^\circ\to\A'_g$ to the boundary $i$, then it is a $\Gm$-torsor over $\X_g\times_{\A_{g-1}}\X_{g-1}$. Let $q$ be the projection from the torsor to $\X_g\times_{\A_{g-1}}\X_{g-1}$.
			\begin{proof}[Proof of Theorem \ref{thm:unit_Ag}]
				From Theorem \ref{thm:Xdual}, let $\fm:\CH^*(\X'_g)\xrightarrow{\cong}\CH^*(\X'_g)$ be the Fourier transform and $\fm^{-1}$ be its inverse with the normalization as in \eqref{eq:FM} and \eqref{eq:fminv}.  By the same argument in Proposition \ref{pro:section}, we get $[e]=\fm^{-1}(1)$. Let $\EE:=\pi_*\Omega_{\pi,\log}^1$ be the Hodge bundle.  By \cite{EV02}, $c(\EE)c(\EE^\vee)=1$ still holds. The Todd class of the residue sheaf $\Rels_\pi$ has the same formula as in Proposition \ref{pro:Todd} with the normal bundle of the codimension two stratum calculated in \cite{EvG}.  Proposition \ref{pro:fm1} and Theorem \ref{thm:Apush} applied to the inductive structure in Section \ref{sec:treelike} give the result. 
			\end{proof}
			%Theorem \ref{thm:unit_Ag} coincides with \cite[Theorem 1.1]{GZ} over $\X_g^\circ$. 
            After pulling back Theorem \ref{thm:unit_Ag} along the unit section, we recover \cite[Theorem 1.1]{EvG}. The formula of the unit section on $\X'_g$ will be discussed in the forthcoming paper \cite{BFLM}.
			
			For more degenerate families $\pi:\Xbar_g\to\Abar_g$, such as Alexeev's compactification \cite{Alex,Olsson}, the singularities of $\Xbar_g\times_{\Abar_g}\Xbar_g$ get more complicated, and $\delta$-regularity breaks down. It remains an interesting question whether we can generalize our construction to such abelian fibrations.
			
			\bibliographystyle{plain}
			\bibliography{main}
			
			\vspace{8pt}
			
			\noindent Department of Mathematics, University of Michigan \\
			\noindent younghan@umich.edu
			
			\vspace{8pt}
			
			\noindent Department of Mathematics, University of Rome ``La Sapienza" \\
			\noindent samouil.molcho@uniroma1.it
			
			\vspace{8pt}
			
			\noindent Department of Mathematics, University of Michigan\\
			\noindent pixton@umich.edu
			
		\end{document}